\documentclass[11pt,reqno]{amsart}

\usepackage{cite,epic,eepic,euscript,verbatim,amsmath,amssymb,amsthm,afterpage,float,bm,color} \usepackage{graphicx,epsfig}
\usepackage{comment}

\usepackage{enumerate}
\usepackage[pdftex,breaklinks,colorlinks, citecolor=blue,urlcolor=blue]{hyperref}
\usepackage{verbatim}
\numberwithin{equation}{section}

\newtheorem{theorem}{Theorem}[section]

\usepackage{array,multirow} 

\newcommand{\R}{\mathbb{R}}
\newcommand{\C}{\mathbb{C}}
\newcommand{\N}{\mathbb{N}}

\DeclareMathOperator{\re}{Re}
\DeclareMathOperator{\im}{Im}

\def \ds  {\displaystyle } 

\def \nn   {\vec{n}}

\def \nnls  {{\rm{NLS$_{\R^2}$ }}}

\def \Nls {{\rm{(NLS$_\Omega$) }}}
\def \NNls {{\rm{NLS}}$_{\Omega}$ }

\def \uu {\bar{u}}
\def \th {\theta}

\title[2D  NLS outside an obstacle]
{Interaction with an obstacle in the 2d focusing\\ nonlinear Schr\"odinger equation}

\author[O. Landoulsi]{Oussama Landoulsi}
\address{LAGA, UMR 7539, Institut Galil\'ee, Universit\'e Sorbonne Paris Nord}
\email{landoulsi@math.univ-paris13.fr}

\author[S. Roudenko]{Svetlana Roudenko}
\address{Department of Mathematics \&  Statistics, Florida International University, Miami, FL 33199, USA}
\email{sroudenko@fiu.edu}

\author[K. Yang]{Kai Yang}
\address{Department of Mathematics \&  Statistics, Florida International University, Miami, FL 33199, USA}
\email{yangk@fiu.edu}



\begin{document}

\begin{abstract}
We present a numerical study of solutions to the $2d$  cubic and quintic focusing nonlinear Schr\"odinger equation in the exterior of a smooth, compact and strictly convex obstacle (a disk) with Dirichlet boundary condition. We first investigate the effect of the obstacle on the behavior of solutions traveling toward the obstacle at different angles and with different velocities directions. We introduce a new concept of weak and strong interactions of the solutions with the obstacle. Next, we study the existence of blow-up solutions depending on the type of the interaction and show how the presence of the obstacle changes the overall behavior of solutions (e.g., from blow-up to global existence), especially in the strong interaction case, as well as how it affects the shape 
of solutions compared to their initial data, (e.g., splitting into transmitted and reflected parts). We also investigate the influence of the size of the obstacle on the eventual existence of blow-up solutions in the strong interaction case in terms of the transmitted and the reflected parts of the mass. Moreover, we show that the sharp threshold for global existence vs. finite time blow-up solutions in the mass critical case in the presence of the obstacle is the same as the one given by Weinstein for {\rm{NLS}} in the whole Euclidean space $\R^d$. Finally, we construct new Wall-type initial data that blows up in finite time after a strong interaction with an obstacle and having a very distinct dynamics compared with all other blow-up scenarios and dynamics for the {\rm{NLS}} in the whole Euclidean space $\R^d$. 

\end{abstract}

\maketitle

{
\hypersetup{linkcolor=black}
\tableofcontents
}

\section{Introduction}

We consider the $2d$ focusing nonlinear Schr\"odinger equation (NLS) outside of a smooth, compact and strictly convex obstacle with Dirichlet boundary conditions:
\begin{equation}
\tag{NLS$_{\Omega}$} 
 \begin{cases}
i\partial_tu+\Delta_{\Omega} u = -|u|^{p-1}u  \qquad   (t,x)\in \R \times\Omega, \\
 u(t_0,x) =u_0(x)   \qquad \qquad  \; \,  \quad \forall x \in \Omega ,  \\
u(t,x)=0   \qquad \qquad \qquad \qquad     (t,x)\in \R \times \partial\Omega ,
\end{cases}
 \end{equation}  
 where $t_0 \in \R$ is the initial time, $\Omega$ is an exterior domain in $\R^2$ and $\Delta_{\Omega}$ is the Dirichlet Laplace operator defined by $\Delta_{\Omega}:= \partial_{x}^2+   \partial_{y}^2, \, (x,y) \in \R^2$. 
Here, $u$ is a complex-valued function,
\begin{align*}
 u:  \R &\times  \Omega \longrightarrow \C \\
 	(t &, x)  \longmapsto  u(t,x) .
 \end{align*}
 We consider $u_0 \in H^1_0(\Omega),$ where the Sobolev space $H^1_0(\Omega)$ is the set of functions in $H^1(\Omega)$ that satisfy Dirichlet boundary conditions, i.e., $u=0$ on $\partial \Omega$. \\
     
 The {\rm{NLS}}$_{\Omega}$ equation is locally well-posed in $H^1_0(\Omega)$ in dimension $d=2,$ see \cite{BuGeTz04a}, \cite{BlairSmithSogge2012} for a non-trapping obstacle and \cite{Ivanovici10}, \cite{Ivanovici07} for a strictly convex obstacle. The solution $u$ can be extended to a maximal time interval $I= (-T_{-},T_{+})$ of existence and the following alternative holds: \\
 
\qquad either $T_+= \infty$ (respectively, $T_{-}=\infty $), or $T_+< \infty$ (respectively, $T_{-}<\infty) $ with 
$$\lim\limits_{t \to T_{+} } \left\|u(t, \cdot) \right\|_{H^1_0(\Omega)}=\infty \quad \bigg(\text{respectively, } \lim\limits_{t \to  T_{-}} \left\|u(t, \cdot) \right\|_{H^1_0(\Omega)}=\infty \bigg).  $$  

During their lifespans, solutions to the nonlinear Schr\"odinger equation outside an obstacle conserve both mass and energy: 
 \begin{align}
 \label{mass-consev}
 M_{\Omega}[u(t)]&:= \int_{\Omega} |u(t,x)|^2 dx = M_{\Omega}[u_0], \\
  \label{energy-consev}
 E_{\Omega}[u(t)]&:= \int_{\Omega} \left| \nabla u (t,x)\right|^2  \,dx - \frac{1}{p+1}\int_{\Omega} |u(t,x)|^{p+1} \, dx =E_{\Omega}[u_0].  
 \end{align}
 
 Unlike the nonlinear Schr\"odinger equation {\rm{NLS}$_{\R^d}$}  posed on the whole Euclidean space $\R^d,$ the \NNls equation does not preserve the momentum $P_{\Omega}[u]= \im \int_{\Omega} \bar{u}(t,x) \nabla u(t,x) dx, $ since the derivative of the momentum $P_{\Omega}$ with respect to the time variable is equal to a non-zero boundary term. \\ 
 
 Furthermore, the \nnls equation, posed on the whole Euclidean space $\R^2,$ is invariant under the scaling transformation, 
 that is, if $u(t,x)$ is a solution to the \nnls equation, then 
$ \lambda^{\frac{2}{p-1}}u(\lambda x,\lambda^2 t)$ is also a solution for $ \lambda >0. $ This scaling identifies the critical Sobolev space $\dot{H}^{s_c}_{x}$, where the critical regularity $s_c$ is given by $$s_c:= \frac{p-3}{p-1}.$$ The equation, when $ s_c=0,$ is referred to as the mass-critical (or the $L^2$-critical), and when $0<s_c<1,$ is called the mass-supercritical (or $L^2$-supercritical) and energy-subcritical (or $H^1$-subcritical). Throughout this paper, we consider the $2d$ cubic ($p=3$) and quintic ($p=5$) \NNls equations. Since the presence of the obstacle does not change the intrinsic dimensionality of the problem, we may regard the cubic \NNls equation as being the mass-critical equation and the quintic one as the mass-supercritical and  energy-subcritical (or intercritical) equation.  \\

The focusing {\rm{NLS}} equation, posed on the whole space, admits soliton solutions that are periodic in time, that is,  $u(t,x)=e^{i t \omega  }Q_{\omega}(x),$ where $\omega>0$ and $Q_{\omega} $ is an $H^1$ smooth solution of the nonlinear elliptic equation,   
 \begin{equation}
 \label{eq_Q}
-\Delta Q_{\omega} +  \omega \, Q_\omega = \left| Q_{\omega} \right|^{p-1} Q_{\omega}.
\end{equation} 
 
In this paper, we denote by $ Q_\omega $ the ground state solution, that is, the unique, positive, vanishing at infinity $H^1$ solution of \eqref{eq_Q}. The ground state solution turns out to be radial, smooth and exponentially decaying function (for $s_c<1$), see  \cite{CharlesVCoffman72}, \cite{GiNiNi81}, \cite{BeLi83a}, \cite{Kwong89}. Moreover, $Q_{\omega}$  characterized as the unique minimizer for the Gagliardo-Nirenberg inequality up to scaling, space translation and phase shift, see \cite{Kwong89}. For simplicity, we denote by $Q$ the ground state solution of \eqref{eq_Q}, when $\omega=1.$ \\ 

The {\rm{NLS}} equation, posed on the whole Euclidean space $\R^d$, also enjoys Galilean invariance: if $u(t,x)$ is a solution, then so is $u(t,x-vt) \, e^{i(\frac{x \cdot v}{2} -   \frac{ | v |^2   }{4} t ) }$, $v \in \R^d .$\\

Applying the Galilean transform to the solution $e^{i  t \omega } Q_{\omega}(x)$ of the {\rm{NLS}} on $\R^d$, we obtain a {\sl{soliton}} solution, moving on the line $x=t v$ with a velocity $v \in \R^d:$ 
\begin{equation}
\label{soliton}
u(t,x) = e^{i(\frac{1}{2} x \cdot v-\frac{1}{4} \left| v\right|^2 t + t \, \omega )}Q_{\omega}(x-t \,v).
\end{equation}

The soliton solution is a global solution of the focusing {\rm{NLS}} equation, but it is {\it not} a soliton solution for the \NNls equation: this soliton solution does not satisfy the Dirichlet boundary conditions. \\

In \cite{Ou19}, the first author constructed a solitary wave solution for the $3d$ focusing $L^2$-supercritical  \NNls equation for large $t,$ which behaves asymptotically as a soliton on the Euclidean space $\R^3$, traveling with a velocity $v,$ and moving away from the obstacle. Indeed, let $T_0>0,\; c_\omega>0$ and let $\Psi$ be a $C^{\infty}$ function such that $\Psi=0$ near the obstacle and $\Psi=1$ for $|x|\gg 1$, then 
 \begin{equation*}
	    \label{solution-u}
	\left\|  u(t,x)- e^{i(\frac{1}{2}(x \cdot v)-\frac{1}{4}|v|^2 t +t \, \omega)} Q_\omega(x-tv) \Psi(x) \right\|_{H^1_0(\Omega)} \leq  e^{- c_\omega |v| t }  \quad \forall (t,x) \in [T_0,+\infty) \times \Omega, 
	 \end{equation*}
	is a solution of the \rm{NLS$_\Omega$} equation. This solution is global in time, however, it does not scatter. For an arbitrary small velocity, this solution proves the optimality of the following threshold for the global existence and scattering given in \cite{KiVisnaZhang16} for the cubic \NNls equation, in dimension $d=3$: let $u_0 \in H^1_0(\Omega)$ satisfy  
\begin{align}
\label{ME-u<ME-u}
E_{\Omega}[u_0] M_{\Omega}[u_0] & < E_{\R^3}[Q] M_{\R^3}[Q], \\
\label{grad-u<grad-Q}
\left\| u_0 \right\|_{L^2(\Omega)} \left\| \nabla u_0 \right\|_{L^2(\Omega)} & <  \left\| Q \right\|_{L^2(\R^3)} \left\| \nabla Q \right\|_{L^2(\R^3)}.
\end{align}

Then the solution $u(t)$ scatters in $H^1_0(\Omega)$ in both time directions.  \\ 
This threshold was first proved for the $3d$ cubic NLS 
equation on the whole space $\R^3$ by the second author with Holmer in \cite{HoRo08} (in the radial setting) and with Duyckaerts and Holmer  in \cite{DuHoRo08} (nonradial case); further generalizations can be found in \cite{FaXiCa11}, \cite{Guevara14}. \\

In \cite{DuOuSvet20}, the first two authors with Duyckaerts studied the dynamics of the focusing $3d$ cubic \NNls equation in the exterior of a strictly convex obstacle at the mass-energy threshold, namely, when $$E_{\Omega}[u_0] M_{\Omega}[u_0] = E_{\R^3}[Q] M_{\R^3}[Q] , $$ with the initial mass-gradient bound on $u_0 \in H^1_0(\Omega)$, $$ \left\| u_0 \right\|_{L^2(\Omega)} \left\| \nabla u_0 \right\|_{L^2(\Omega)} < \left\| Q \right\|_{L^2(\R^3)} \left\| \nabla Q \right\|_{L^2(\R^3)},$$ where $Q$ is the ground state solution of \eqref{eq_Q}, with $\omega=1$. The same problem was studied in the whole Euclidean space by Duyckaerts and the second author in \cite{DuRo10} for the focusing cubic \rm{NLS} equation on $\R^3$. The dynamics of the \rm{NLS} equation on the whole Euclidean space is more involved. Indeed, the authors proved that if the initial datum $u_0 \in H^1(\R^3)$ satisfies the same mass-gradient condition as above, then the solution $u(t)$ scatters or $u(t)$ is a ``special solution" $Q^+$ , up to symmetries, that scatters in negative time and converges to the soliton $e^{it} Q$ (up to symmetries) in positive time. We showed in \cite{DuOuSvet20} that this special solution does not have an analogue for the problem in the exterior of an obstacle and  prove that such solutions are globally defined and scatter in the positive time direction. The existence of blow-up solutions at the mass-energy threshold for the \rm{NLS} equation on the whole space was also proved in \cite{DuRo10} and the behavior of  solutions is related to another special solution $Q^-$. It was proved that if $E_{\R^3}[u_0] M_{\R^3}[u_0] = E_{\R^3}[Q] M_{\R^3}[Q] $ and $ \left\| u_0 \right\|_{L^2(\R^3)} \left\| \nabla u_0 \right\|_{L^2(\R^3)} > \left\| Q \right\|_{L^2(\R^3)} \left\| \nabla Q \right\|_{L^2(\R^3)},$ then the solution $u(t)$ blows up in finite time or $u(t)$ is a special solution $Q^{-},$ up to symmetries. The existence of blow-up solutions at the mass-energy  threshold $E_{\Omega}[u_0] M_{\Omega}[u_0] = E_{\R^3}[Q] M_{\R^3}[Q] $ and $ \left\| u_0 \right\|_{L^2(\Omega)} \left\| \nabla u_0 \right\|_{L^2(\Omega)} > \left\| Q \right\|_{L^2(\R^3)} \left\| \nabla Q \right\|_{L^2(\R^3)},$ for the \NNls equation is currently an open question. \\

All results obtained for the \NNls equation are for the globally existing and scattering solutions, however, the existence of blow-up solutions has been an open question for some time. The classical proof by the convexity argument on the Euclidean space $\R^d$ fails in the exterior of an obstacle due to the appearance of the boundary terms with an unfavorable sign in the second derivative of the variance ${\rm{V}}(u(t))$, that is, if
\begin{equation}
      {\rm{V}}(u(t)):=\int_{\R^d} \left| x \right|^2 |u(t,x)|^2\, dx ,
\end{equation}
then
\begin{equation}
\label{Variance}
  \frac{1}{16}  \frac{d^2}{dt^2}{\rm{V}}(u(t))= E [u]- \frac{1}{2} \left( \frac{d}{2}  - \frac{d+2}{p+1} \right)\int_{\Omega} |u|^{p+1} \, dx - \frac{1}{4 } \int_{\partial \Omega} \left| \nabla u \right|^2 \, (x \cdot \nn) \, d\sigma(x),
\end{equation}
where $\nn$ is the unit outward normal vector. One can see that in the last term 
$$
x \cdot \nn \leq 0,  \,  \text{ for all\,  } x \in  \partial \Omega.
$$

Recently, the first author in \cite{Ouss20} (see also \cite{OL20}) proved the existence of blow-up solution to the \NNls equation in the exterior of a smooth, compact, convex obstacle. 
This was the first step in the study of the existence of blow-up solutions to the {\rm{NLS}$_{\Omega}$} equation. A new modified variance $\mathcal{V}(u(t)),$ which is bounded from below and is strictly concave for the solutions considered, was introduced
\begin{equation} 
\mathcal{V}(u(t)):=\int_{\Omega} \Big( d(x,\Omega^c)+10 \Big)^2 \, |u(t,x)|^2 \, dx ,
\end{equation}
where $d(x,\Omega^c)=|x|-R$ is the distance to the obstacle and $R$ is the radius of the obstacle (a ball in $\R^d$). \\
In \cite{Ouss20} (see also \cite{OL20}), it was shown that solutions with finite variance and negative energy blow up in finite time (for a ball and also any smooth, compact and convex obstacle). Furthermore, it was proved that finite variance solutions to the \NNls equation for $p\geq 1+\frac 4d,$ which satisfy \eqref{ME-u<ME-u}, $\left\| u_0 \right\|_{L^2( \Omega )} \left\| \nabla u_0 \right\|_{L^2(\Omega) } > \left\| Q  \right\|_{L^2{\left(\R^3 \right)}} \left\| \nabla Q \right\|_{L^2{\left(\R^3 \right)}}$ and a certain symmetry condition, will blow up in finite time.  \\

From the above review one notices that an obstacle does influence the behavior of solutions and while some properties and criteria remain robust and almost unchanged (except for the restriction of the whole space to the exterior domain problem $\Omega$),
other properties either get significantly modified or even more, become unclear in the obstacle setting. Further analytical investigations of the interaction between solutions and an obstacle are needed, though the presence of the obstacle breaks down quite a few invariance properties of the equation, creating additional difficulties for theoretical study. \\ 

The purpose of this paper is to investigate this question numerically and to gain further insights of the obstacle influence. We are specifically interested in how solitary wave-type data (e.g., as in \eqref{soliton}) interacts with an obstacle, depending on the distance and the angle to the obstacle as well as the size of the obstacle. \\

In our simulations, we distinguish two types of interaction between a solitary wave solution moving on the line $\vec{x}=t \vec{v}$ ($\vec{v}$ is the velocity vector) and the obstacle:  \\

\textbf{Strong interaction:} We call the interaction \textit{strong}, when a soliton-type solution is moving, towards the obstacle, in the same direction as the outward normal $\vec{n}$ vector of the obstacle, i.e., the velocity vector is collinear to the normal vector, $\vec{v}=\alpha \vec{n}, \; \alpha>0$, (e.g., see Figure \ref{directStrongInterac}). In this case, after the collision  or the shock, the solitary wave solution does not preserve the shape of the initial or the original soliton but the solution splits into several solitons or bumps, with a substantial amount of backward reflected waves. \\ 
 
\textbf{Weak integraction:} We call the interaction \textit{weak}, when the velocity vector of the moving soliton solution, towards the obstacle, is not pointed in the same direction as the outward normal vector, i.e., the solution hits the obstacle at an angle $0<\theta\leq \frac{\pi}{2}$ between the velocity vector and the outward normal vector, see Figure \ref{weakVelocityinteractiondomain}. In this case, after the interaction, the solitary wave solution is transmitted almost with the same shape and with backward reflected waves of insignificant size.  \\ 

The interaction between a solitary wave-type solution and an obstacle does not depend only on the direction of the velocity vector and the angle of the collision, it also depends on the initial distance between the solitary wave solution and the obstacle. For that, we also study the dependence on the distance.  
Throughout this paper, we denote by                                
\begin{equation}
\label{min-distance} 
 \ds d^{*}:=\min_{x \in \, \rm{supp\,}({u_0})} \mbox{dist}(x, \Omega^c),  
\end{equation} 
the distance between the obstacle $\Omega^c$ and the essential support of the initial data $u_0$ 
such that $u_0$ is well-defined, i.e., $u_0$ is smooth and satisfies Dirichlet boundary conditions. Note that, if we consider the initial condition with the distance $d>> d^{*},$ then the presence of the obstacle does not effectively influence the behavior of the solution (provided that $u_0$ has essentially a compact support, for instance, the Gaussian $u_0=Ae^{-x^2}$ will suffice for computational purposes). If we consider $u_0$ with a large mass such that $d>>d^*,$ then the solution will blow-up in finite time before it could reach the obstacle for all velocity directions, see Figures \ref{NointeractTranslVariab} and \ref{NointeraclVariabY0} for such scenarios, in such a case, there is no interaction between the obstacle and the solution. Moreover, the computation of the boundary value terms in \eqref{Variance} vanishes to $0$ when $d>>d^{*}$,  and one can see that the expression of the second derivative of the variance $V(t)$ in \eqref{Variance} is close to the corresponding value of the variance defined on the whole space $\R^2$. In this case, numerically, the soliton-type solution behaves as a solution posed on a computational domain without an obstacle, see Figures \ref{NLScritBlow-up} and \ref{NLScritBlow-up2}. \\

For the purpose of this work, and in order to study the influence and the interaction of a generic solution (a solitary wave-type solution) with an obstacle, we always consider the distance $d$ to be the minimal distance $d^*$ such that even a slight modification of the velocity direction or the translation parameters would produce at least a weak interaction.   \\

In this paper we present our numerical results about the behavior of solutions influenced by an obstacle in the \NNls equation outside of a ball or a disk or radius $r_{\star}$, in dimension $d=2.$ Our goal is to understand the \textit{interaction} between a solitary wave (for example, traveling with a velocity $v$) and the obstacle, as well as the \textit{influence} of the obstacle on the nonlinear dynamics of the \NNls equation. We also study the existence of blow-up solutions to the \NNls equation, in dimension $d=2,$ in particular, we investigate the influence of the obstacle on the \textit{behavior} of finite time blow-up solutions, and its dependence on different types of interaction, which is affected by the direction of the velocity $v$ and the angle at the collision. According to our numerical simulations, the solitary wave amplitudes decrease at the collision or at any interaction (even small) between the soliton and the obstacle. This could be explained by the appearance of reflection, or reflection waves, due to the Dirichlet boundary conditions at the obstacle. After the collision, our numerical results show that, if there is a weak or small interaction, then the solitary wave is transmitted almost completely with little or insignificant backward reflection. If there is a strong interaction, then the solution does not typically preserve the shape of the original solitary wave. First, a single bump will split into two bumps with some substantial backward reflection. After that, the two bumps will start to merge together with a creation of a third bump in the middle and then all of that will continue as a sum of several solitary waves. Typically, the first two bumps will have a dispersive behavior due to a creation of the third middle bump, which will be the main part of the (after-interaction) solitary wave. 

We also observe that the leading reflected wave has a dispersive behavior, which radiates away. The reflection phenomenon, the loss of the amplitude and the change in the shape of the solitary wave, make it very challenging to show (even numerically) the existence of blow-up solutions. Nevertheless, we confirm numerically the existence of blow-up solutions \textit{after the collision} for the $2d$ focusing \NNls equation in several cases of: (i) the weak interaction, depending on the velocity direction, see Section \ref{Weak interaction between soliton and obstacle}; (ii) the strong interaction, depending on  the radius $r_{\star}$ of the obstacle. Moreover, we investigate the influence of the size of the obstacle on the behavior of the solutions in terms of the transmitted and reflected mass: (i) if the radius of the obstacle is very small (e.g., $r_{\star} \approx 0.1$), so that the interaction region of the solution with the obstacle is insignificant or negligible, then the solution is mostly transmitted with a small reflected part and it blows up in finite time after the interaction; (ii)  if the radius of the obstacle is large enough, so that the interaction region is relevant and it is larger than the contour of the solution (e.g., $r_{\star} \approx 2$), then there is almost no transmission of the solution and it blows up in finite time at the boundary of the obstacle, see Section~\ref{sec-Obstacle size dependance}. Furthermore, we construct new \textit{Wall-type} initial data (with a single maximum peak bump), where after a strong interaction with the obstacle, 
the corresponding solutions blow up in finite time in two different locations (this happens even for a larger obstacle size). These specific solutions to the \NNls equation have a very distinct dynamics compared with all other blow-up scenarios and dynamics we observed, since they can produce a blow up not at a single location, see Section \ref{sec-special-solution}; there is also interesting dynamics for certain parameters when the blow-up happens at a single  point after the strong interaction (Figure \ref{Fig-supecial-critic-A0-1.5-r0-10}).  \\

In addition, we study the \textit{sharp threshold} for global existence vs. finite time blow-up solutions (in the $2d$ focusing mass-critical \NNls equation), see Section \ref{Perturbations of the soliton}. This threshold was first obtained by Weinstein in \cite{Weinstein82} for the focusing mass-critical {\rm{NLS}} equation in the whole Euclidean space $\R^d$ (for example, $2d$ cubic {\rm{NLS}}).
He showed a sharp threshold for the global existence using the Gagliardo-Nirenberg inequality combined with the energy conservation, 
$$ \left\| \nabla u \right\|_{L^2}^2 \leq \left(  1- \frac{\left\| u \right\|_{L^2}^2 }{\left\| Q \right\|_{L^2}^2} \right)^{-1} E[u],
$$
which implies that (i) if $ \left\| u_0\right\|_{L^2} < \left\| Q \right\|_{L^2}, $ then an $H^1$ solution exists globally in time and (ii) if $   \left\| u_0\right\|_{L^2} \geq  \left\| Q \right\|_{L^2}, $
then the solution may blow up in finite time. Recently, Dodson proved in \cite{Dodson15} that initial data $u_0 \in L^2(\R^d)$ with $ \left\| u_0\right\|_{L^2} < \left\| Q \right\|_{L^2}$ generates a corresponding solution that is global and scatters in~$L^2(\R^d).$ To confirm this threshold in our setting of the NLS with an obstacle, we consider initial data $u_0$ as a small perturbation of a shifted soliton, $u_0 = A \, Q(x-x_0)$, with either $A<1$ (e.g. $A=0.9$) for global existence or $A>1$ (e.g., $A=1.1$) for a finite time of existence.    \\

 In physics, the study of the reflected, diffracted or scattered (mechanical, electromagnetic or gravitational) waves after encountering an object, an obstacle or a body, is related to the study of boundary value problems. These problems are usually described mathematically as an exterior domain or obstacle problem for the wave-type equations with Dirichlet or Neuman boundary conditions. The study of the wave-type equations in the exterior of an obstacle started in the late 1950s and early 1960s and until now the understanding of the dynamics of the evolution equations on exterior domains has been a widely open area for investigations. Let us mention, some relevant works on the wave-type equation in an exterior domain. H.~W.~Calvin and Morawetz have studied the local-energy decay of the solutions to the linear wave equation in an exterior of a sphere
and star-shape obstacles, with Dirichlet and Neuman boundary conditions, see \cite{Wilcox59} and \cite{Morawetz61}, \cite{Morawetz62}. For later works see \cite{LaxMorawPhillip62}, \cite{LaxMorawPhillip63}. Different results were obtained for almost-star shape, non-trapping and moving obstacles, see \cite{Ivri69},
\cite{MoraTalstonStrauss77}, \cite{MoraRalstonStraussCorec78} and  \cite{CooperStrauss76}. In that period of time, the authors considered a classical solutions with $C^2$ initial data.
In $2004,$ the Cauchy theory in $H^1_0(\Omega)$ for the \NNls equation was initiated by Burq, G\'erard and Tzvetkov in  \cite{BuGeTz04a}, for a non-trapping obstacle. After that, the well-posedness problem for the \NNls equation was investigated by others, see for example, \cite{AnRa08}, \cite{MR2683754}, \cite{PlVe09}, \cite{Ivanovici10}, \cite{BlairSmithSogge2012}. In \cite{Ou19}, the first author proved the local well-posedness for the $3d$ \NNls equation in the critical Sobolev space using the fractional chain rule in the exterior of a compact convex obstacle given in \cite{killip2015riesz}. \\   

This paper is organized as follows: in Section \ref{Numerical-Method} we present the numerical method that we design for this study. In Section \ref{NLS-space}, we show several numerical simulations of scattering and blow-up solutions for the focusing nonlinear Schr\"odinger equation on the whole Euclidean space $\R^2$ for later comparison. In Section \ref{Perturbations of the soliton}, we study the sharp threshold for global existence and blow-up solutions for the critical \NNls equation in terms of the ground state perturbations. In Section \ref{Dependence on the distance}, we fix the radius $r_{\star}$ of the obstacle (e.g., $r_{\star}=0.5$) and study the dependence of the interaction on the initial distance between the solution and the obstacle, depending on the velocity direction. 
In Sections \ref{Weak interaction between soliton and obstacle} and \ref{Strong interaction between soliton and obstacle}, we study the weak and strong interactions between the traveling solutions and the obstacle. In Section \ref{sec-Obstacle size dependance}, we investigate the influence of the size of the obstacle $r_{\star}$ on the behavior of solutions, especially in the strong interaction case. Finally, in Section \ref{sec-special-solution}, we study the existence of  blow-up solutions, with the new \textit{Wall-type} initial data, for a variety of large size obstacles and different initial amplitude of the data. We summarize our findings in conclusions' Section \ref{S:Conclusions}. In all our simulations we consider both the cubic (mass-critical) and quintic (mass-supercritical)  \NNls equations. \\

{\bf Acknowledgments.}
All authors would like to thank Thomas Duyckaerts for fruitful discussions on this problem. 
Most of the research on this project was done while O.L. was visiting
the Department of Mathematics and Statistics at Florida International University, Miami, FL,
during his PhD training. He thanks the department for hospitality and support. The initial numerical investigations started when K.Y. visited T. Duyckaerts at IHP and LAGA, Paris-13. S.R. and K.Y. were partially supported by the NSF grant DMS-1927258, and part of O.L.'s research visit to FIU was funded by the same grant DMS-1927258 (PI: Roudenko).

\section{Numerical approach}
\label{Numerical-Method} 
\subsection{The scheme and Initial data}
Various numerical methods are used in order to approximate the nonlinear Schr\"odinger equation ranging from the explicit and implicit schemes in time to the finite difference or Fourier pseudo-spectral methods in space. There are different methods for the time discretization, for example, the Crank-Nicolson scheme \cite{DelForPay81},  Runge-Kutte type \cite{AgDvAk97}, \cite{Ag93}, \cite{KoAGdDvA93}, symplectic and splitting type, \cite{SsJmvJg86}, \cite{SsJMCMP94} and \cite{WjacHbm86}, \cite{BessBbDs2002} and relaxation methods \cite{BesCris98} and \cite{Besse2004}.  \\ 

We use the well-known Crank-Nicolson scheme for the time discretization of the \NNls equation. The scheme is based on a time centering approximation $u^{n+\frac{1}{2}} \approx \frac{u^{n+1}+u^n}{2}$.  The Crank-Nicolson-type scheme is the 2nd order implicit method. On the plus side, this scheme preserves both the discrete mass and the discrete energy exactly during the time evolution. On the negative side, the schemes have to deal with solving the resulting {\it nonlinear} algebraic system, consequently, the Newton's iterative method \eqref{Newton-iteration} is used for solving the nonlinear system at each time step. \\ 

We consider exponentially decaying data only and we take a large enough computational domain to approximate the convex domain $\Omega$ containing an obstacle (of radius $r^{\star}$). To be specific, we consider the polar coordinates $(r,\theta)$ with $0<r_{\star}<r<R$ and $0 \leq \theta \leq 2\pi$ and use the following domain in our simulations: 
$$  
\Omega=\{ (r,\theta) \in \R^2 : \;  r_{\star} \leq  r \leq R , \; 0\leq \theta \leq 2\pi \}.   \\ 
$$

\begin{figure}[ht]
\centering
 \includegraphics[width=15cm,height=5cm]{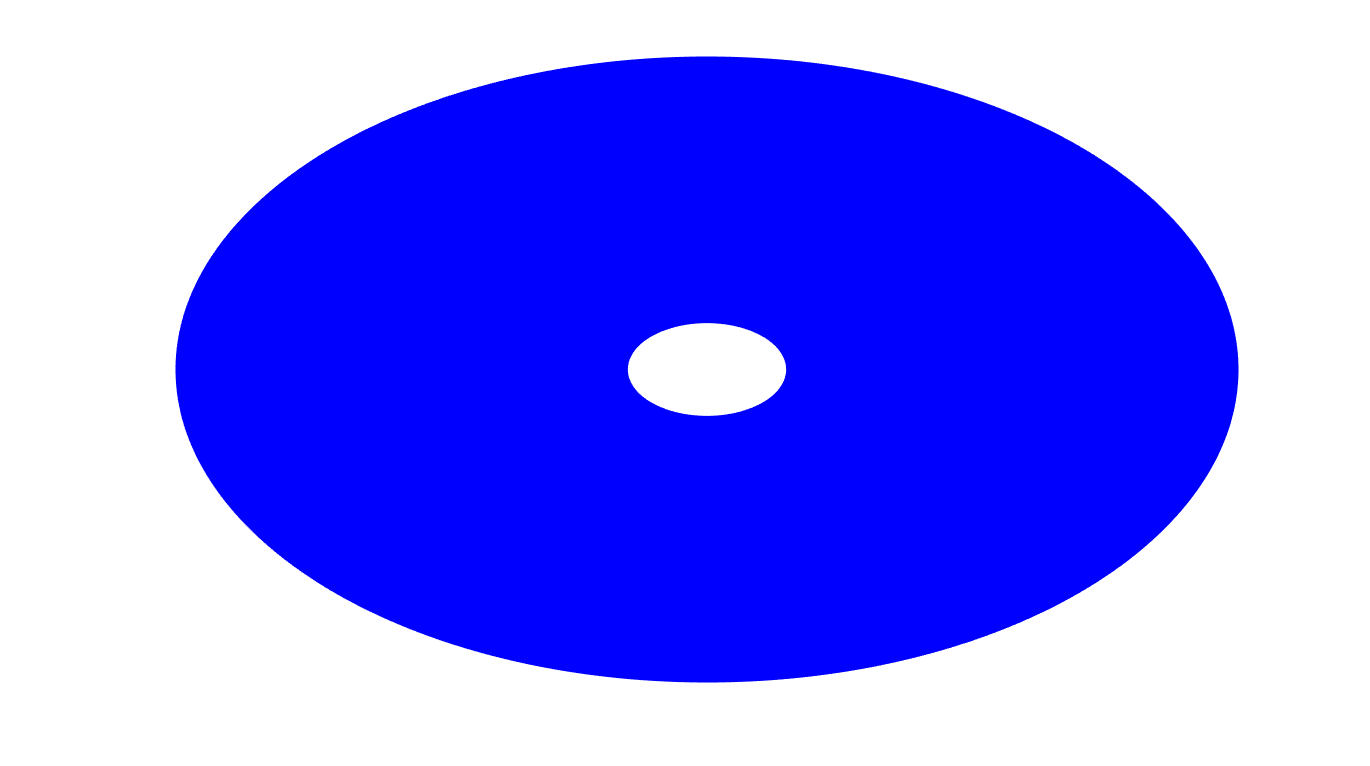}
 \caption{The computational domain $\Omega.$}
 \label{DomainOmega}
 \end{figure}

The obstacle is the white region or the disk $[0,r_{\star}] \times [0,2\pi]$, see Figure \ref{DomainOmega}. We note that the obstacle size is fixed as $r_{\star}=0.5$, unless stated otherwise, for example, a variable parameter or an investigating variable in Sections \ref{sec-Obstacle size dependance} and \ref{sec-special-solution}.  \\ 

It is straightforward to see that to approximate our model, we need to impose the Dirichlet boundary condition on the variable $r$, i.e.,  $u(r_{\star},\theta,t)=u(R,\theta,t)=0$, and the periodic boundary conditions on the variable $\theta$, i.e., $u(r,0,t)=u(r,2\pi,t)$. \\

{\bf{Initial data}\,}:
We consider the following initial data,
\begin{equation}
\label{E:data-1}
u_0(r,\theta)= A_0 \, F( r \cos \theta  - x_c, r \sin \theta -y_c) \, e^{i \, \frac{1}2 (v_x  \cdot r \cos \theta  +v_y \cdot r \sin \theta  )}, 
\end{equation}
where $A_0$ is the initial amplitude, $(x_c,y_c)$ is the translation, $v=(v_x,v_y)$ is the velocity vector,  
and $F$ is the profile of the solutions, which is typically taken as the Gaussian; for example, $F = e^{-r^2/2}$. To be precise, we consider
\begin{equation}
\label{u0NLS}
u_0(r ,\theta )=   A_0 \, e^{- \frac 12[ ( r \cos \theta -x_c)^2+(r \sin \theta -y_c)^2 ] } \; e^{i \, \frac 12 (v_x  \cdot r \cos \theta  + v_y \cdot r \sin \theta )}.
\end{equation}

Throughout  the paper, in most cases we use the same Gaussian initial data \eqref{u0NLS},  
(unless indicated otherwise as in Sections \ref{Perturbations of the soliton} and \ref{sec-special-solution}). The amplitude, translation and velocity parameters vary according to the specific examples considered.

We next describe our algorithm. We first consider the time discretization. Let $T_{max}$ be the existence time of the solution and $T_{\Delta t}$ be the computational time ($T_{\Delta t}<T_{max}$). We use $N$ points for the time discretization, thus, defining a time step $\Delta t=\frac{T_{\Delta t}}{N}$. We discretize the \NNls equation at times $ t_n=n \Delta t,\; n=0,.. , N,$ by  considering the semi-discretization in time $u^{n} \approx u(x,t_n)$ with $u^{0}:=u_0.$ This yields  the following time evolution: 
\begin{equation}
 \label{time-Scheme}
 i \frac{ u^{n+1} - u^{n}}{\Delta t } + \frac 12 \Delta u^{n+1} + \frac 12 \Delta u^{n} =- F(u^{n+1},u^{n}),
\end{equation}
where $F$ is the nonlinear term $|u|^{p-1}u$ approximated by 
 \begin{equation}
 \label{Nonlinear-term-approx}
 F(u^{n+1},u^n):=\frac{2}{p+1} \frac{ \left| u^{n+1} \right|^{p+1} -  \left| u^{n} \right|^{p+1} }{  \left| u^{n+1} \right|^{2} -  \left| u^{n} \right|^{2}    }     \frac{u^{n+1}+ u^{n}}{2} .
 \end{equation}
  
 Note that, $u^{n}$ is the known variable and for $n=0,$ $u^{0}=u_0$ is the given initial condition. We compute the evolution $u^n \longrightarrow u^{n+1}$ by solving the above system \eqref{time-Scheme}. For that, we use the Newton iteration to solve the nonlinear implicit system \eqref{time-Scheme}. We denote $u^{n+1}$ at the iteration $l$ by $u^{n+1,l}$ and assume $u^{n+1}=u^{n+1,\infty}$, which gives
\begin{equation}
\label{Newton-iteration}
\begin{cases}
u^{n+1,l+1}=u^{n+1,l}  - J^{-1} \cdot G(u^{n+1,l}), \\
u^{n+1,0}=1.001 \cdot u^{n},
\end{cases}
\end{equation} 
where $G(u^{n+1})=u^{n+1}-u^{n} + \frac{\Delta t }{2i}  \Delta u^{n+1}+ \frac{\Delta t }{2i}  \Delta u^{n}  + F(u^{n+1},u^{n})$ and $J$ is the Jacobian of $G$ . 

The stopping criterion for \eqref{Newton-iteration} is $\left\|   u^{n+1,l+1} -  u^{n+1,l}  \right\|_{L^{\infty}} <Tol$ for some small constant $Tol$. In our simulation, we take $Tol<10^{-13}$, which is close to the machine precision. In order to reach the blow-up time (or the closest time), we slightly decrease the tolerance according to the examples treated. 
    
We employ the polar transformation in space $x=r \, \cos \theta $ and $y=r\, \sin \theta $ to convert the problem into the polar coordinate setting. Thus, we write the Laplacian in polar coordinates as 
\begin{equation}
\label{laplacian}
 \Delta u (r,\theta)= \frac{\partial }{ \partial r } \left(  r \frac{\partial   }{ \partial r } u(r,\theta) \right)  + \frac{1}{r^2} \frac{\partial^2 }{\partial \theta^2} u(r,\theta), \qquad (r,\theta) \in \Omega.
\end{equation}
We then rewrite the \NNls equations for $  t\in (0,T), \; (r,\theta)\in \Omega, $ as
\begin{align*}
i \frac{\partial}{\partial t} u(t,r,\theta)+      \frac{\partial }{ \partial r } \left(  r \frac{\partial   }{ \partial r } u(t,r,\theta) \right)  + \frac{1}{r^2} \frac{\partial^2 }{\partial \theta^2} u(t,r,\theta)   = -|u(t,r,\theta)|^{p-1}u(t,r,\theta)  
 \end{align*}  
with the periodic boundary condition on $\theta$
\begin{equation}
\label{period-condition-data}
u(t,r,0)=u(t,r,2\pi),  \quad \forall t \in [0,T], \; \forall r \in [r_{\star},R], 
\end{equation}
and the Dirichlet boundary condition on $r$
\begin{equation}
\label{Dirichlet-bound-data}
u(t,r_{\star},\theta)=u(t,R,\theta)=0, \quad \forall t \in [0,T], \; \forall \theta  \in [0,2\pi]. 
\end{equation}

We use $N_r$ and $ N_{\theta}$ to denote the number of points for the space discretization, setting
$$   \Delta r=\frac{R-r_{\star}}{N_r} \qquad   \text{and} \qquad \Delta \theta = \frac{2\pi}{N_{\theta}}  .$$

We denote the full discretization by $u^{n}_{k,j} \approx u(r_k,\theta_j,t_n)$, 
$r_k=r_{\star} + k\Delta r$, $\theta_j=j\Delta \theta$,
for $n=0,...,N$, $k=0, \cdot \cdot \cdot , N_r$ and $j=0,\cdot \cdot \cdot,N_{\theta}.$  

We use the second order finite difference scheme in space, to approximate the \NNls equation:
\begin{equation}
\label{schemelast}
 i \frac{ u^{n+1}_{k,j} - u^{n}_{k,j} }{\Delta t } + \frac 12 \big[ D_{r} + D_{\theta}  \big] u^{n+1}_{k,j} + \frac 12\big[ D_{r} + D_{\theta}  \big]u^{n}_{k,j} =- F(u^{n+1}_{k,j},u^{n}_{k,j}),
\end{equation}
where  $F(u^{n+1}_{k,j},u^{n}_{k,j})$ is defined in \eqref{Nonlinear-term-approx} and
 \begin{align}
\label{D_r}
D_r u^{n}_{k,j}&= \frac{1}{\Delta r} \frac{1}{r_k} \left( r_{k+\frac 12}   \frac{u^{n}_{k+1,j} - u^{n}_{k,j}}{\Delta r}  - r_{k-\frac 12}   \frac{u^n_{k,j} - u^n_{k-1,j}}{\Delta r}    \right) , \\
\label{D_th}
D_{\theta} u^n_{k,j}&=\frac{1}{r_{k}^2} \frac{ u^n_{k,j+1}-2 u^n_{k,j}+u^n_{k,j-1}       }{ (\Delta \theta)^2}, \\ 
\label{Fu}
 F(u^{n+1}_{k,j},u^n_{i,j})&=\frac{2}{p+1} \frac{ \big| u^{n+1}_{k,j} \big|^{p+1} -  \big| u^{n}_{k,j} \big|^{p+1} }{  \big| u^{n+1}_{k,j} \big|^{2} -  \big| u^{n}_{k,j} \big|^{2}    }     \frac{u^{n+1}_{k,j}+ u^{n}_{k,j}}{2} ,
\end{align}
where $r_{k+\frac 12}=\frac{r_k+r_{k+1}}{2}$, and using the convention that $u^n_{0,j}=u^n_{N_{r},j}=0$, $u^n_{k,0}=u^n_{k,N_{\theta}}$ and $u^n_{k,1}=u^n_{k,N_{\theta}+1}$ from the boundary settings. 
 
To solve the above system \eqref{schemelast} with \eqref{D_r} and \eqref{D_th}, we consider the initial condition such that $u_0$ satisfies Dirichlet boundary conditions. We typically consider a shifted Gaussian as an initial condition, therefore, we define the translation parameters $(x_c,y_c)$ such that $u_0$ is smooth and vanishes to 0 near both the obstacle and the boundary of the computational domain.   

\subsection{Mass and Energy conservation}
The Crank-Nicolson scheme \eqref{time-Scheme} conserves the following discretized quantities: 
the discretized $L^2$-norm, or often referred to as the {\sl discrete} mass, and the discretized energy, called the {\sl discrete} energy, which are the discrete analogues of the mass and energy conservations in \eqref{mass-consev} and  \eqref{energy-consev}. 
  
If we consider the rectangular coordinates $(x,y)$, and the discretization of the Laplacian term $\Delta u$ by the standard five points stencil finite difference approximation (e.g., see \cite{DelForPay81}), i.e.,
\begin{align}\label{NLS scheme x y}
\Delta u_{k,j} \approx \frac{u_{k-1,j}+u_{k+1,j}+u_{k,j-1}+u_{k,j+1}-4u_{k,j}}{\Delta x \Delta y},
\end{align}
by assuming $\Delta x=x_{k+1}-x_{k}= \Delta y$, then the conservation of the discrete mass for the scheme \eqref{time-Scheme} with the Laplacian $\Delta u$ approximated by \eqref{NLS scheme x y} on the rectangular domain is given by 
\begin{equation} 
M[u^{n}]=    \sum_{k =0}^{N_{x}} \sum_{j =0}^{N_{y}}  |u^n_{k,j}|^2  \Delta x\, \Delta y =M[u^0], \quad \text{for} \;  n \geq 0,
\end{equation} 
which can be proved by multiplying the equation \eqref{time-Scheme} by $(\bar{u}_{k,j}^{n+1}+\bar{u}_{k,j}^{n}) \Delta x \Delta y$ and summing from $k=0$ to $N_x$, and  $j=0$ to $N_y$ for each $k,j$.

Similarly, the conservation of the discrete energy in rectangular coordinates $(x,y)$ is obtained by multiplying \eqref{time-Scheme} with 
$(\bar{u}_{k,j}^{n+1}-\bar{u}_{k,j}^{n}) \Delta x \Delta y$, summing from $k=0$ to $N_x$ and  $j=0$ to $N_y$ for each $k,j$ and taking the real part:

\begin{equation}
E[u^{n}]= \frac 12 \sum_{k =0}^{N_{x}-1} \sum_{j =0}^{N_{y}-1} \left(  \left| \frac{u^n_{k+1,j}-u^n_{k,j}}{\Delta x} \right|^{2}+ \left| \frac{u^n_{k,j+1}-u^n_{k,j}}{\Delta y} \right|^{2}  - \frac{2}{p+1} |u^n_{k,j}|^{p+1} \right) \Delta x \Delta y  =E[u_0] ,  \quad \text{for} \;  n \geq 0.
\end{equation}
For brevity, we omit the above standard proofs. 

In polar coordinates $(r, \theta),$ the scheme \eqref{schemelast} also conserves the discrete mass and energy exactly, similarly to the above. More specifically, we define the discrete mass at $t=t_n$ by
\begin{equation}
\label{mass-conserv}
M[u^{n}]= \sum_{k=0}^{N_{r}} \sum_{j =0}^{N_{\th}}  |u^n_{k,j}|^2 \; r_k  \;  \Delta r \; \Delta  \theta , \quad \text{for} \;  n \geq 0.   
\end{equation}
One can see that the definition \eqref{mass-conserv} is an analog to the mass in \eqref{mass-consev}. In the same spirit, we define the discrete energy as
\begin{multline}
\label{energy-conserv}
E[u^{n}]=  \frac 12 \sum_{k=0}^{N_r} \sum_{j=0}^{N_\theta} \left( \left|  \frac{ u^{n}_{k+1,j} -  u^{n}_{k,j} }{\Delta r} \right|^{2} r_{k+\frac 12} \,  \Delta r \; \Delta \theta + \frac 12 \frac{1}{r_k} \left| \frac{ u^{n}_{k,j+1} -  u^{n}_{k,j} }{\Delta \theta} \right|^{2} \Delta r \; \Delta \theta \right)  \\  -  \frac{1}{p+1}  \sum_{k=0}^{N_r} \sum_{j=0}^{N_\theta} |u^n_{k,j}|^{p+1} \; r_k  \;  \Delta r \; \Delta \theta ,  \quad \text{for} \;  n \geq 0,
\end{multline}
which is an analog of the energy conservation in \eqref{energy-consev}.  
 
We have the following theorem:
\begin{theorem}
The numerical scheme \eqref{schemelast} conserves the discrete mass \eqref{mass-conserv} and the discrete energy \eqref{energy-conserv} for all $n\in \N$, i.e., 
$$M[u^n]= M[u_0] \quad \quad \text{and} \quad \quad E[u^n]= E[u_0].$$
\end{theorem}
\begin{proof}
The proof of the mass conservation is similar to the proof in the case of the rectangular coordinates $(x,y), $  it suffices to multiply \eqref{schemelast} with $(\bar{u}^{n+1}_{k,j}+\bar{u}^{n}_{k,j}) \, r_k \Delta r \Delta \theta ,$ sum up over $k$ and $j$ from $0$ to $N_r$, $0$ to $N_{\theta}$, respectively, and then take the imaginary part.  

For the energy-conservation, the proof is slightly different than the one in the rectangular coordinates $(x,y)$, due to the space discretization of the Laplacian in \eqref{D_r}. 
For that, we write the scheme \eqref{schemelast} for $u_{k,j}^{n}$, using \eqref{D_r}, \eqref{D_th} and \eqref{Fu}, to obtain
\begin{multline}
\label{2d-NLS-Scheme}
  \underbrace{  i \frac{u_{k,j}^{n+1} - u_{k,j}^{n}}{ \Delta t}  }_{(I_1)_{k,j}}+ \underbrace{\frac 12 \frac{1}{r_k}  \frac{1}{\Delta r} 
    \left(  r_{k+ \frac 12} \frac{u_{k+1,j}^{n+1}-u_{k,j}^{n+1}}{\Delta r}  - r_{k - \frac 12} \frac{u_{k,j}^{n+1} - u_{k-1,j}^{n+1}}{\Delta r}  \right)  }_{(I_{2,1})_{k,j}}  \\ \underbrace{ + \frac 12 \frac{1}{r_k}  \frac{1}{\Delta r} 
    \left(  r_{k+ \frac 12} \frac{u_{k+1,j}^{n}-u_{k,j}^{n}}{\Delta r} - r_{k - \frac 12} \frac{u_{k,j}^{n} - u_{k-1,j}^{n}}{\Delta r}     \right)  }_{(I_{2,2})_{k,j}}  
   +   \underbrace{ \frac 12 \frac{1}{(r_k)^2} \frac{u_{k,j+1}^{n+1}-2u_{k,j}^{n+1}+u_{k,j-1}^{n+1}}{ (\Delta \th)^2} }_{(I_{3,1})_{k,j}} \\ + \underbrace{\frac 12 \frac{1}{(r_k)^2}\frac{u_{k,j+1}^{n}-2u_{k,j}^{n}+u_{k,j-1}^{n}}{ (\Delta \th)^2}  }_{(I_{3,2})_{k,j}}     = \underbrace{-\frac{1}{p+1}
    \frac{|u_{k,j}^{n+1}|^{p+1}- |u^{n}_{k,j}|^{p+1}}{|u^{n+1}_{k,j}|^2-|u^{n}_{k,j}|^2}(u_{k,j}^{n+1}+u_{k,j}^{n}) }_{(I_4)_{k,j}}
\end{multline}

Multiplying \eqref{2d-NLS-Scheme} with $(\bar{u}^{n+1}_{k,j}-\bar{u}^{n}_{k,j} ) \, r_k \Delta r \Delta \theta $, taking the real part and summing up over $k$ and $j$, yields 
\begin{equation}
\label{I_1}
\re \left[ \sum_{k=0}^{N_r} \sum_{j=0}^{N_{\th}}  (I_1)_{k,j} \times (\uu^{n+1}_{k,j} - \uu^n_{k,j})   r_k \Delta r \Delta \theta \right]=0.  
\end{equation}
  
\begin{multline}
\label{I_{2,1}+I_{2,2}}
\re \left[ \sum_{k=0}^{N_r}  \sum_{j=0}^{N_{\th}}  ( (I_{2,1})_{k,j}+(I_{2,2})_{k,j})  \times (\uu^{n+1}_{k,j} - \uu^n_{k,j})   r_k \Delta r \Delta \theta\right]= \sum_{k=0}^{N_r} \sum_{j=0}^{N_{\th}} 
- \frac 12 r_{k + \frac 12}{\left|   \frac{u^{n+1}_{k+1,j}-u^{n+1}_{k-1,j}}{\Delta r} \right|^2 }    \Delta r \Delta \theta  \\ +   \sum_{k=0}^{N_{r}} \sum_{j=0}^{N_{\th}} 
  \frac12 r_{k+\frac 12}  {\left|   \frac{u^{n}_{k+1,j}-u^{n}_{k,j}}{\Delta r} \right|^2 }   \Delta r \Delta \theta. 
\end{multline}
  
Using that $ \uu^{n}_{k,N_{\th}}=\uu^{n}_{k,0}, \; u^{n}_{k,N_{\th}+1}=u^{n}_{k,1}, \; u^{n}_{k,N_{\th}}=u^{n}_{k,0}$ and $u^{n}_{0,j}=u^{n}_{N_{r},j}=0,$ for all $n\in \N,$ we get
  
\begin{multline}
\label{I_{3,1}}
\re \left[\sum_{k=0}^{N_r} \sum_{j=0}^{N_{\th}}  (I_{3,1})_{k,j} \times (\uu^{n+1}_{k,j}-\uu^{n}_{k,j})  r_k \Delta r \Delta \theta \right]= - \frac 12 \frac{1}{r_k} \re \sum_{k=0}^{N_r}  \sum_{j=0}^{N_{\th}} \left| \frac{u^{n+1}_{k,j}-u^{n+1}_{k,j-1}}{\Delta \th} \right|^{2} \Delta r \Delta \th \\ -\frac 12 \frac{1}{r_k } \frac{1}{ (\Delta \th)^2} \re \bigg[\sum_{k=0}^{N_r}  \sum_{j=0}^{N_\th} {u_{k,j+1}^{n+1} \uu^n_{k,j}}
+ {2u_{k,j}^{n+1} \uu^n_{k,j}} - {u_{k,j-1}^{n+1} \uu^n_{k,j}}  
\bigg]\Delta r \Delta \th . 
\end{multline}

Similarly, we deduce 
\begin{multline}
\label{I_{3,2}}
\re \left[\sum_{k=0}^{N_r} \sum_{j=0}^{N_{\th}}  (I_{3,2})_{k,j} \times (\uu^{n+1}_{k,j}-\uu^{n}_{k,j})  r_k \Delta r \Delta \theta \right]= \frac 12 \frac{1}{r_k} \re \sum_{k=0}^{Nr} \sum_{j=0}^{N_{\th}} \left| \frac{u^{n}_{k,j}-u^{n}_{k,j-1}}{\Delta \th} \right|^{2} \Delta r \Delta \th \\ -\frac 12 \frac{1}{r_k} \frac{1}{  (\Delta \th)^2} \re \sum_{k=0}^{Nr} \sum_{j=0}^{N_\th} \bigg[ -\uu_{k,j+1}^{n+1} u^n_{k,j} -2\uu_{k,j}^{n+1} u^n_{k,j} + \uu_{k,j-1}^{n+1} u^n_{k,j}  \bigg]\Delta r \Delta \th .
\end{multline}

By \eqref{I_{3,2}} and \eqref{I_{3,1}}, we obtain

\begin{multline}
\label{I_{3,1}+I_{3,2}}
\re \left[ \sum_{k=0}^{N_r} \sum_{j=0}^{N_{\th} } (   (I_{3,1})_{k,j} + (I_{3,2})_{k ,j }  )    \times ( \uu^{n+1}_{k,j} - \uu^{n}_{k,j} )  r_k \Delta r \Delta \theta  \right]=- \frac 12 \frac{1}{r_k} \re \sum_{k=0}^{Nr}   \sum_{j=0}^{N_{\th}} \left| \frac{u^{n+1}_{k,j}-u^{n+1}_{k,j-1}}{\Delta \th} \right|^{2} \Delta r \Delta \th \\ +  \frac 12 \frac{1}{r_k} \re \sum_{k=0}^{Nr} \sum_{j=0}^{N_{\th}} \left| \frac{u^{n}_{k,j}-u^{n}_{k,j-1}}{\Delta \th} \right|^{2} \Delta r \Delta \th,
\end{multline}

\begin{equation}
\label{I_4}
\re \left[ \sum_{k=0}^{Nr} \sum_{j=0}^{N_{\th}} (I_4)_{k,j} \times (\uu^{n+1}_{k,j} - \uu^n_{k,j})  r_k \Delta r \Delta \theta \right]     = -\frac{1}{p+1} \sum_{k=0}^{Nr}  \sum_{j=0}^{N_{\th}} 
\left(  {|u^{n+1}_{k,j}|^{p+1}- |u^{n}_{k,j}|^{p+1}}   \right)  r_k \Delta r \Delta \theta.
\end{equation}

Summing up \eqref{I_1}, \eqref{I_{2,1}+I_{2,2}}, \eqref{I_{3,1}+I_{3,2}} and \eqref{I_4}, we finally arrive at  
\begin{align*}\label{energy-2d}
E[u^{n}]&=\frac12 \sum_{k=0}^{N_{r}}  \sum_{j=0}^{N_\th} 
  r_{k+\frac 12}  {\left|   \frac{u^{n}_{k+1,j}-u^{n}_{k,j}}{\Delta r} \right|^2 }  \Delta r \Delta \theta+ \frac12  \frac{1}{r_k}   {\left|   \frac{u^{n}_{k,j+1}-u^{n}_{k,j}}{\Delta \th} \right|^2 }  \Delta r \Delta \th \\ 
  & -\frac{1}{p+1}  \sum_{k=0}^{N_{r}} \sum_{j=0}^{N_{\th}}  |u^{n}_{k,j}|^{p+1} r_k \Delta r \Delta \theta \\ 
  & =  \frac 12 \sum_{k=0}^{N_{r} }\sum_{j=0}^{N_{\th}} r_{k + \frac 12} {\left|   \frac{u^{n+1}_{k+1}-u^{n+1}_{k}}{\Delta r} \right|^2 }  \Delta r \Delta \th+ \frac12  \frac{1}{r_k}  {\left|   \frac{u^{n+1}_{k,j+1}-u^{n+1}_{k,j}}{\Delta \th} \right|^2 }  \Delta r \Delta \theta       \\
   & - \frac{1}{p+1} \sum_{k=0}^{N_{r}} \sum_{j=0}^{N_{\th}}  |u^{n+1}_{k,j}|^{p+1}  r_k \Delta r \Delta \theta 
   =E[u^{n+1}].
\end{align*}
\end{proof}

We note that in our simulations the mass and the energy are well preserved: the relative mass-error and energy-error are bounded by at least $10^{-14}$, at the end of the simulations at $T=20$ with the time step $10^{-2},$ as shown in Figure~\ref{Mass-Energy-Error1}. The evolution of the relative mass and energy errors can be tracked by 
\begin{equation}\label{Error_M}
\max_{0 \leq m  \leq n}( M[u^m])-\min_{0 \leq m \leq n}(M[u^m])  \qquad \text{and} \qquad  \max_{0 \leq m \leq n}( E[u^m])-\min_{0 \leq m \leq n}(E[u^m]),
\end{equation}
or
\begin{equation}\label{ErrorM_1}
\frac{ M[u^{n}] -M[u_0]}{M[u_0]}  \qquad \text{and} \qquad     \frac{ E[u^{n}] -E[u_0]}{E[u_0]} .
\end{equation}

\begin{figure}[!h]
\includegraphics[width=8.5cm,height=4.5cm]{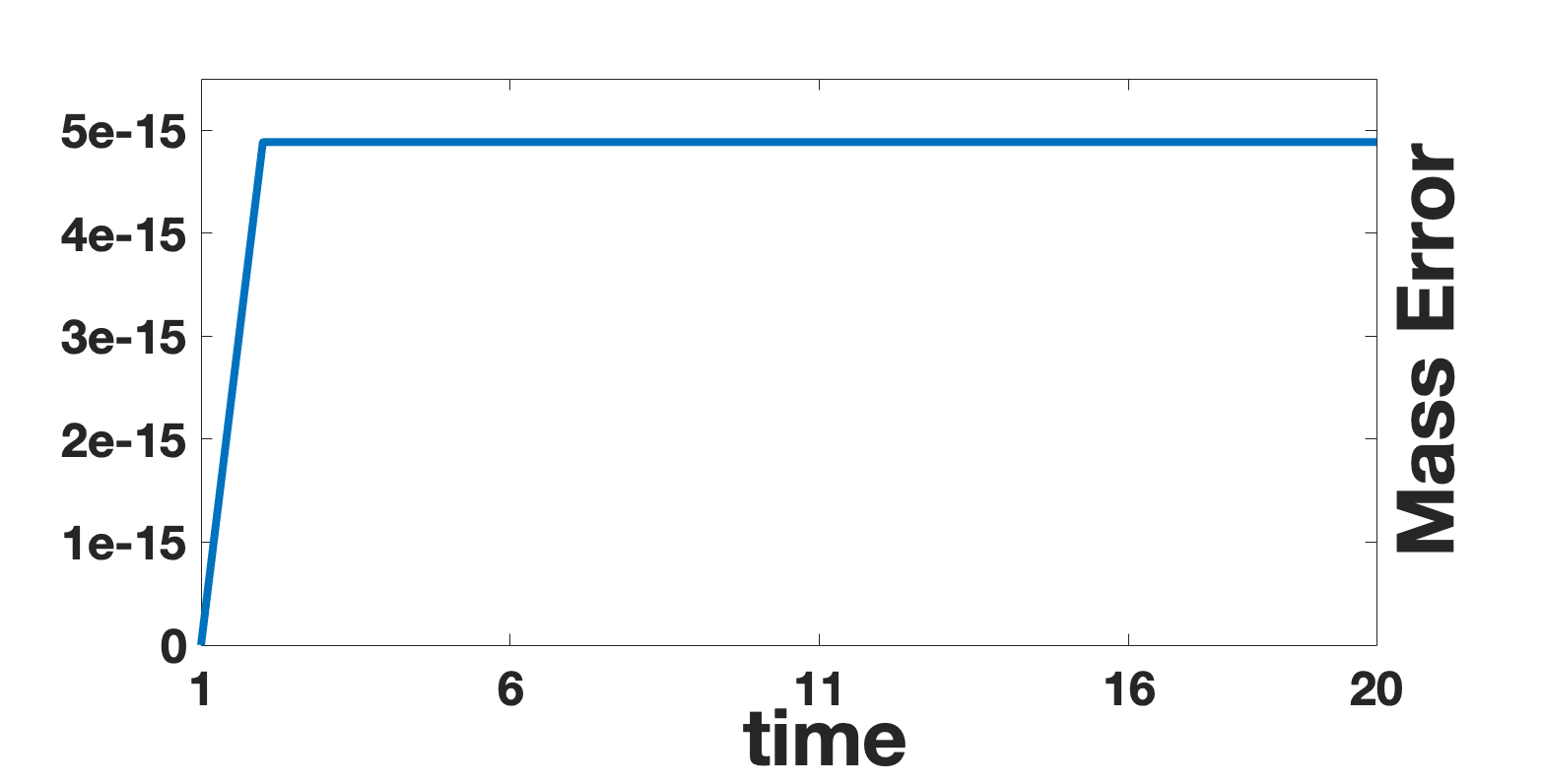}
\includegraphics[width=8.5cm,height=4.5cm]{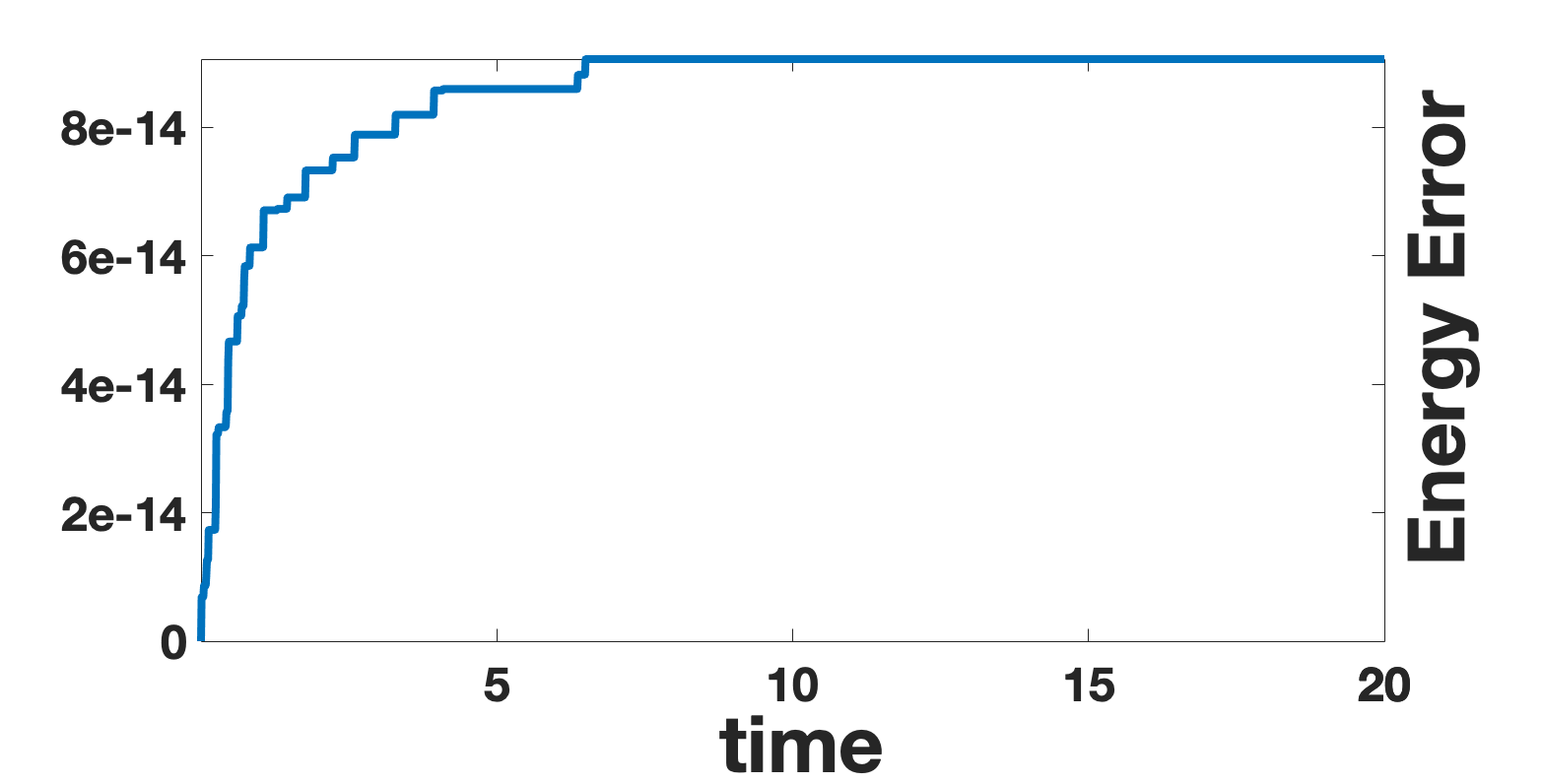}
\caption{Evolution of the relative mass and energy errors \eqref{Error_M} in the scheme \eqref{schemelast} for $d=2$ and $p=3.$}
\label{Mass-Energy-Error1}
\end{figure}
    
In Figure \ref{Mass-Energy-Error1}, we show that the error in \eqref{Error_M} for the discrete mass and energy are on the order of $10^{-15}$ and $10^{-14},$ correspondingly.  The errors of mass and energy in the case for $p=5$ are on the same order. 

\section{The NLS equation on the whole plane $\mathbb R^2$}
\label{NLS-space}
In this section, we show different numerical solutions of the focusing nonlinear Schr\"odinger equation on the whole Euclidean space $\R^2.$ We give several single bump solitary wave-type examples (and their evolution on the whole $\mathbb R^2$ plane) that we later consider in the following sections when investigating the influence of the obstacle on the behavior of the NLS$_{\Omega}$ solutions with the same initial data for comparison. 

Here, we consider a sufficiently large regular bounded rectangular domain without an obstacle and impose the Dirichlet boundary conditions on the boundary. In order to approximate the \nnls equation, the Laplacian term $\Delta u$ is discretized in the way of \eqref{NLS scheme x y}, and the temporal discretization is achieved by the Crank-Nicolson scheme in \eqref{time-Scheme}, with the Newton's iteration solving the resulting fully discretized nonlinear algebraic system. 

\subsection{The $L^2$-critical case: $2d$ cubic NLS on $\mathbb R^2$} 
We consider the focusing cubic \nnls equation with initial data as in \eqref{u0NLS}, moving on the line $y=x$ with the following parameters: 
\begin{equation} \label{para:NLS-critical}  
A_0=2.25, \quad v=(15,15) \quad \mbox{and} \quad (x_c,y_c)=(-4.5,-4.5).  
\end{equation}

\begin{figure}[!ht]
\includegraphics[width=5cm,height=5.6cm]{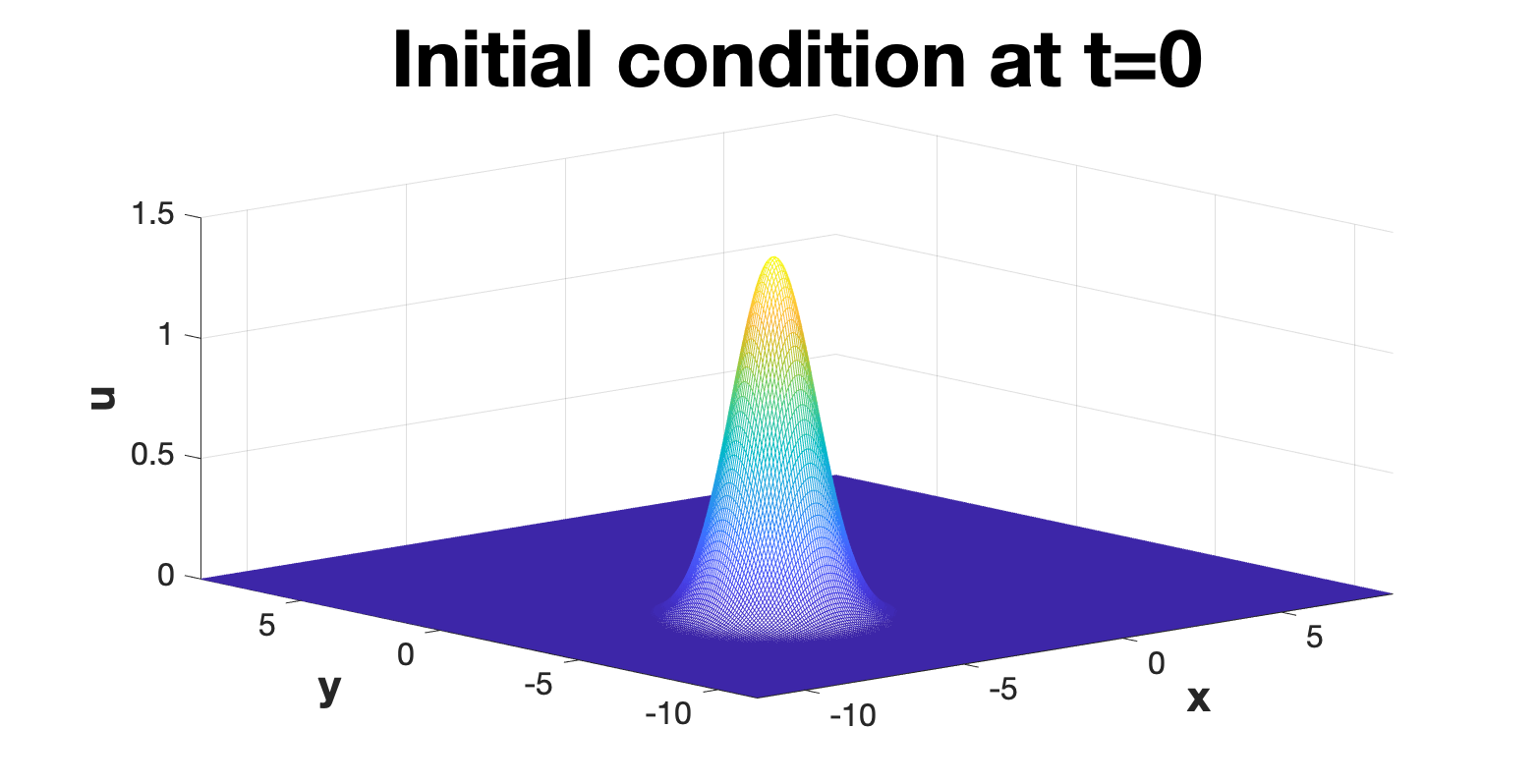} 
\includegraphics[width=5.1cm,height=5.6cm]{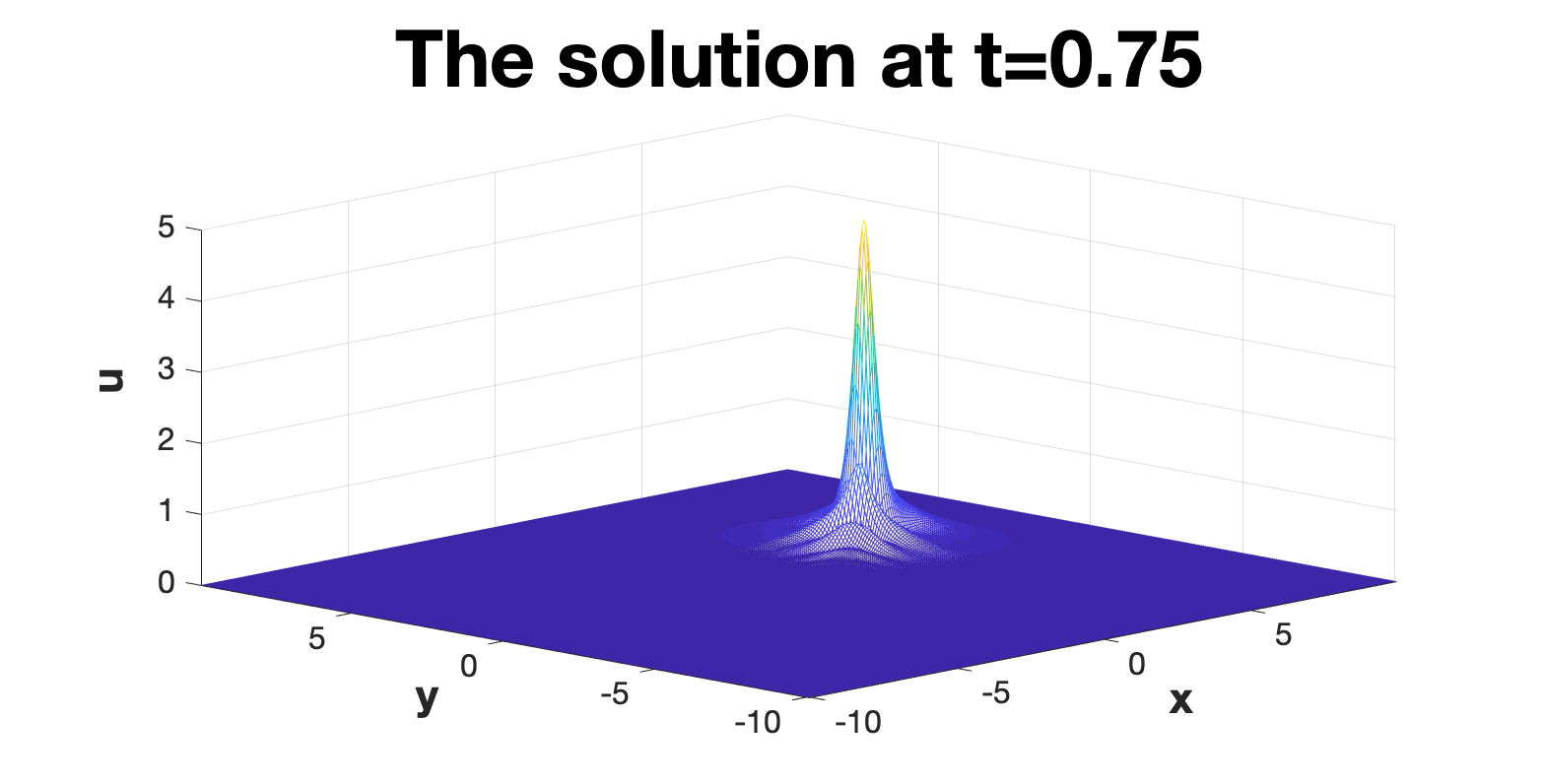}
\includegraphics[width=5cm,height=5cm]{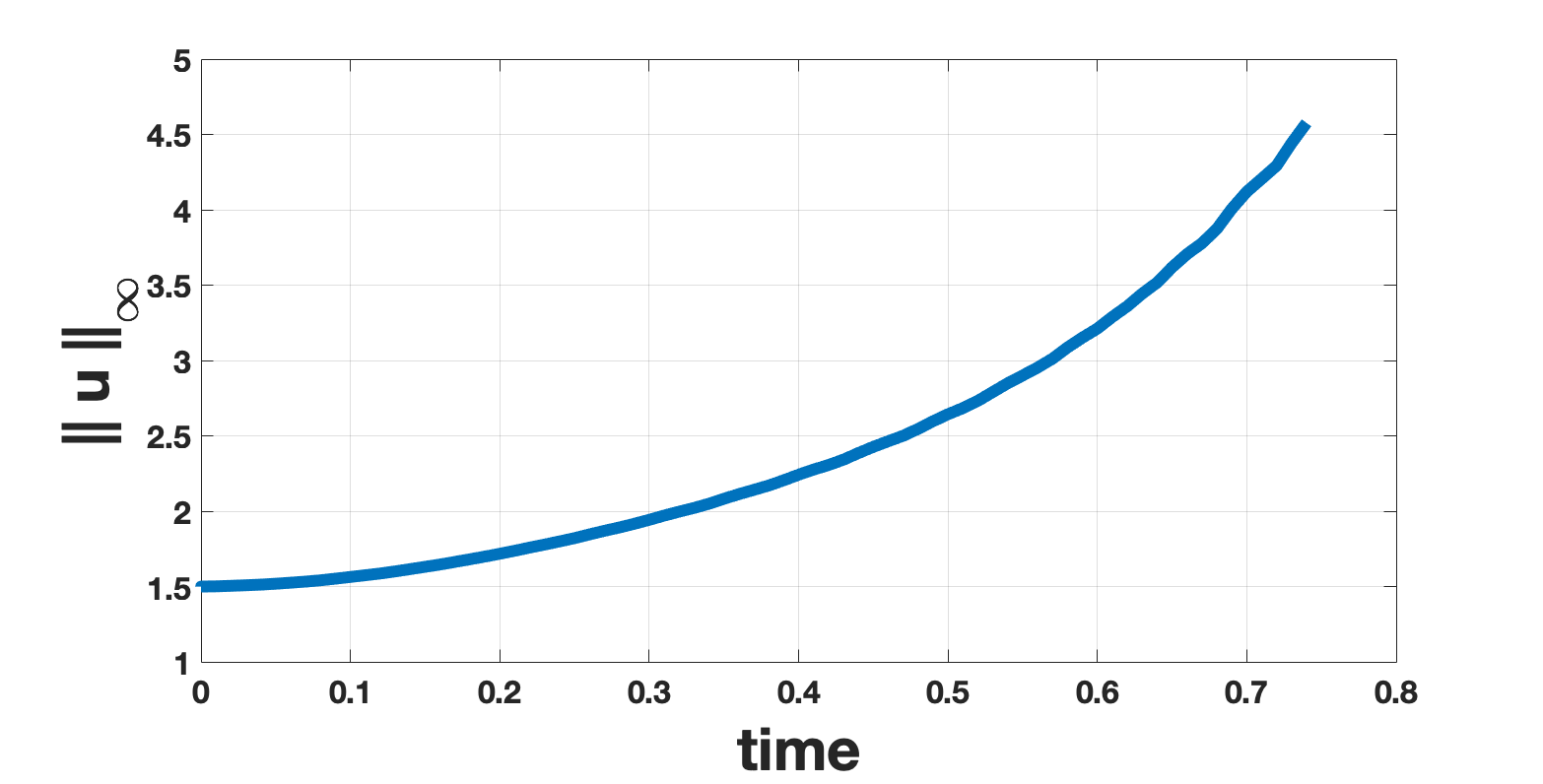} 
\caption {Solution to the $2d$ cubic NLS 
with initial data $u_0$ as in \eqref{u0NLS} with $A_0=2.25, v=(15,15)$ and $(x_c,y_c)=(-4.5,-4.5)$ at times $t=0$ and $t=0.75$ moving along the line $y=x$ (left and middle); time dependence of the $L^{\infty}$-norm of this solution (right).}
\label{NLScritBlow-up2}
\end{figure}

Snapshots of the initial data and the solution $u(t)$ to \nnls equation, which blows up in finite time (shortly after $t= 0.75$), is plotted in Figure \ref{NLScritBlow-up2} on the left ($t=0$) and middle ($t=0.75,$ the last computational time before the blow-up) subplots. The right subplot is the $L^{\infty}$-norm of the solution, which increases toward $\infty,$ indicating the blow-up.  We note that the same initial condition with exactly the same parameters \eqref{para:NLS-critical}, will have a different evolution when an obstacle is present, see Figures \ref{Solution-directionVelocityStrong}, \ref{Snapshots-Solution-directionVelocityStrong}, see also  Figures \ref{test3Critique-weak}, \ref{test4Critique-weak}, for similar initial condition with $v=(12,15)$ and $v=(15,12).$ Note that, in the presence of the obstacle, we also impose Dirichlet boundary condition for the initial data on the obstacle boundary, see \eqref{Dirichlet-bound-data}.


\subsection{The $L^2$-supercritical case: $2d$ quintic NLS on $\mathbb R^2$}
We consider the focusing quintic \nnls equation with initial data $u_0$ as in \eqref{u0NLS} with 
$$ A_0=1.25, \quad v=(15,0) \quad \mbox{and} \quad (x_c,y_c) =(-4.5,0).$$ 

When there is no obstacle present, the solution $u(t)$ to the \nnls equation, which is moving along the line $y=0$, blows up in finite time slightly after $t=0.64$, see Figure \ref{NLScritBlow-up} snapshots of the solution at $t=0$ and $t=0.64$ (initial time and the last computational time before the blow-up) on the left and middle subplots and the $L^\infty$-norm on the right subplot. We observe later that the solution to the \NNls equation with the same initial data, when the obstacle is present, does not blow-up in finite time and has a completely different dynamics. Furthermore, the solitary wave does not preserve its shape or profile after the collision or interaction with the obstacle (see Figure \ref{F:28}, \ref{F:29}, and Table \ref{T:2}). 
\begin{figure}[!ht]
\includegraphics[width=5cm,height=5.8cm]{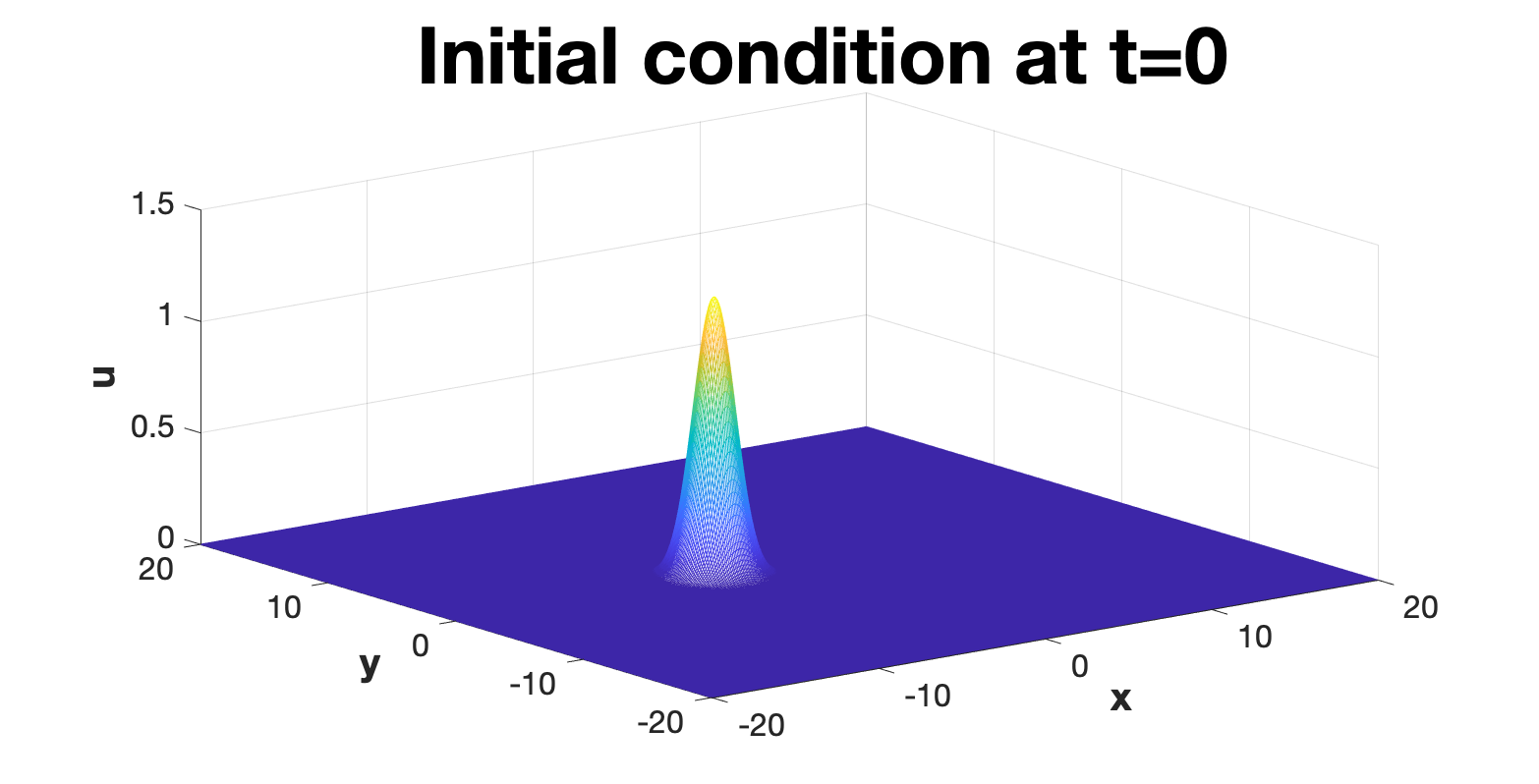} 
\includegraphics[width=5cm,height=5.8cm]{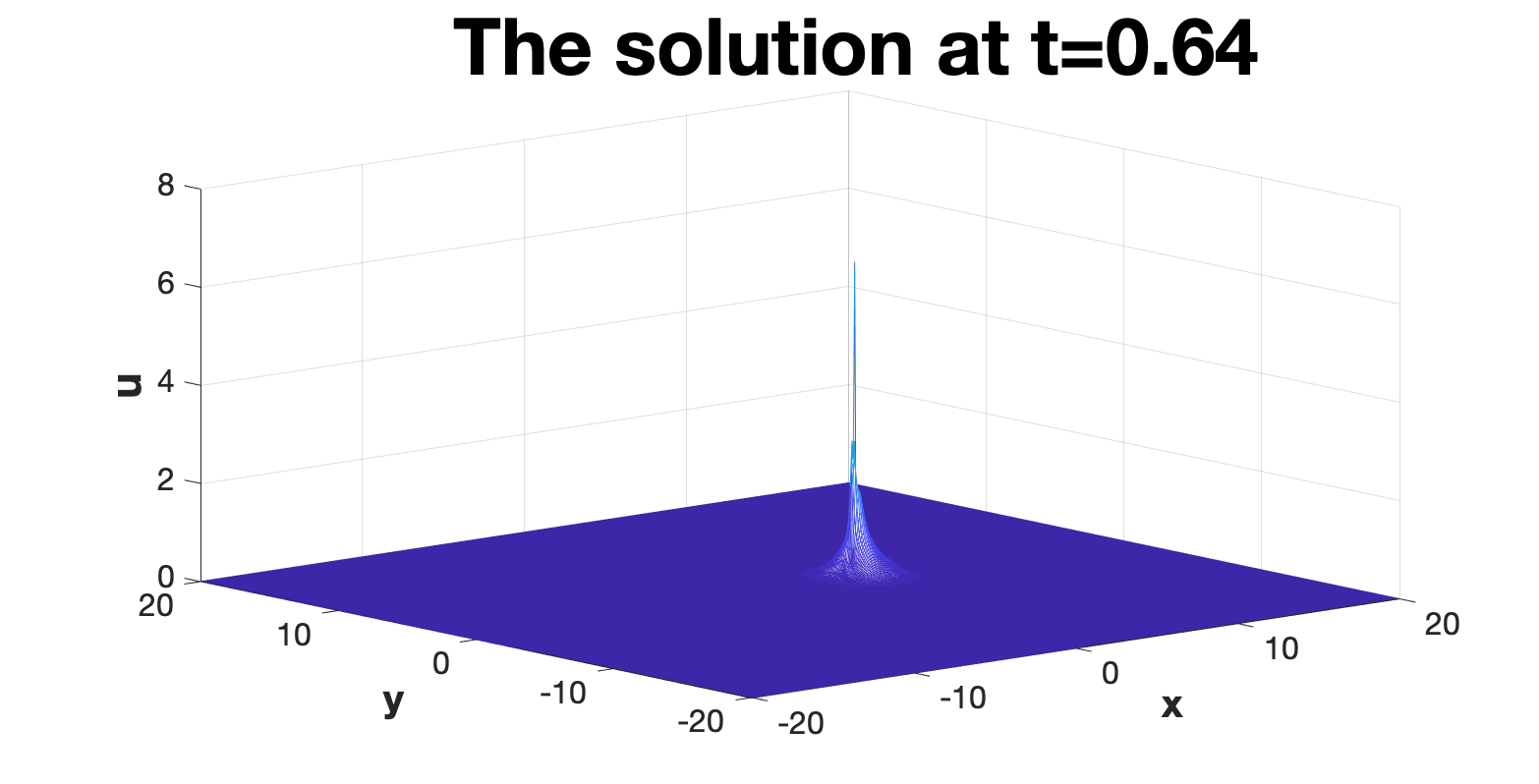}
\includegraphics[width=5cm,height=5cm]{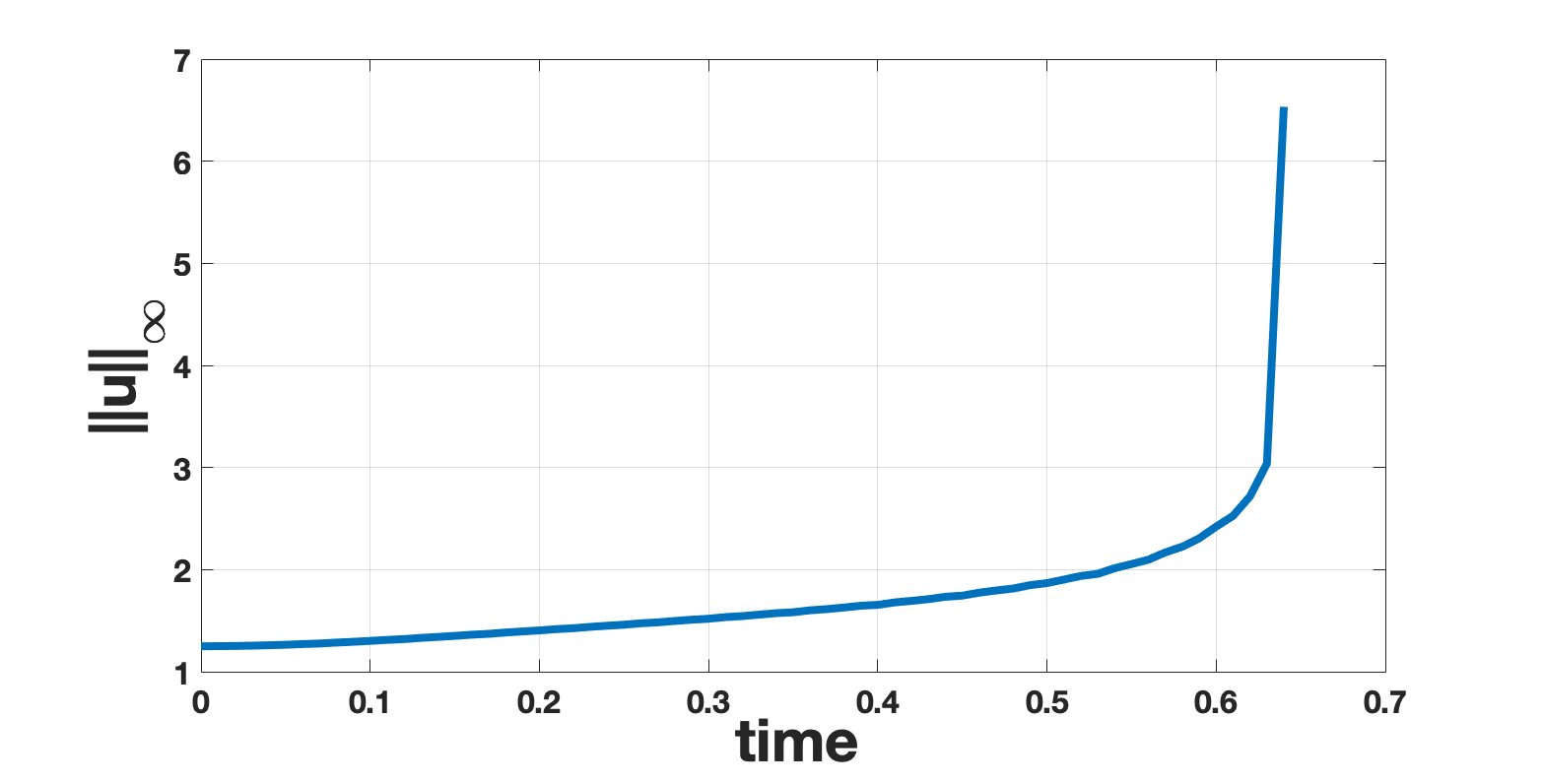} 
\caption {Solution to the $2d$ {quintic} NLS 
with initial data $u_0$ as in \eqref{u0NLS} with $A_0=1.25, v=(15,0)$ and $(x_c,y_c)=(-4.5,0)$ at times $t=0$ and $t=0.64$ moving on the line $y=0$ (left and middle); time dependence of the $L^{\infty}$-norm of this solution (right).}
\label{NLScritBlow-up} 
\end{figure}

Let us mention that, for a sufficiently large amplitude $A_0$, the solution for the \nnls equation, considered on the whole Euclidean space $\R^2,$ blows up in finite time (for more details about thresholds values of $A_0$ for scattering and blow-up in the focusing NLS, see \cite{HPR2010}). We also mention that if $A_0$ is sufficiently small (for example, in the $L^2$-critical case when $\|u_0\|_{L^2} < \|Q\|_{L^2}$, or in the $L^2$-supercritical case with initial data under the so-called mass-energy threshold, see for example \cite{HoRo08}), then the solution will scatter (in both cases with or without an obstacle), and while there maybe some small reflecting waves formed in the latter (obstacle) case, we omit this case here, and point out reflecting waves later.

\section{Perturbations of the soliton in the NLS with obstacle}
\label{Perturbations of the soliton}    
We now start investigating the $2d$ NLS on the domain with an obstacle, \NNls. In this section we consider a multiple of the ground state that is shifted by $(x_c,y_c)$ as in \eqref{E:data-1}, we refer to this evolution as a perturbed soliton, since the initial data have the form
$$
u_0(r,\theta)=\lambda \, Q(r \cos \theta  -x_c, r \sin \theta -y_c), \quad  \lambda \in \R,
$$
where $Q$ is the ground state solution to \eqref{eq_Q}. This ground state solution is obtained numerically via the Petviashvili's iteration, for example, see \cite{Petviashvili1976}, \cite{PS2004}, \cite{OSSS2016} or the work of the last two authors in \cite{RWY2021} or \cite{YRZ2018, YRZ2019}.

Perturbations of the soliton solution to the \NNls equation with a `large mass' initial condition, for example, $\lambda =1.1$, lead to blow-up solutions. For example, 
$$
u_0(r,\theta)=1.1 \, Q(r \cos \theta +4.5, r \sin \theta )
$$
blows up  at time $t=0.84$ with the diverging $L^{\infty}$ norm as shown in Figure \ref{Q1.1}. We use the Newton iteration to solve the implicit scheme \eqref{time-Scheme} and to reach the desired accuracy. It is quite challenging to approach the blow-up time while maintaining the convergence of the Newton iteration \eqref{Newton-iteration}. 
To address this issue, we run the scheme with a more refined mesh in order to maintain the convergence of \eqref{Newton-iteration}. This is not a simple task to perform in the $2d$ non-radial case and can become computationally prohibitive. However for our results of identifying the blow-up and the type of interaction, it suffices to investigate the $L^\infty$-norm (or the height) and label the solution as `the blow-up' if, for example, the amplitude becomes 3 times higher than the initial one (e.g., see the right graph in Figure \ref{Q1.1}).  
\begin{figure}[!ht]
\includegraphics[width=9.7cm,height=6.7cm]{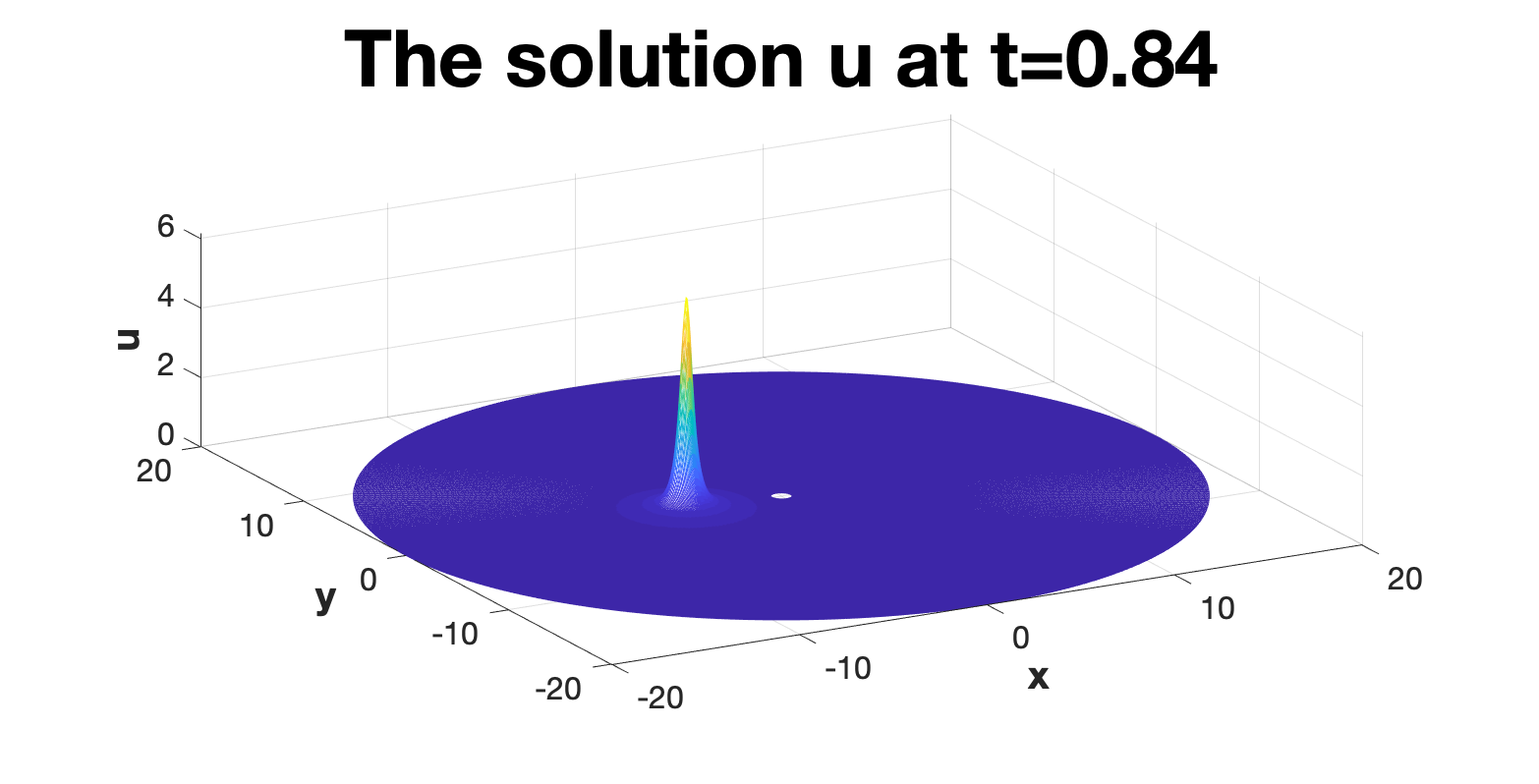}
\includegraphics[width=6.3cm,height=5.7cm]{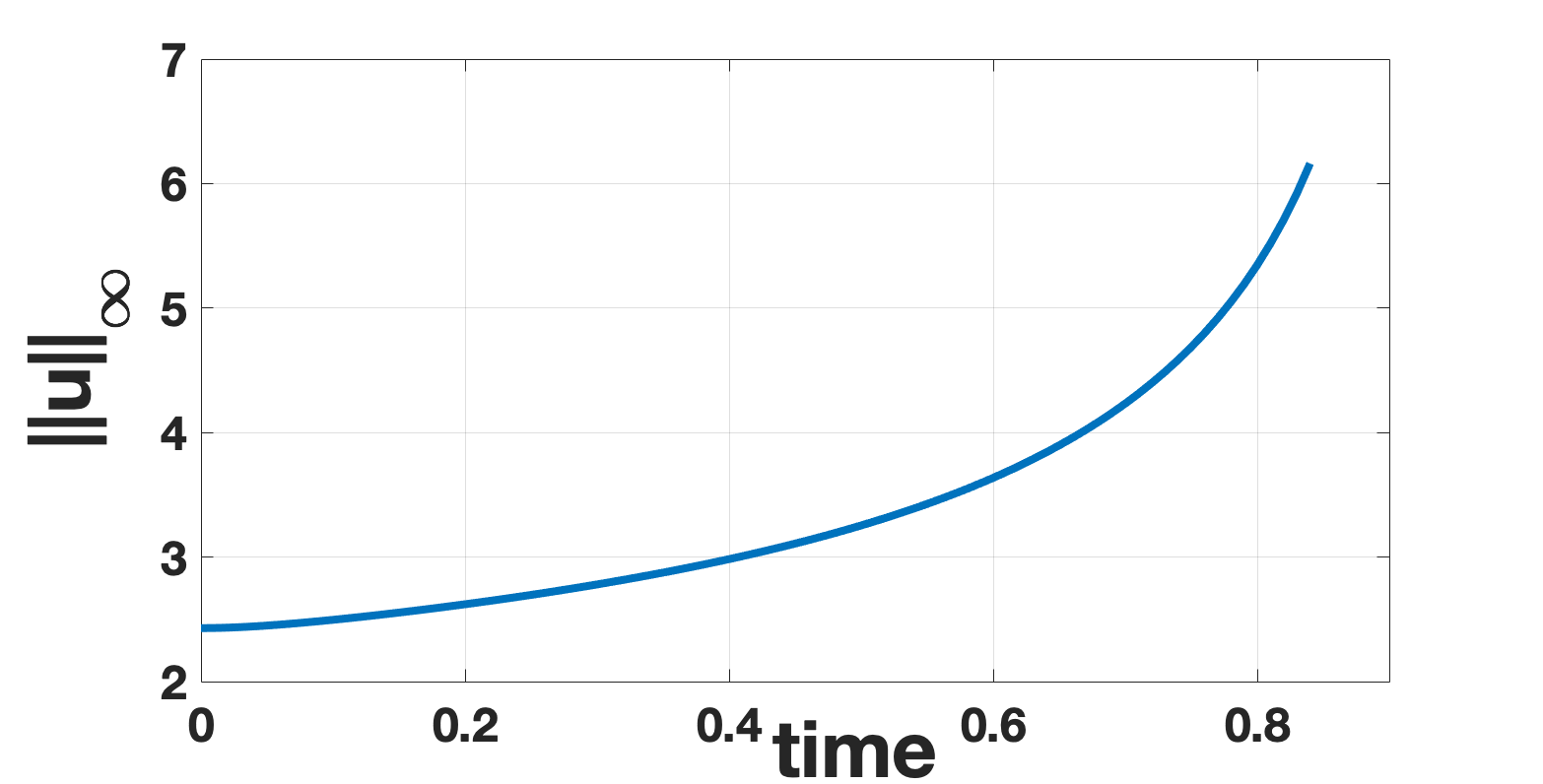}
\caption{Ground state solution to the $2d$ cubic \NNls equation with $u_0(x,y) = 1.1\, Q(x+4.5,y) $ at $t= 0.84$ (left) and its $L^{\infty} $ norm depending on time (right). }
\label{Q1.1}
\end{figure}

We next examine the initial condition of the perturbed soliton with the mass smaller than that of the ground state (i.e., $\|u_0\|_{L^2} < \|Q\|_{L^2}$), for example, 
$$
u_0(r,\theta)= 0.9 \,  Q(r\cos \theta +4.5, r \sin \theta ).
$$ 

\begin{figure}[!ht]
\includegraphics[width=9.7cm,height=6.7cm]{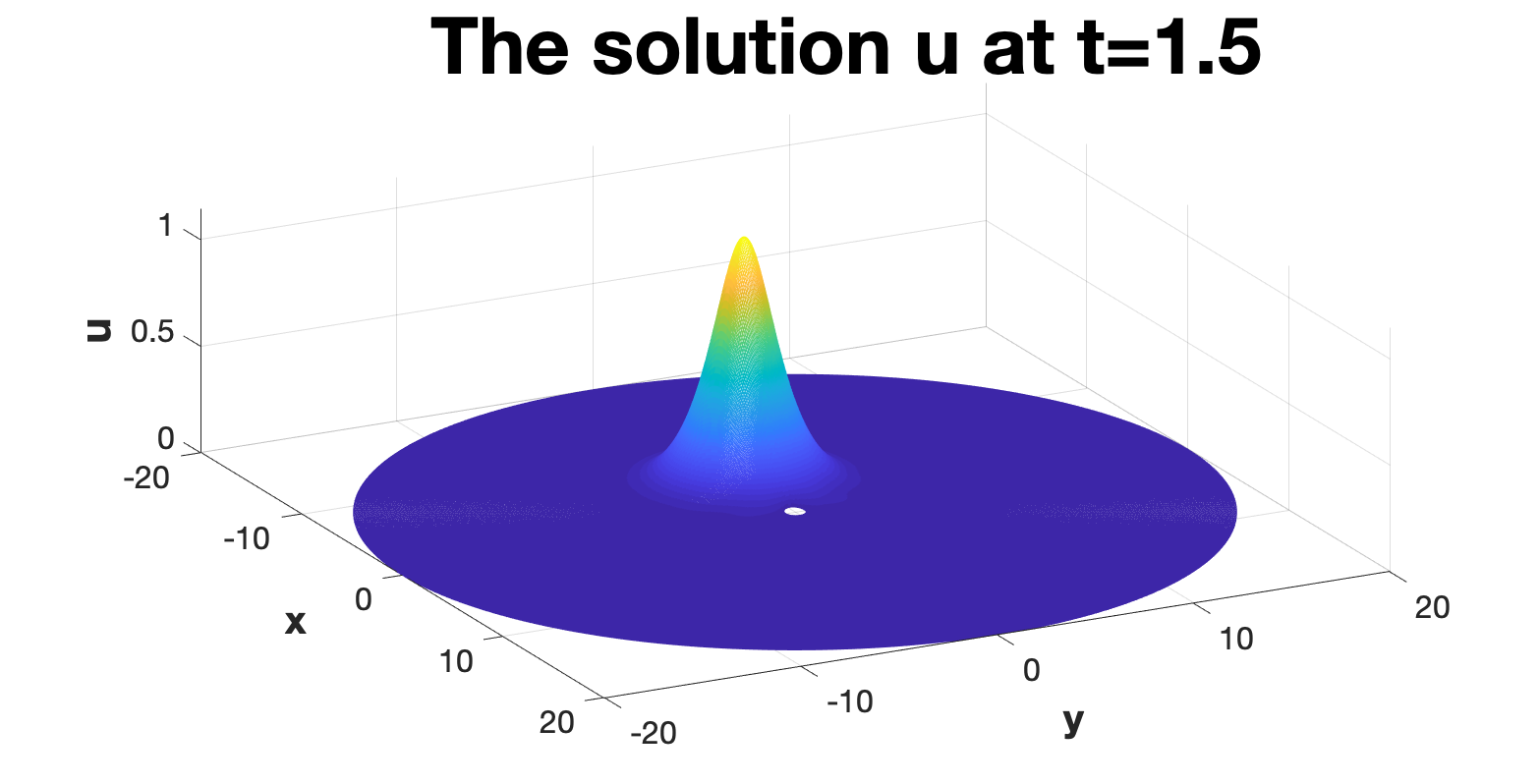}
\includegraphics[width=6.3cm,height=5.7cm]{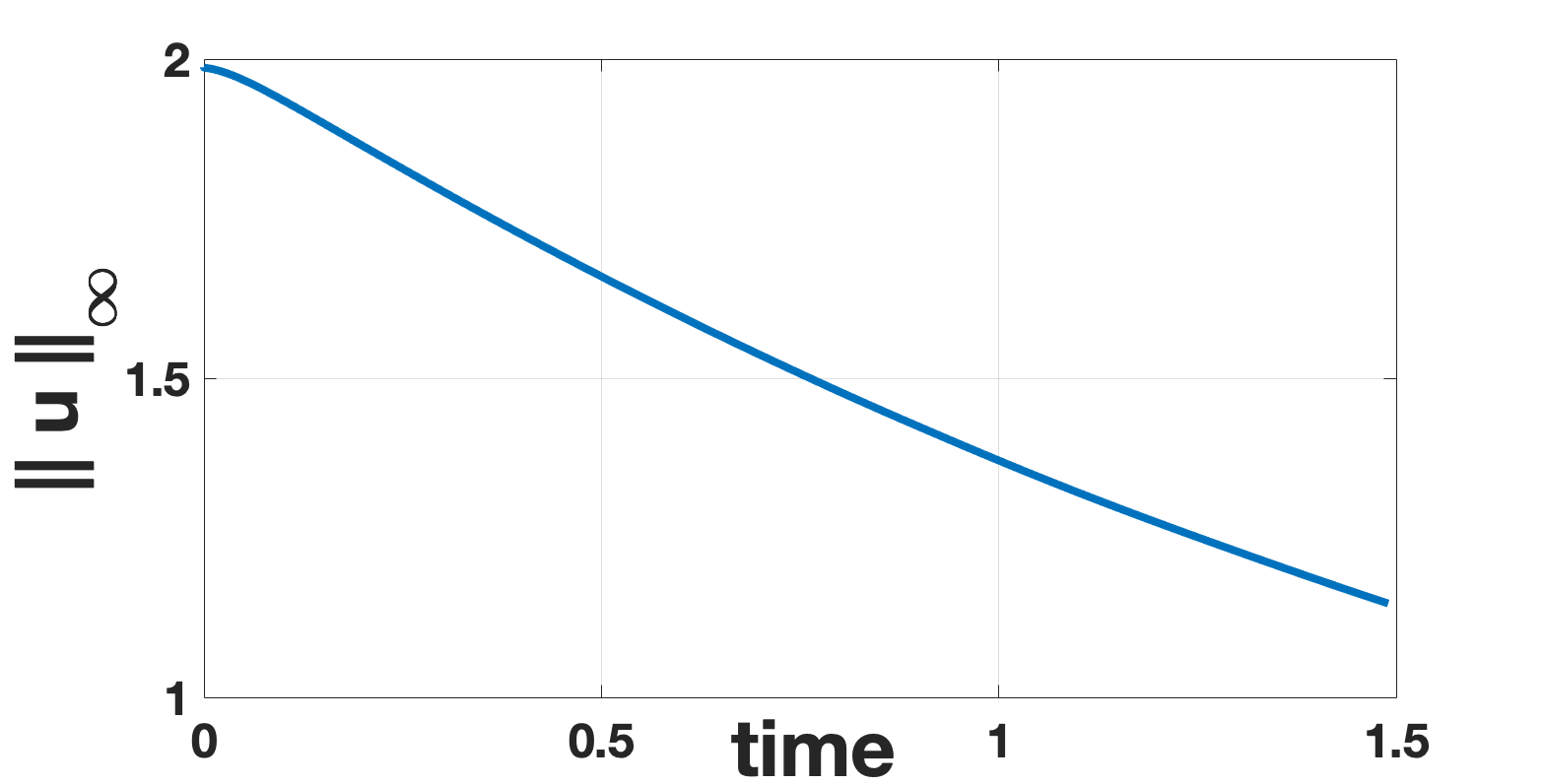}
\caption{Ground state solution to the $2d$ cubic \NNls with $u(x,y,0) = 0.9 \, Q(x+4.5,y) $ at $t= 1.5$ (left) and the $L^{\infty} $ norm for $0<t<1.5$ (right).} 
\label{Q0.9}
\end{figure}
A snapshot of the corresponding solution to the \NNls equation at time $t=1.5$ and the time dependence of the $L^{\infty}$ norm are shown in Figure \ref{Q0.9}. 
In this example, one can see that the $L^{\infty}$ norm is monotonously decreasing with a definite negative slope. It is plausible to conclude that this solution disperses in a long run, as expected in the $L^2$-critical case for the perturbations with smaller mass than that of the soliton (note that $\lambda = 0.9<1$). To further confirm our expectations we run this example with the same initial condition for longer times and the next Figure \ref{Linftylongyime} shows that the $L^{\infty}$ norm indeed keeps decreasing to $0.$

\begin{figure}[!h]
\centering
\includegraphics[width=9cm,height=5.5cm]{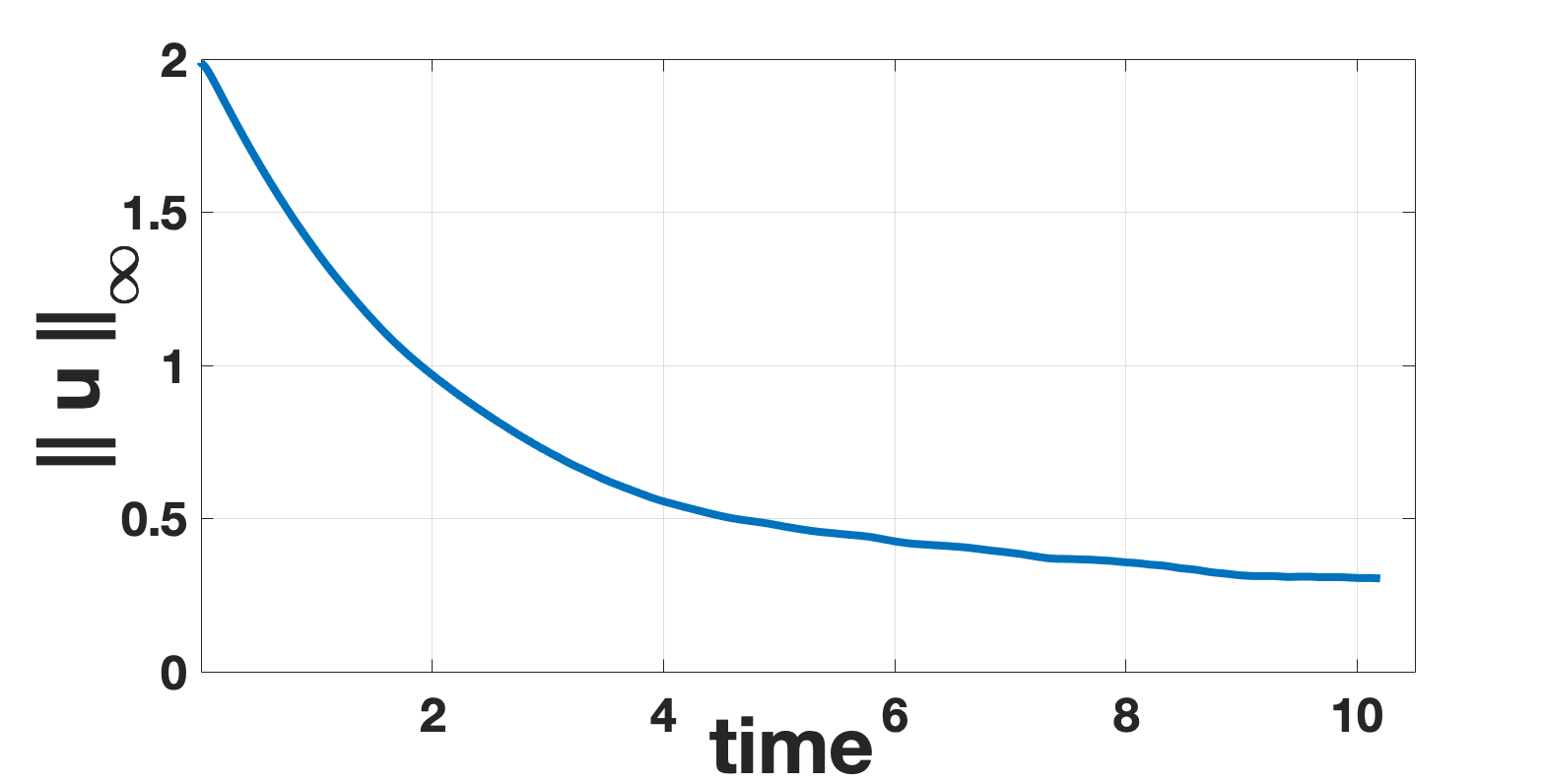}
\caption{Time dependence of the $L^{\infty}$-norm for the solution in Figure \ref{Q0.9}.} 
\label{Linftylongyime}
\end{figure}

We have tried other perturbations of $Q$ with various shifts and observed similar behavior. Therefore, we confirm in this section that the threshold for blow-up vs. scattering for the perturbed soliton data is indeed given by the ground state. 

{\bf Remark.} In the rest of the paper, except Section \ref{sec-special-solution}, we consider a shifted Gaussian initial data $u_0$ as in \eqref{u0NLS}. One of the reasons is that it has a faster decay than the ground state (though both decay exponentially), which ensures that the simulations close to the obstacle satisfy Dirichlet boundary condition (even a slightly faster exponential decay makes computations easier). Another reason is that in order to  study various interactions with an obstacle, we consider initial data $u_0$ with the minimal possible distance $d^{\star}$  to the obstacle (as defined in \eqref{min-distance}) so that $u_0$ is smooth and still satisfies Dirichlet boundary condition. 
 
\section{Dependence on the distance} 
\label{Dependence on the distance}
From now on we study both the $2d$ cubic and quintic  \NNls equations $(p=3,5)$ with  the radius of the obstacle $r_{\star}=0.5$. Our goal in this section is to consider solutions with data $u_0$ such that the distance $d$ between the obstacle and the initial condition is larger than the minimal distance $d^*$.
\medskip

We take \eqref{u0NLS} with $x_c$ and $y_c$ such that $d>>d^{*}$.
As before $v=(v_x,v_y)$ is the velocity vector, which governs the moving direction of the initial bump. Figure \ref{NointeractTranslVariab} shows different directions of propagation for this solitary wave-type data depending on  the velocity $\vec{v}$.  
\begin{figure}[!ht]
\centering
\includegraphics[width=15cm,height=5.5cm]{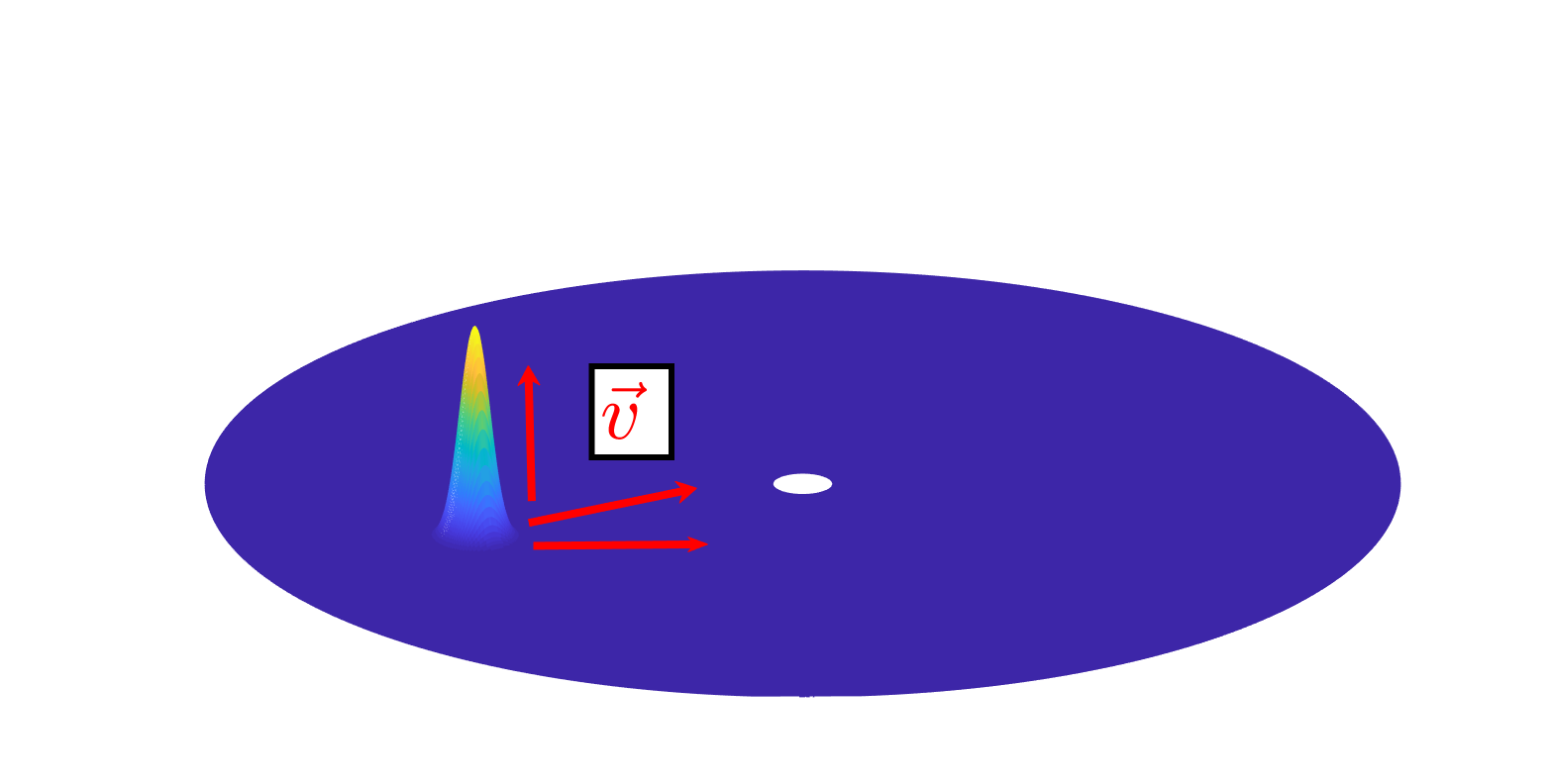}
\caption {Directions of the velocity $\vec{v}$ of the initial data (single bump) relative to an obstacle (the white region represents the location of the obstacle). If the initial bump is relatively far from the obstacle, $d>>d^{*}$, then the blow-up occurs in any direction of the initial velocity shown on the picture, for example, as it is shown in Figure \ref{NointVeloDirec}.}
\label{NointeractTranslVariab}
\end{figure}

\subsection{The $L^2$-critical case}
For the $2d$ cubic \NNls equation we take the initial data  \eqref{u0NLS} with large enough mass and $d >>d^{*}.$ Then the corresponding solution to \eqref{time-Scheme} blows up in finite time before reaching or interacting with an obstacle in any direction of the velocity vector $v$, see Figure \ref{NointVeloDirec}. 
\begin{figure}[ht]
\centering
\includegraphics[width=6.2cm,height=5.5cm]{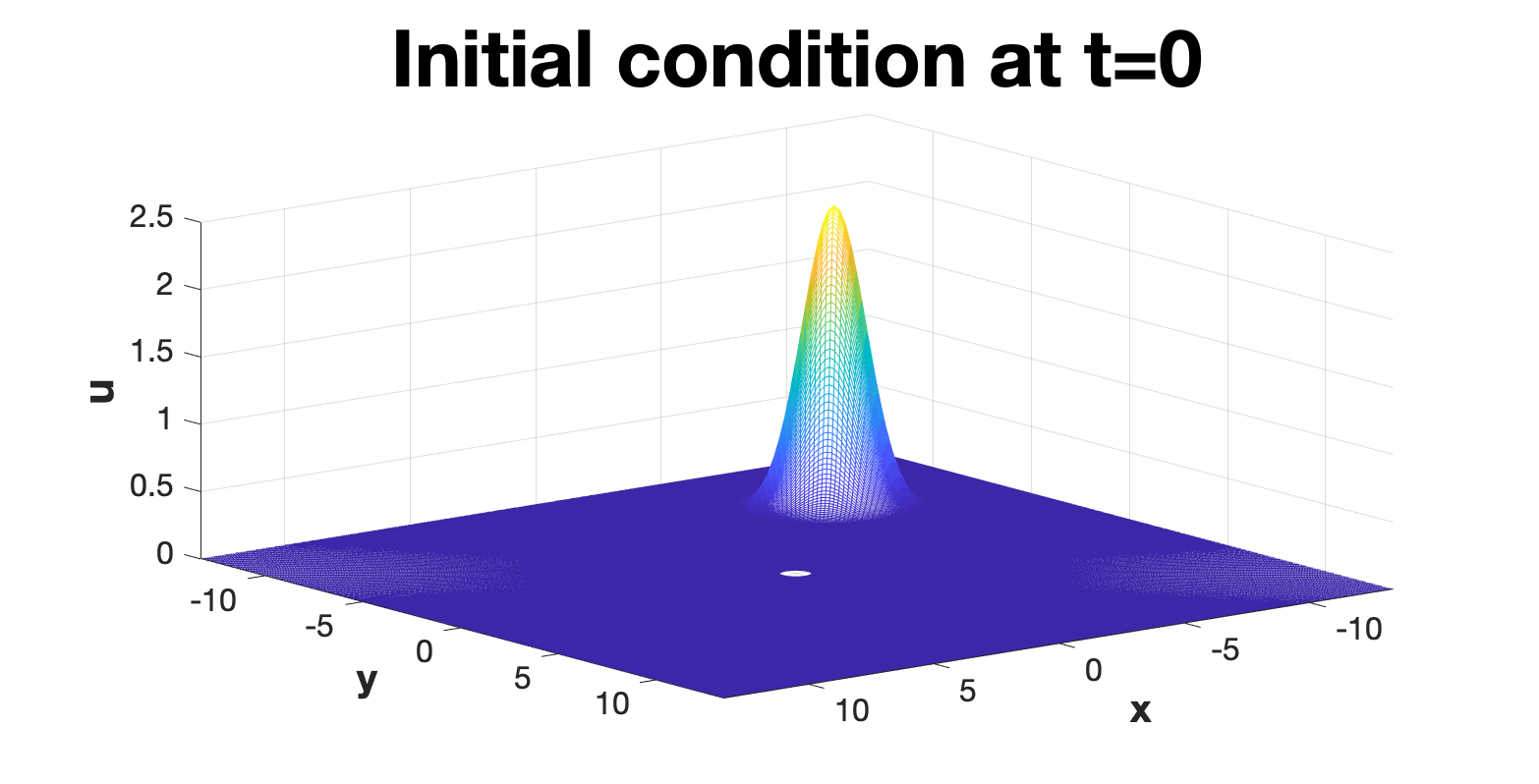}    
\includegraphics[width=6.3cm,height=5.5cm]{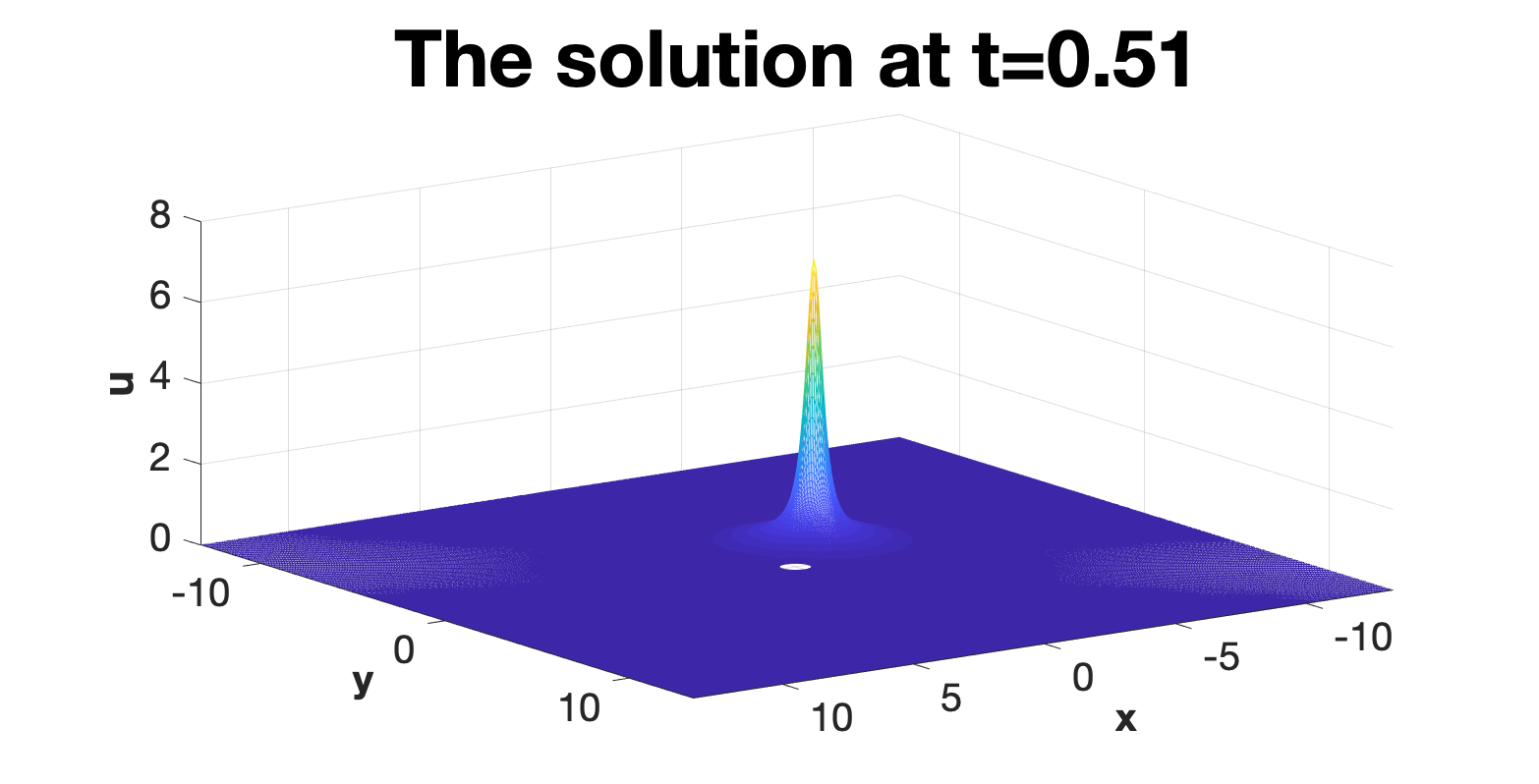}
\includegraphics[width=5.2cm,height=5.1cm]{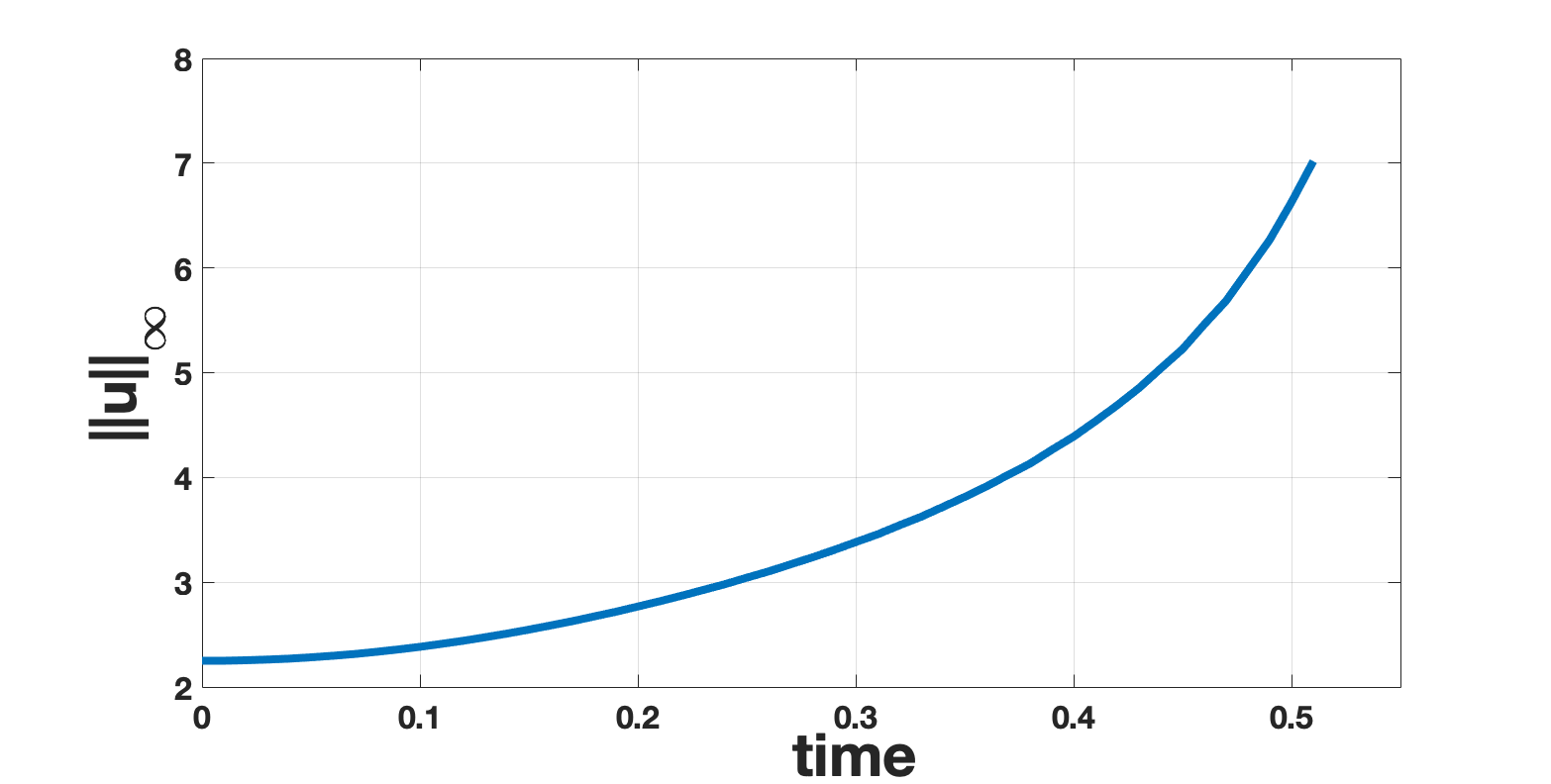}
\caption {Snapshots of the initial data $u_0$ as in \eqref{u0NLS} with large enough mass and the solution $u(t)$ to the $2d$ cubic \NNls at $t=0$ and $t=0.51$ with $d >>d^{*}$ (left and middle); the $L^{\infty}$-norm depending on time (right).}
\label{NointVeloDirec}
\end{figure}

Later we study the case when $d \equiv d^{*}$ and the solution concentrates in its (blow-up) core after the obstacle, for the same initial data but with a different velocity direction. We also investigate the influence of the obstacle when there is an interaction between the traveling wave and the obstacle. In Section \ref{L2critWeak}, we consider the weak interaction for the $2d$ cubic \NNls equation ($L^2$-critical case) and in Section \ref{L2critStrong} we study the strong interaction. We observe that in those cases the solution exhibits a different behavior on a longer time interval.  

\subsection{The $L^2$-supercritical case}
Next, we consider the $2d$ quintic \NNls equation and take the initial condition \eqref{u0NLS} with a large mass and $d >>d^{*}$. In the following scenario, we fix parameters $A_0$ and $v=(v_x,0)$ and vary 
the translation parameter $y_c$ in the translation $(x_c,y_c)$ 
as shown in Figure \ref{NointeraclVariabY0}.   

\begin{figure}[ht] 
\centering
\includegraphics[width=15cm,height=6.6cm]{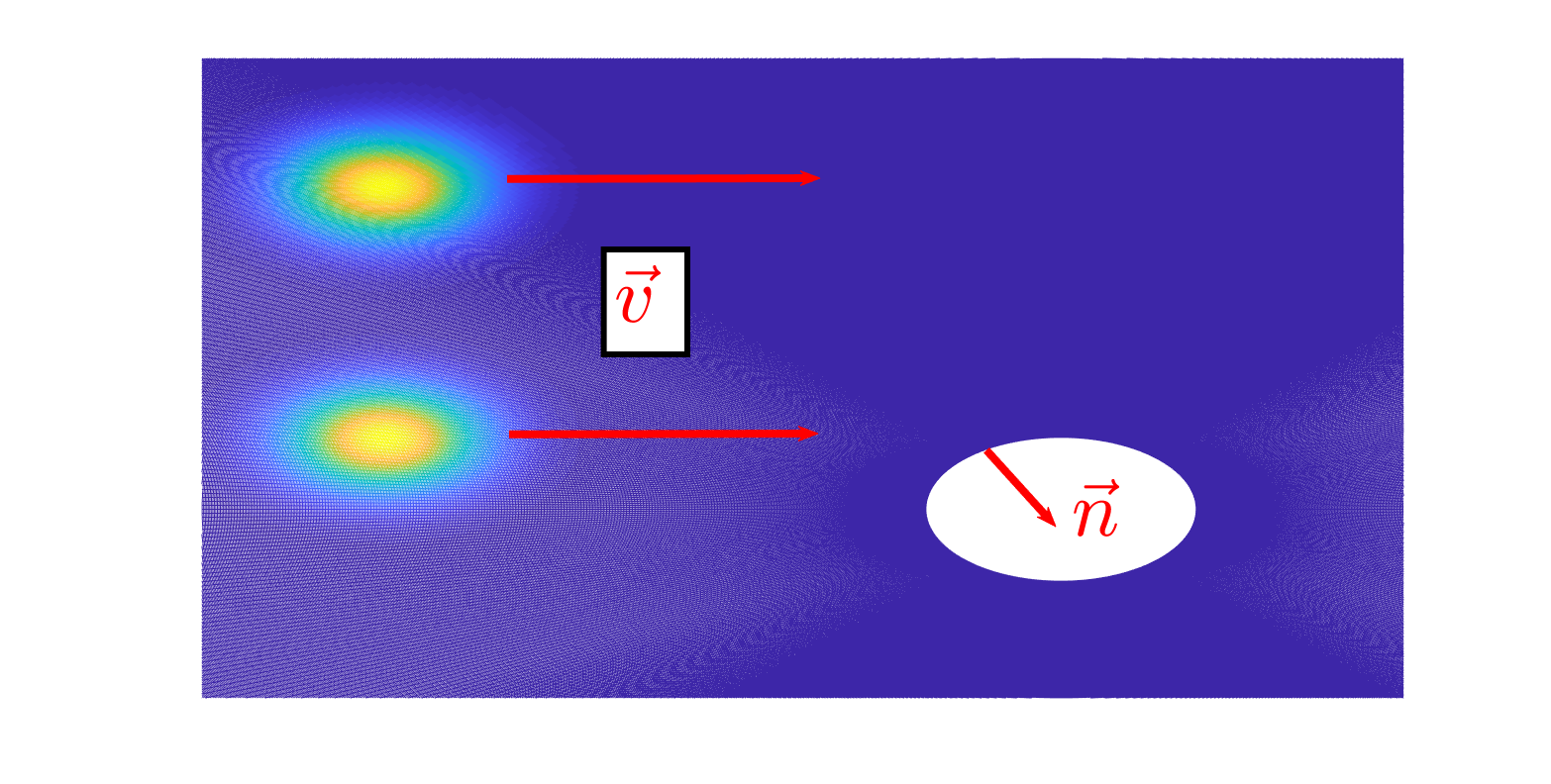}
\caption {Directions of the velocity $\vec{v}$ of the initial two single bumps that  move along either the line $y=5$ or $y=2$ with $d >>d^{*}$.}
\label{NointeraclVariabY0}
\end{figure}
 
 
Snapshots of the corresponding solution to the $2d$ quintic \NNls equation are plotted in Figure \ref{SolutionNointeraclVariabY0}. As in the previous example, the solution blows up in finite time before the obstacle (for large $x_c$). \\

\begin{figure}[ht]
\includegraphics[width=6.2cm,height=5.6cm]{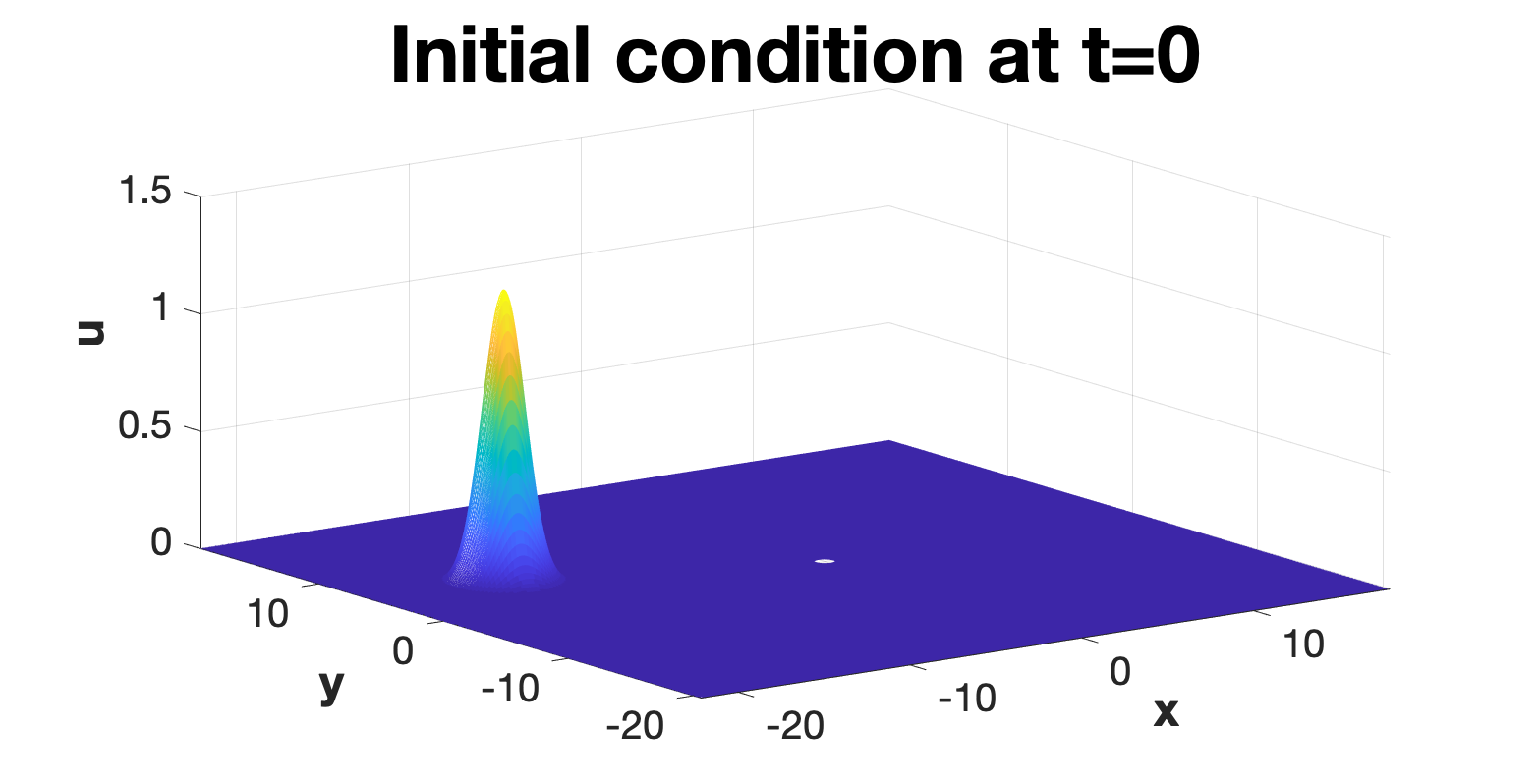}
\includegraphics[width=6.3cm,height=5.6cm]{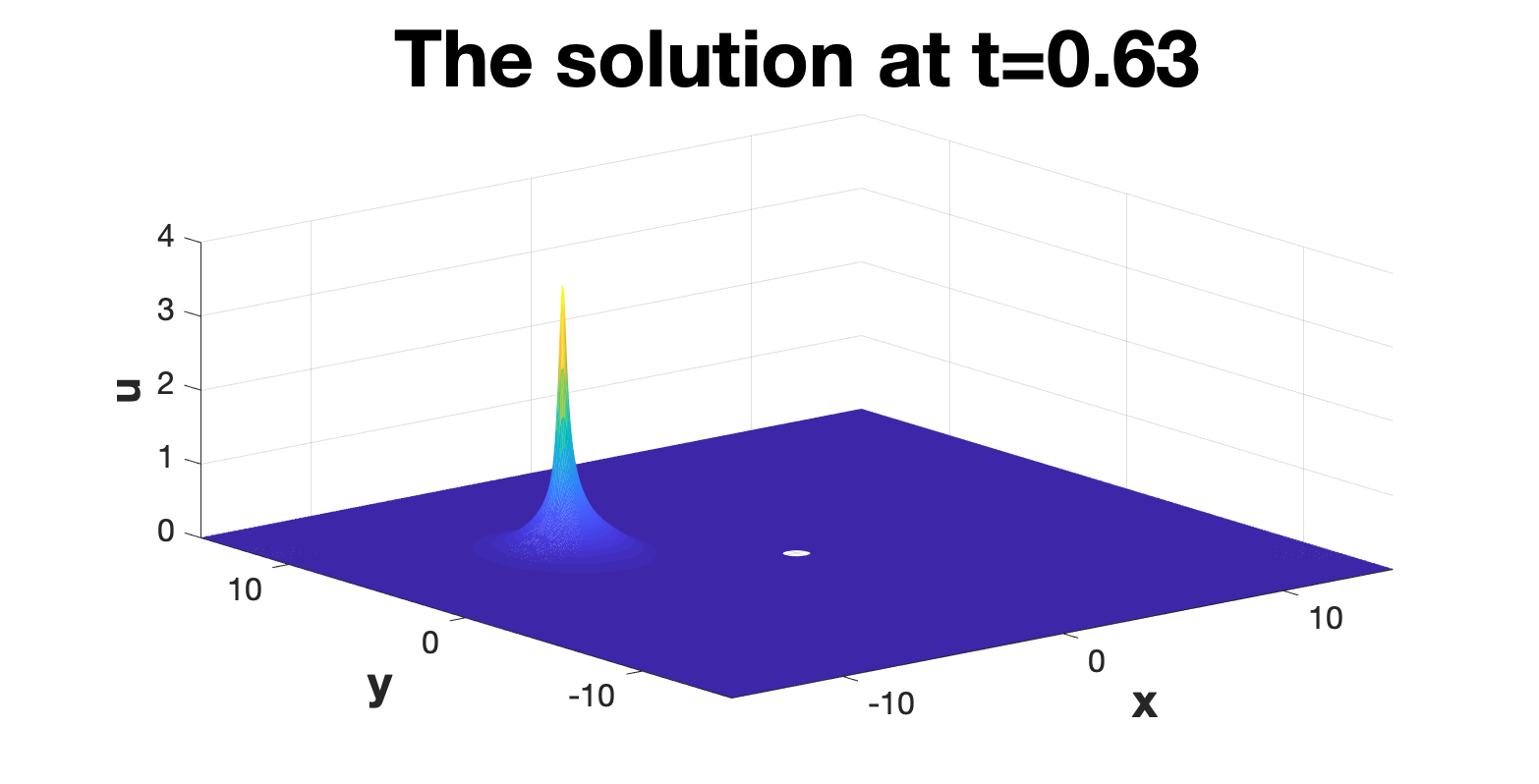}
\includegraphics[width=5.2cm,height=5cm]{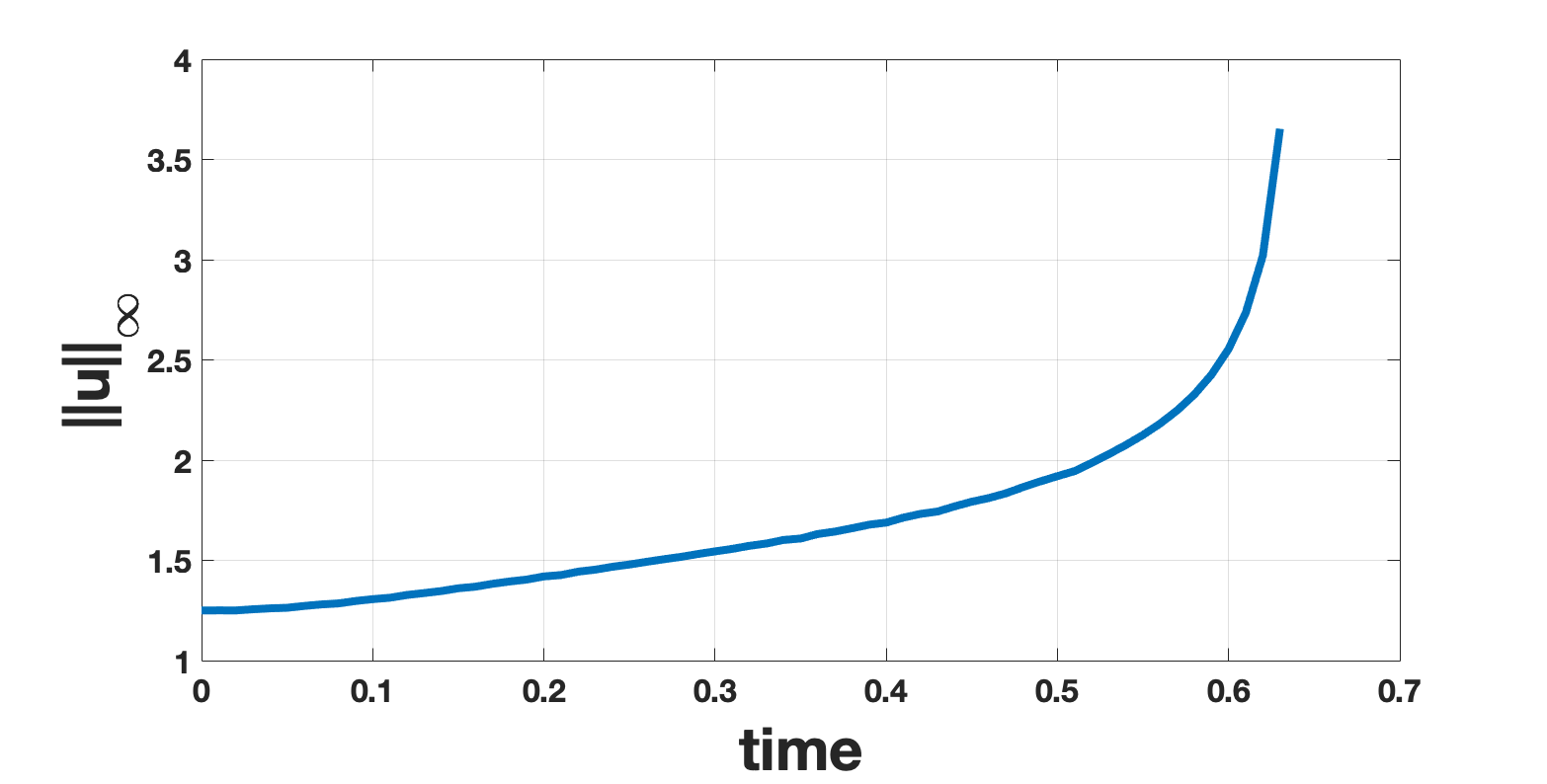}
\caption {A blow-up solution to the $2d$ quintic \NNls equation with the initial condition \eqref{time-Scheme} and $d >>d^{*}.$ The initial profile (left), a snapshot of the solution at $t=0.63$ (middle), the time dependence of the $L^{\infty}$-norm (right).} 
\label{SolutionNointeraclVariabY0}
\end{figure}

We later investigate the case when the solution blows up in finite time, after the obstacle and when $d \equiv d^{*}$, with the same initial data $u_0$ as in \eqref{u0NLS} for a fixed amplitude $A_0=1.25$ and velocity direction $v=(15,0),$ and $x_c=-4.5,$ but for different space translation $y_c,$ see Table \ref{T:2}. This will lead to the weak or strong interaction for the $2d$ quintic \NNls equation ($L^2$-supercritical case), see Sections \ref{L2supercritWeak} and  \ref{L2supercritStrong}.

\section{Weak interaction with an obstacle} 
\label{Weak interaction between soliton and obstacle}
\subsection{The $L^2$-critical case}
\label{L2critWeak}
We return to the cubic \NNls setting and consider initial data \eqref{u0NLS} with $A_0=2.25$, $x_c=-4.5,$  $y_c=-4.5$ and we vary the direction of the velocity vector. Considering the two scenarios, as shown in Figure \ref{veloctyNointeraction}.  
\begin{figure}[ht]
\centering
\includegraphics[width=15cm,height=5.5cm]{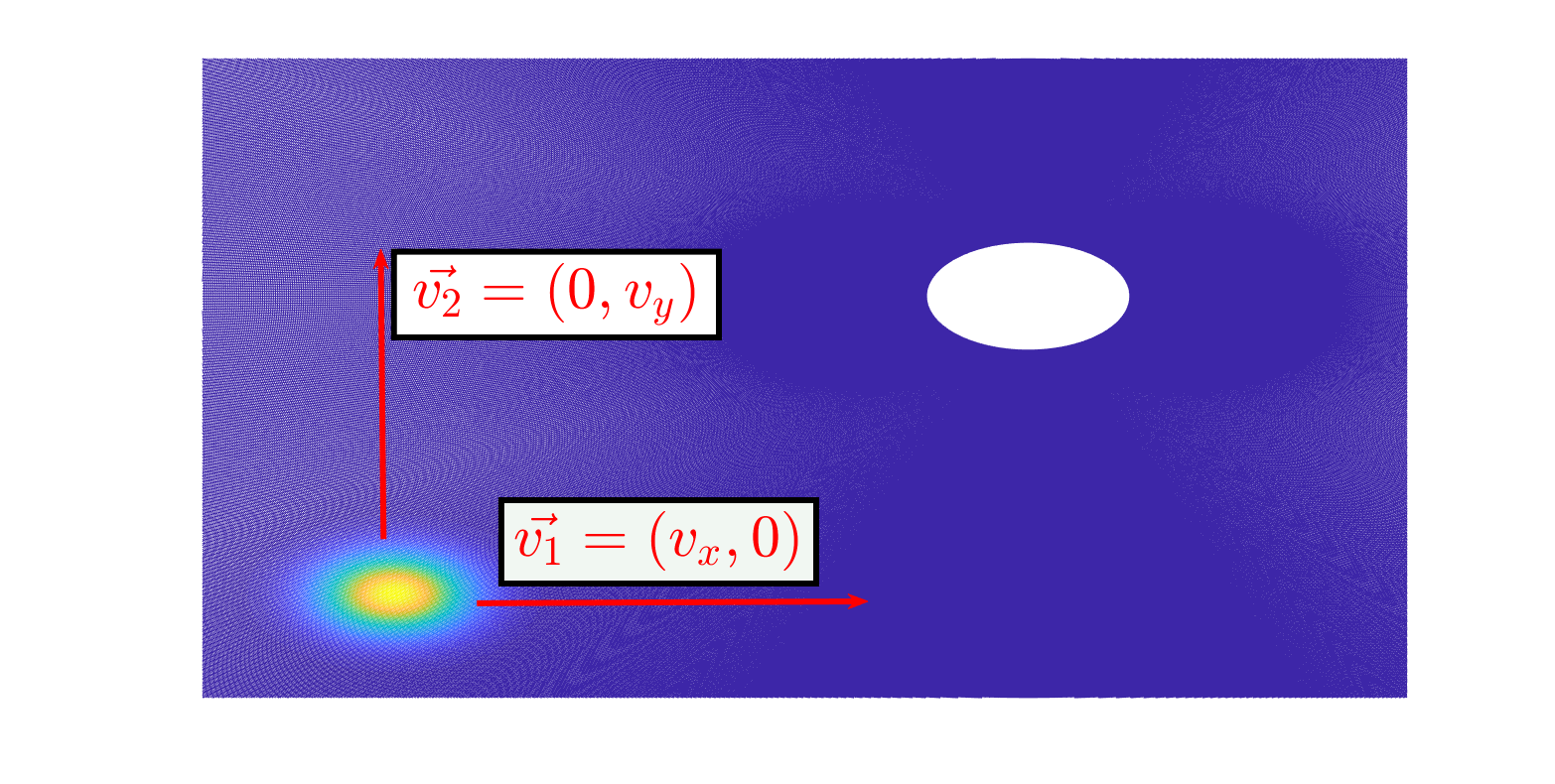}
\caption{Directions of the velocity for the examples in Section \ref{L2critWeak}.}
\label{veloctyNointeraction}
\end{figure}

We start with ${v_1}=(v_x,0)$, $v_x=15$, and observe that the solution blows up at time $t=0.52$. It does not interact with the obstacle; its behavior is the same as it would be of a solitary wave on the whole space, see Figure \ref{test1Criti}.

\begin{figure}[!ht]
\centering
\includegraphics[width=6.2cm,height=7.5cm]{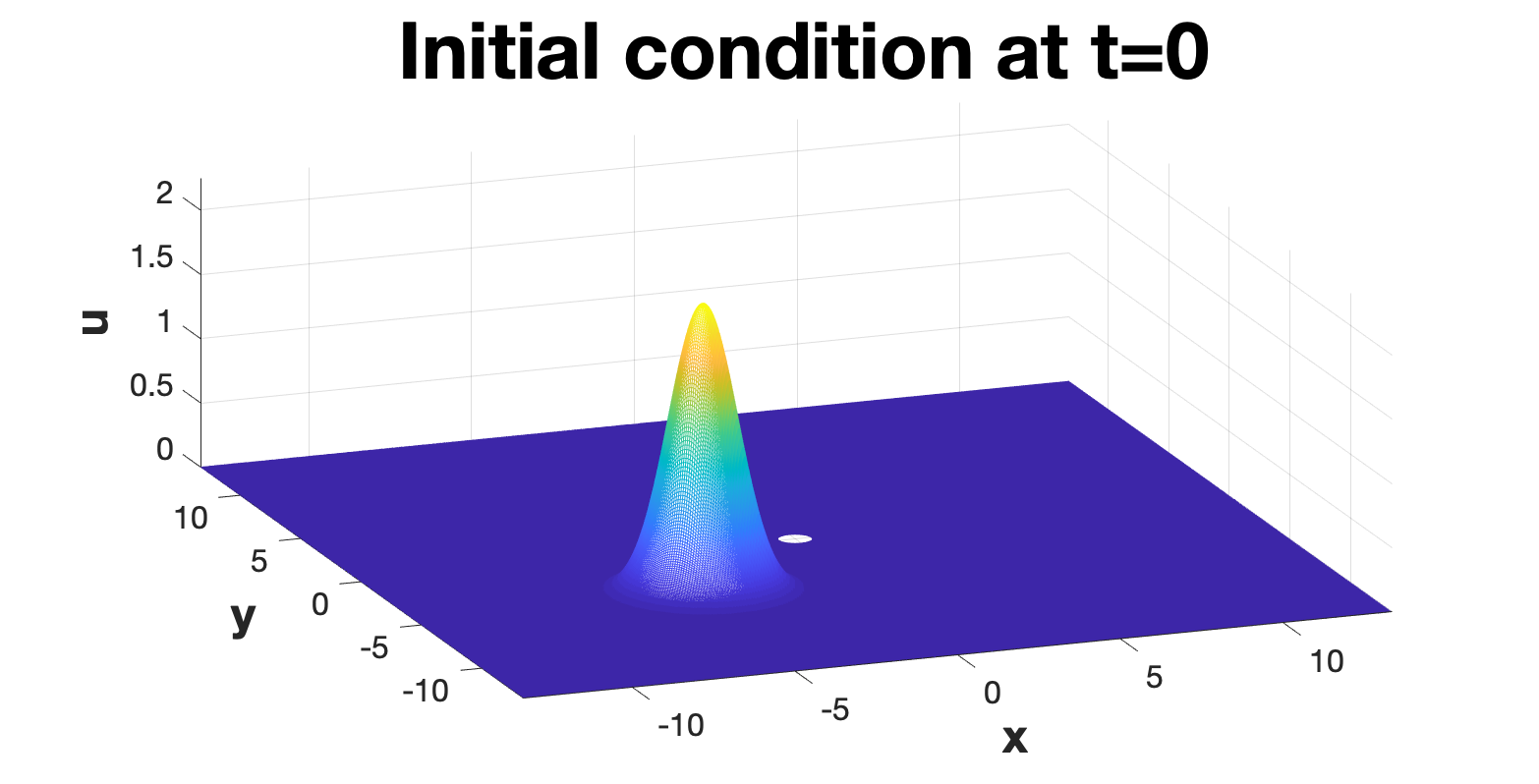}
\includegraphics[width=6.4cm,height=7.5cm]{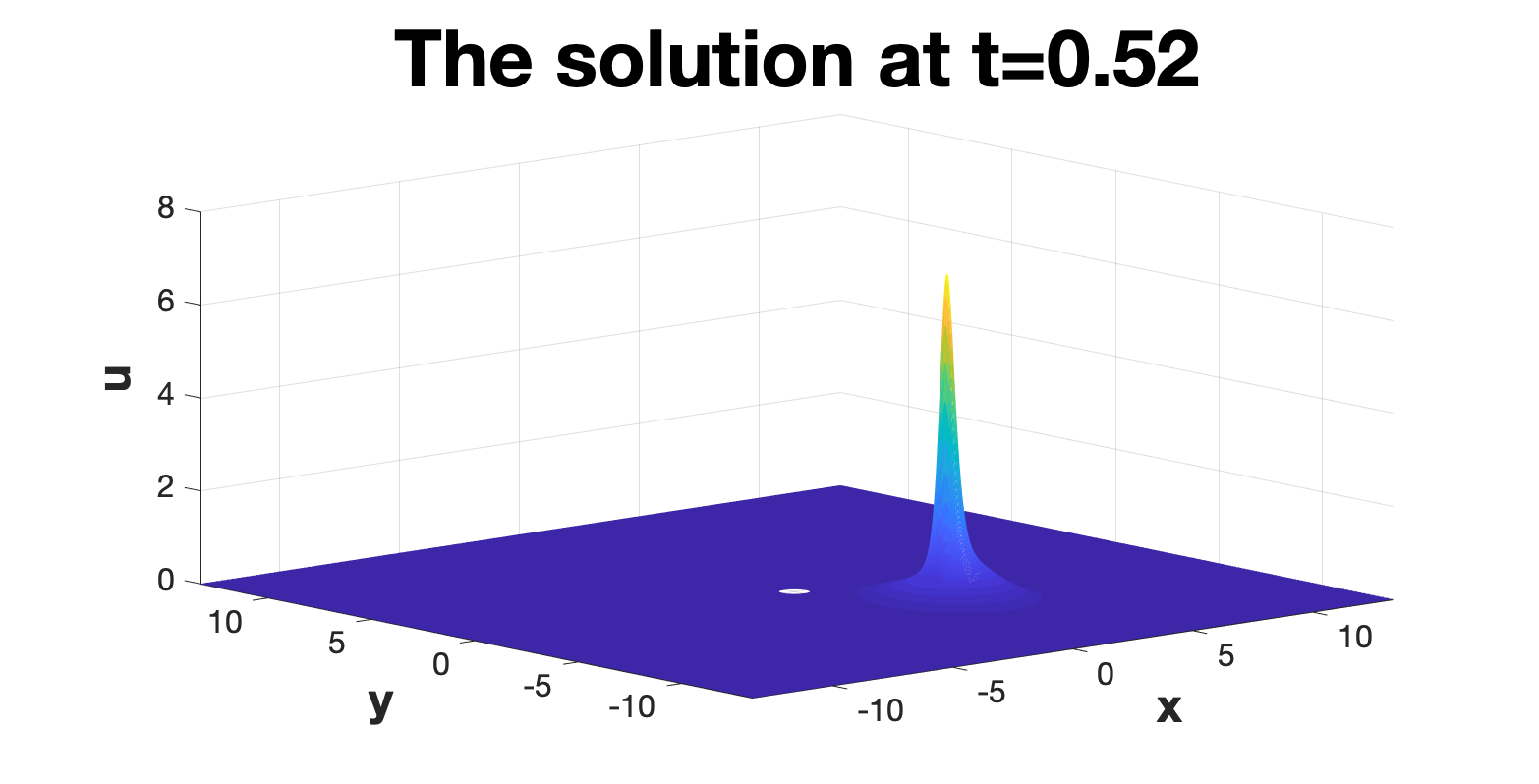}
\includegraphics[width=5.2cm,height=6.2cm]{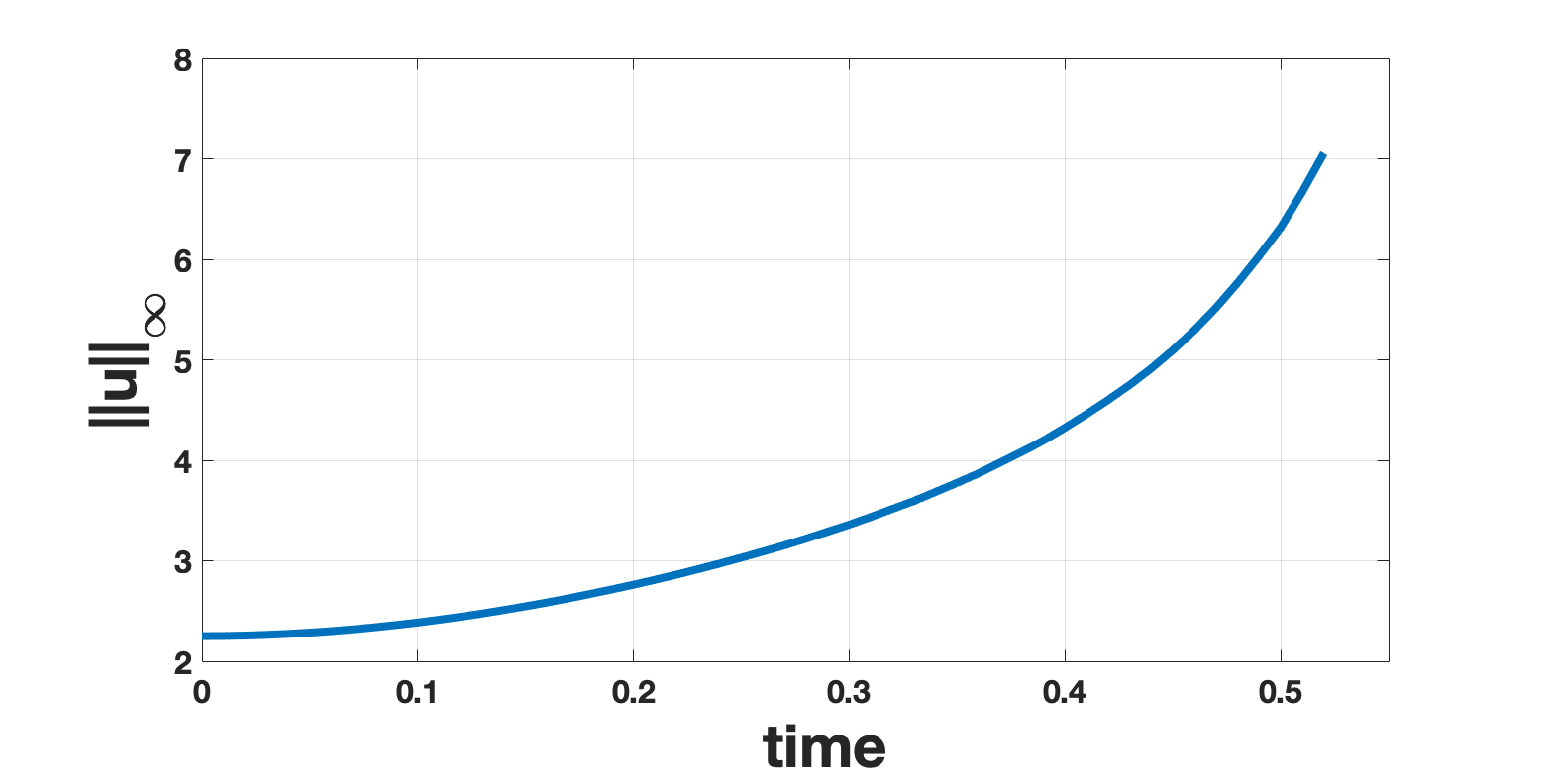}
\caption{The $2d$ cubic \NNls with $u_0$ from \eqref{u0NLS},  $A_0=2.25,$ $x_c=-4.5, \; y_c=-4.5$ and ${v_1}=(15,0)$ (left); the time evolution at $t=0.52$ (middle); the time dependence of the $L^{\infty}$-norm (right).}
\label{test1Criti}
\end{figure}

\newpage
Next, we take the same initial condition but with the velocity vector $v_2=(0,v_y)=(0,15)$, that is, perpendicular to the direction used in the previous example, as shown in Figure \ref{veloctyNointeraction}. We observe, see Figure \ref{test2Criti}, that the solution blows up at the same time.  \\

\begin{figure}[!h]
\centering
\includegraphics[width=6.2cm,height=7.5cm]{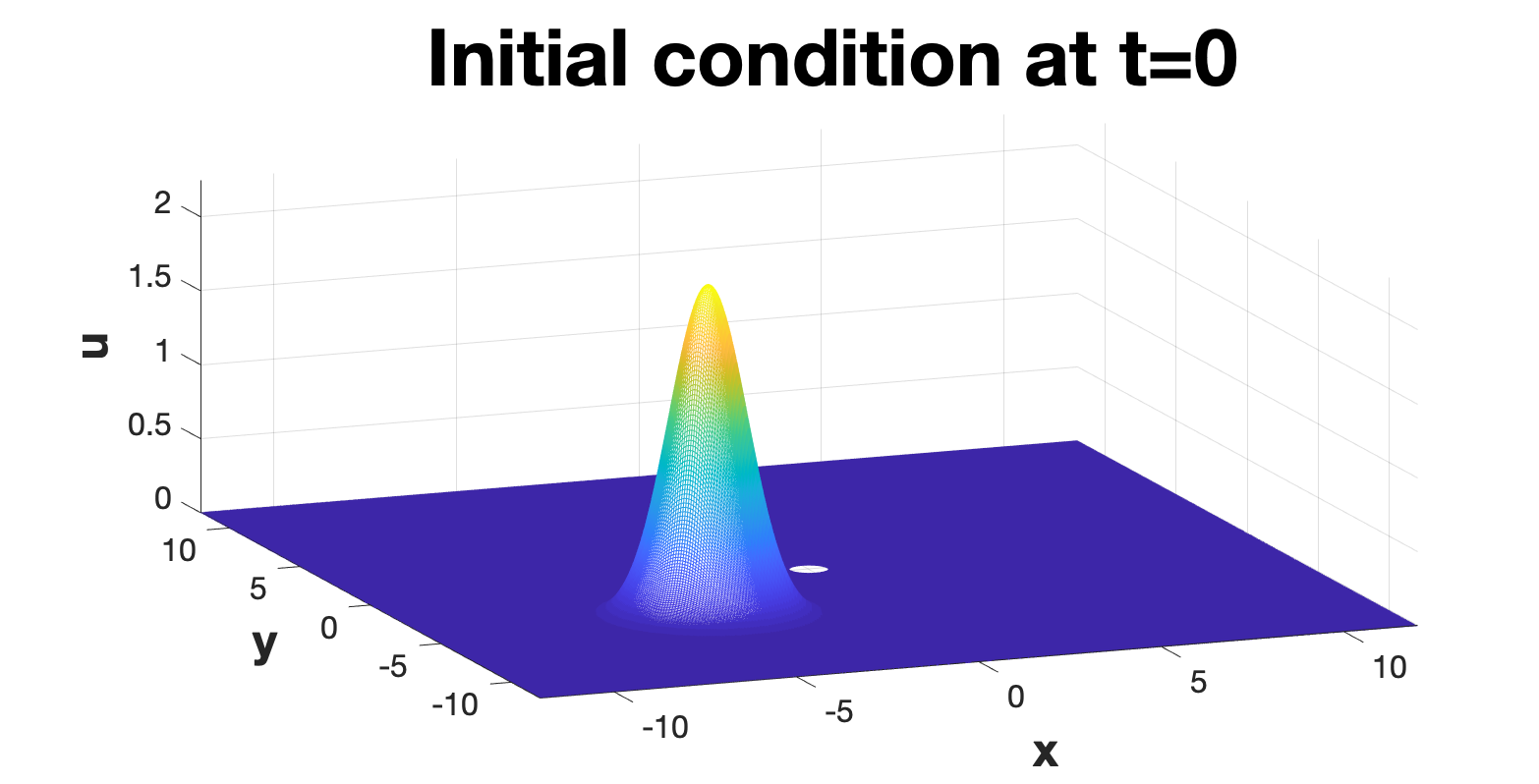}
\includegraphics[width=6.4cm,height=7.5cm]{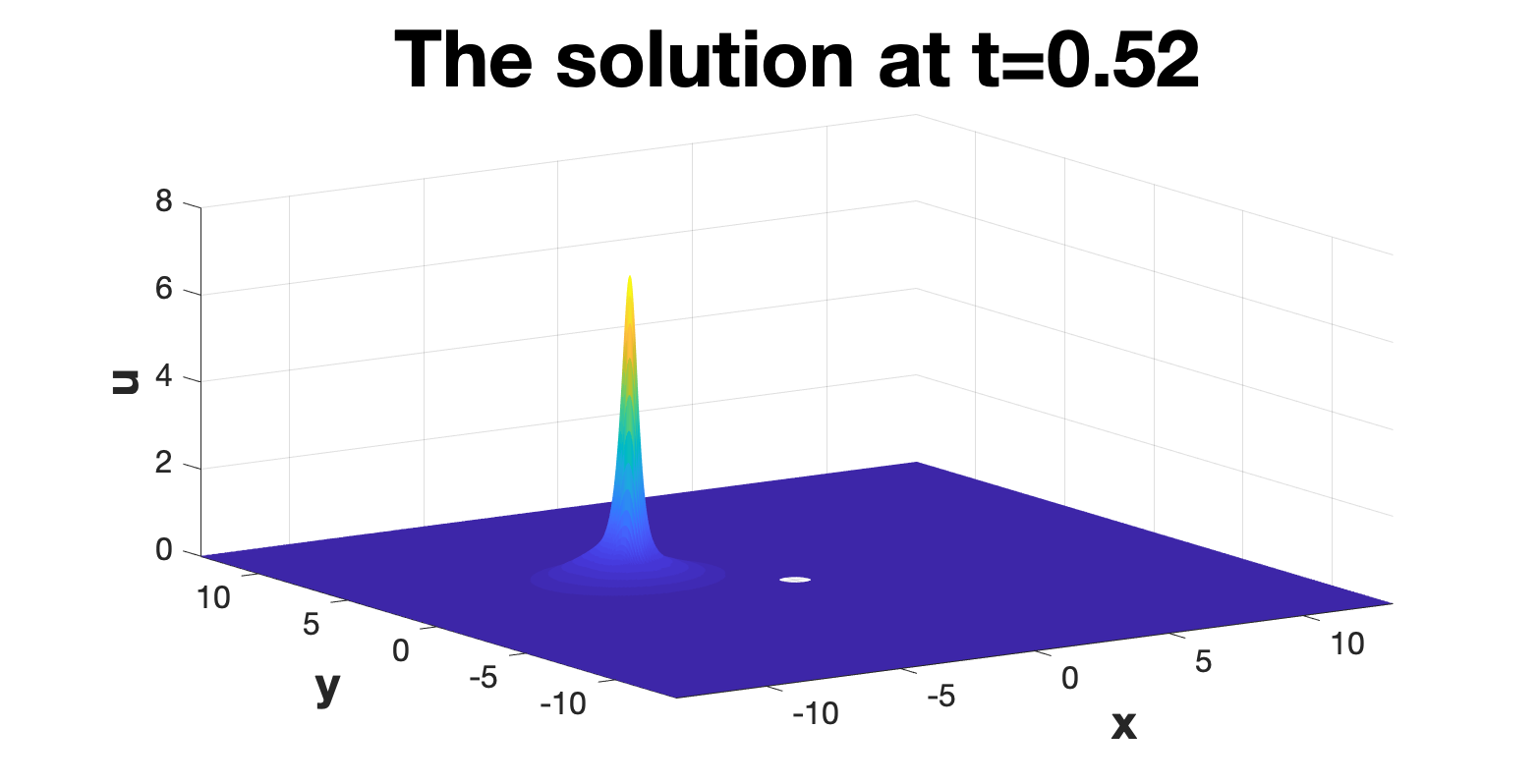}
\includegraphics[width=5.2cm,height=6.2cm]{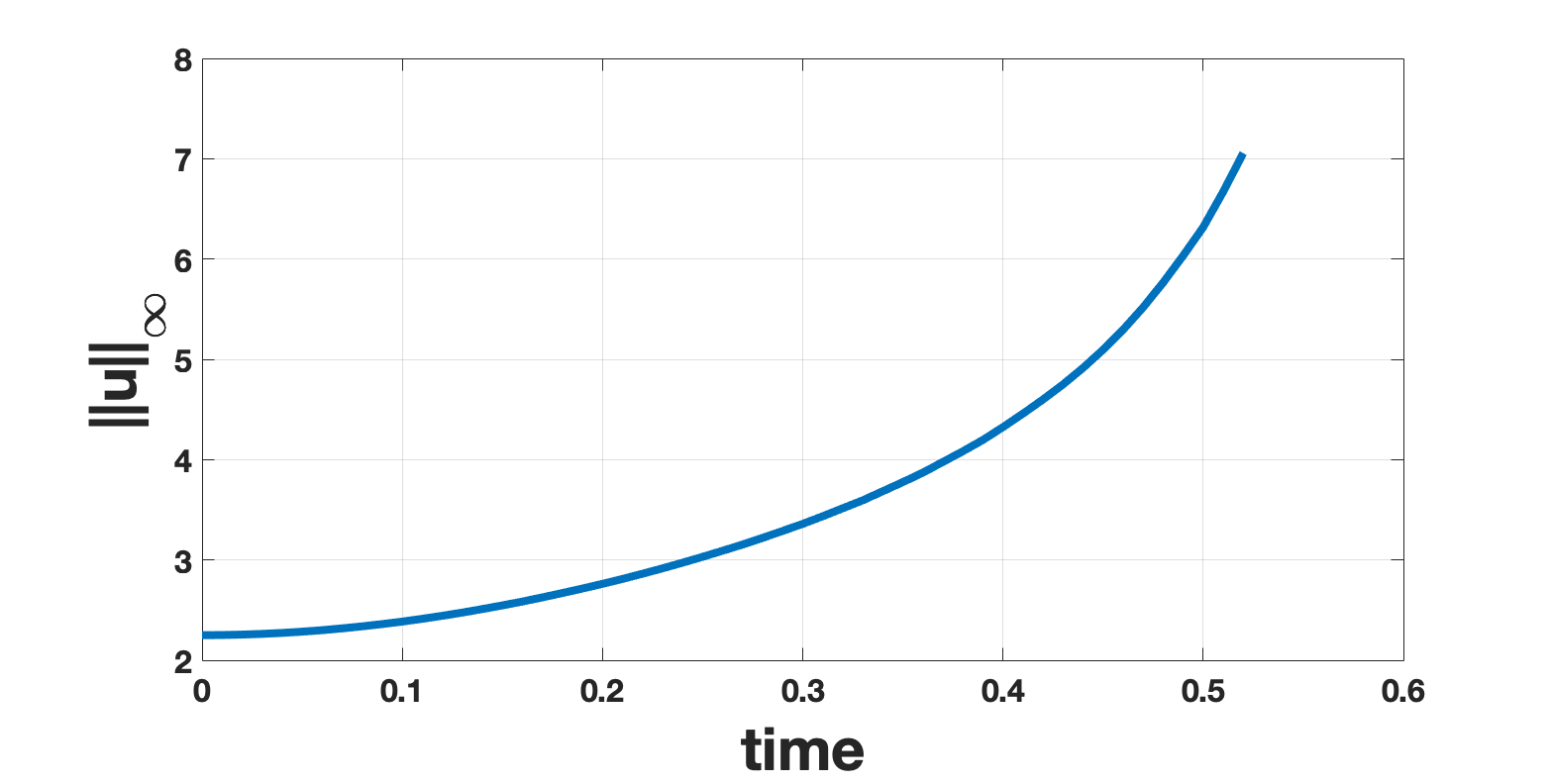}
\caption{The $2d$ cubic \NNls equation with $u_0$ from \eqref{u0NLS}, $A_0=2.25,$ $x_c=-4.5, \; y_c=-4.5$ and ${v_2}=(0,15)$ (left); its time evolution  at $t=0.52$ (middle); the time dependence of the $L^{\infty}$-norm (right).}
\label{test2Criti}
\end{figure}

In our third example, we take the initial condition $u_0$ from \eqref{u0NLS} with $A_0=2.25,\; x_c=-4.5,\; y_c=-4.5$ and the velocity $\vec{v}$ that has a different direction but has  the same magnitude $|\vec{v}|,$ as in the previous two examples: we choose $v_1=(v_x, v_y)$ and $v_2=(v_y,v_x)$ as shown on Figure \ref{weakVelocityinteractiondomain}.     

\begin{figure}[!h]
\centering
\includegraphics[width=15cm,height=7cm]{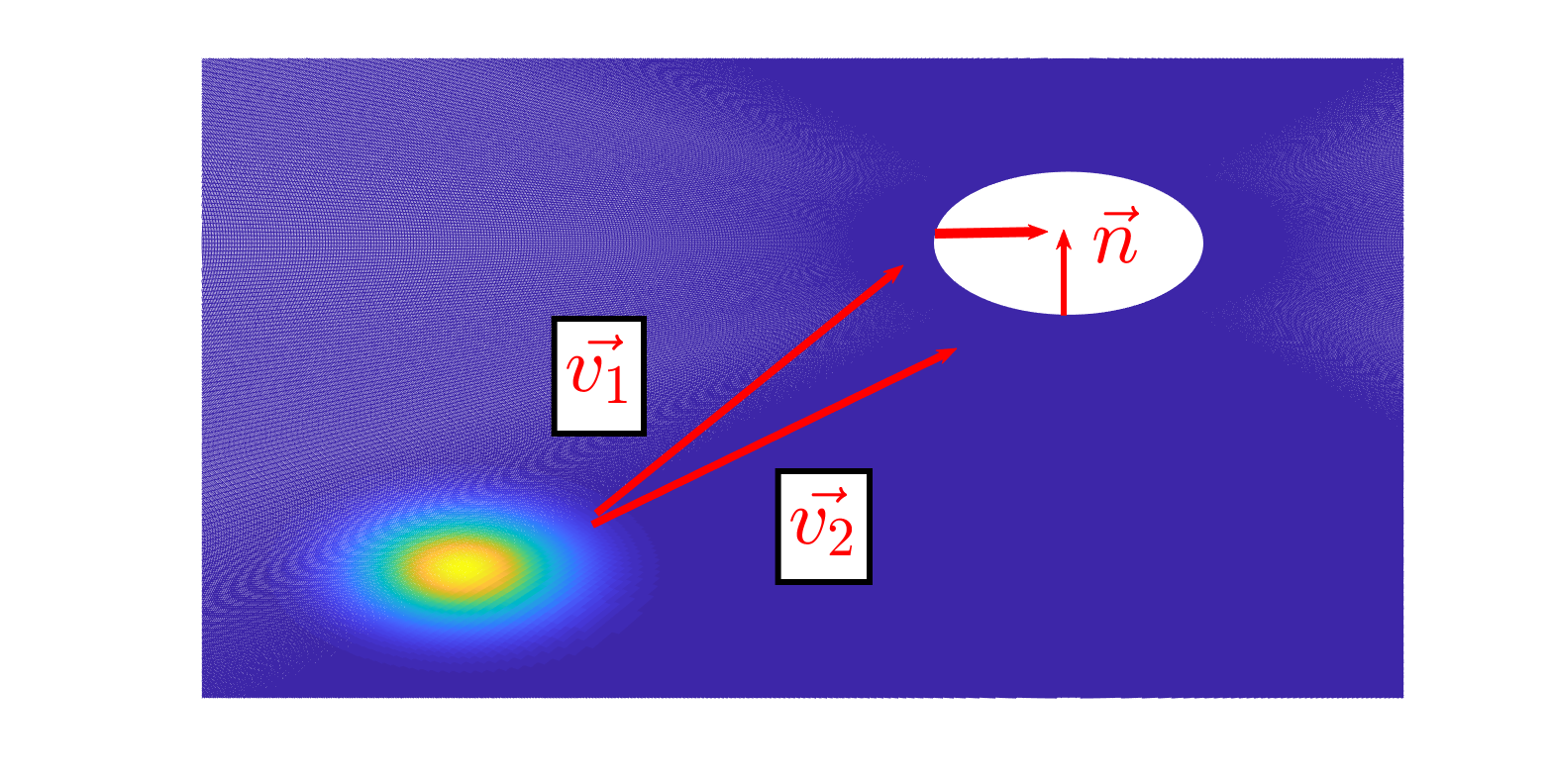}
\caption{The directions of the velocity for the examples in Figures \ref{test5Critique-weak} to \ref{test4Critique-weak}.}
\label{weakVelocityinteractiondomain}
\end{figure}

We choose $v_1=(15,9)$ such that the solution has a small interaction with the obstacle. After the collision, we observe that 
the solution has almost the same behavior (as in Figures \ref{test1Criti}, \ref{test2Criti}), i.e., it blows up but with slightly dispersive reflection part, preserving the shape of the soliton, similar to the two previous cases. The solution blows up in finite time $t=0.57$ after the interaction with the obstacle, see Figure \ref{test5Critique-weak}. Moreover, we see that at the collision time the $L^{\infty}$-norm has a slight perturbation (or a small oscillation), however, afterwards it continues to increase: such perturbation it not sufficient to prevent the overall growing of the $L^{\infty}$-norm and the occurrence of the blow-up, however, it affects the blow-up time compared to the previous examples. Besides Figure \ref{test5Critique-weak} with the $L^\infty$ norm dependence, we also provide snapshots of the behavior of the solution for different time steps for $v_2=(9,15)$ in Figure \ref{test6Critique-weak}.

\begin{figure}[ht]
\centering
\includegraphics[width=8.6cm,height=6.5cm]{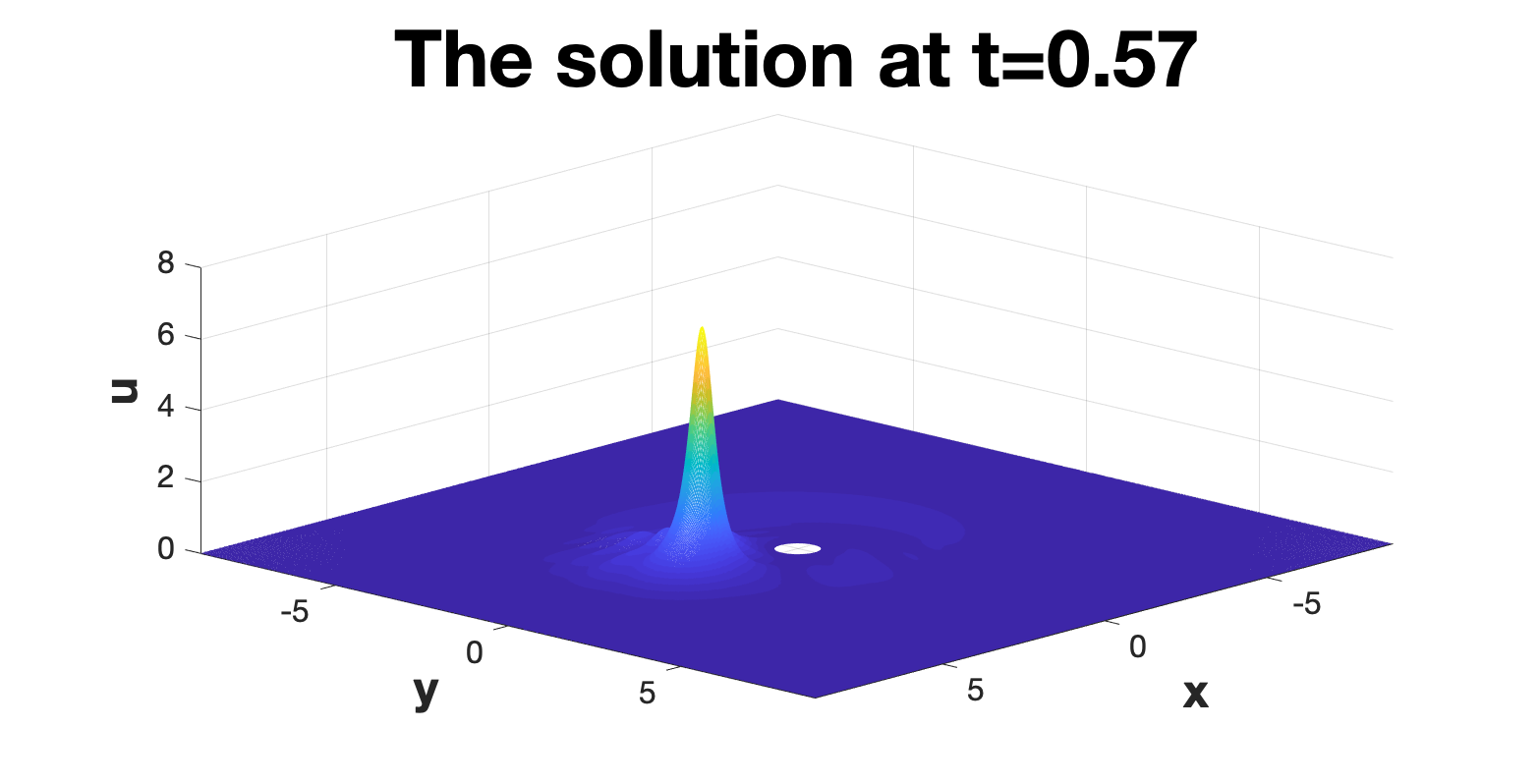}
\includegraphics[width=7.6cm,height=5.7cm]{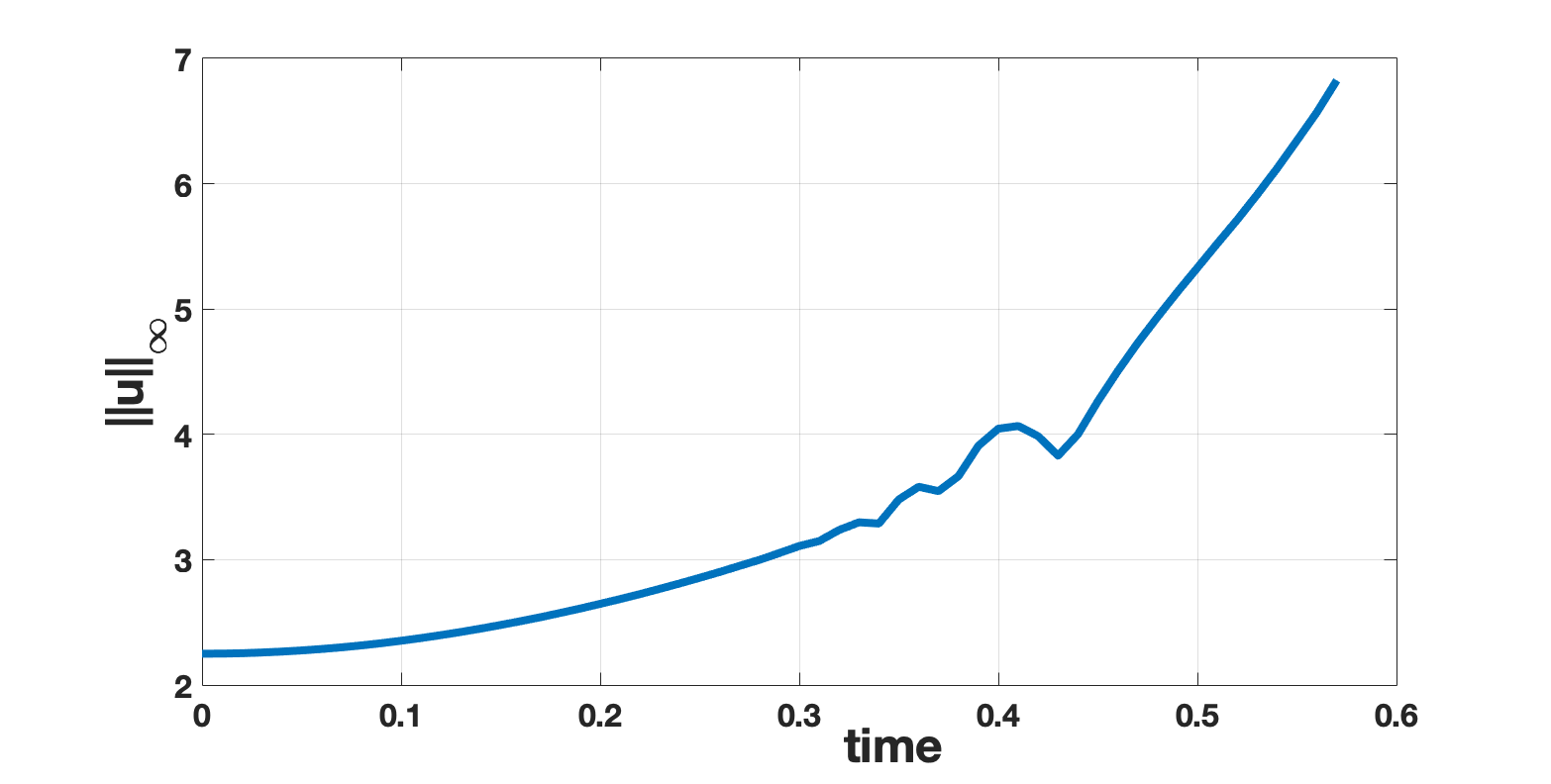}
\caption{Solution to the $2d$ cubic \NNls equation with the initial condition $u_0$ as in \eqref{u0NLS},  $A_0=2.25,\; x_c=-4.5,\; y_c=-4.5,$ and $v_1=(15,9)$ at time $t=0.57,$ moving on the line $y=\frac 35 x$ (left); the time dependence of $L^{\infty}$-norm (right). }
\label{test5Critique-weak}
\end{figure}

\begin{figure}[ht]
\centering
\includegraphics[width=5.3cm,height=5.6cm]{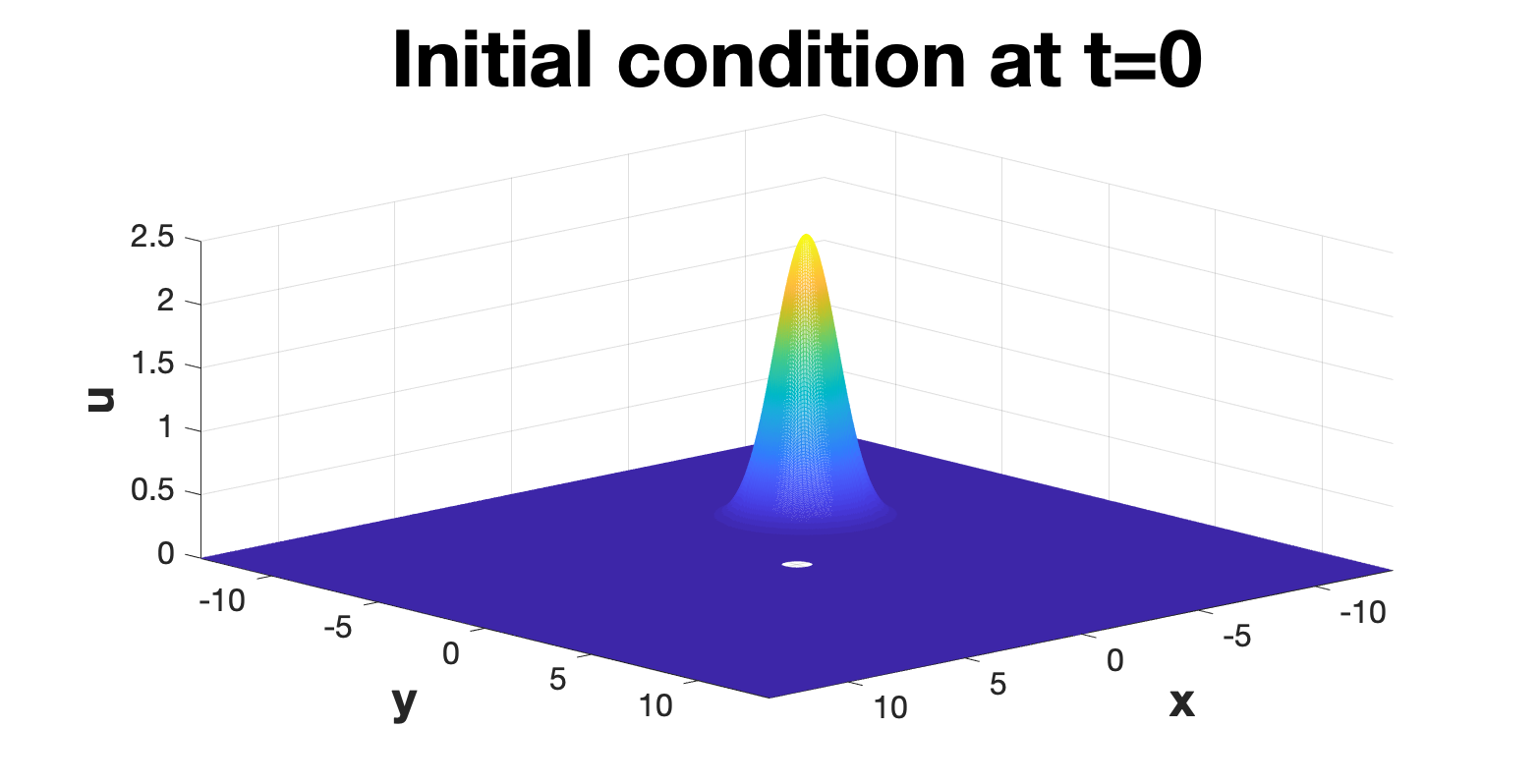}
\includegraphics[width=5.3cm,height=5.4cm]{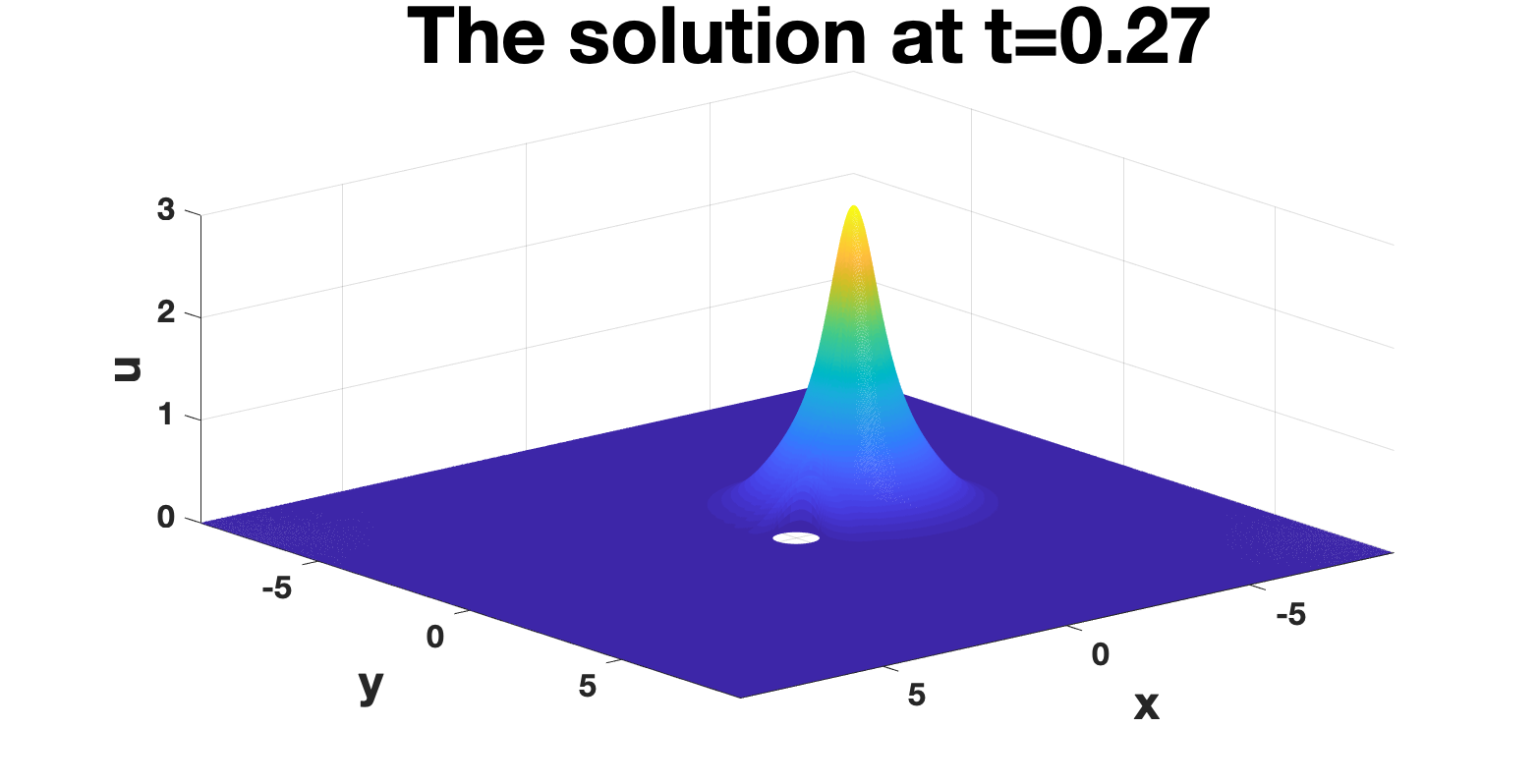}
\includegraphics[width=5.3cm,height=5.6cm]{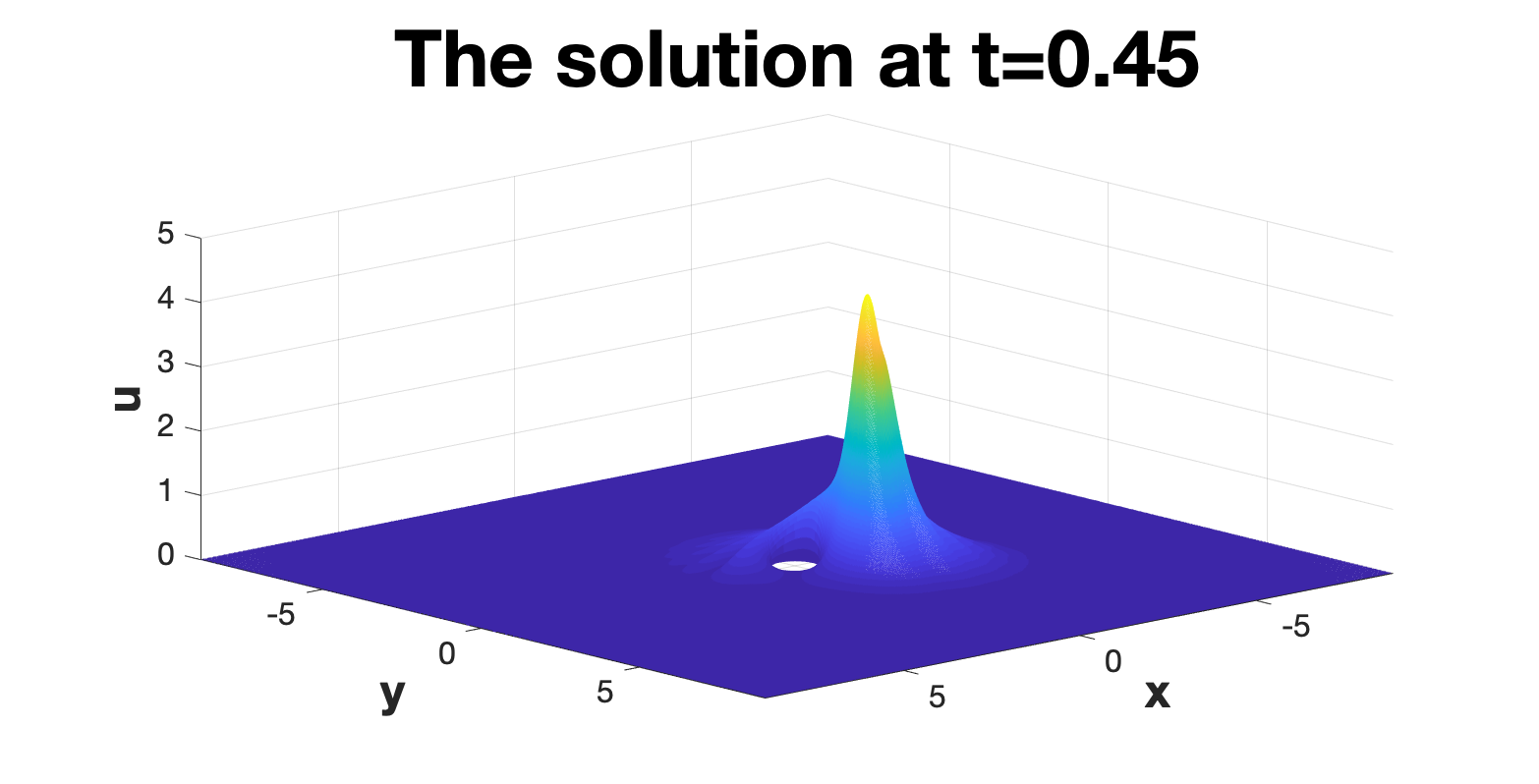}
\includegraphics[width=5.3cm,height=5.6cm]{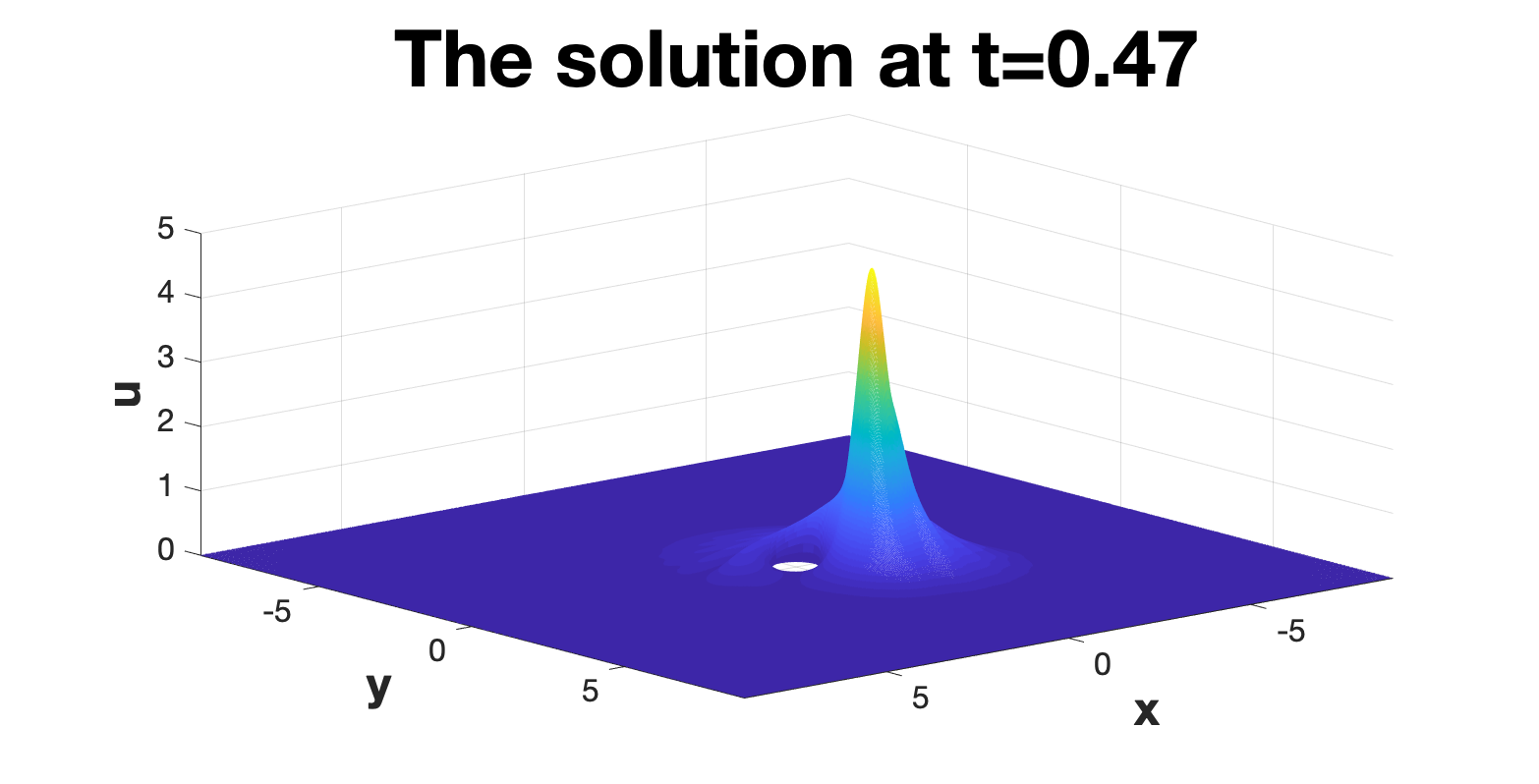}
\includegraphics[width=5.3cm,height=5.6cm]{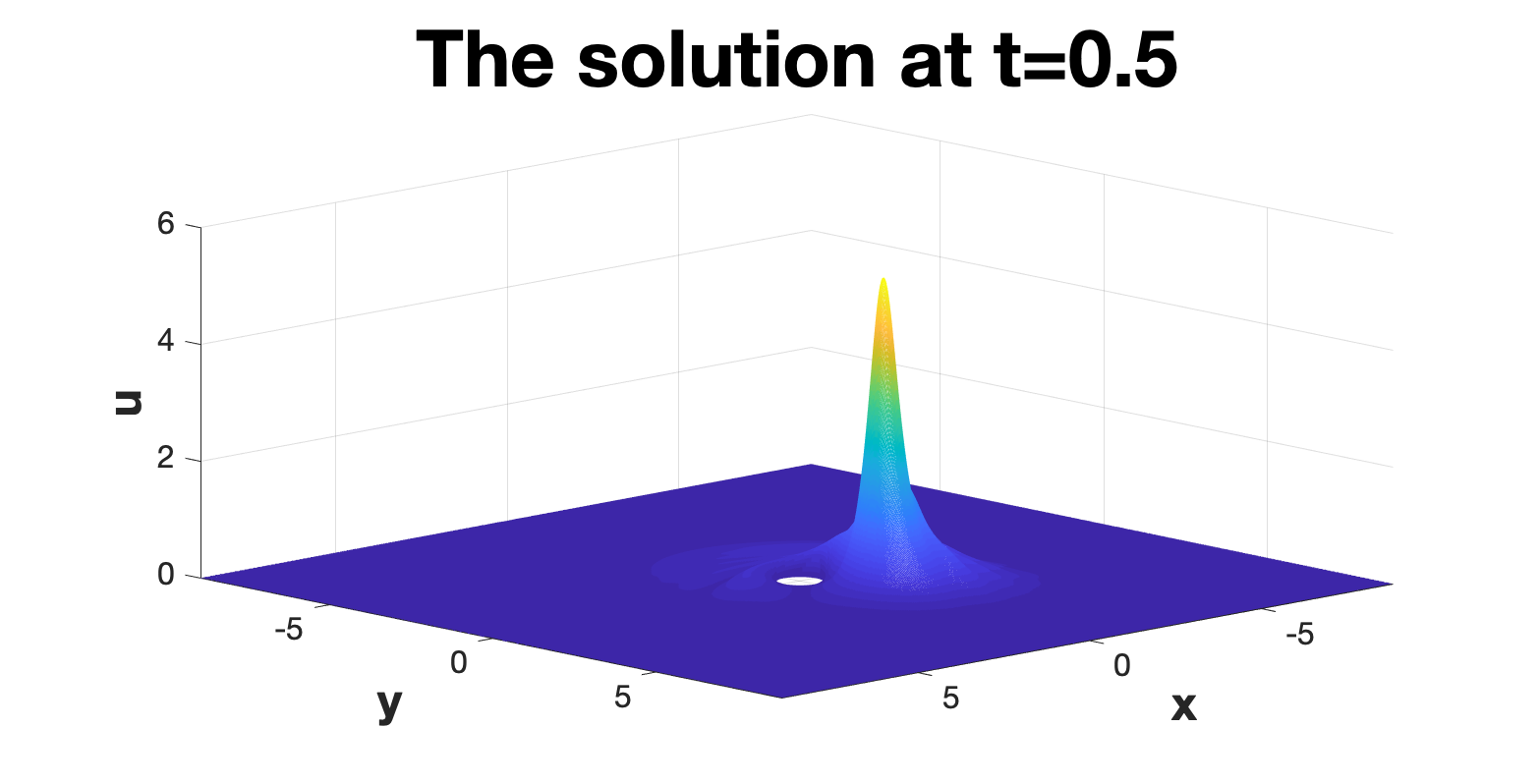}
\includegraphics[width=5.3cm,height=5.6cm]{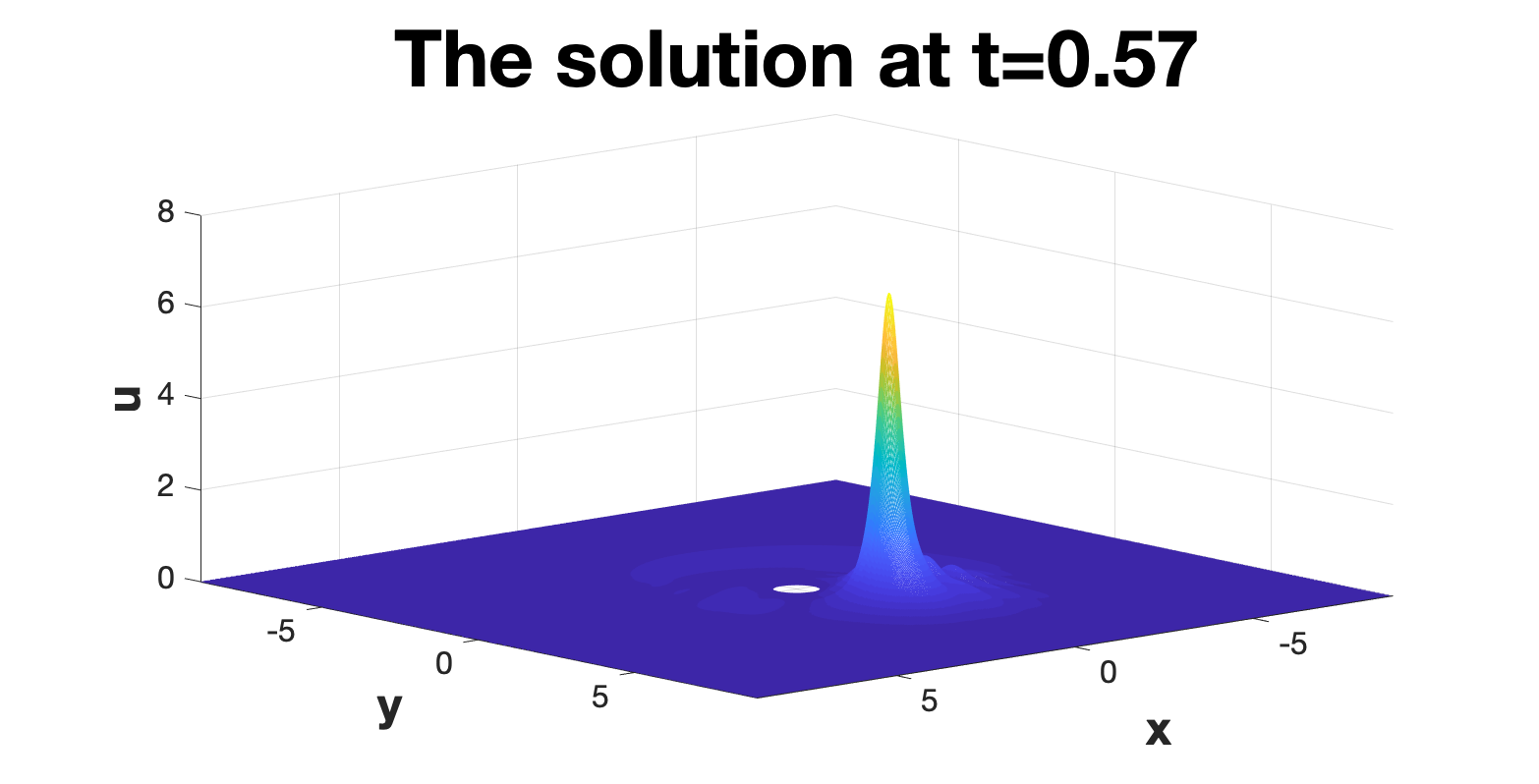}
\caption{Snapshots of the solution $u(t)$ to the $2d$ cubic \NNls equation with the initial condition $u_0$ as in \eqref{u0NLS}, $A_0=2.25, \; x_c=-4.5, \; y_c=-4.5$ and $v_2=(9,15),$ moving on the line $y= \frac 53 x$.}
\label{test6Critique-weak}
\end{figure}


\newpage
In our fourth example here, we consider the initial condition  $u_0$ from \eqref{u0NLS} with $A_0=2.25, \, x_c=-4.5,$ $y_c=-4.5$ and the velocity vector $v_1=( 12,15)$ and $v_2=(15,12).$ For $v_1=( 12,15),$ a snapshot of the solution at time $t = 1.2$  is plotted on the left of Figure \ref{test3Critique-weak}. 
\begin{figure}[ht]
\centering
\includegraphics[width=7.5cm,height=6.2cm]{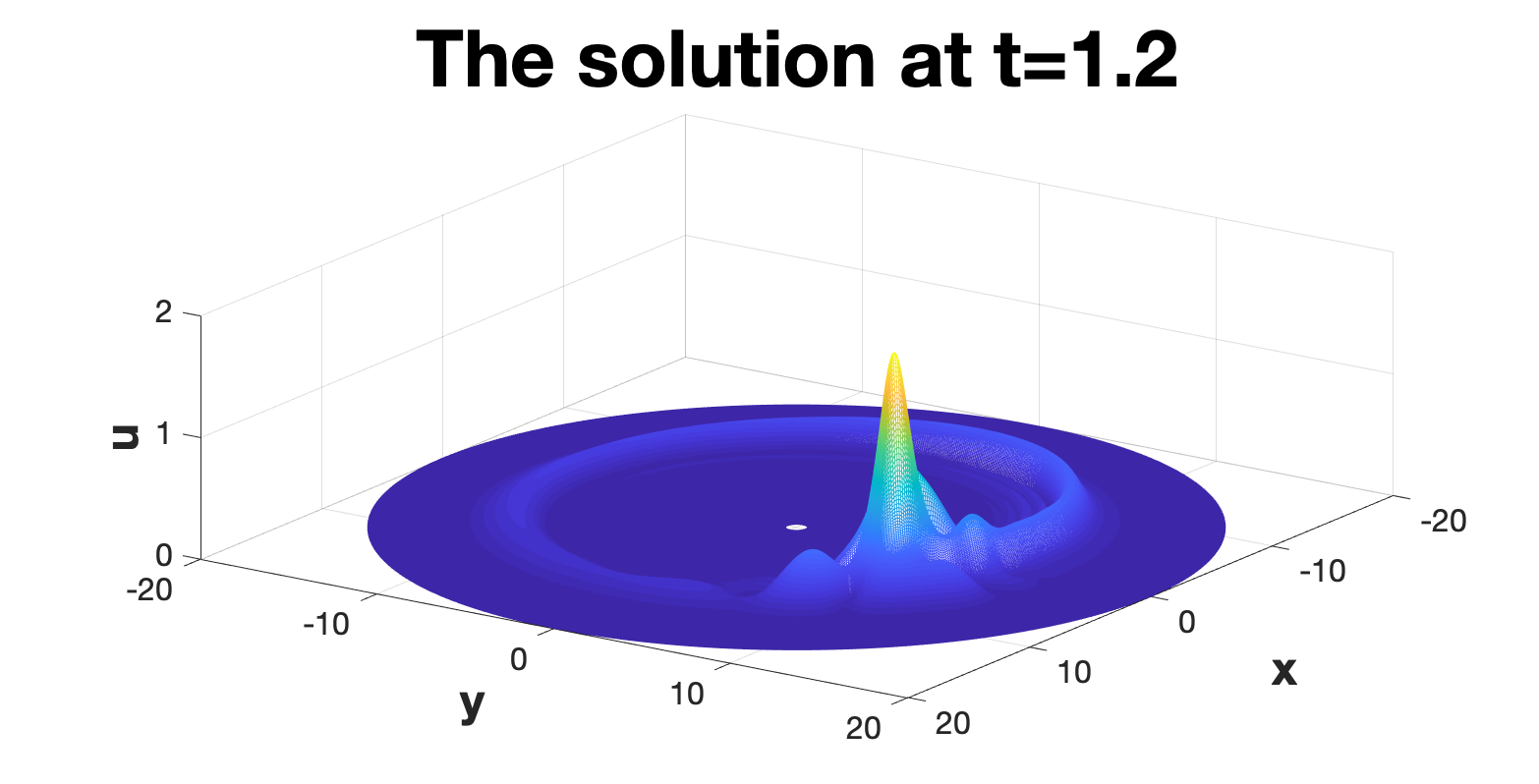}
\includegraphics[width=7.5cm,height=5.0cm]{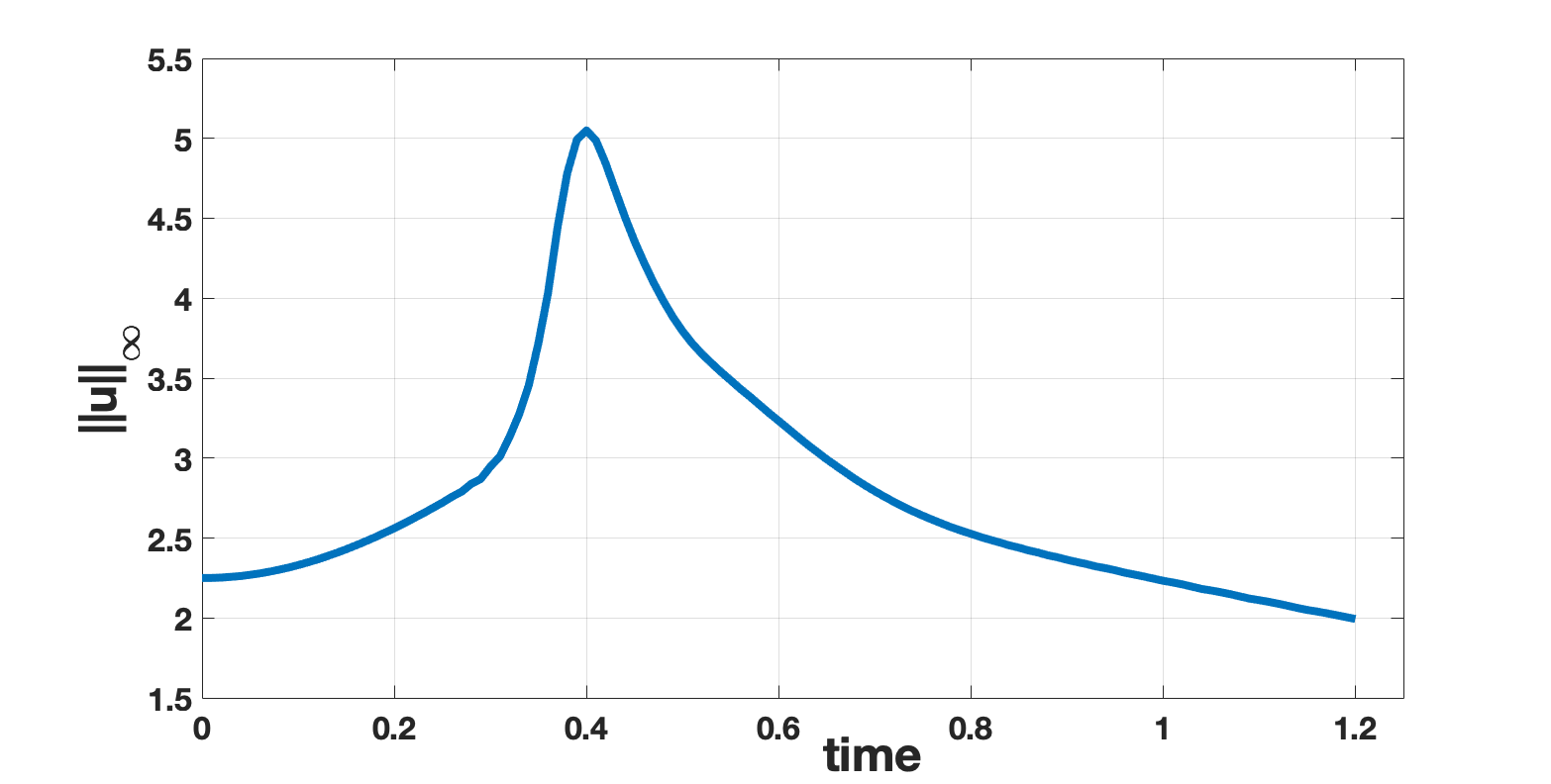}
\caption{Solution to the $2d$ cubic \NNls equation with the initial condition $u_0$ from \eqref{u0NLS}, $A_0=2.25, \, (x_c,y_c)=(-4.5,-4.5)$
and $v_1=(12,15)$ at time $t=1.2$ moving on the line $y=\frac 54 x$ (left), the time dependence of the $L^{\infty}$-norm of the solution (right). }
\label{test3Critique-weak}
\end{figure}
The right subplot shows the $L^{\infty}$-norm depending on time, which appears to grow quite fast in the beginning of the simulation, but after the collision it starts to decrease  monotonically. This solution disperses, or in other words, it becomes a scattering solution. Thus, the obstacle arrests the blow-up. This is a \textbf{different behavior} compared to the previous examples, where the solutions were transmitted almost with the same shape after the interaction and the soliton core was preserved. Unlike the previous examples, the collision of the solution with the obstacle here creates reflected waves, which then disperse the solution. The reflection causes the loss of the mass in the main part of the solution, arresting the blow-up in finite time unlike the examples above, where the reflection does not affect the blow-up of the solution and only delays the blow-up time. In this case the interaction between the soliton and the obstacle has a substantial influence on the behavior of the solution, which is a completely new dynamics compared to the dynamics on the whole space. For better understanding of this dynamics, we provide snapshots of the behavior of the solution for different time steps for $v_2=(15,12),$ see Figure \ref{test4Critique-weak}. This is an example where the solution has a behavior close to the \textit{strong} interaction case. 


\begin{figure}[!h]
\centering
\includegraphics[width=5.4cm,height=6cm]{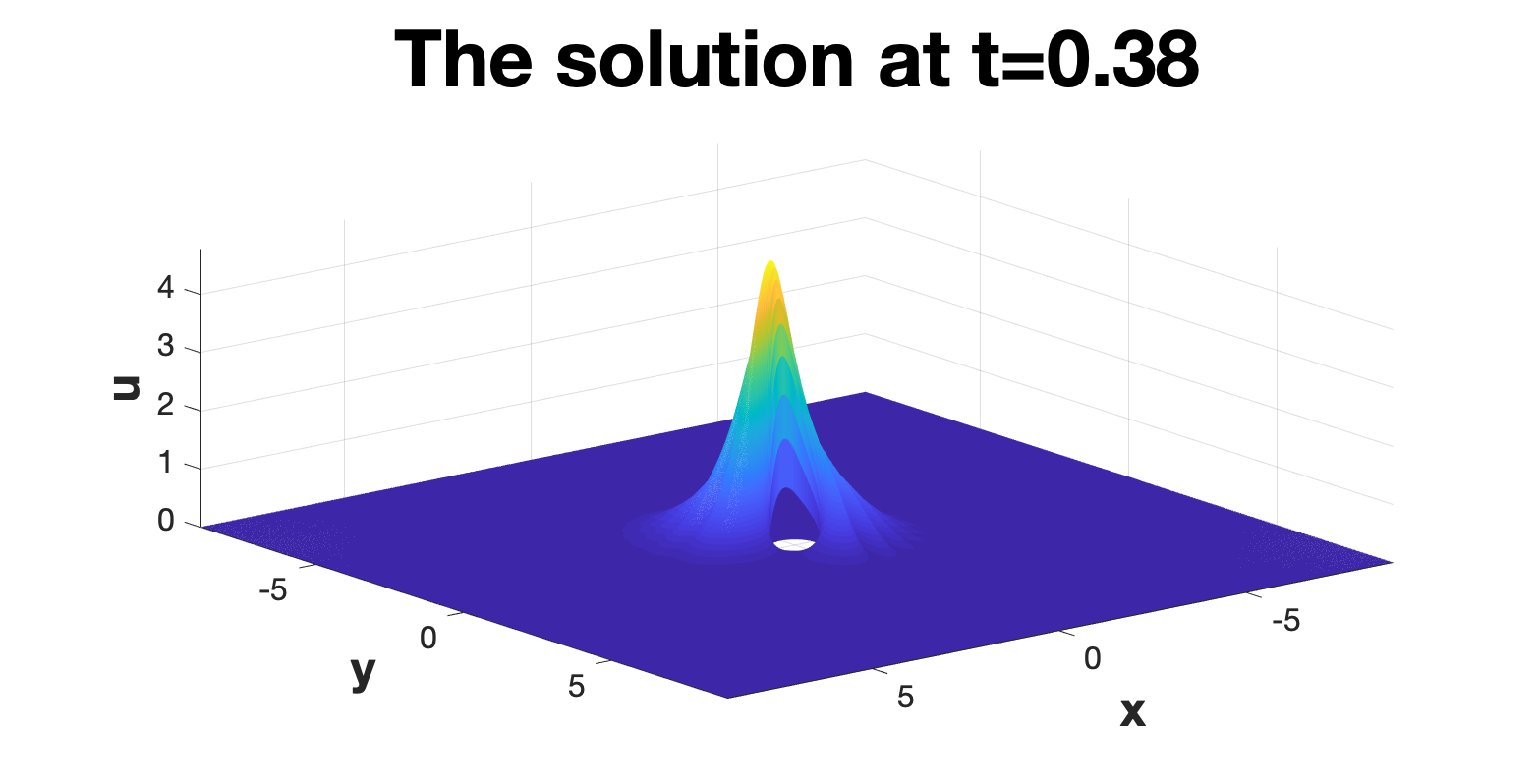}
\includegraphics[width=5.4cm,height=6cm]{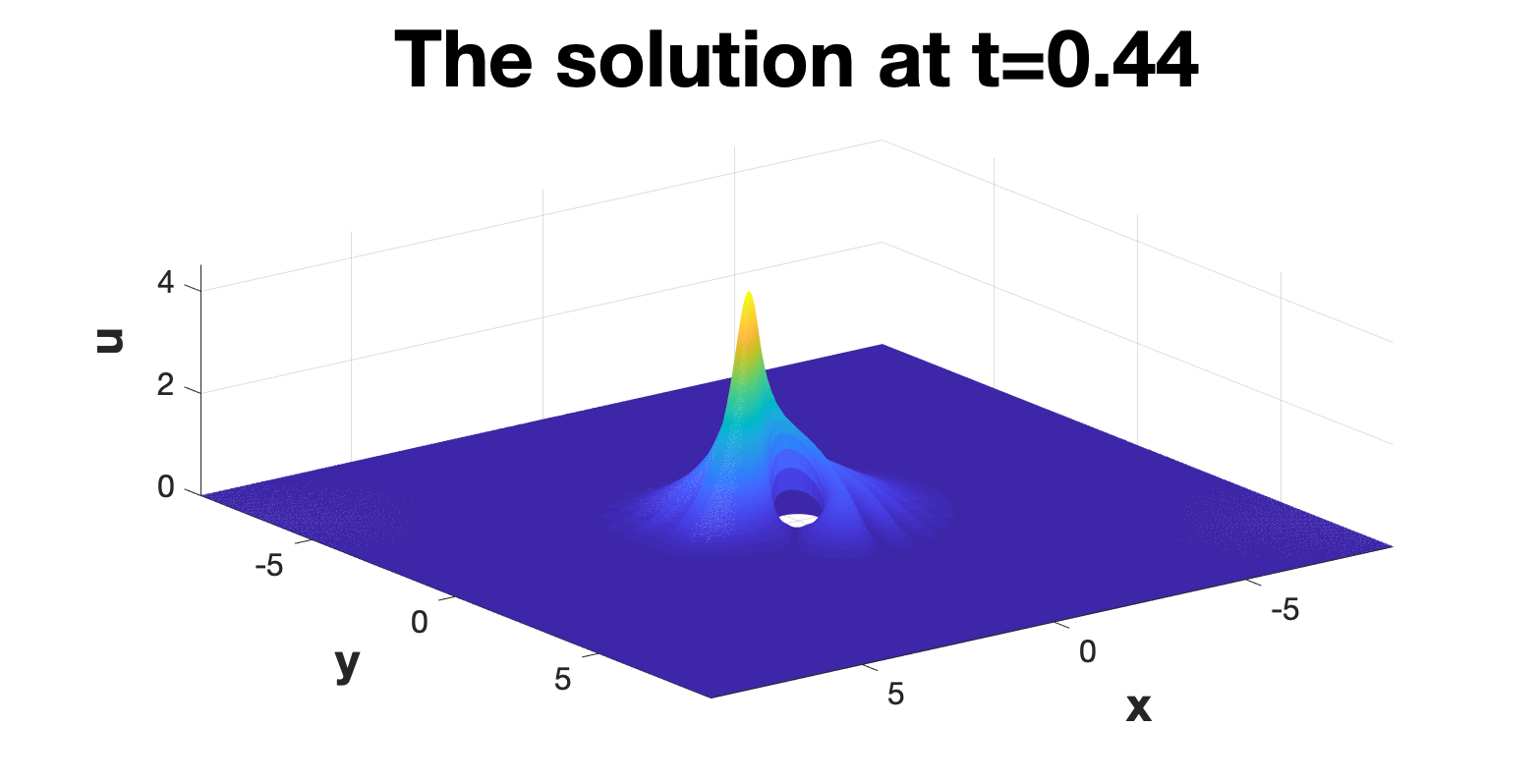}
\includegraphics[width=5.4cm,height=6cm]{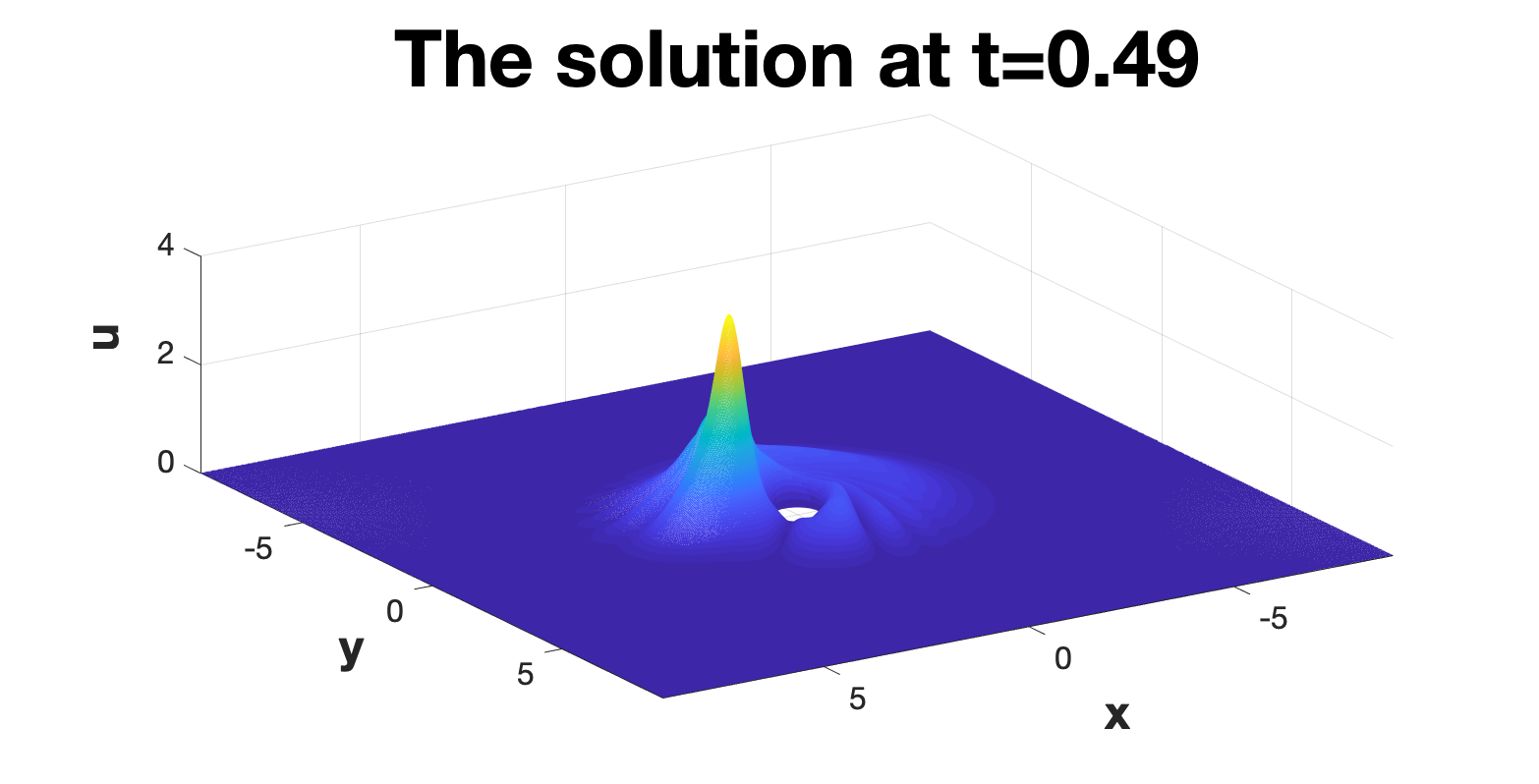}
 \includegraphics[width=5.4cm,height=6cm]{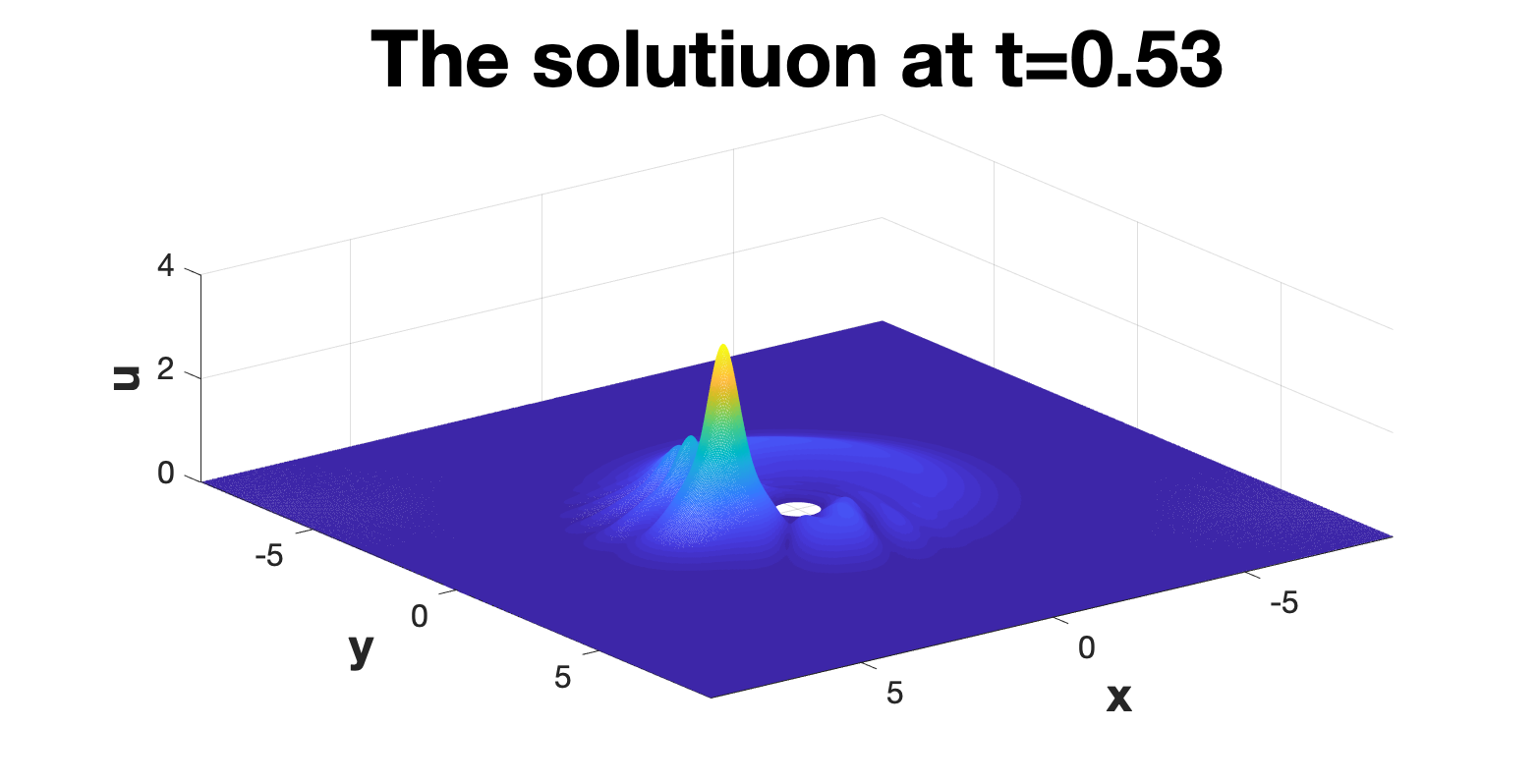}
\includegraphics[width=5.4cm,height=6cm]{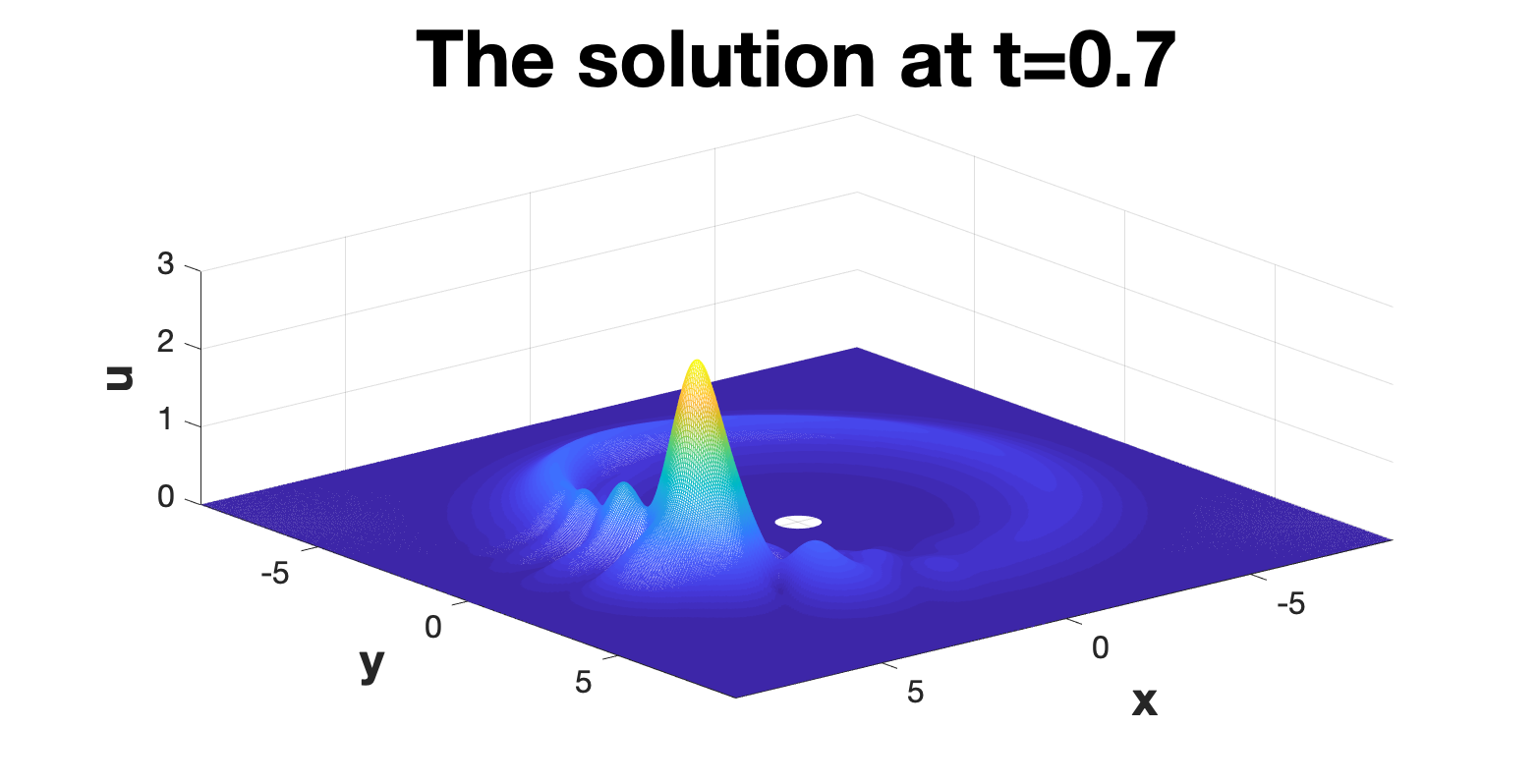}
\includegraphics[width=5.4cm,height=6cm]{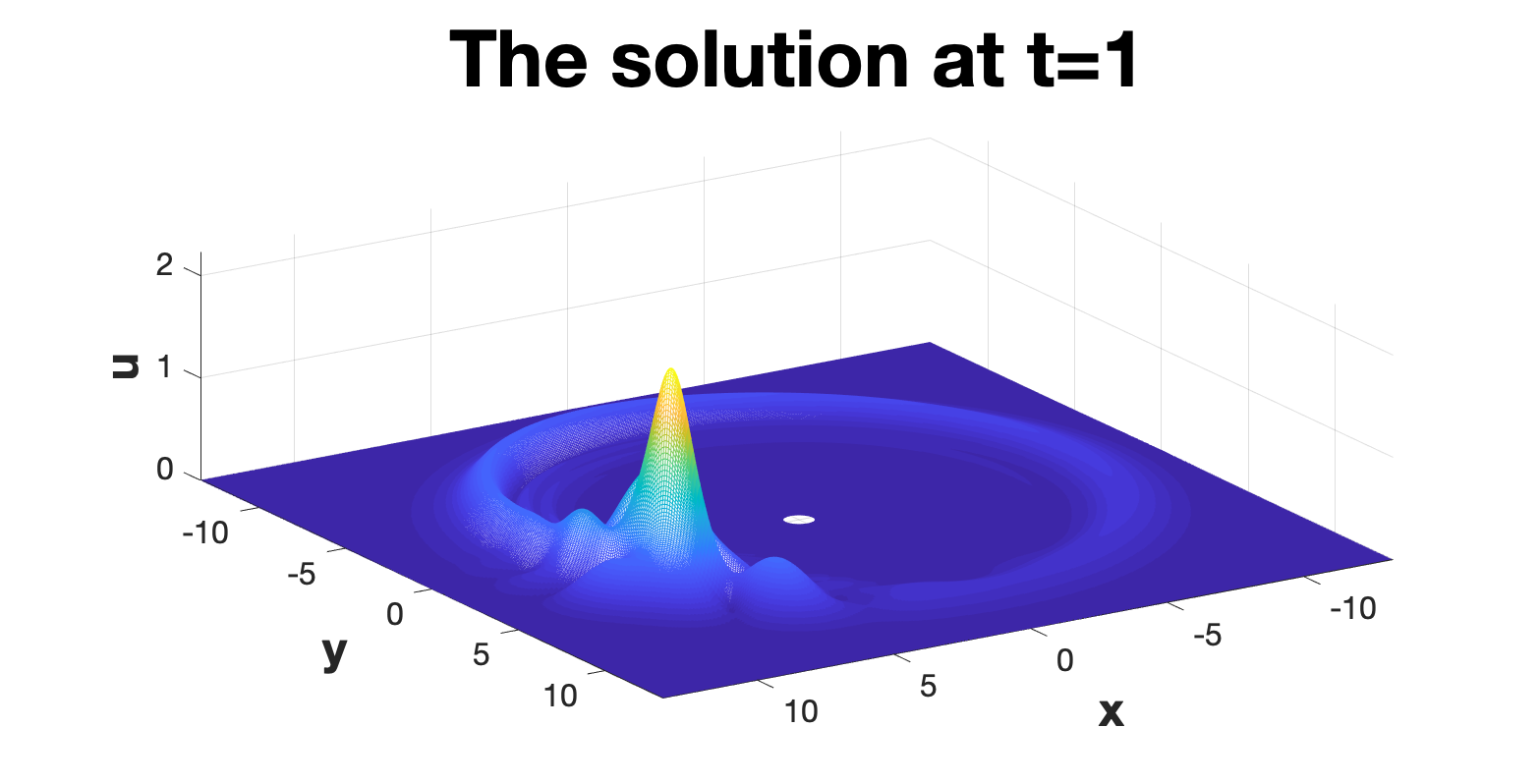}
\caption{Snapshots of the solution $u(t)$ to the $2d$ cubic \NNls equation, with the initial condition $u_0$ from \eqref{u0NLS}, $A_0=2.25, \, (x_c,y_c)=(-4.5,-4.5)$ and 
$v_2=(15,12)$ moving on the line $y= \frac 45 x$.}
\label{test4Critique-weak}
\end{figure}



In the following simulations, we study the behavior of various examples of solutions to the cubic \NNls equation, depending on the interaction between the solution and the obstacle. 
We consider the same initial condition as in \eqref{u0NLS}, we fix the following parameters as follows: 
 $ A_0=2.25, \; x_c=-4.5, \; y_c=-4.5$, and vary the direction of the velocity vector $(v_x,v_y). $  
 
We record the results of our simulations for this $L^2$-critical case in Table \ref{T:1}.  

{\small
 \begin{table}[h!]
 \centering 
\begin{tabular}{|c|c| c|l|c|} 
  \hline
   $(v_x, v_y)$     & Discrete mass &Discrete energy   &  Behavior of the solution                           & Type of interaction  \\
  \hline
  \hline
  $(15 ,0 )$          & 15.9043                       &  442.9353                    &  Blow-up at $t \approx  0.52\; \;$              &  no-interaction  \\  
  \hline
 $( 0,15 )$           &  15.9043                      &  442.9353                    &  Blow-up at $t \approx 0.52\;\;$                & no-interaction                           \\  
  \hline
 $( 15,8 )$          & 15.9043                        &  570.4814                     &  Blow-up at $t \approx 0.56\;\;$                &  weak-interaction       \\  
  \hline
  $(8,15 )$          & 15.9043                        &  570.4814                     &  Blow-up at $t \approx  0.56\;\; $                 & weak-interaction    \\    
  \hline 
$(9,15 )$            & 15.9043                        &   604.0182                     &  Blow-up up $t \approx 0.57\; $               & weak-interaction    \\  
  \hline
 $(15,9 )$           &  15.9043                      &   604.0182                    &  Blow-up at $t \approx 0.57\; \;$              & weak-interaction        \\  
  \hline
 $(10,15)$          &  15.9043                       &   641.4737                     &  Blow-up at $t \approx 0.63\; \;$         & weak-interaction   \\  
  \hline
 $(15,10)$           &  15.9043                    &  641.4737                       &  Blow-up at $t \approx 0.63\; \;$         & weak-interaction \\  
  \hline
  $(15,12)$           &  15.9043                    & 728.1404                   &  Scattering $\qquad \quad \quad \,$        & weak-interaction \\  
  \hline
 $(12,15)$           &  15.9043                    & 728.1404                   &  Scattering $\qquad \quad \quad \,$        & weak-interaction \\  
  \hline
   $(15 , 15 )$        & 15.9043                       &  887.5015                &  Scattering $\qquad \quad \quad \,$       & strong-interaction                   \\
   \hline
\end{tabular}
\caption{Different velocity directions $\vec{v}=(v_x,v_y)$ and the corresponding behavior of the solution $u(t)$ to the $2d$ cubic \Nls with the indicated discrete mass and energy (the value of energy differs due to the phase).}
\label{T:1}
\end{table}
}

\subsection{The $L^2$-supercritical case}
\label{L2supercritWeak}
We now consider the $2d$ quintic \NNls equation $(p=5).$ Again, we try to carefully examine the interaction between the obstacle and the solution.  In the following simulations, we fix $A_0$ and the velocity $v$, but vary  the translation parameters. 
We start with an example, where there is no interaction in order to compare the behavior of the solution for different scenarios later, especially when there will be a strong interaction. For that, we consider the initial data $u_0$ from \eqref{u0NLS} with 
\begin{equation}
\label{A_0=1.25+y_c=5}
A_0=1.25, \quad  x_c=-4.5 , \quad  y_c=5 ,\quad   \text{ and} \quad v=(15,0),
\end{equation} 
\begin{figure}[!ht]
\centering
\includegraphics[width=5.4cm,height=6cm]{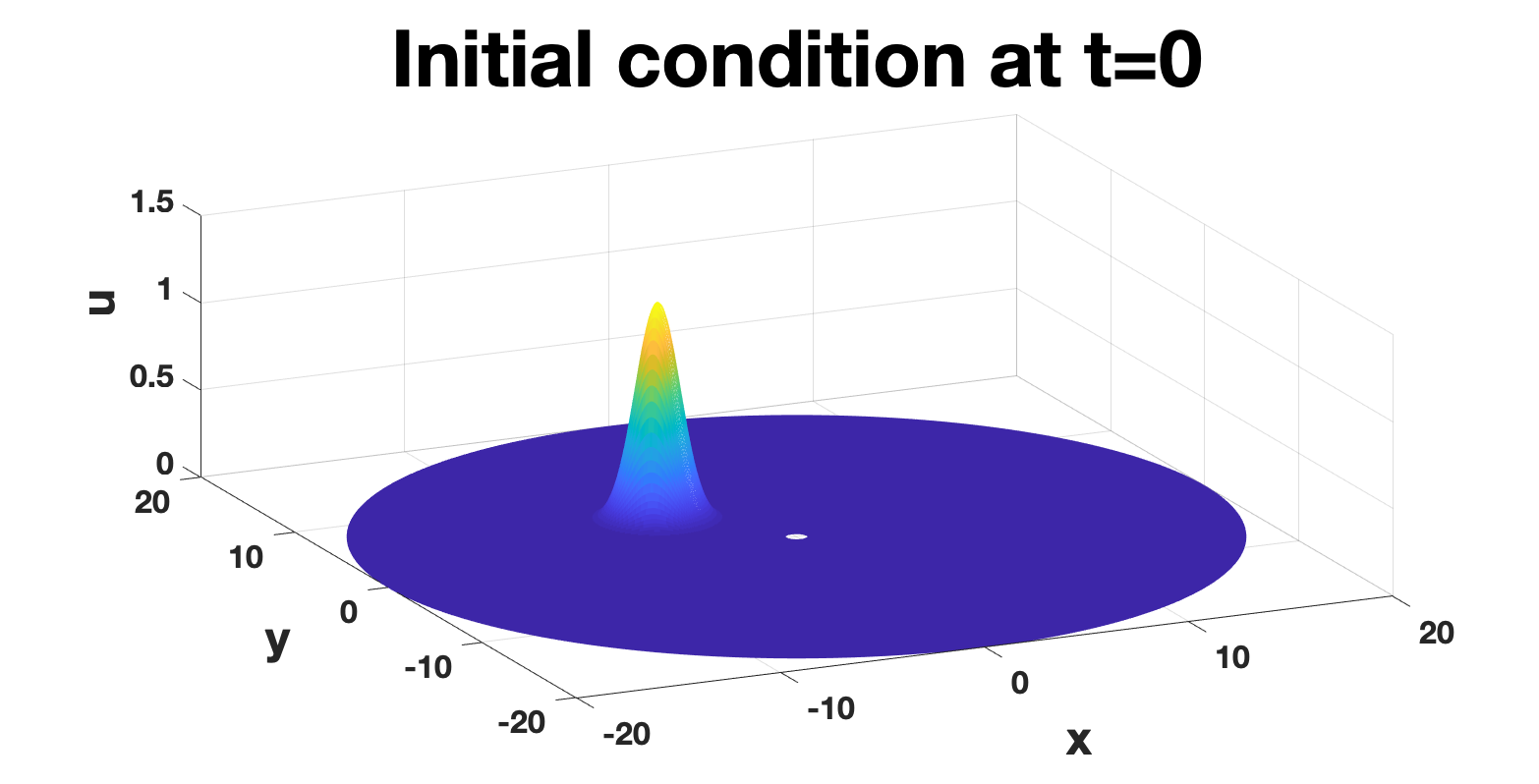}
  \includegraphics[width=5.4cm,height=6cm]{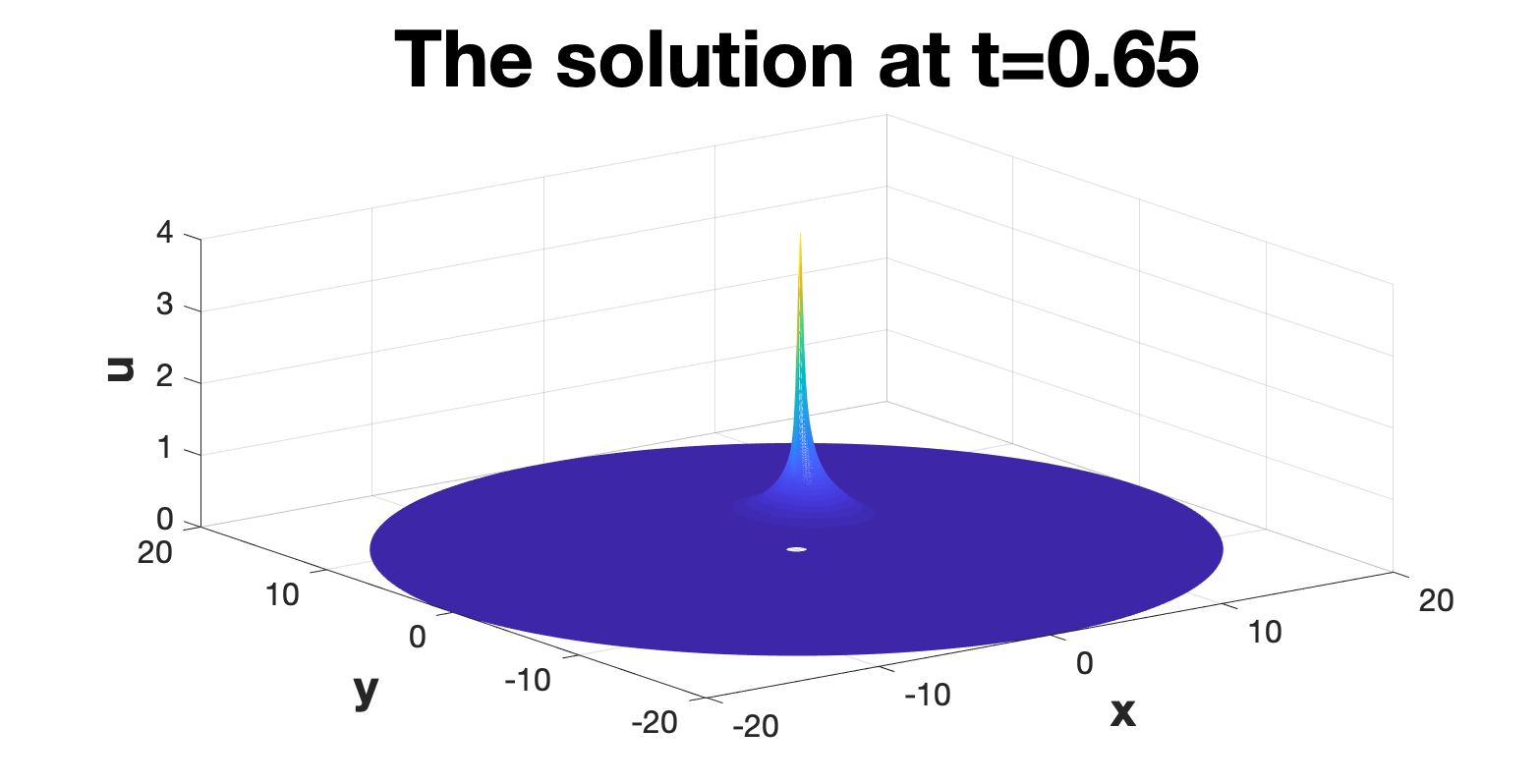}
   \includegraphics[width=5.4cm,height=5cm]{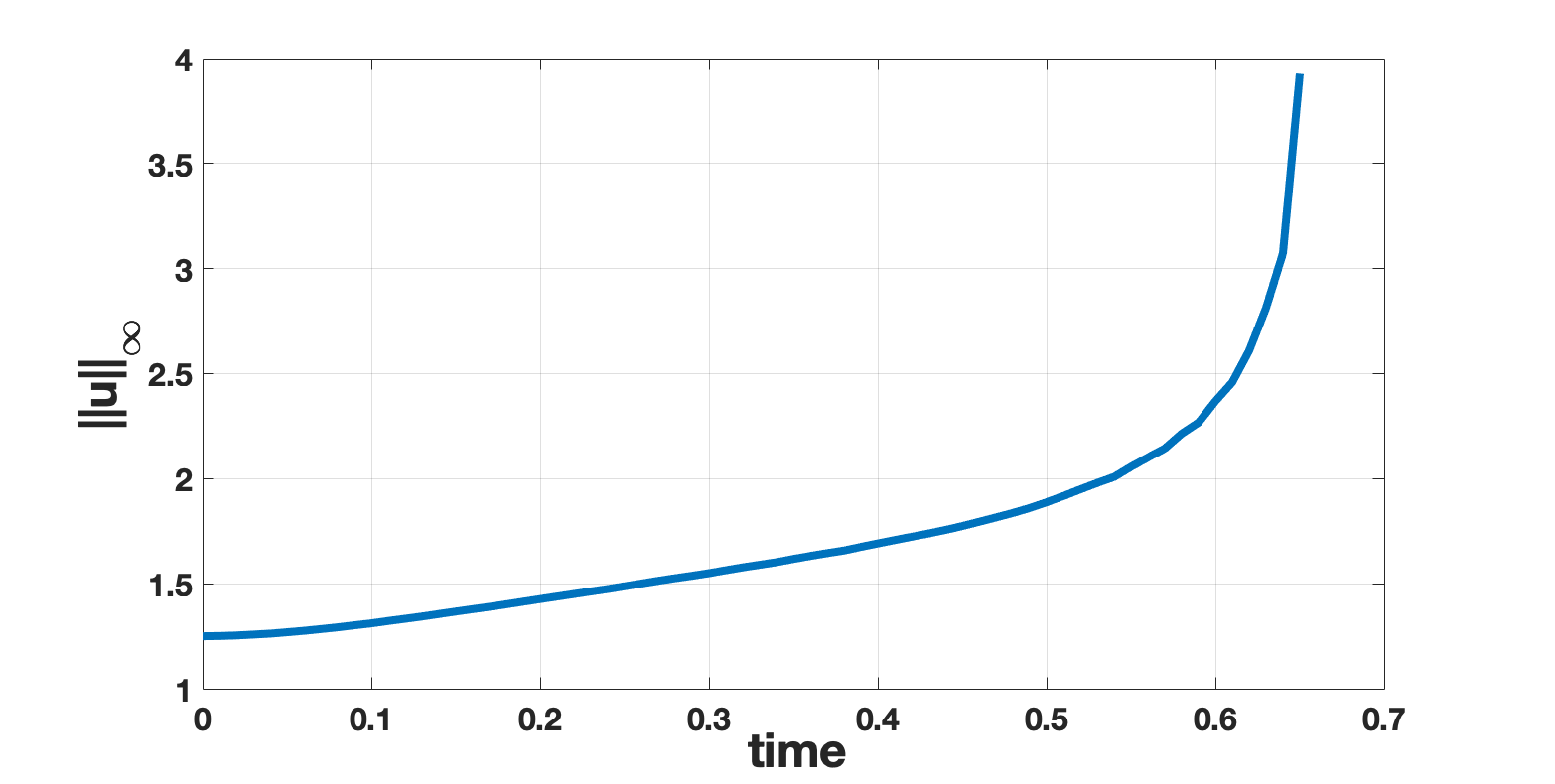}
\caption{Solution to the $2d$ quintic \NNls equation with $u_0$ from \eqref{u0NLS} and \eqref{A_0=1.25+y_c=5} (left) close to blow-up time (middle); time dependence of the $L^{\infty}$-norm (right).}
\label{u1-no-interac}
\end{figure}
which can be seen on the left of Figure \ref{u1-no-interac}. The middle subplot shows that the corresponding solution to the $2d$ quintic \NNls equation
blows up in finite time at $t=0.65$ with the diverging $L^{\infty}$-norm. Snapshots of the solution in time (top view onto the $xy$-plane) are plotted in Figure \ref{u1-no-interac-90}. We observe that the solution blows up in finite time and there is no interaction between the solution and the obstacle.  

\begin{figure}[!h]
\centering
\includegraphics[width=5.4cm,height=5.4cm]{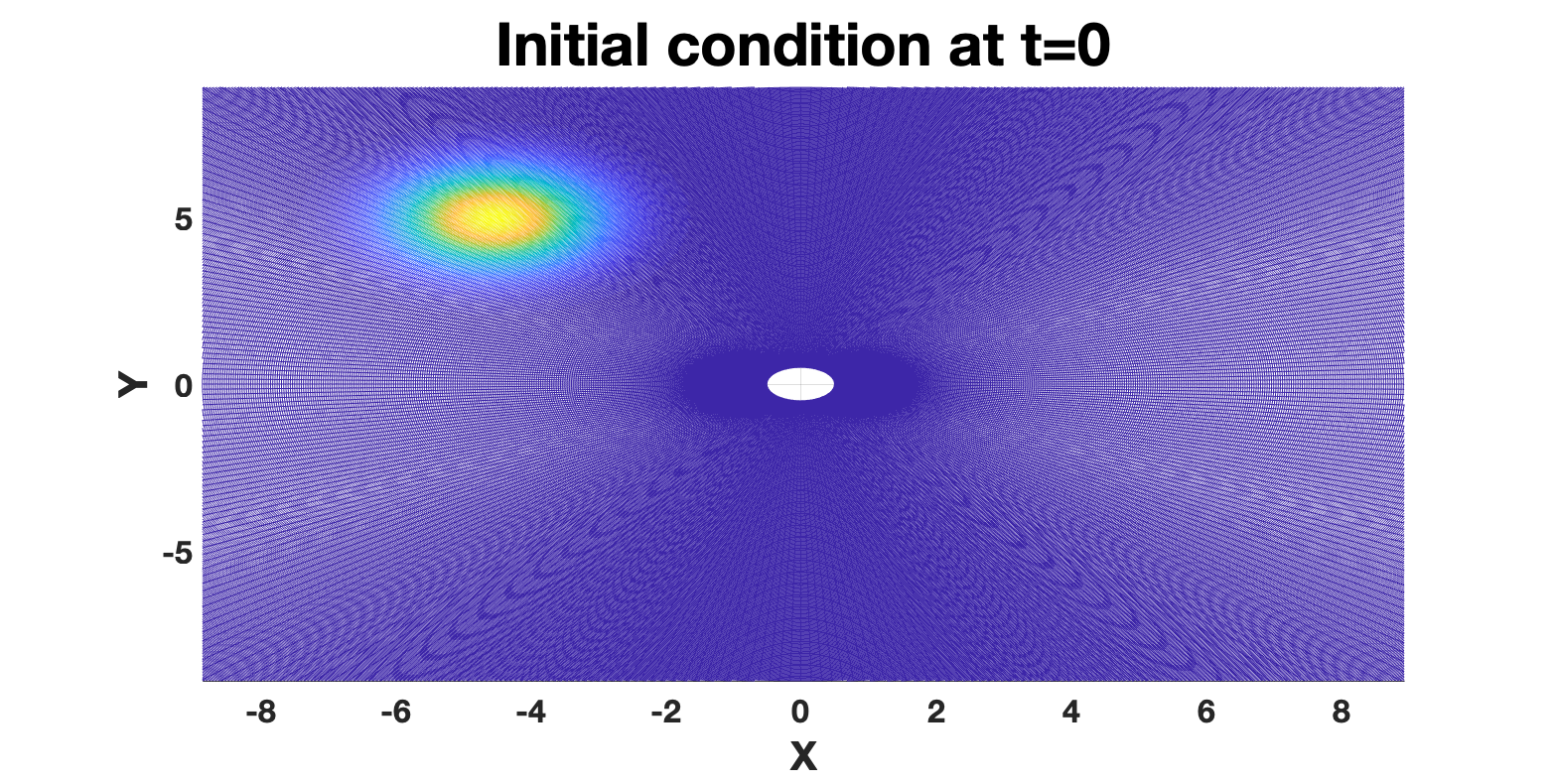}
\includegraphics[width=5.4cm,height=5.4cm]{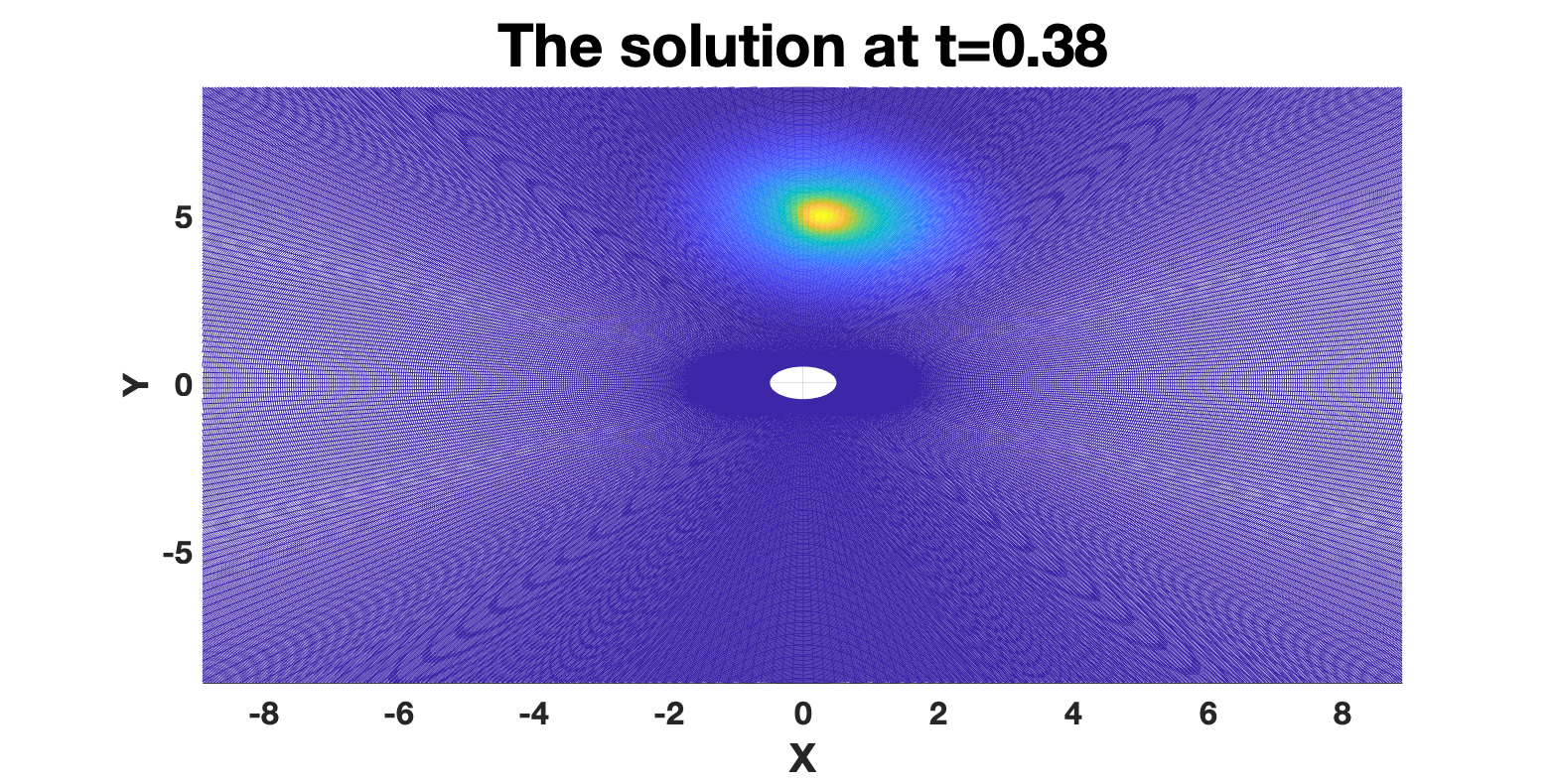}
\includegraphics[width=5.5cm,height=5.2cm]{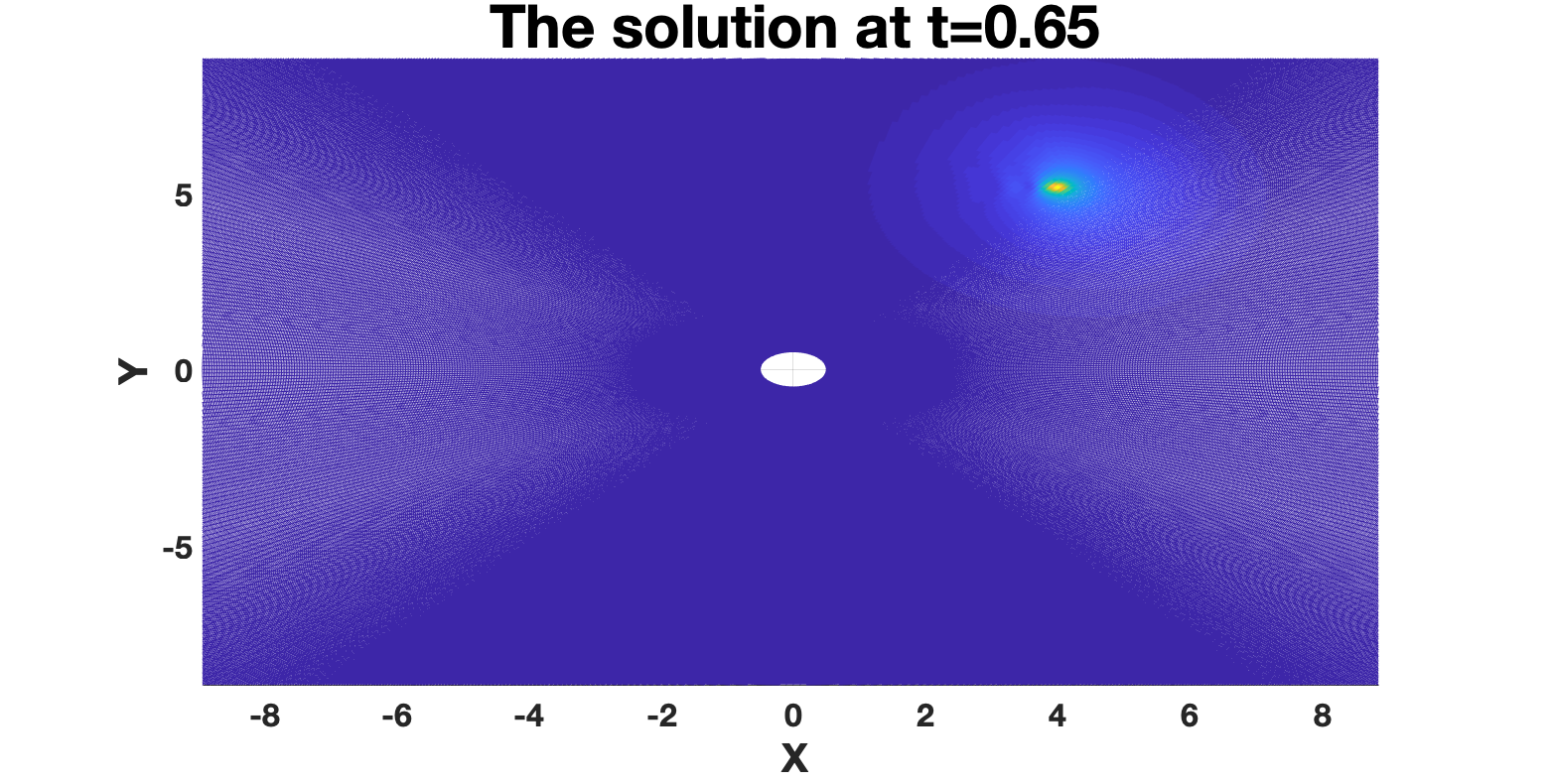}
\caption{Snapshots of the evolution of $u_0$ from \eqref{u0NLS} with \eqref{A_0=1.25+y_c=5} at $t=0,\;  t=0.38$ and $t=0.65.$ }
\label{u1-no-interac-90}
\end{figure}

\newpage

Next, we take the same initial data $u_0$ as in the previous example (i.e., $A_0=1.25, \, v=(15,0),\, x_c=-4.5$) except for the $y_c$ value we choose $y_c=2$ as shown in Figure \ref{NointeraclVariabY0}. In this case, we expect that the traveling wave solution has some weak interaction with the obstacle, see Figure \ref{Blow-up-t=0.66-WI}.     

\begin{figure}[!h]
\centering
\includegraphics[width=8.7cm,height=6.6cm]{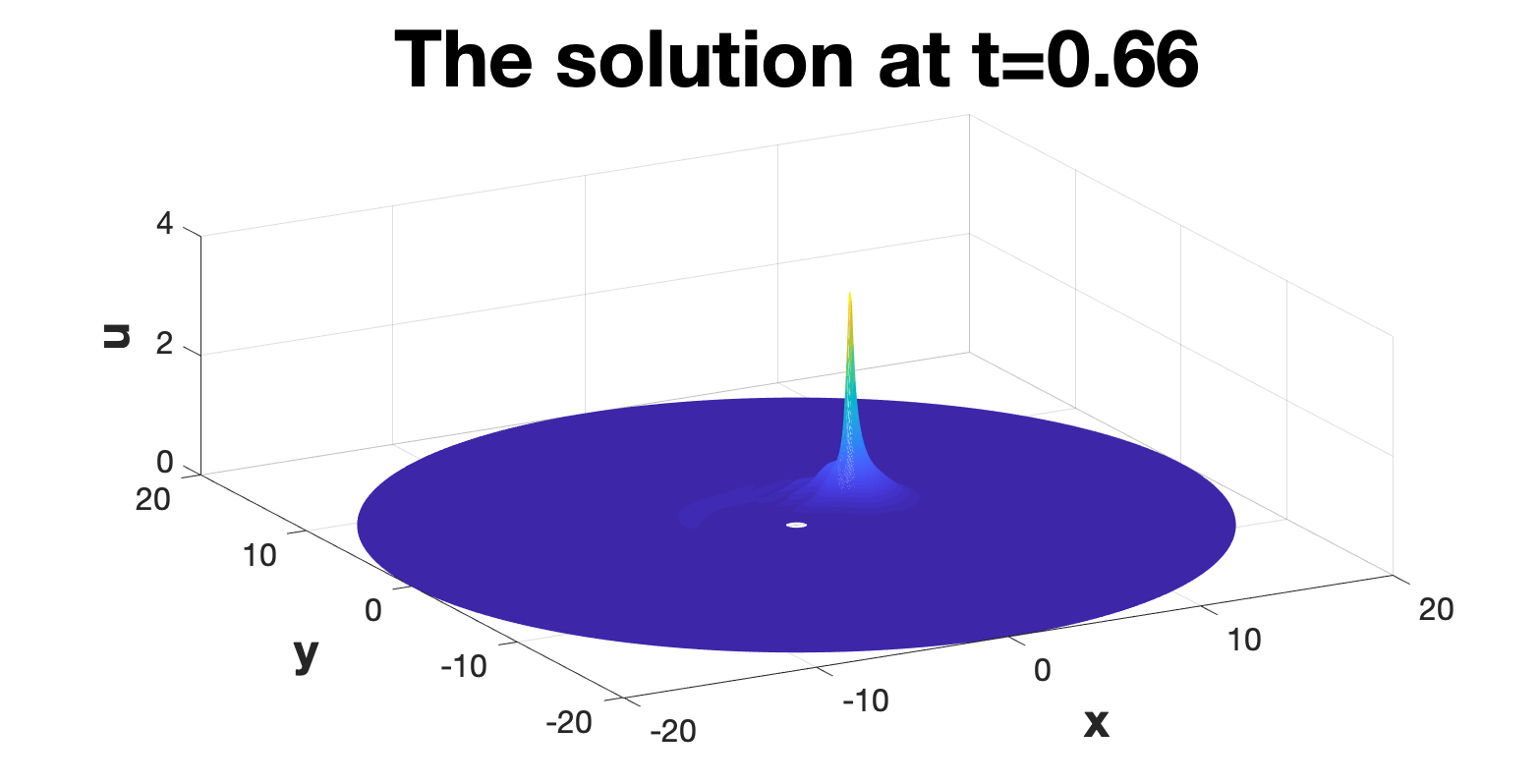}
\includegraphics[width=6.4cm,height=5.2cm]{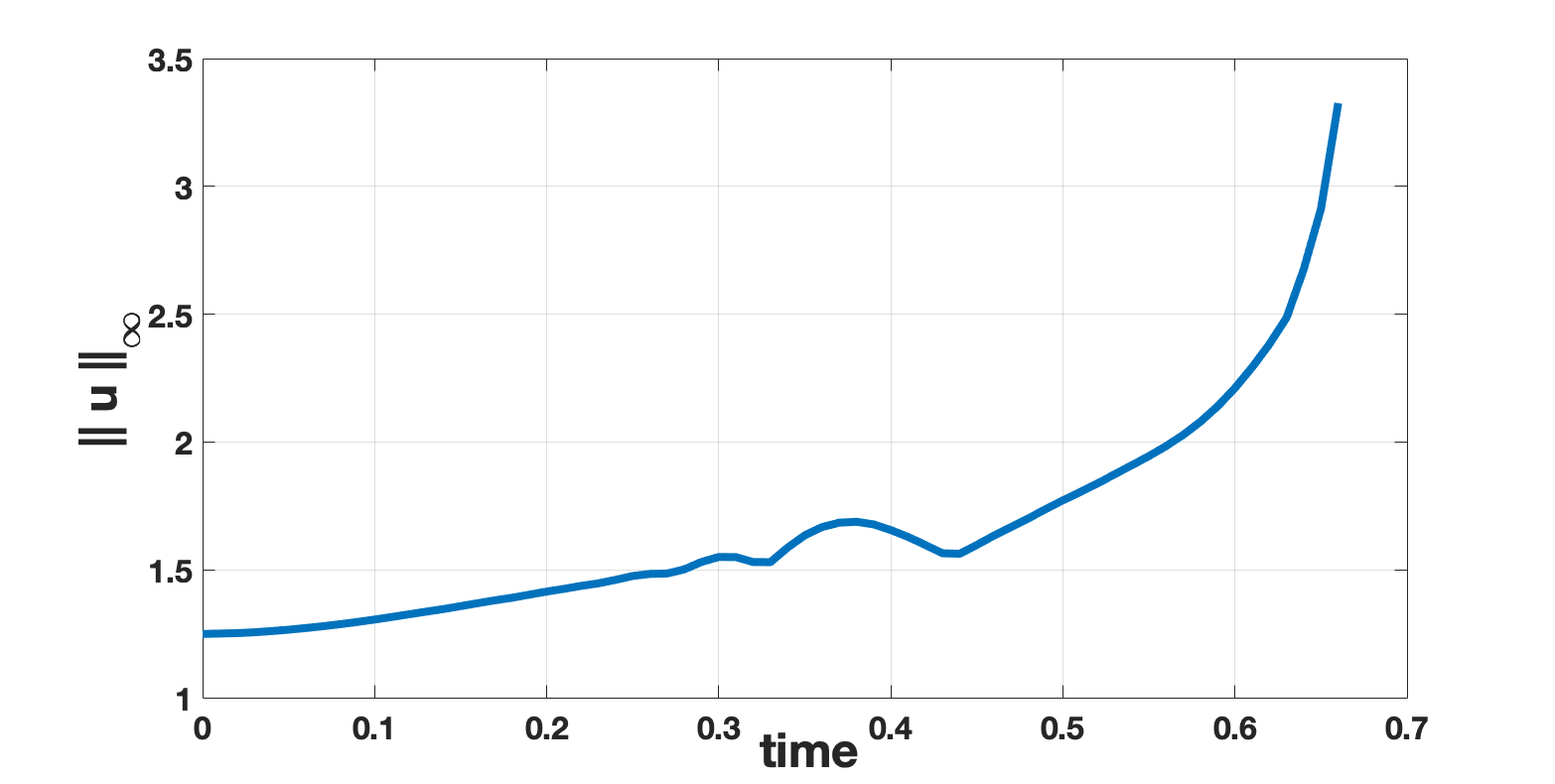}
\caption{Solution $u(t)$ to the $2d$ quintic \NNls equation with initial condition $u_0$ from \eqref{u0NLS}, where $A_0=1.25$, $v=(15,0)$ and $(x_c,y_c)=(-4.5,2)$. A snapshot of $u(t)$ at $t=0.66$ (left), the time dependence of the $L^{\infty}$-norm (right).}
\label{Blow-up-t=0.66-WI}
\end{figure}

We observe that with this weak interaction, the solution still blows up in finite time at $t=0.66$, but the blow-up time is delayed compared to the case, where there was no interaction between the solution and the obstacle, compare Figures \ref{SolutionNointeraclVariabY0} and \ref{Blow-up-t=0.66-WI}.  Moreover, we observe a slight perturbation of the growth in the $L^{\infty}$-norm: at the collision, the amplitude of the solution starts decreasing but after the weak interaction, the solution is back to the concentration leading to the blow-up. This can be explained by the appearance of small reflected waves after the collision, which scatter at the end of the simulation. They can be seen in the snapshots of the solution in Figure \ref{Snap-of-Blow-up-t=0.66-WI} with the view onto the $xy$-plane and zooming near the obstacle.      

\begin{figure}[!ht]
\centering
\includegraphics[width=5.4cm,height=5.4cm]{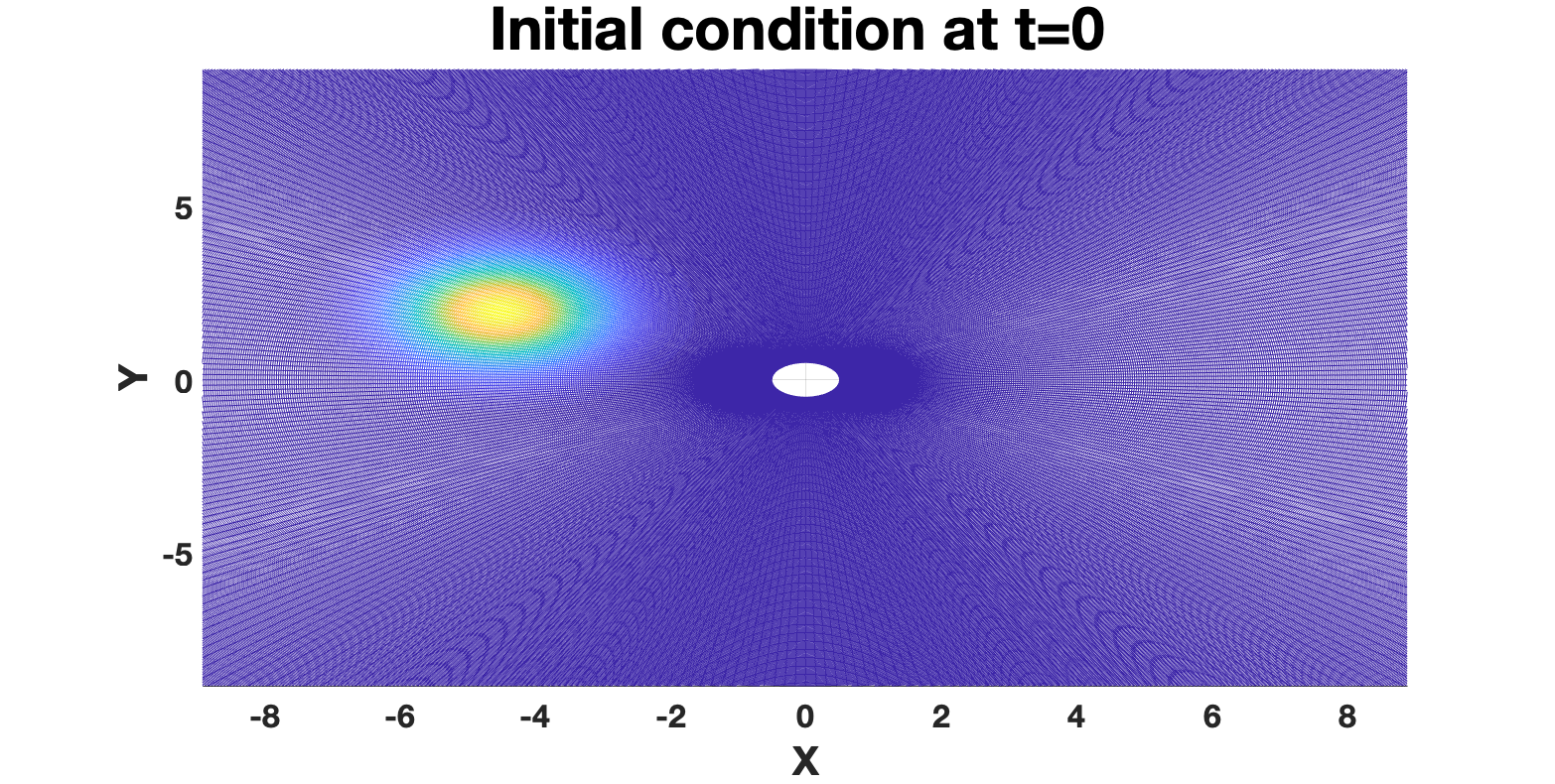}
\includegraphics[width=5.4cm,height=5.4cm]{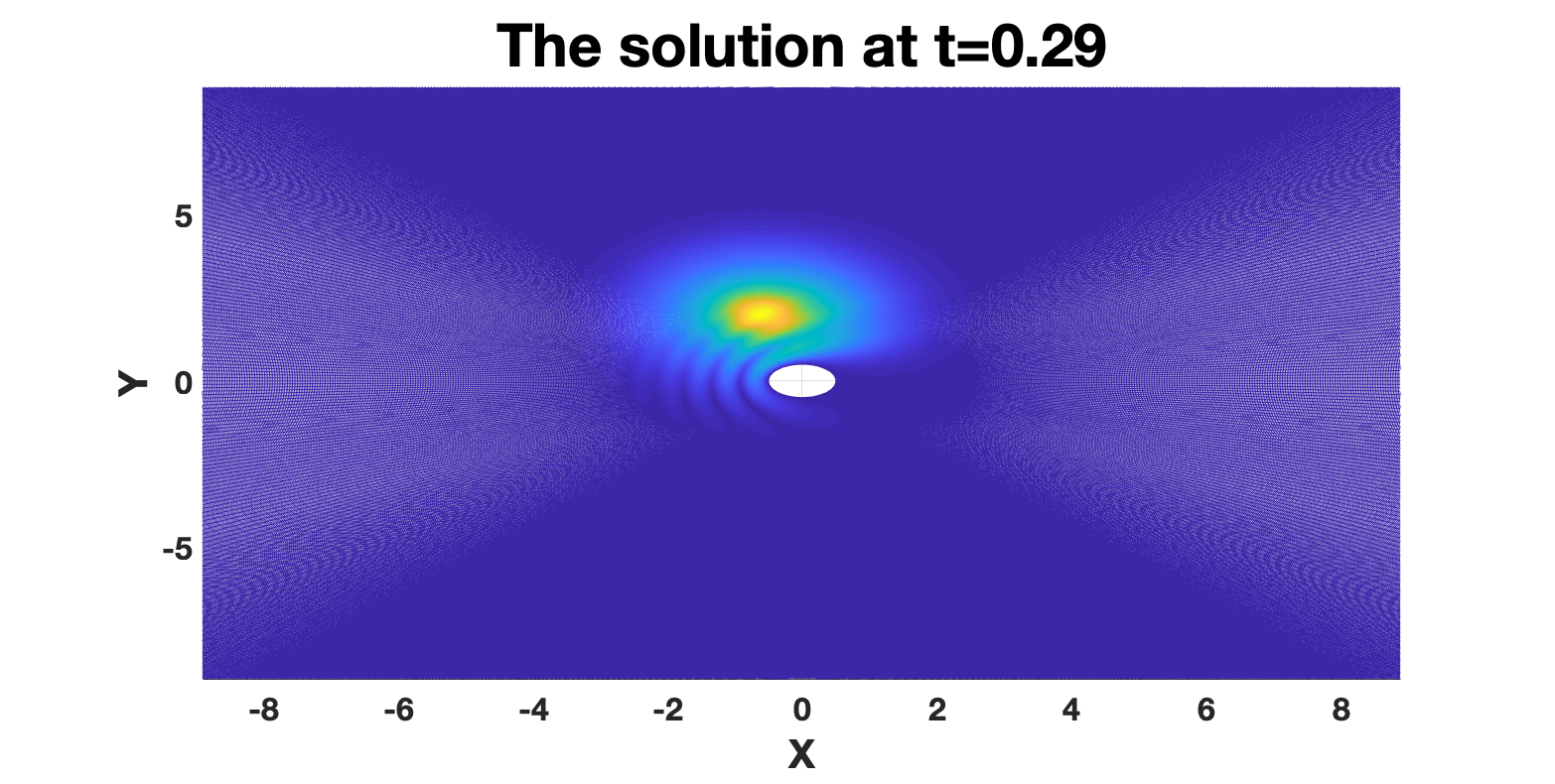}
\includegraphics[width=5.4cm,height=5.4cm]{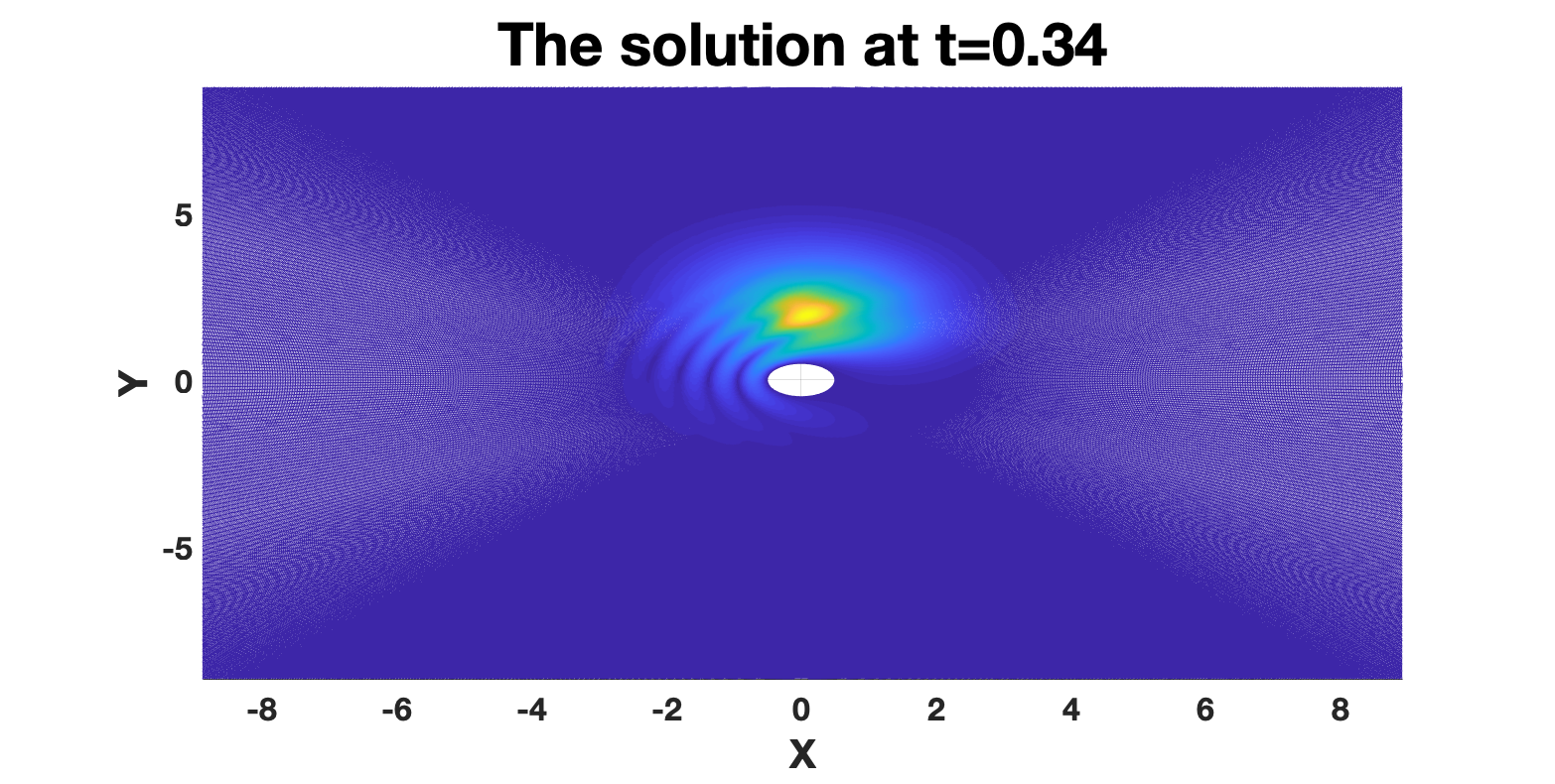}
\includegraphics[width=5.4cm,height=5.3cm]{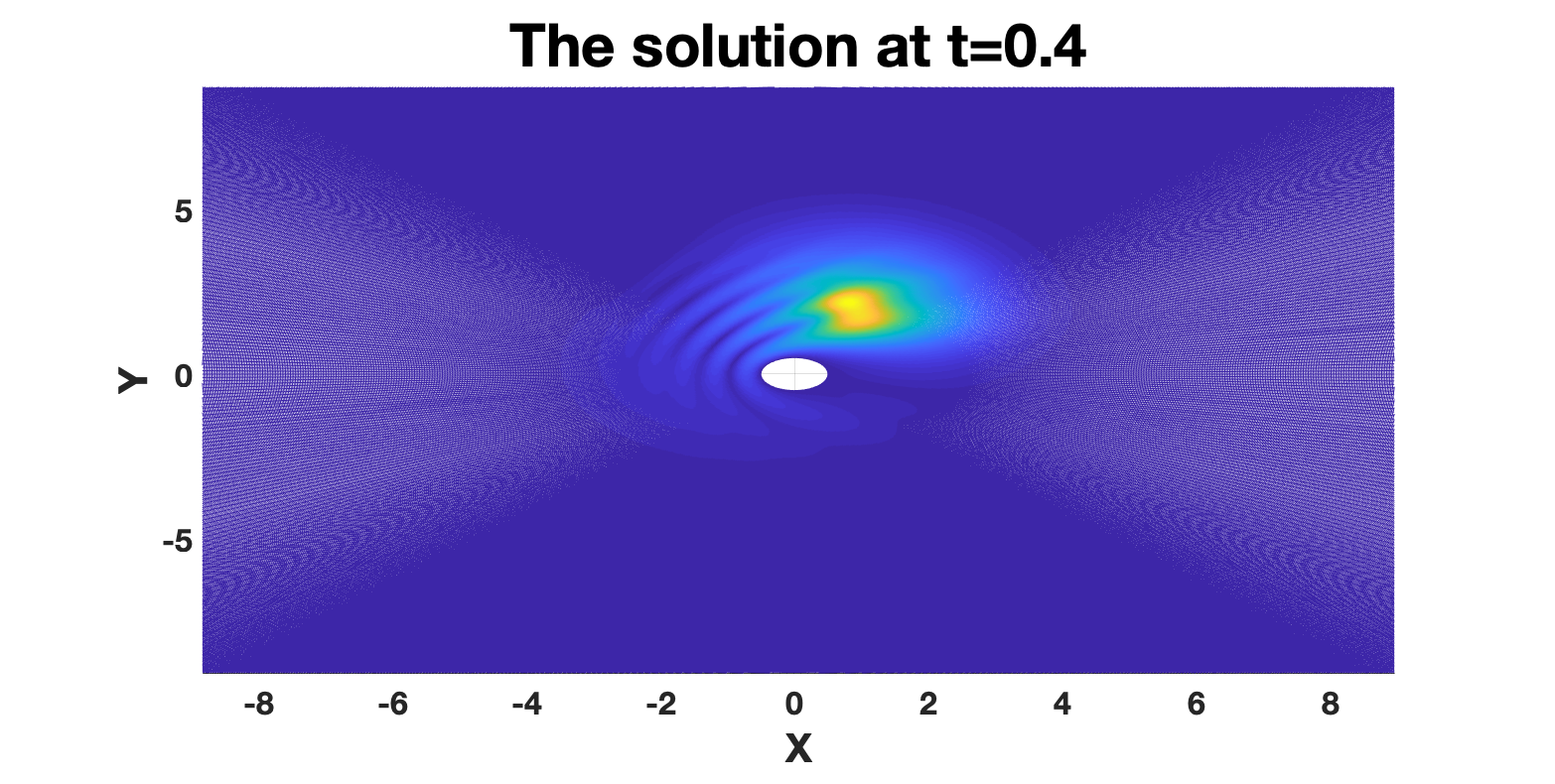}
\includegraphics[width=5.4cm,height=5.3cm]{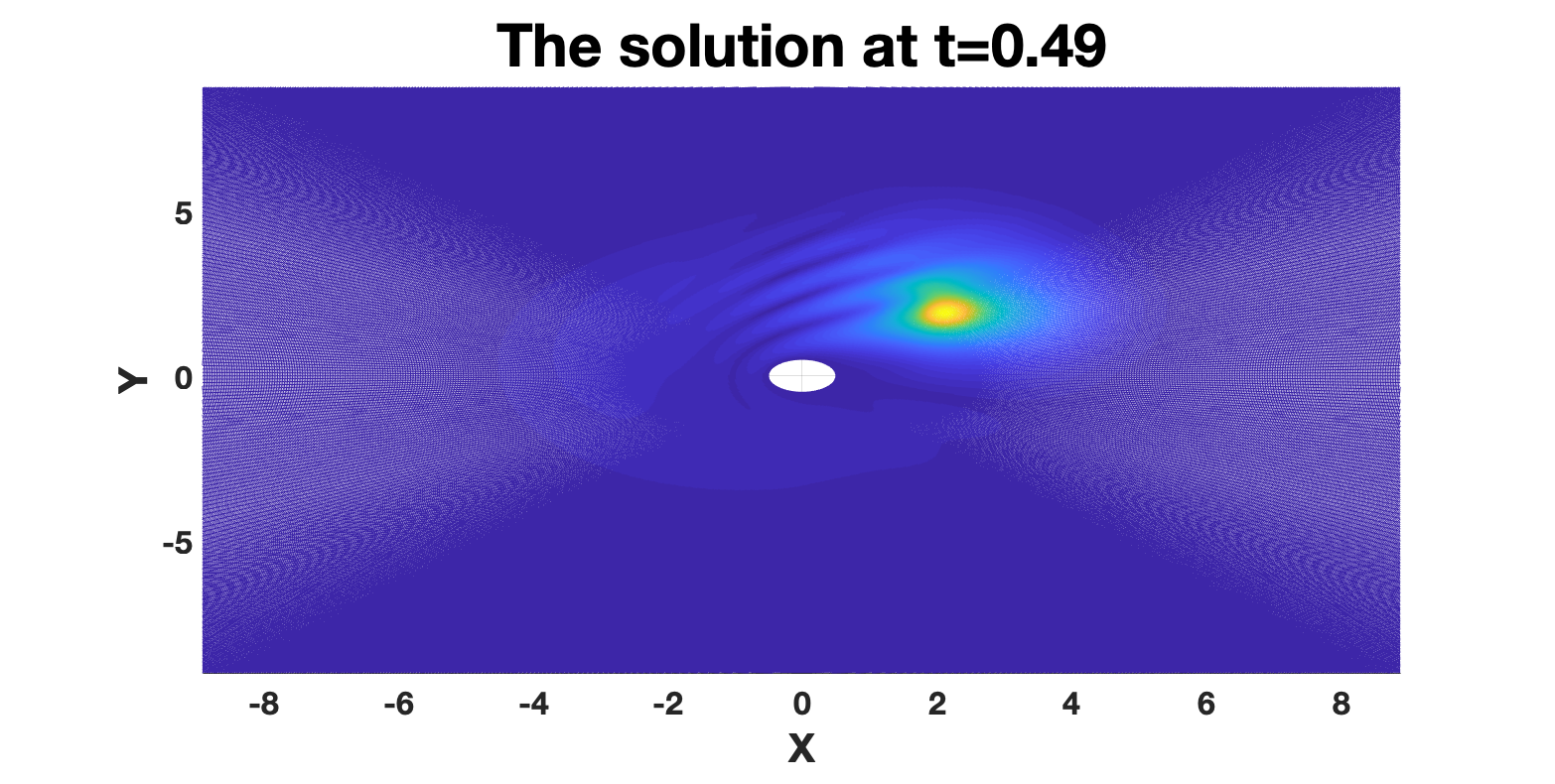}
\includegraphics[width=5.4cm,height=5.3cm]{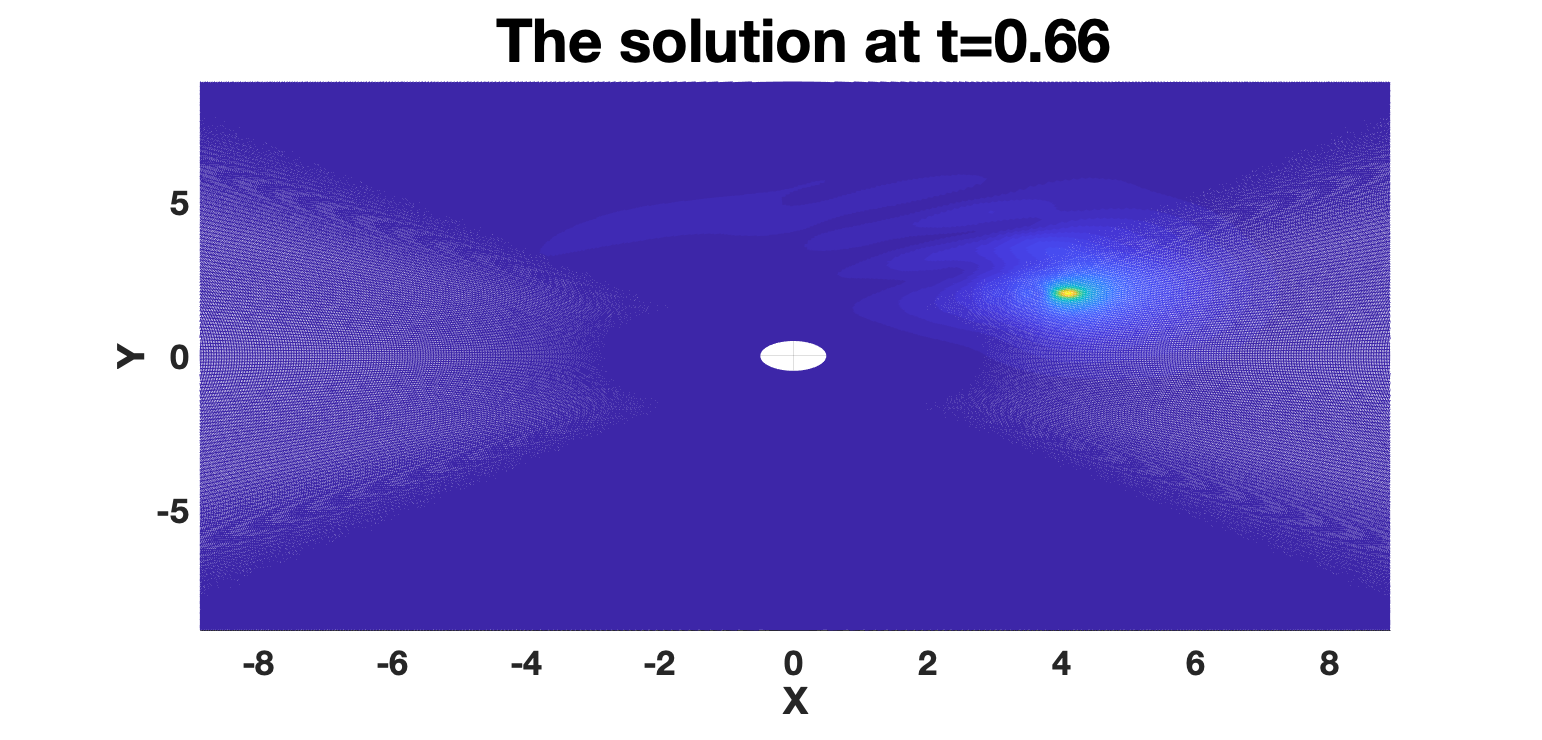}
\caption{Snapshots of the time evolution of $u(t)$ with initial condition $u_0$ from \eqref{u0NLS}, $A_0=1.25$, $v=(15,0)$ and $(x_c,y_c)=(-4.5,2)$, which eventually blows up in finite time.}
\label{Snap-of-Blow-up-t=0.66-WI}
\end{figure}


Next, we summarize the behavior of the solution to the $2d$ quintic \NNls equation, depending on the initial parameters. We take different values for the space translation $y_c$ in the initial condition \eqref{u0NLS} and fix the following parameters:
$$ 
A_0=1.25, \quad v_x=15, \quad v_y=0, \quad x_c=-4.5 .
$$ 
The results are given in Table \ref{T:2}. \\

{\footnotesize
\begin{table}[!ht]
\centering 
\begin{tabular}{|c|c|c|l|c|} 
  \hline
   $(x_c ,y_c)$           & Discrete mass      & Discrete energy   &  Behavior of the solution                                &Type of interaction    \\
  \hline \hline
 $(-4.5,  5 )$              &  4.9087                           &  138.9766                   & Blow-up at $t \approx 0.65$                         & no interaction \\
   \hline
   $ (-4.5  ,4)$            &  4.9087                           &   139.2859                  &    Blow-up at $t \approx 0.68\; $                           & no interaction\\  
  \hline
   $(-4.5  ,3)$             &  4.9087                            &   139.4553                 &    Blow-up  at  $t \approx 0.65$                       & weak-interaction  \\  
  \hline
  $(-4.5 ,2)$               &  4.9087                            &  139.5022                 &  Blow-up at $t \approx 0.66$                                &weak-interaction     \\  
  \hline
  $(-4.5 ,1.5)$            &  4.9087                            &  139.4946                  & Blow-up at $t \approx 0.51 \;$                                 &weak-interaction    \\  
  \hline  
    $(   -4.5 ,1)$           &  4.9087                           &  139.4784                   &  Blow-up at $t \approx 0.41$                                 &weak-interaction \\  
  \hline
  $(   -4.5 ,0.5)$          &  4.9087                            &  139.4636                   & Scattering       \qquad  \qquad  \qquad                   &weak-interaction \\  
  \hline 
  $(  -4.5 ,0 )$             &  4.9087                             & 139.4578                    & Scattering         \qquad        \qquad  \qquad           & strong-interaction  \\  
  \hline                                         
 $(  -4.5 ,-0.5)$           &  4.9087                             & 139.4636                   &  Scattering        \qquad         \qquad  \qquad          &weak-interaction \\  
  \hline
  $( -4.5 ,-1)$              &  4.9087                              & 139.4784                       &       Blow-up at $t \approx 0.4 $                      & weak-interaction  \\  
  \hline  
 $(  -4.5 ,-1.5)$            &  4.9087                             &  139.4946                &   Blow- up at $t \approx 0.5$                              & weak-interaction \\  
  \hline
   $( -4.5 ,-2)$               &  4.9087                             &  139.0924                 &  Blow-up at $t \approx 0.63$                                  &weak-interaction \\  
  \hline 
\end{tabular}
\caption{Influence of the translation parameter $y_c$ on the behavior of the solution $u(t)$ with initial data \eqref{u0NLS}, $A_0=1.25$, $v=(15,0)$.
Note that a tiny difference in the values of the discrete energy for different $y_c$ results from the varying density of the mesh grid: this is due to the fact that the polar coordinates $(r,\theta)$ form a uniform mesh, however, $(x,y)=(r\cos \theta, r \sin \theta)$ will not be uniform. Nevertheless, the discrete energy is conserved in each case from the start, i.e., $E[u^n]=E[u^0]$ in each simulation.}
  \label{T:2}
\end{table}
}
 

\section{Strong interaction with an obstacle}
\label{Strong interaction between soliton and obstacle}
\subsection{The $L^2$-critical case}
\label{L2critStrong}

We now consider a direct interaction of the solution with an obstacle, which we term as a {\it strong} interaction, starting with the $L^2$-critical case. The depiction of the velocity direction and the initial location is in Figure \ref{directionVelocityStrong}. 

\begin{figure}[!ht]
\centering
\includegraphics[width=16cm,height=7cm]{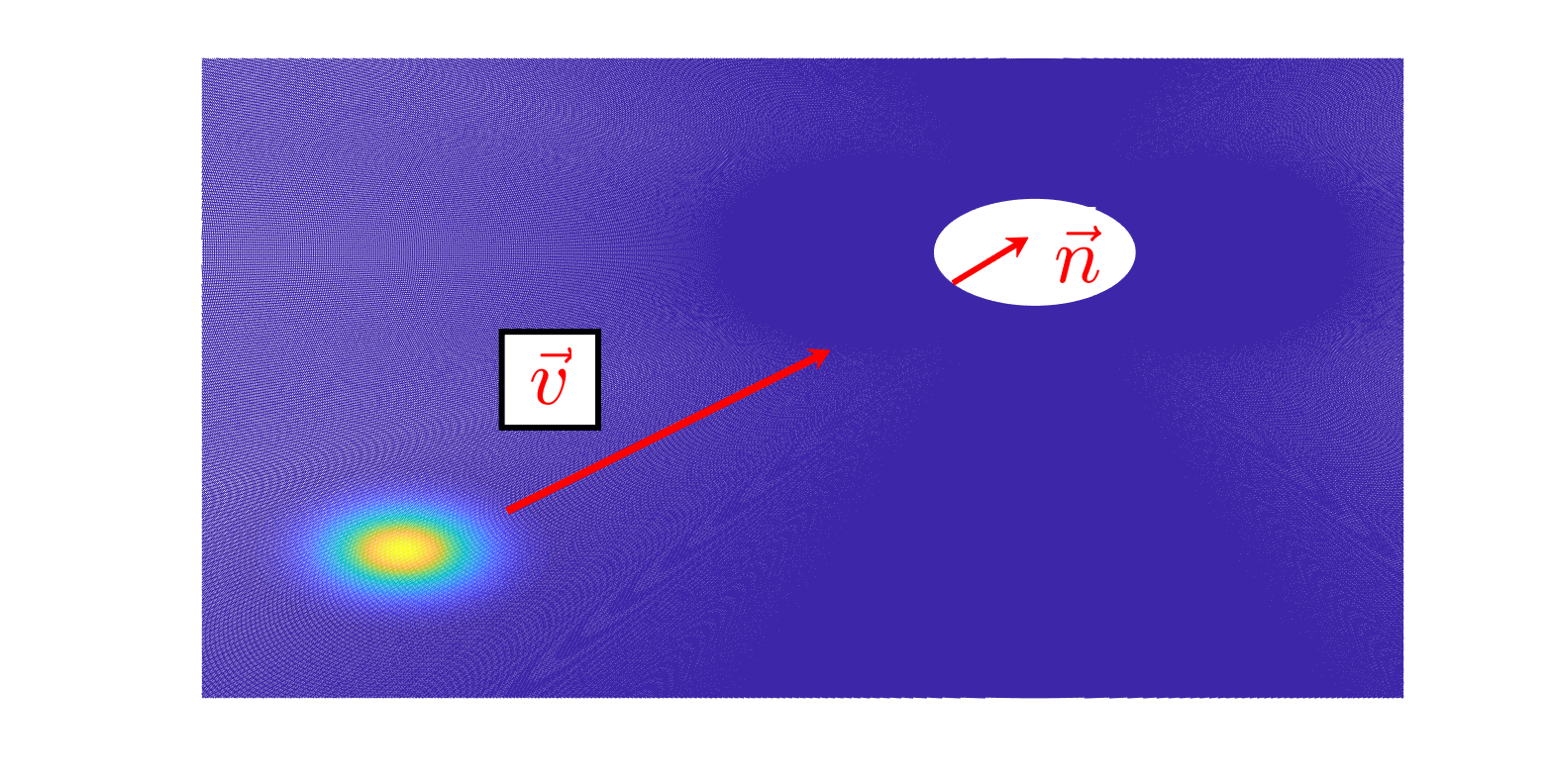}
\caption{The directions of the velocity of the solution $u(t)$ on the line $y=x$ and the same direction of the outward normal vector $\vec{n}$.}
\label{directionVelocityStrong}
\end{figure}

We consider the cubic \NNls equation with the same initial data \eqref{u0NLS}, with the same amplitude and space translation ($A_0=2.25,\,(x_c,y_c)=(-4.5,-4.5)$), as in Subsection \ref{L2critWeak} but in this case we take the velocity directly pointed at the obstacle $v=(15,15)$, meaning that the solution $u(t)$ is moving along the line $y=x,$ i.e., in the same direction as the outward normal vector $\vec{n}$ as shown in Figure \ref{directionVelocityStrong}.   \\ 

\begin{figure}[ht]
\centering
\includegraphics[width=5.9cm,height=5.7cm]{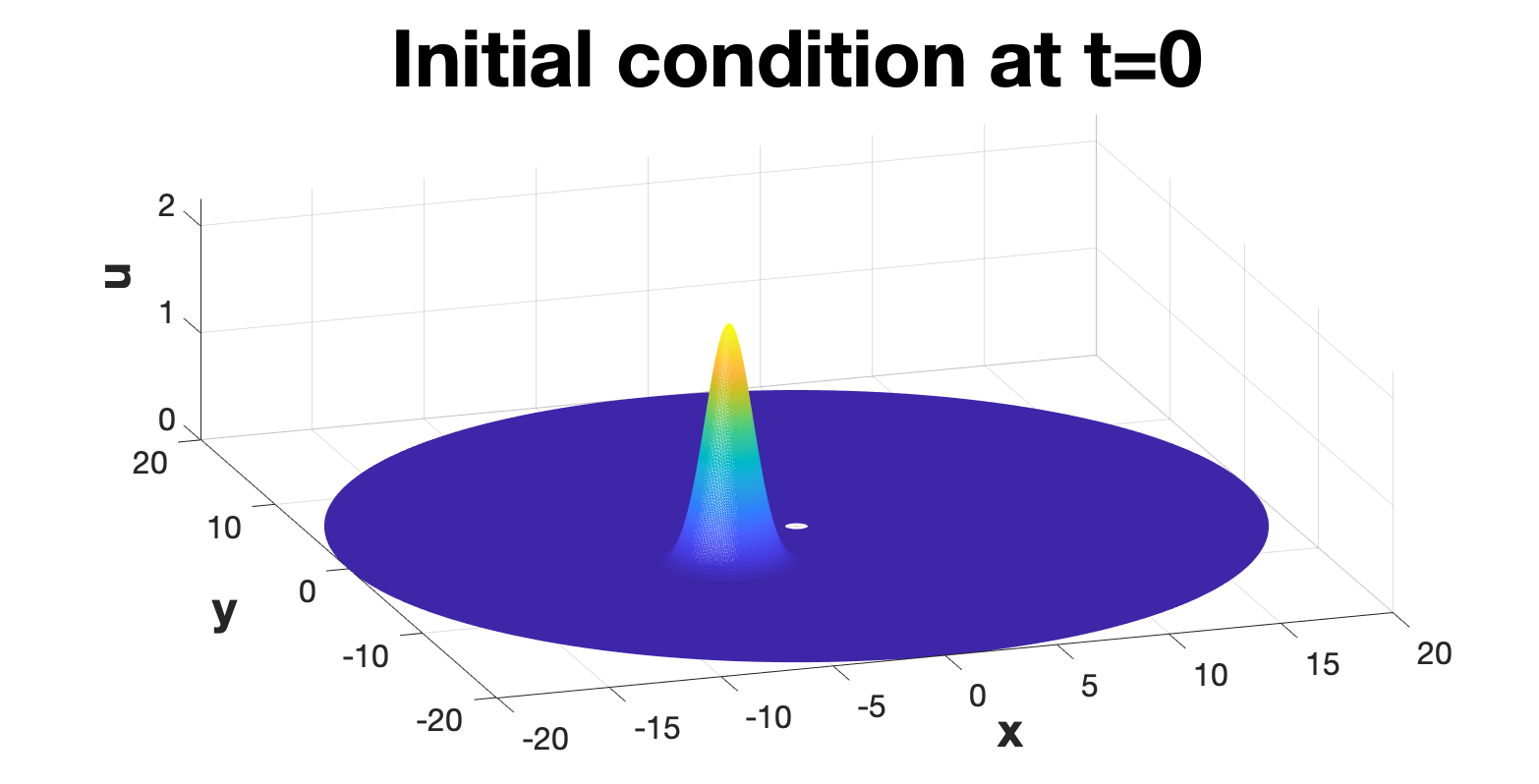}
\includegraphics[width=6.4cm,height=5.7cm]{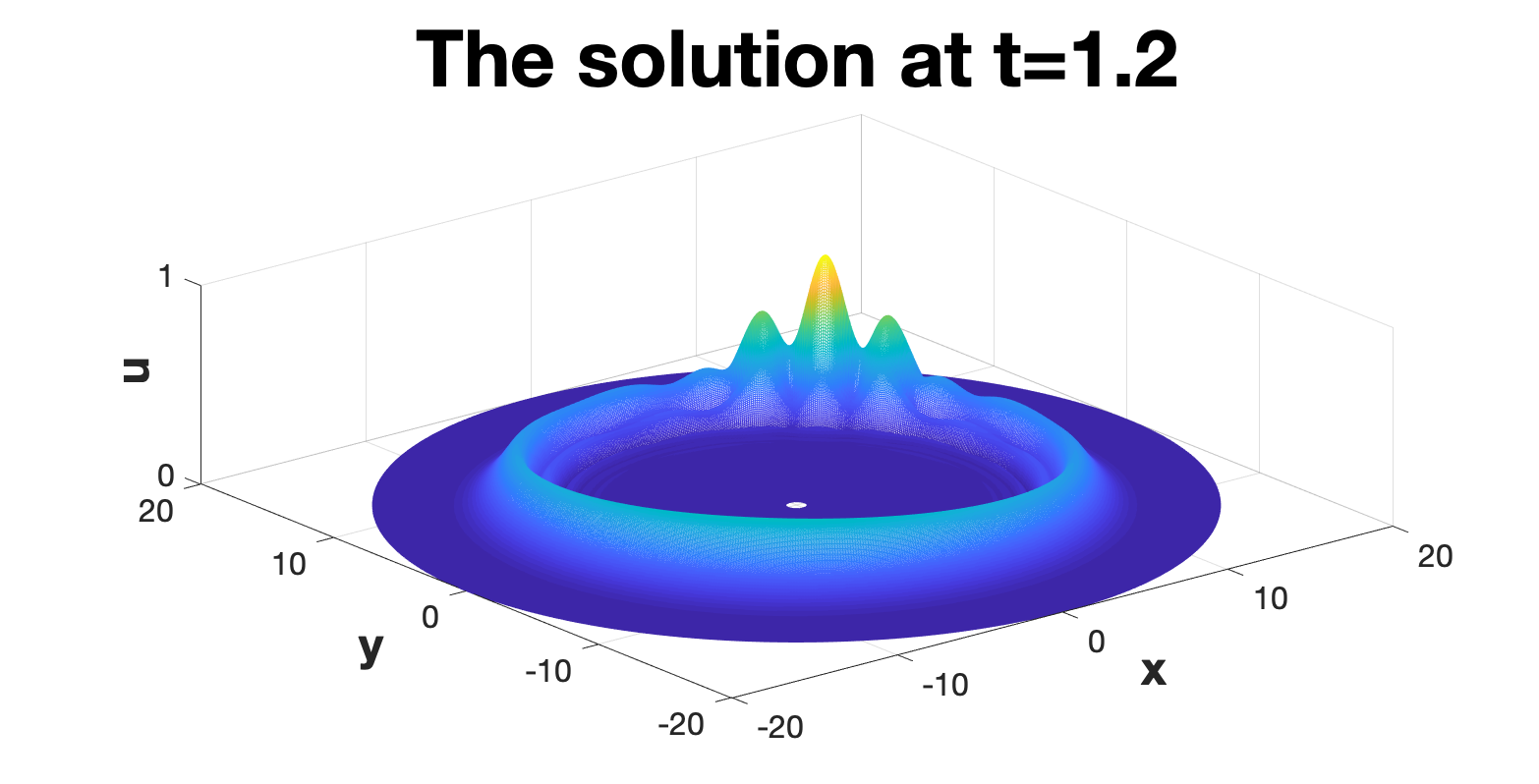} 
\includegraphics[width=5.5cm,height=5.6cm]{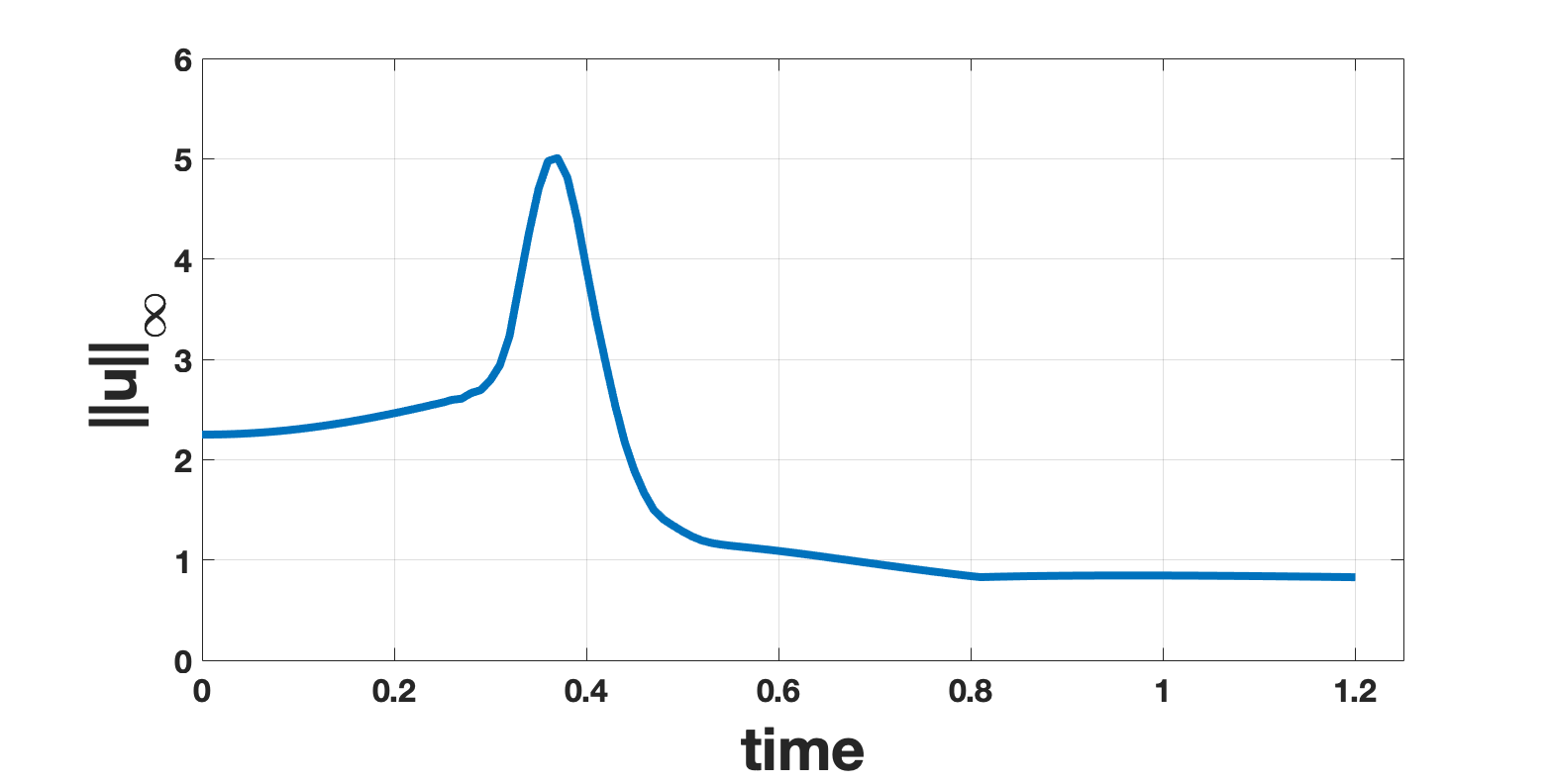}
\caption{The initial condition $u_0$ at $t=0$ from \eqref{u0NLS}, $A_0=2.25,\,(x_c,y_c)=(-4.5,-4.5)$ and $v=(15,15)$ (left); the corresponding solution $u(t)$ to the $2d$ cubic \NNls equation at $t=1.2$ (middle); time dependence of $L^\infty$ norm (right).}
\label{Solution-directionVelocityStrong}
\end{figure}

In this scenario, we observe that the solution has a scattering behavior and does not conserve the same profile or shape of the initial form of the solitary wave, see Figure \ref{Solution-directionVelocityStrong}, thus, exhibits the strong interaction.  
Snapshots of the time evolution of this solution $u(t)$ to show the strong interaction are given in Figure \ref{Snapshots-Solution-directionVelocityStrong}: the solution hugs the obstacle while transferring the mass forward, then it forms the two main bumps in front of the obstacle, later they connect together, which creates a third (middle) bump and shifts more and more mass into this central lump while propagating it forward along the main velocity line ($y=x$ in this case). The circle of dispersive reflective waves (including backward reflective waves) forms and expands, radiating out all of the reflective waves. \\

\begin{figure}[ht]      
\centering
\includegraphics[width=5.5cm,height=5.2cm]{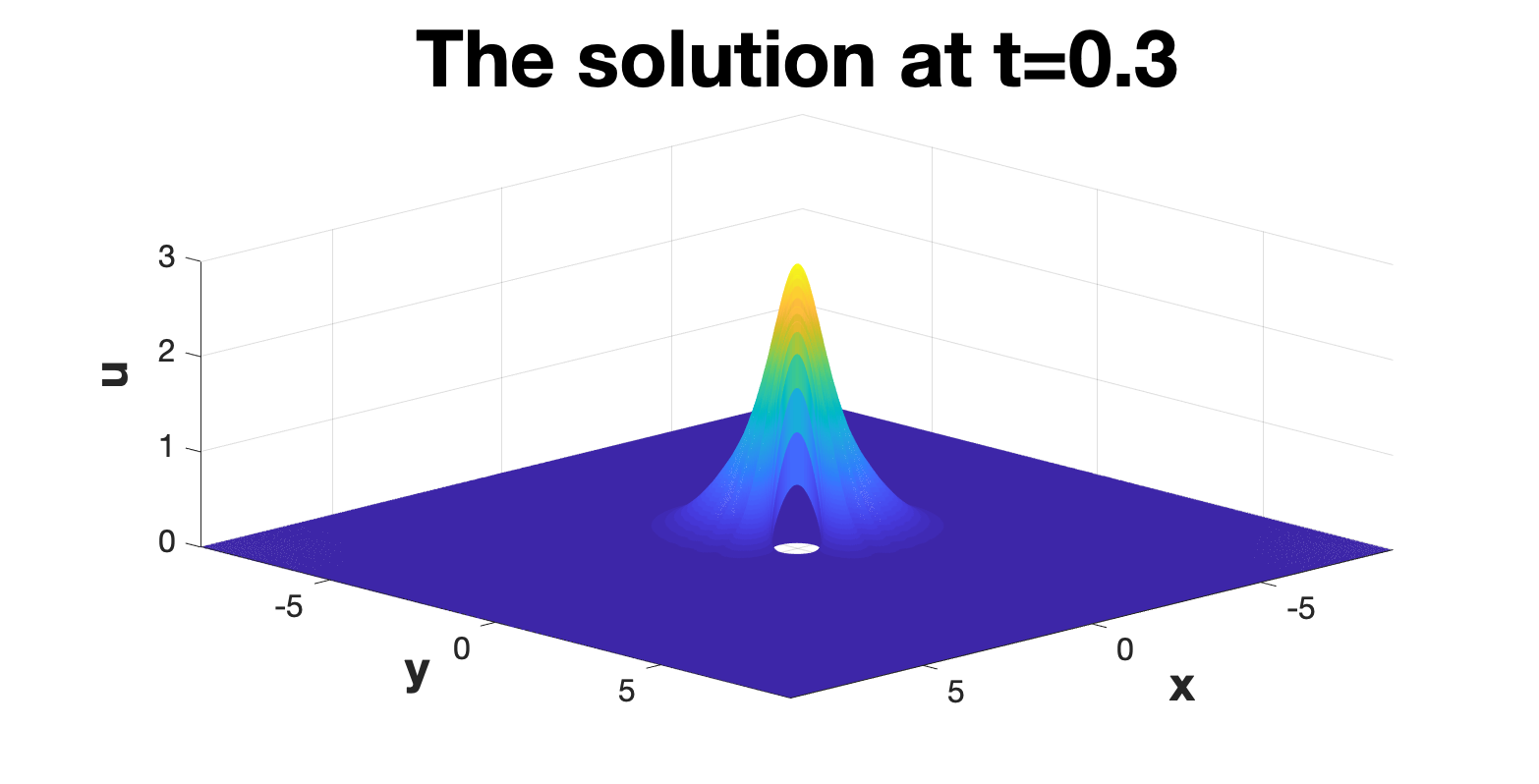}
\includegraphics[width=5.5cm,height=5.2cm]{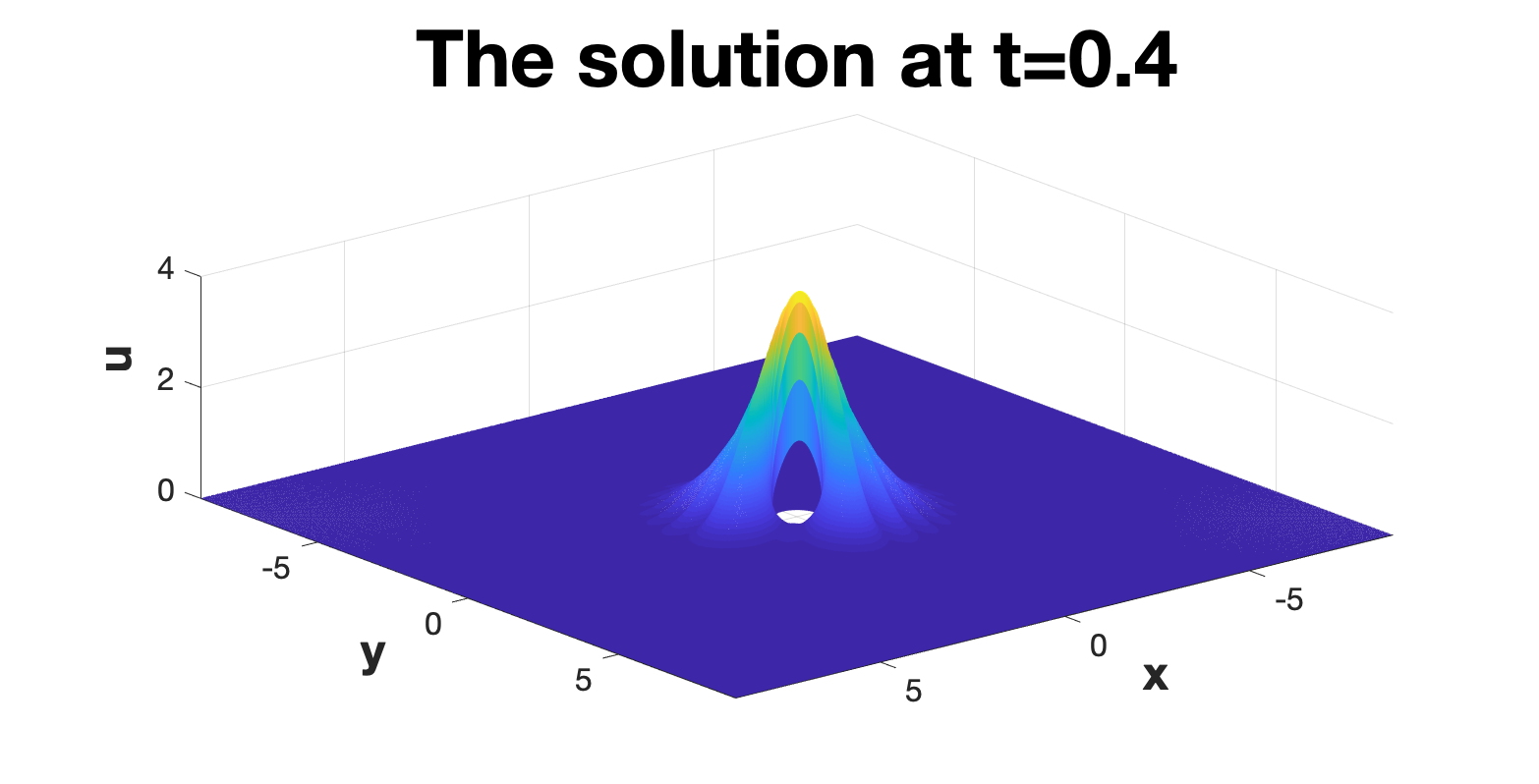}
\includegraphics[width=5.5cm,height=5.2cm]{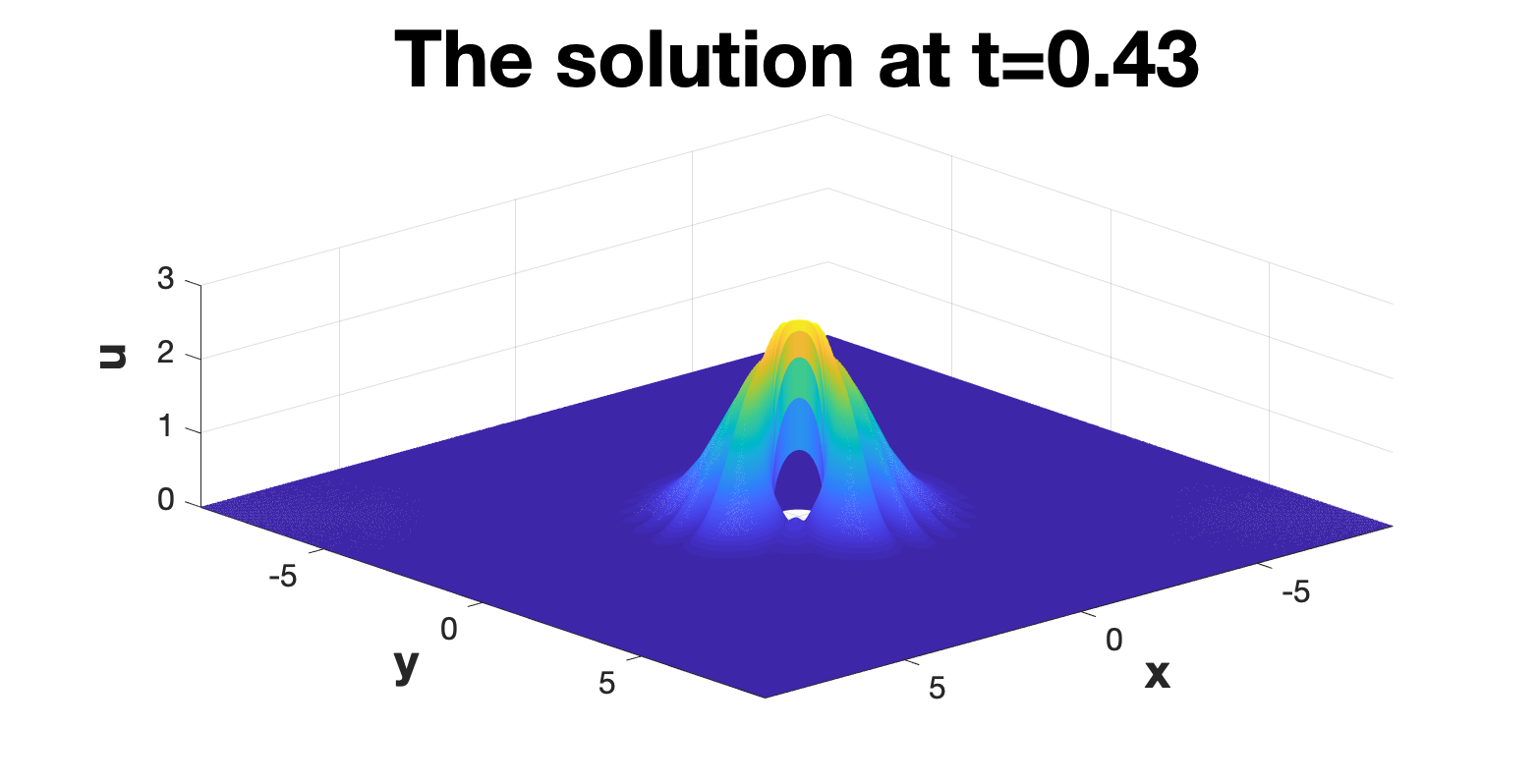}
\includegraphics[width=5.5cm,height=5.2cm]{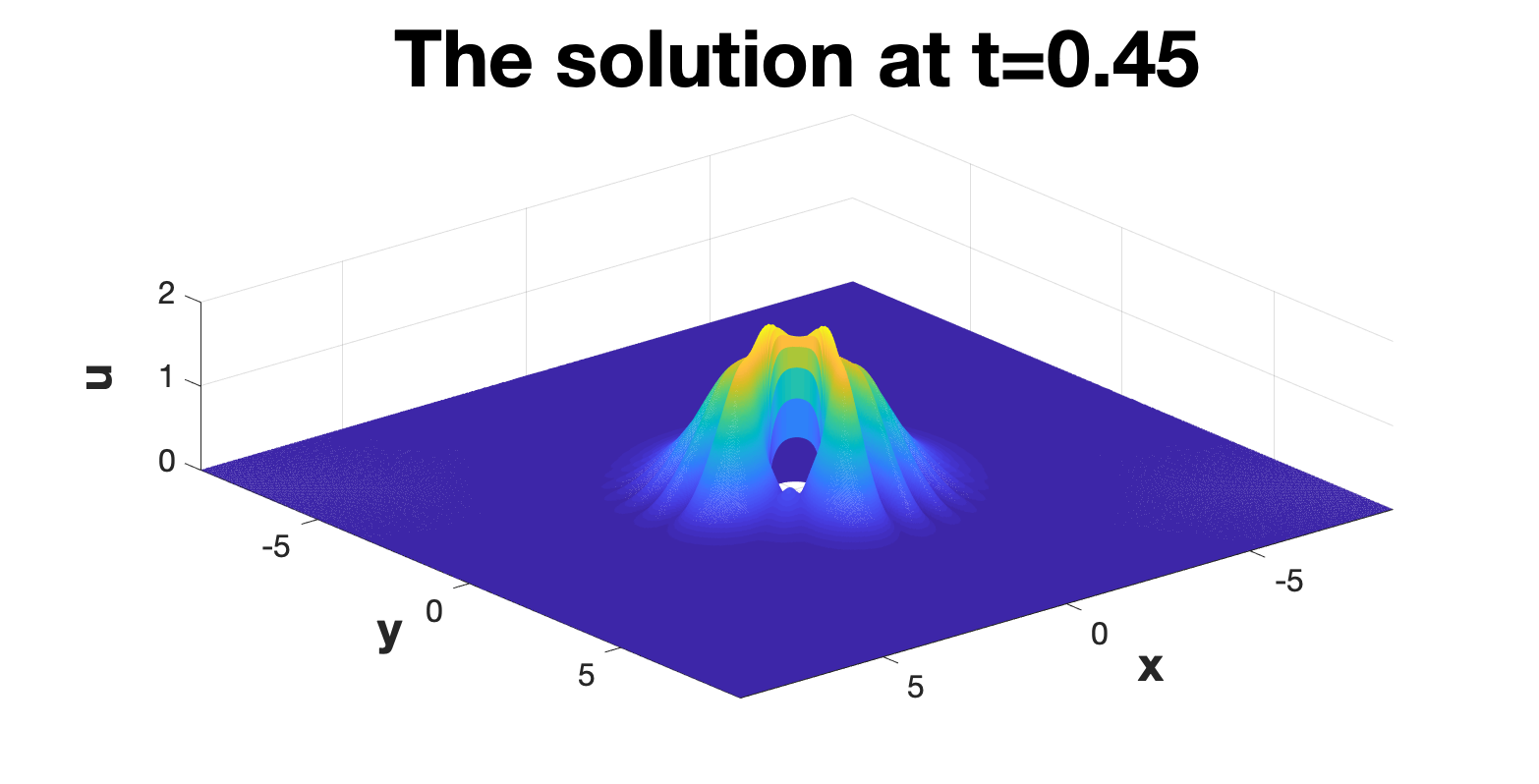}
\includegraphics[width=5.5cm,height=5.2cm]{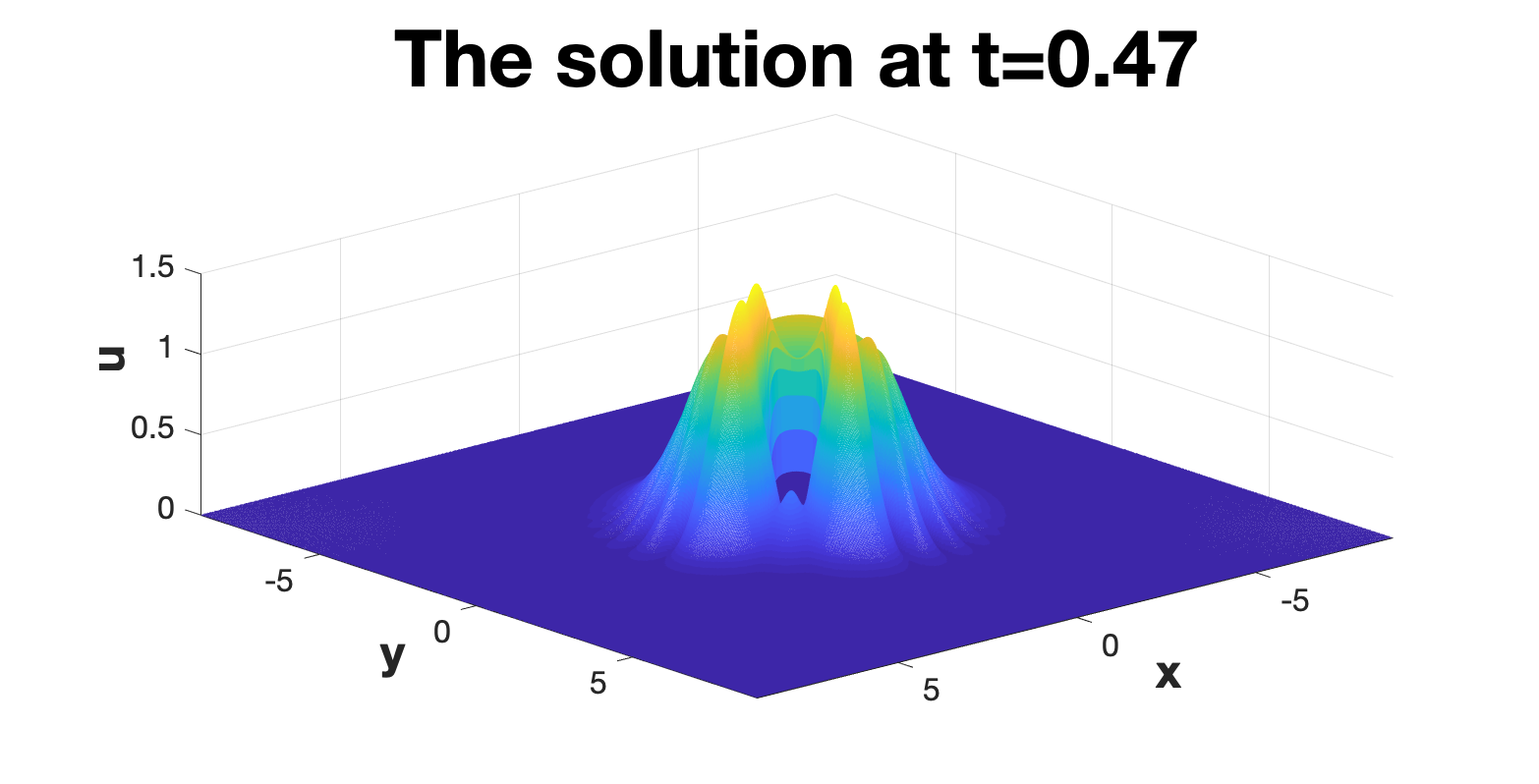}
\includegraphics[width=5.5cm,height=5.2cm]{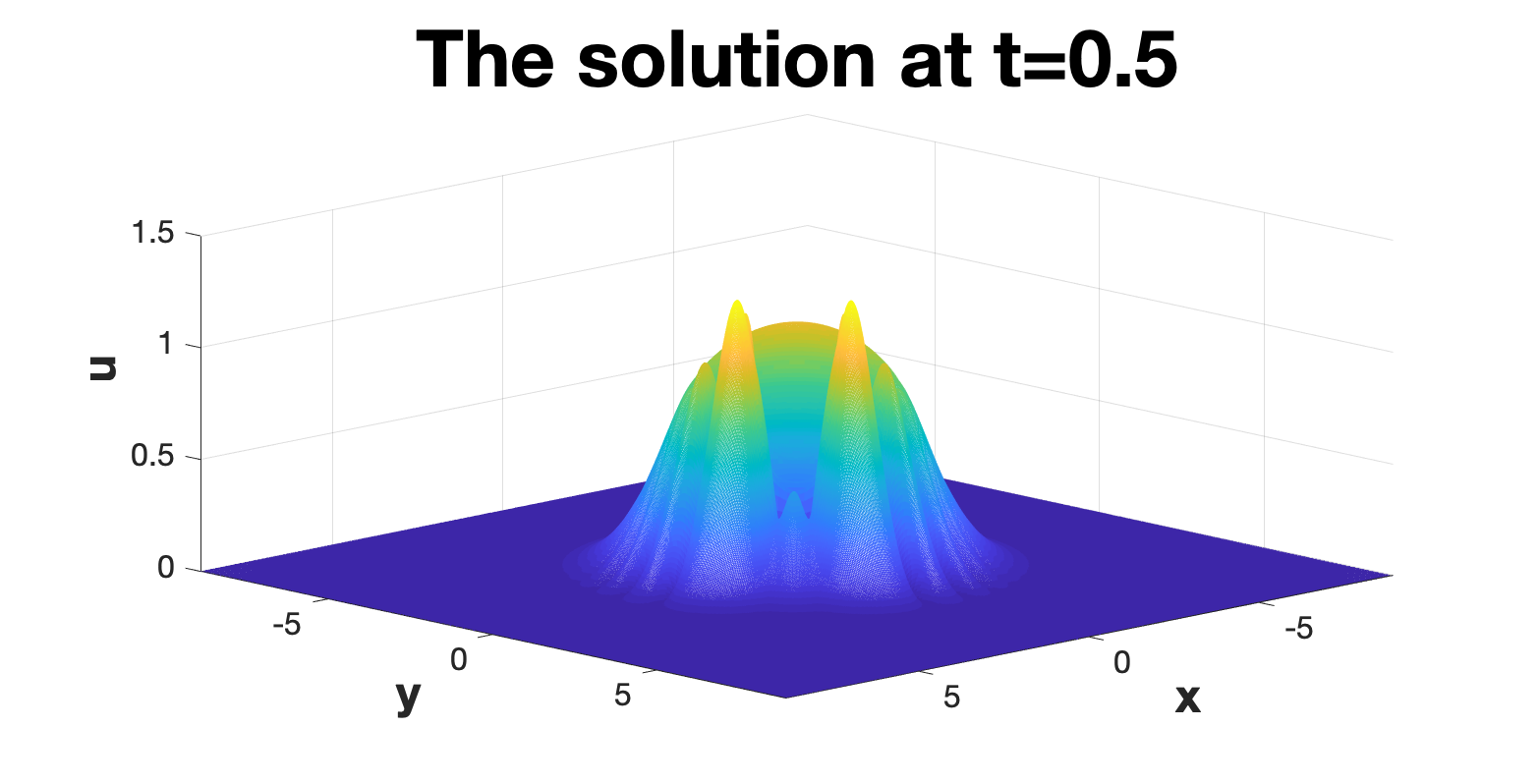}
\includegraphics[width=5.5cm,height=5.2cm]{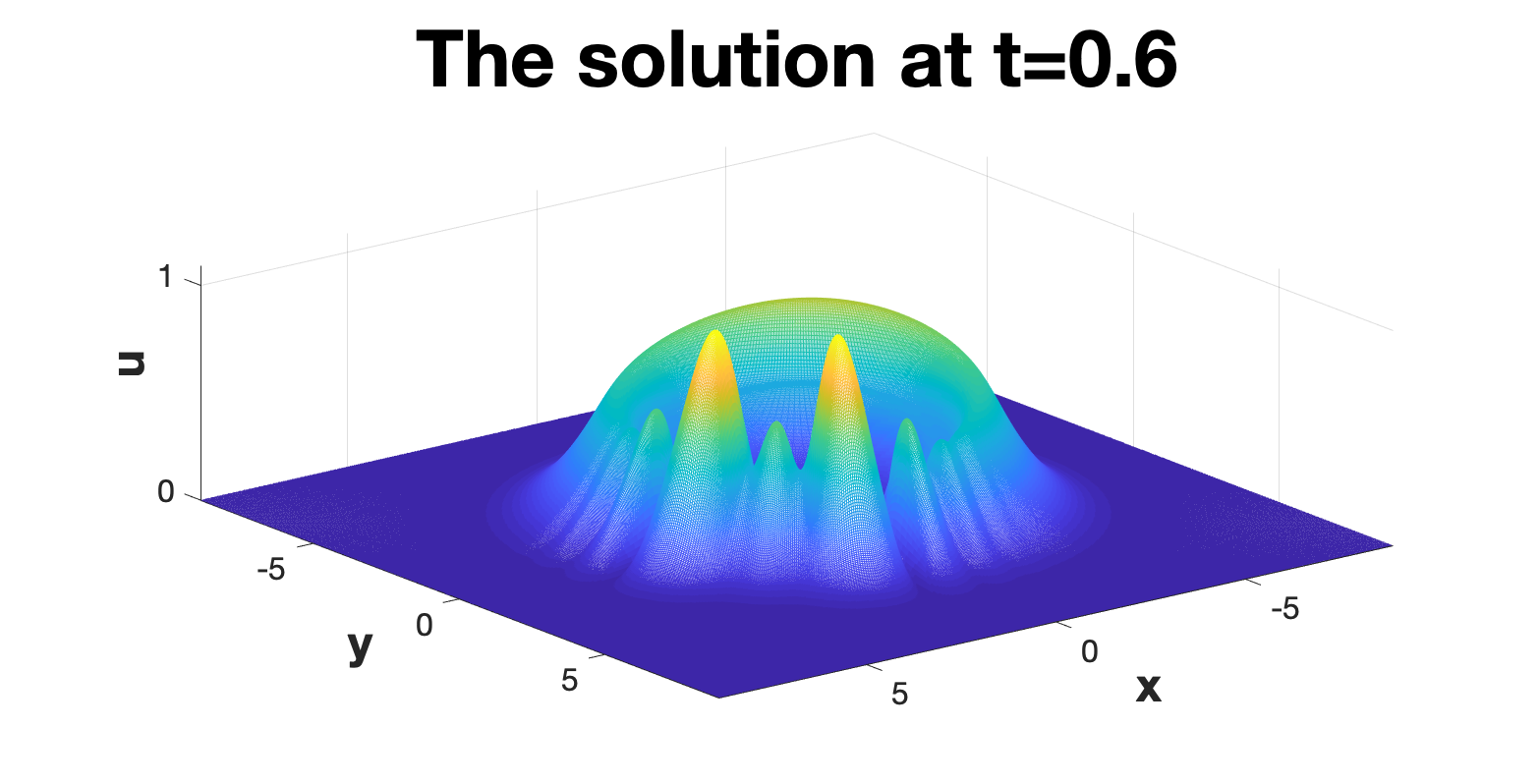}
\includegraphics[width=5.5cm,height=5.2cm]{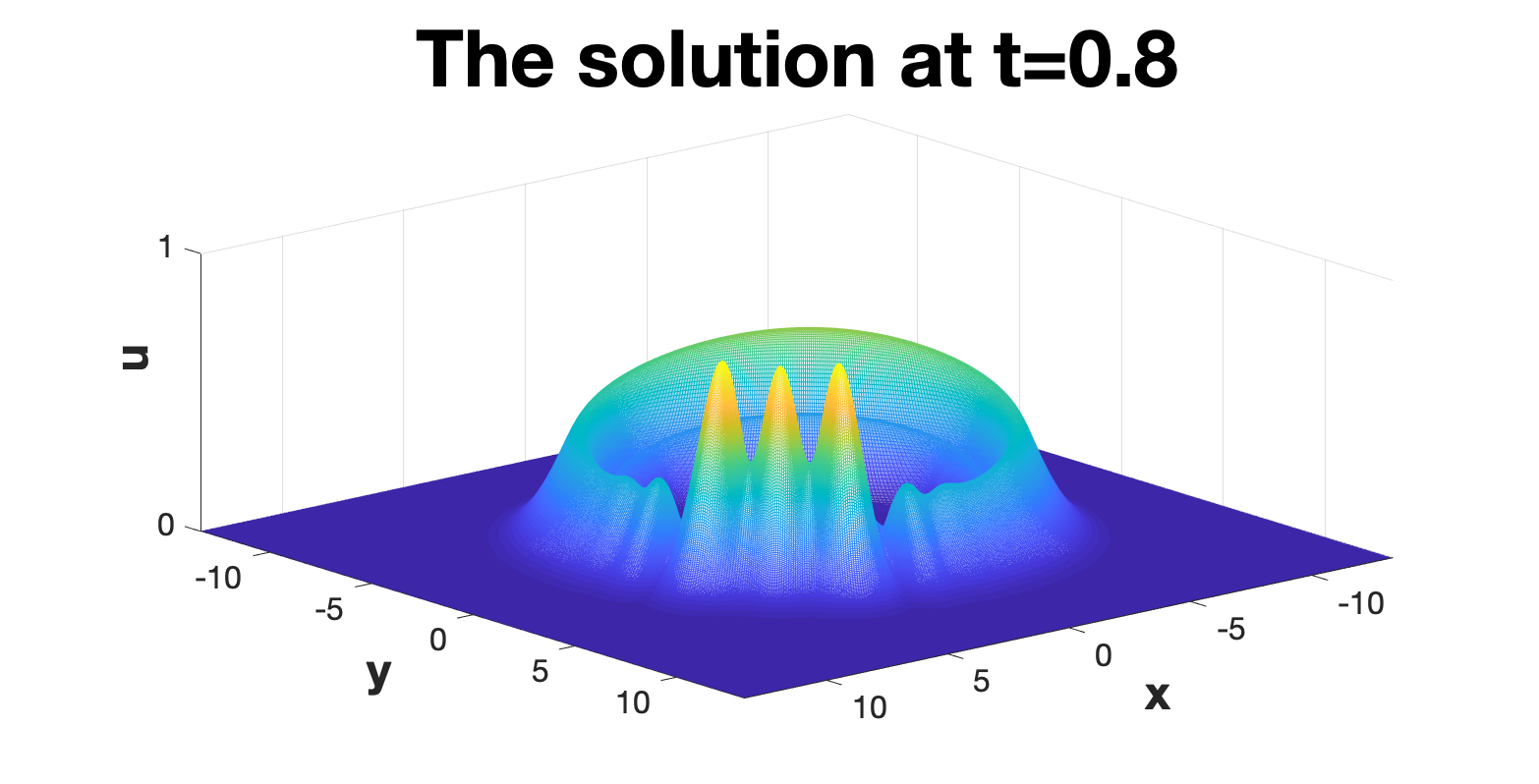}
\includegraphics[width=5.5cm,height=5.2cm]{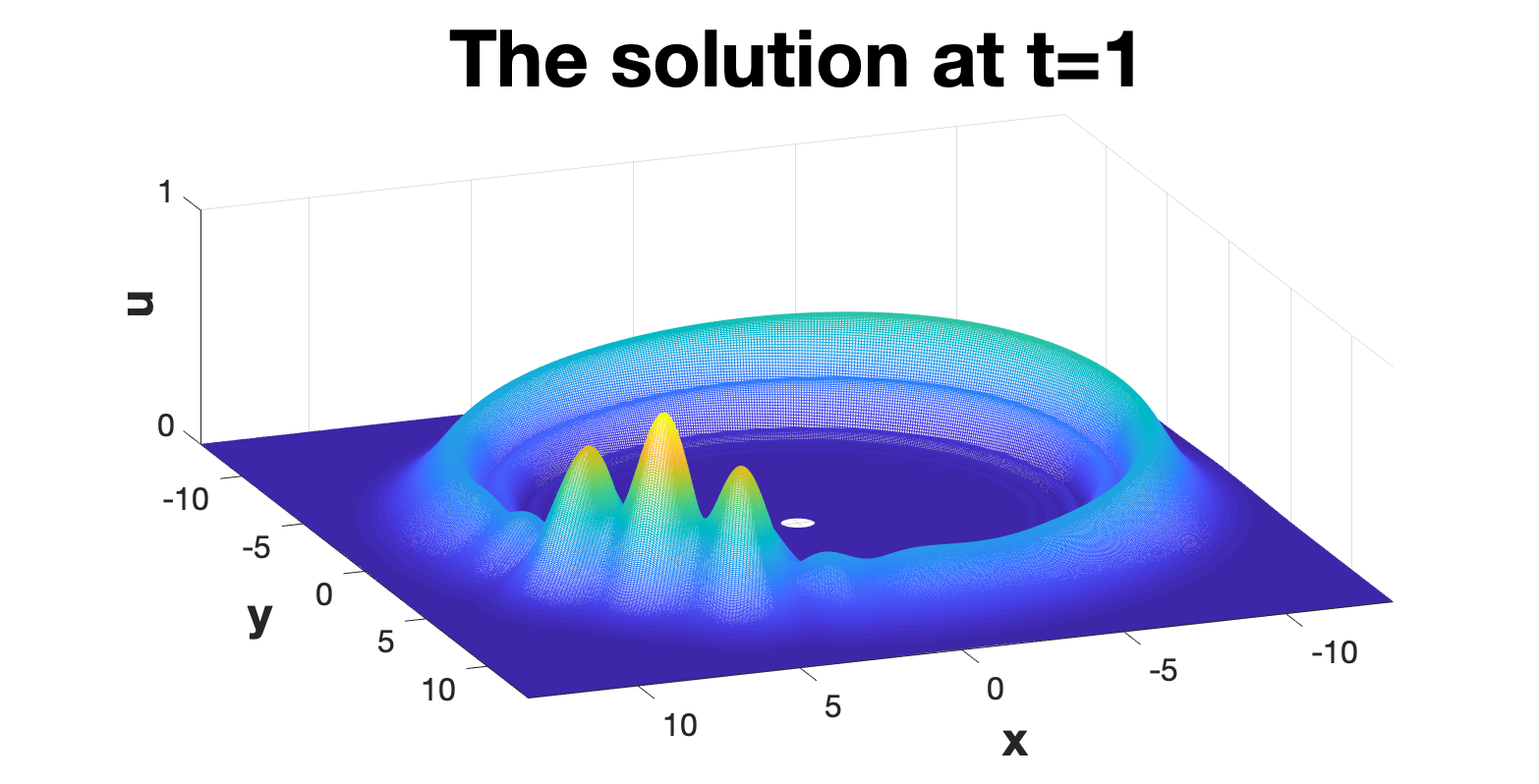}
\caption{Snapshots of the solution $u(t)$ to the $2d$ cubic \NNls equation for different time steps of the strong-interaction between the solution and the obstacle with the initial data as in Figure \ref{Solution-directionVelocityStrong}.}
\label{Snapshots-Solution-directionVelocityStrong}
\end{figure}

The obstacle  transforms the blow-up behavior into what seems to be scattering, that we investigate further.   
Before doing that, we emphasize that the strong interaction has a substantial influence on the dynamics of the solution. 
We point out that in the weak interaction case in a similar example in Section \ref{L2critWeak} ($L^2$-critical case) the solution blows up in finite time. In the considered case, 
the $L^{\infty}$-norm starts increasing, manifesting a blow-up behavior, see Figure \ref{Solution-directionVelocityStrong}, however, after the collision, the amplitude of the solution start decreasing. After that, we observe that the $L^{\infty}$-norm appears to stabilize, as shown in the right graph in Figure \ref{Solution-directionVelocityStrong}, where $\left\| u(t) \right\|_{\infty} \approx 0.8$ after the interaction, indicating that the solution might not be scattering (for example, it could approach a rescaled soliton). We check this case simulating it for a longer time and observe that the $L^{\infty}$-norm continues to decrease as shown in the left subplot of Figure \ref{norminf+log+solution}. In the right graph we show 
the solution amplitude change in time on a log scale for $t\in[2,3.32]$, which shows that it is decreasing as $\frac{1}{\sqrt{t}}$. Thus, it is pleasible to conclude that the solution scatters (however, this would have to be proved analytically, as it is possible that it might approach an asymptote at a later time).

\begin{figure}[ht]
\centering
\includegraphics[width=8.2cm,height=5.9cm]{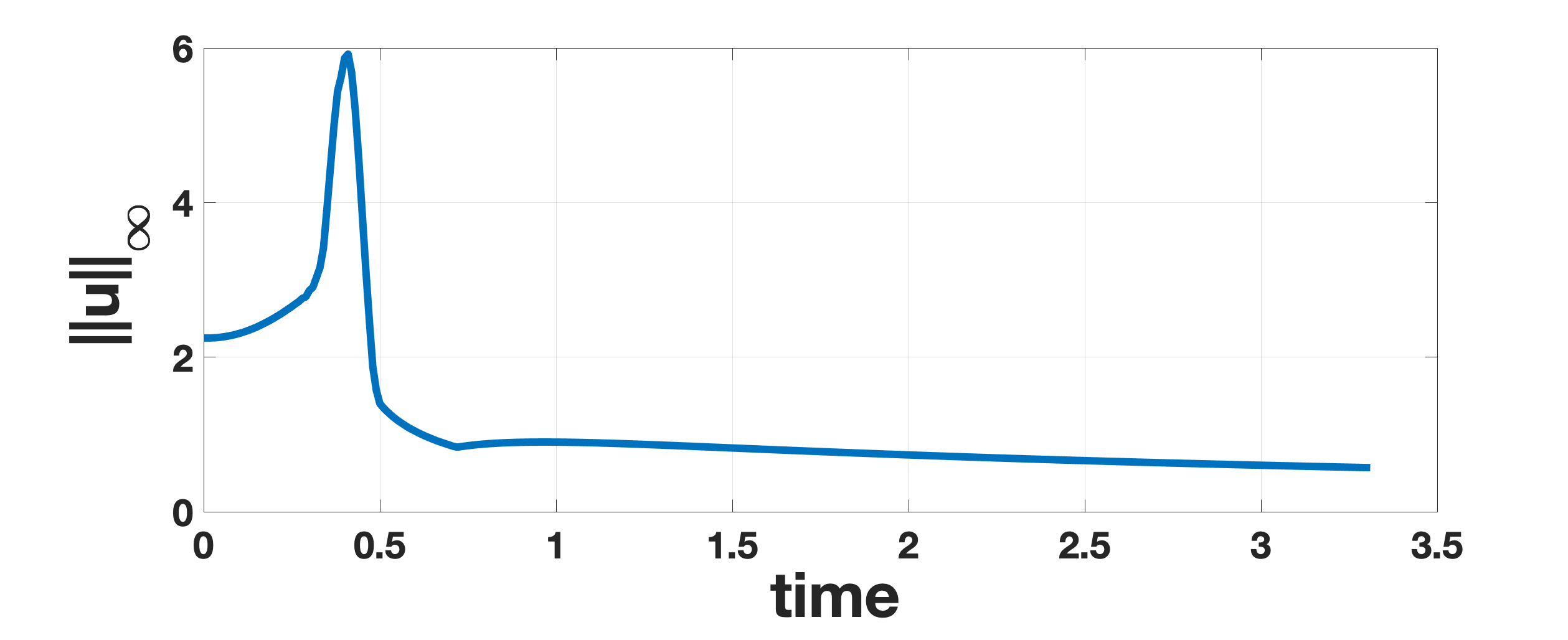}
\includegraphics[width=8.2cm,height=5.9cm]{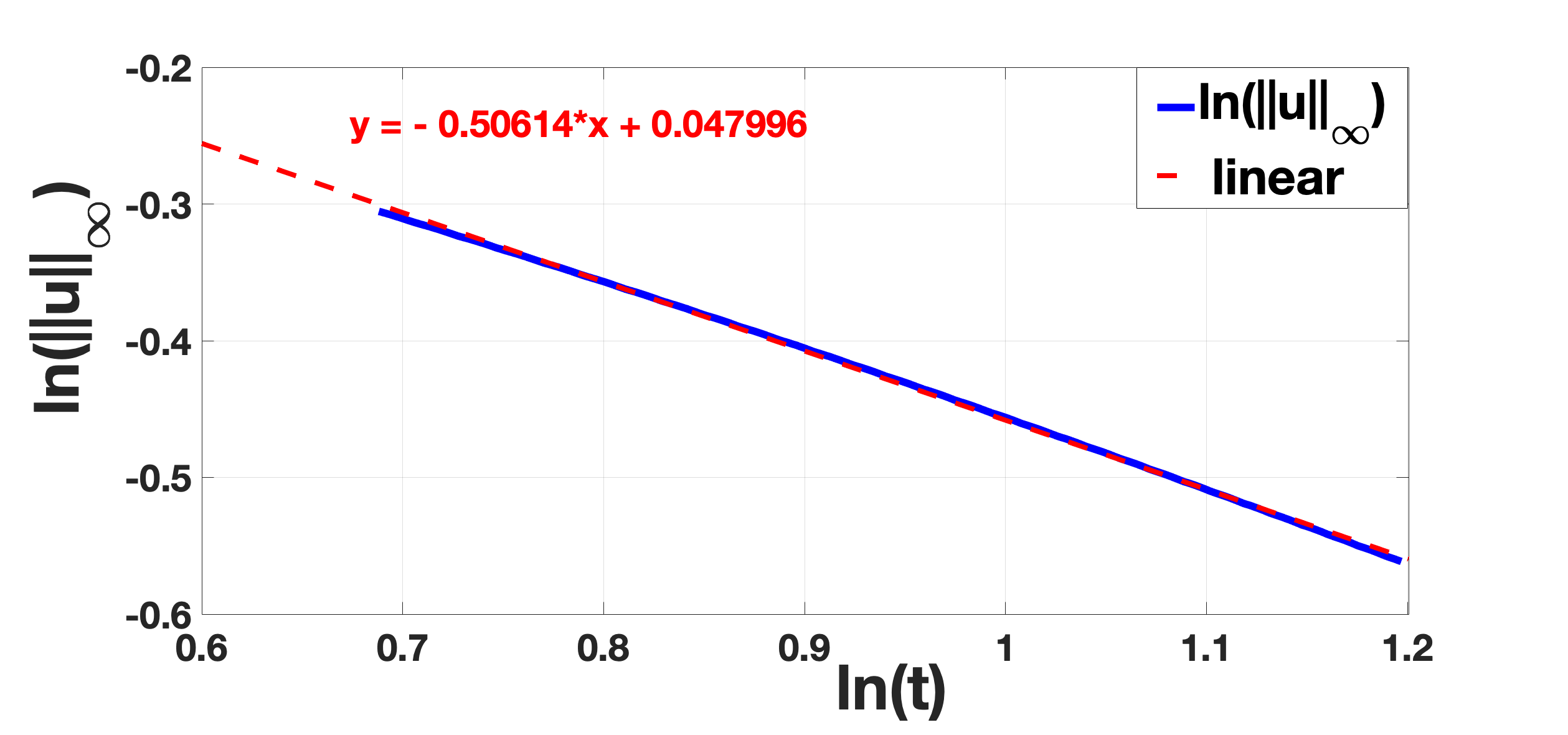}
\caption{The time dependence of $L^\infty$ norm (left); the slope of the $L^\infty$ norm from $t=2$ to $t=3.32$ on a log scale (right).}
\label{norminf+log+solution}
\end{figure}

\subsection{The $L^2$-supercritical case}
\label{L2supercritStrong}
In the $2d$ quintic \NNls equation $(p=5)$,  we also investigate the {\it strong} interaction between the obstacle and the solution, where 
the solution is moving in the same direction as the outward normal vector of the obstacle, see Figure \ref{directStrongInterac}. 

\begin{figure}[ht]         
\centering
\includegraphics[width=16cm,height=5.9cm]{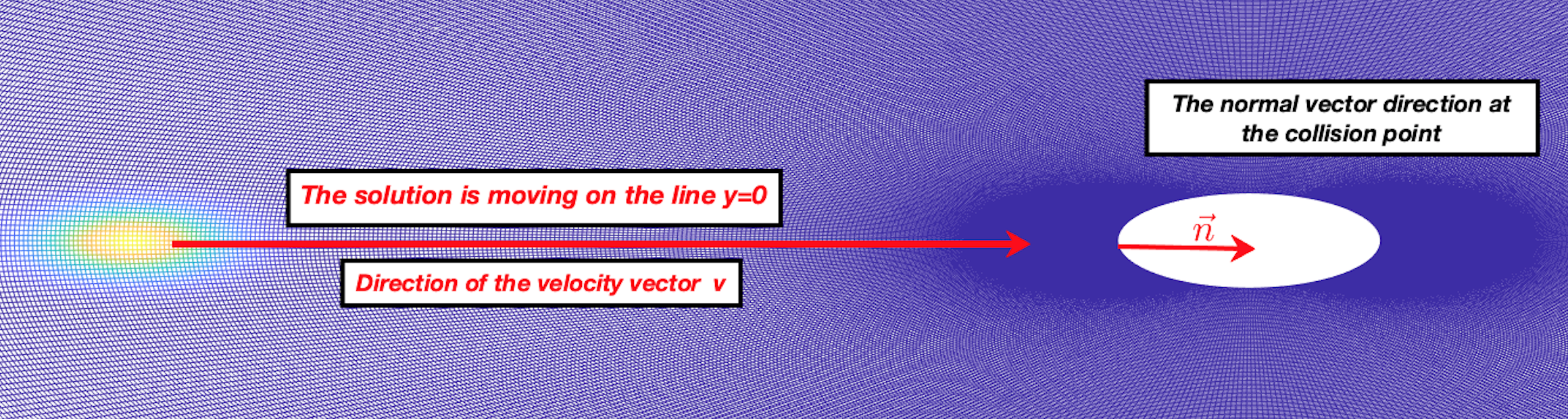}
\caption{The direction of movement of the solution, on the line $y=0$ with outward normal vector.} 
\label{directStrongInterac}
\end{figure}


We consider the same initial data \eqref{u0NLS}, with the same phase as for the quintic \NNls equation described in Section \ref{L2supercritWeak} but now with $y_c=0$ (i.e., $A_0=1.25,\, x_c=-4.5$ and $v=(v_x,0)$ are fixed parameters). In the present situation, the solution is moving on the line $y=0,$ i.e.,  in the same direction as the outward normal vector of the obstacle. The solitary wave hits the obstacle straight on, causing a strong interaction between the wave and the obstacle, see Figure \ref{F:28}.  
\begin{figure}[ht]
\centering
\includegraphics[width=5.9cm,height=6.5cm]{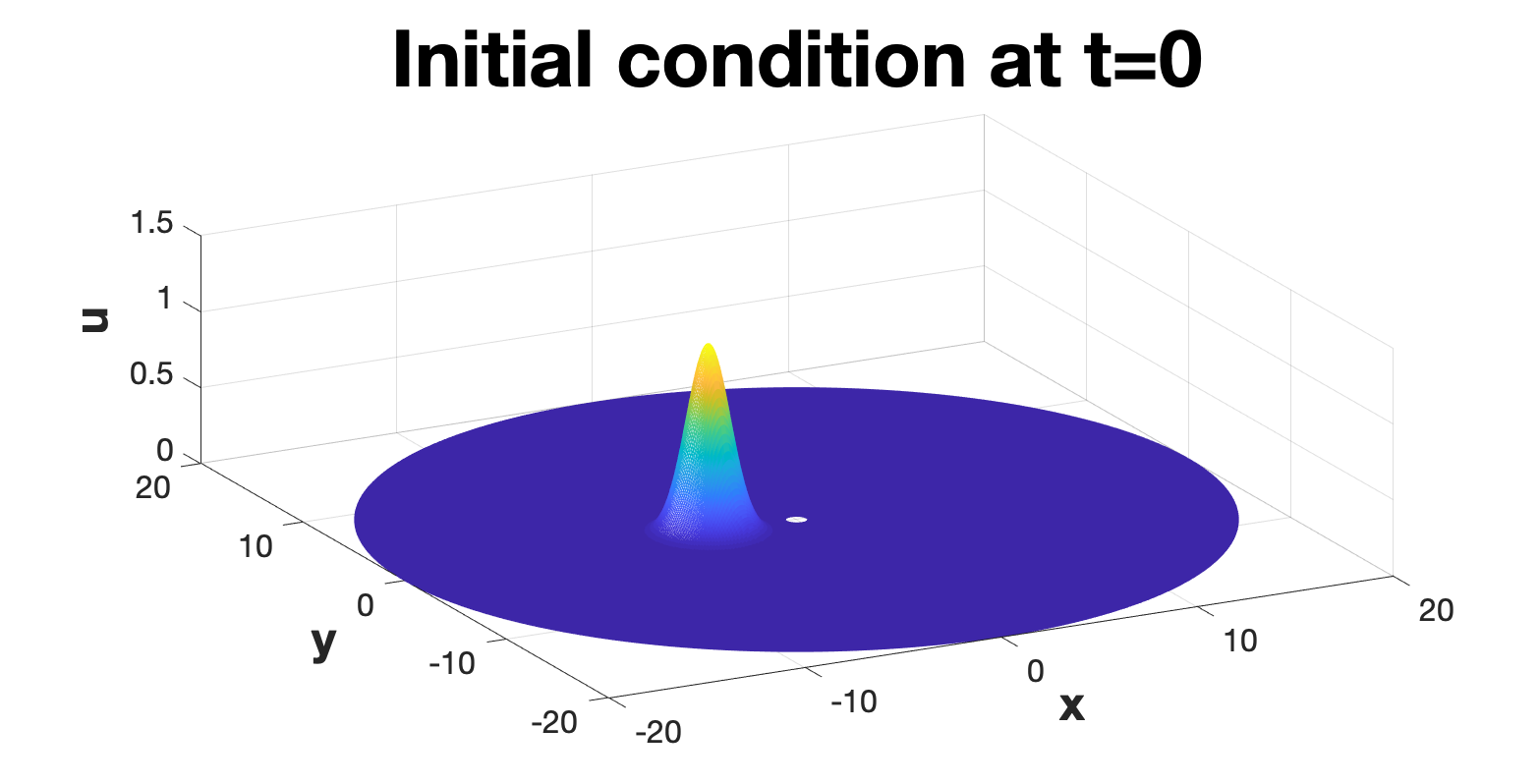}
\includegraphics[width=6.3cm,height=6.5cm]{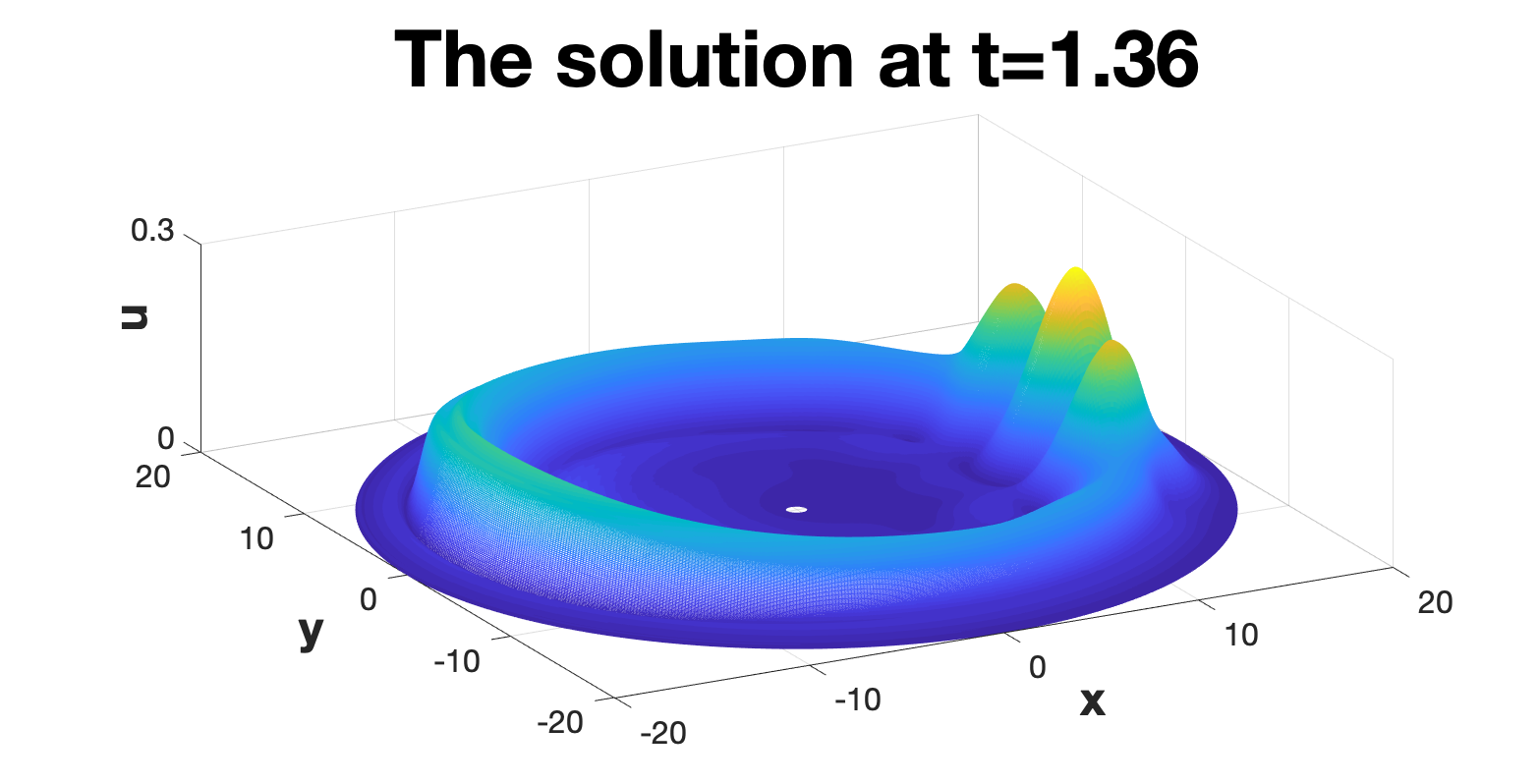}
\includegraphics[width=5.3cm,height=5.5cm]{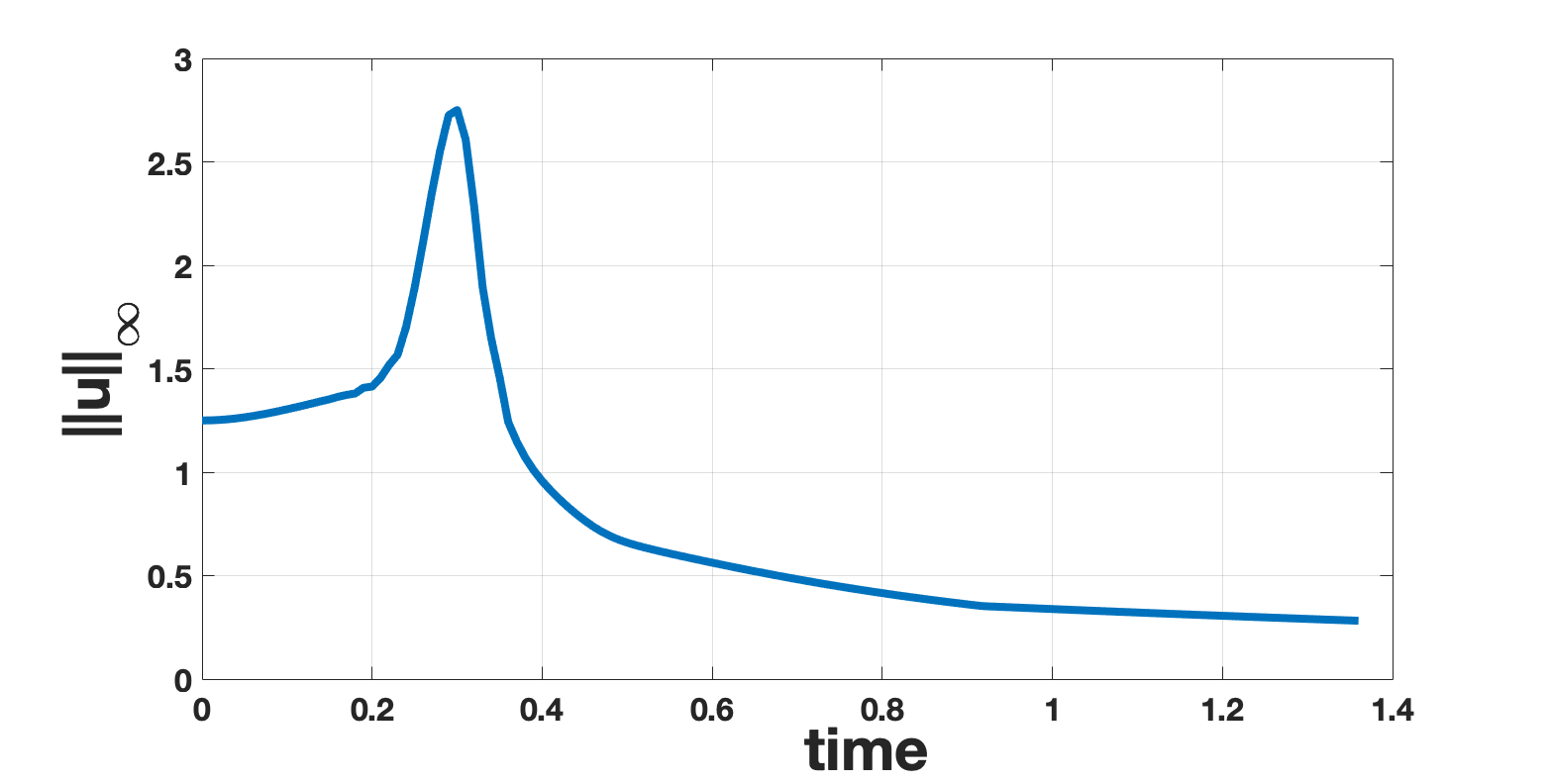}
\caption{Solution $u(t)$ to the $2d$ quintic \NNls equation with the initial condition $u_0$ from \eqref{u0NLS}, $A_0=1.25,\, (x_c,y_c)=(-4.5,0)$ and $v=(15,0)$ (left), the solution $u(t)$ moving on the line $y=0$ at time $t=1.36$ (middle); time dependence of the $L^{\infty}$-norm (right).}
\label{F:28}
\end{figure}

In this case, the solution scatters and does not preserve the shape of the original solitary wave. After the collision, the solitary wave solution forms two (then later possibly three, and eventually just one ) bumps with a circular reflecting waves, part of which reflects backward, see Figure \ref{F:29}. We observe also that the leading reflected wave has a dispersive behavior. Moreover, one can see that the presence of the obstacle  completely prevents blow-up. Before the interaction, the $L^{\infty}$-norm of the solution starts increasing, indicating a possible blow-up behavior, however, after the interaction with the obstacle, the amplitude of the solution decreases toward $0$, which confirms the dispersion of the solution in a long term, thus, scattering.


\begin{figure}[!h]      
\centering
\includegraphics[width=5.5cm,height=5cm]{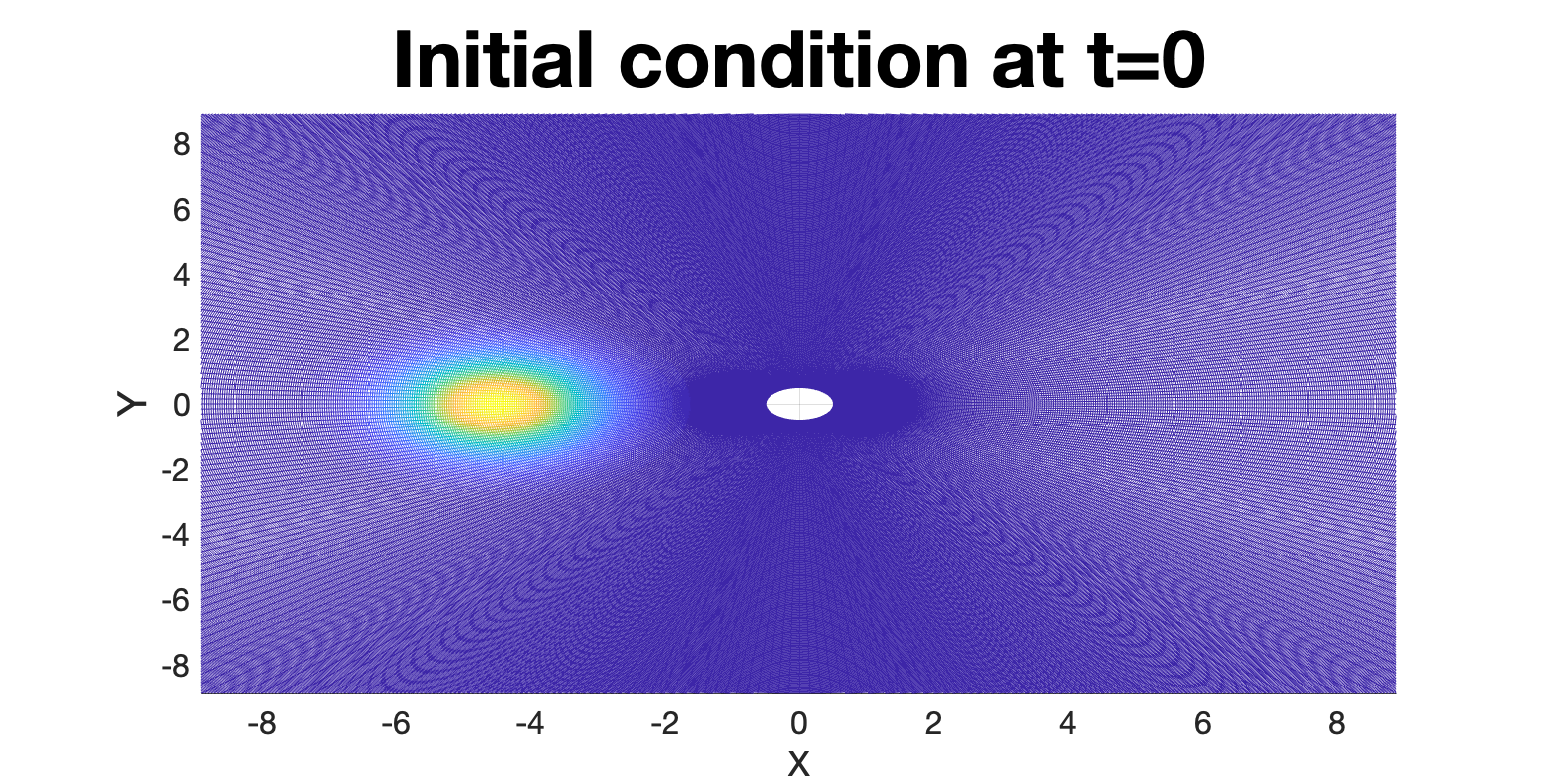}
\includegraphics[width=5.5cm,height=5cm]{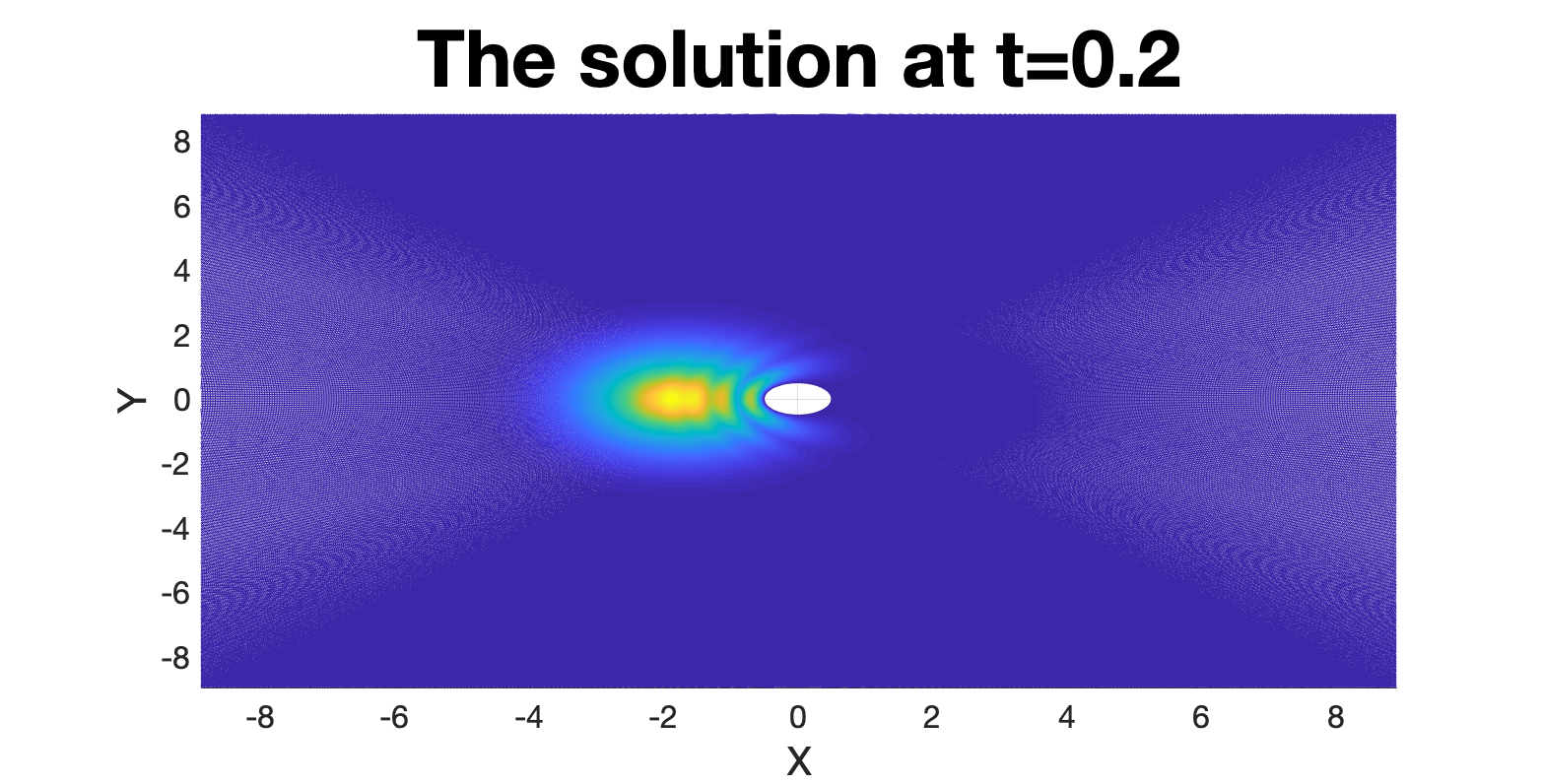}
\includegraphics[width=5.5cm,height=5cm]{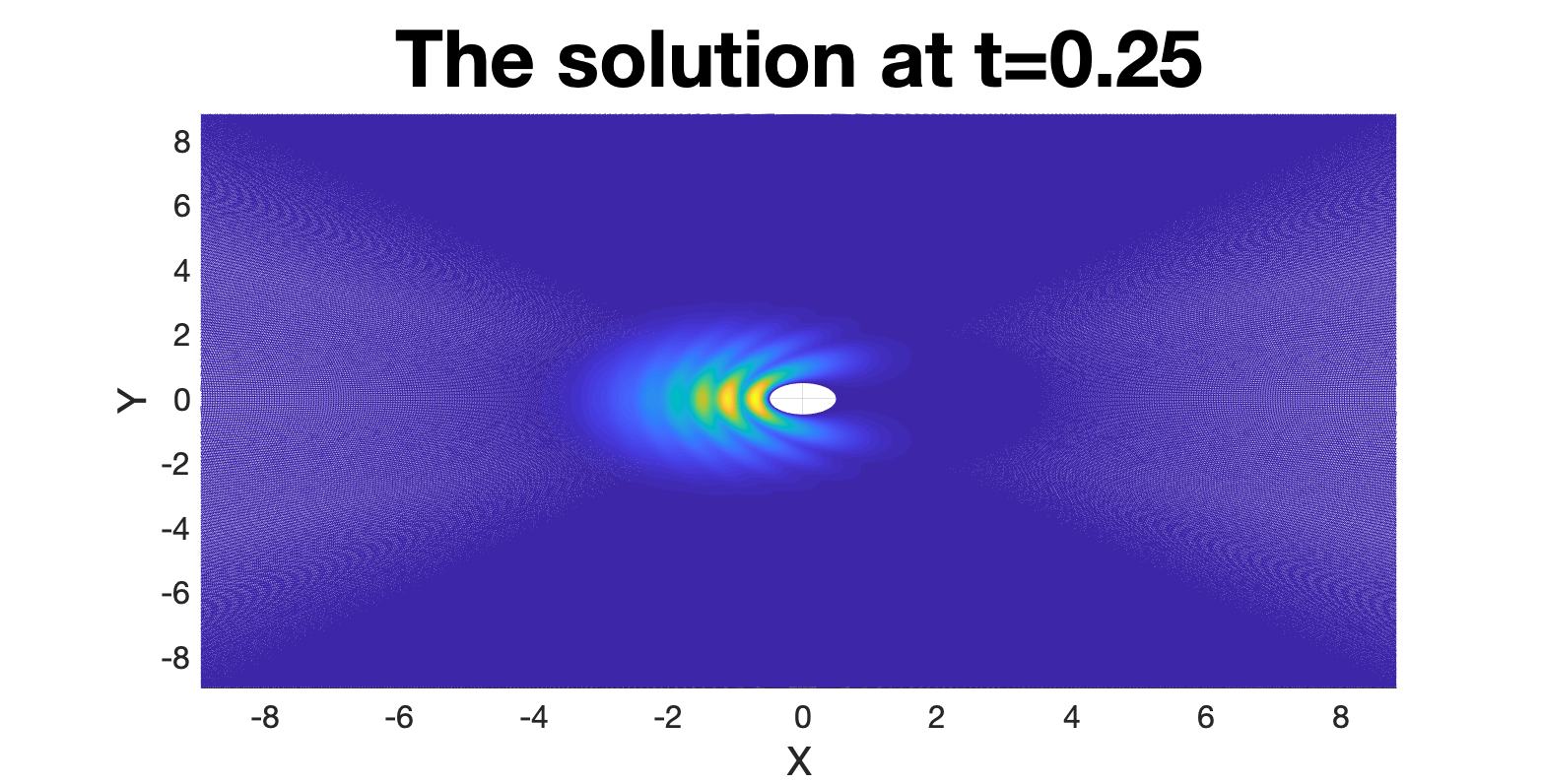}
\includegraphics[width=5.5cm,height=5cm]{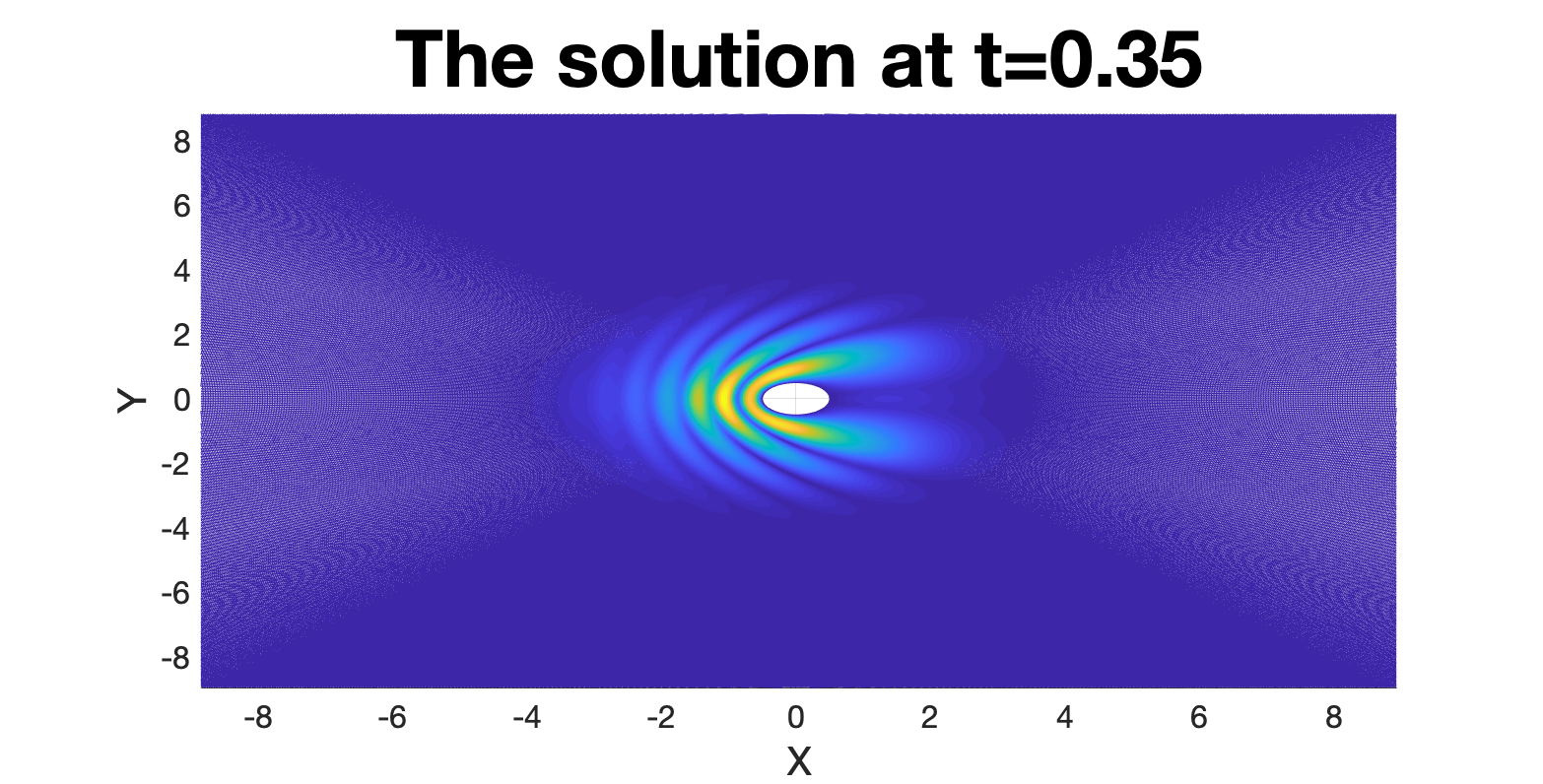}
\includegraphics[width=5.5cm,height=5cm]{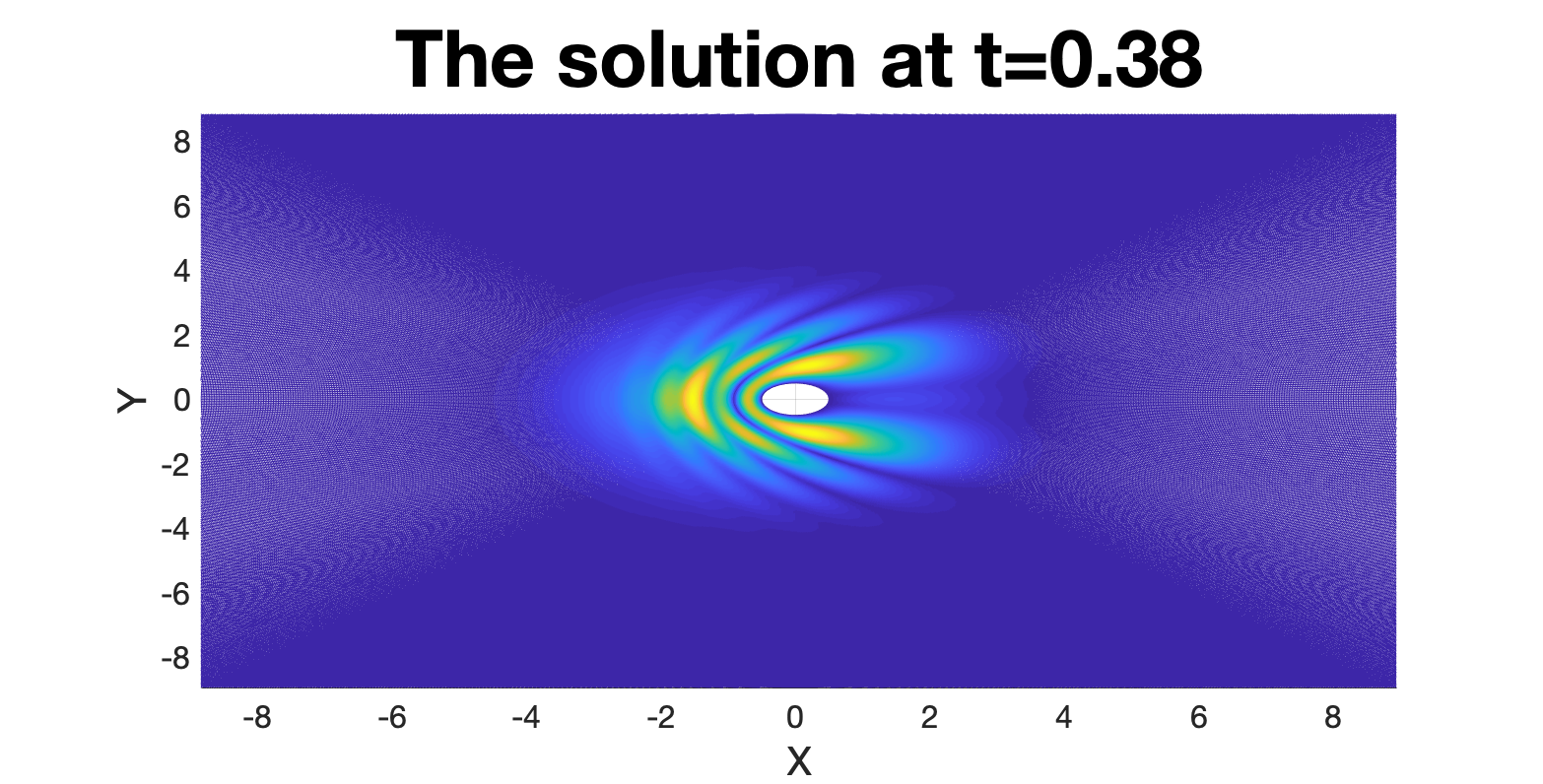}
\includegraphics[width=5.5cm,height=5cm]{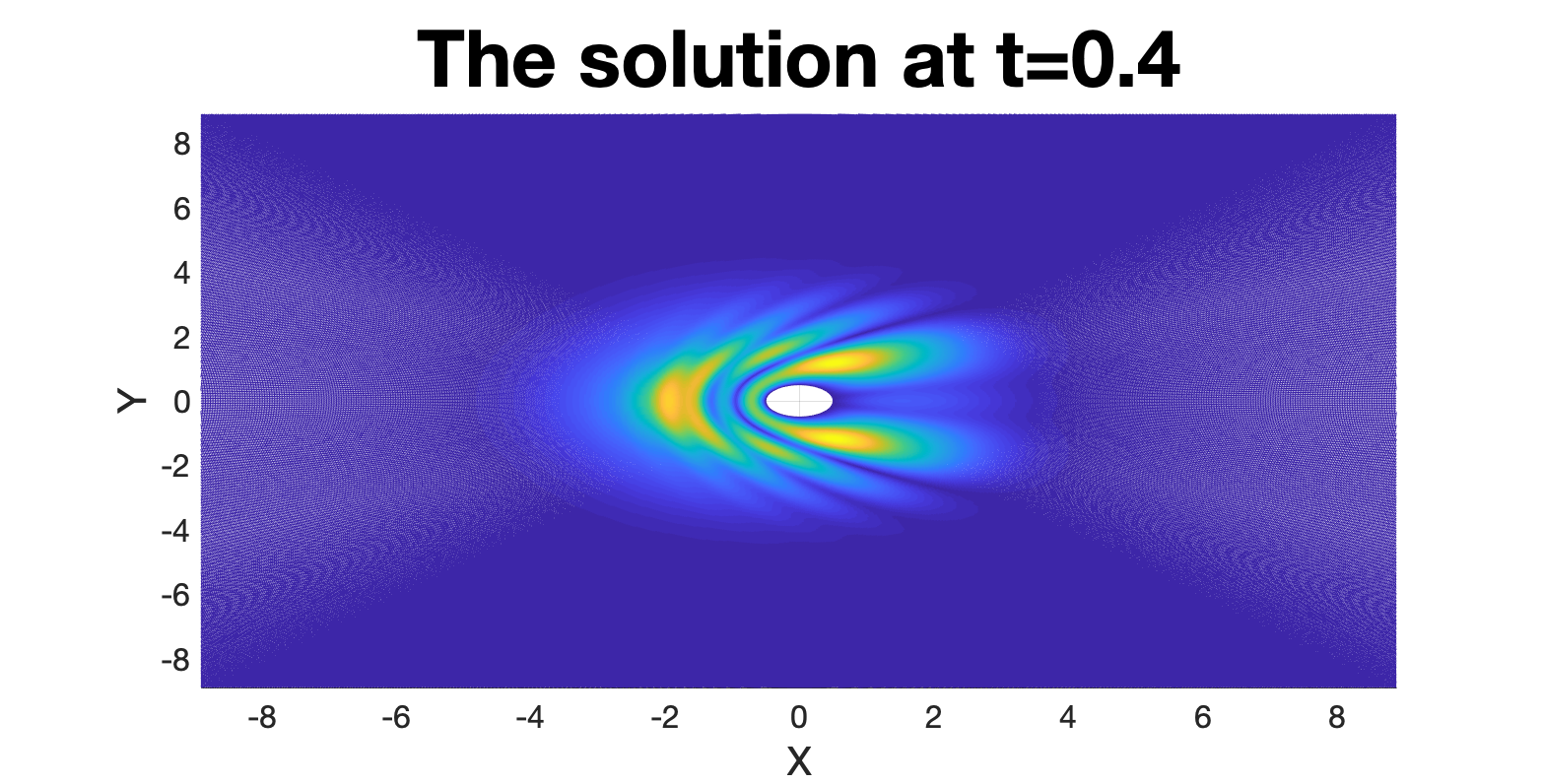}
\nonumber
\end{figure}
\begin{figure}[!h] 
\includegraphics[width=5.5cm,height=5cm]{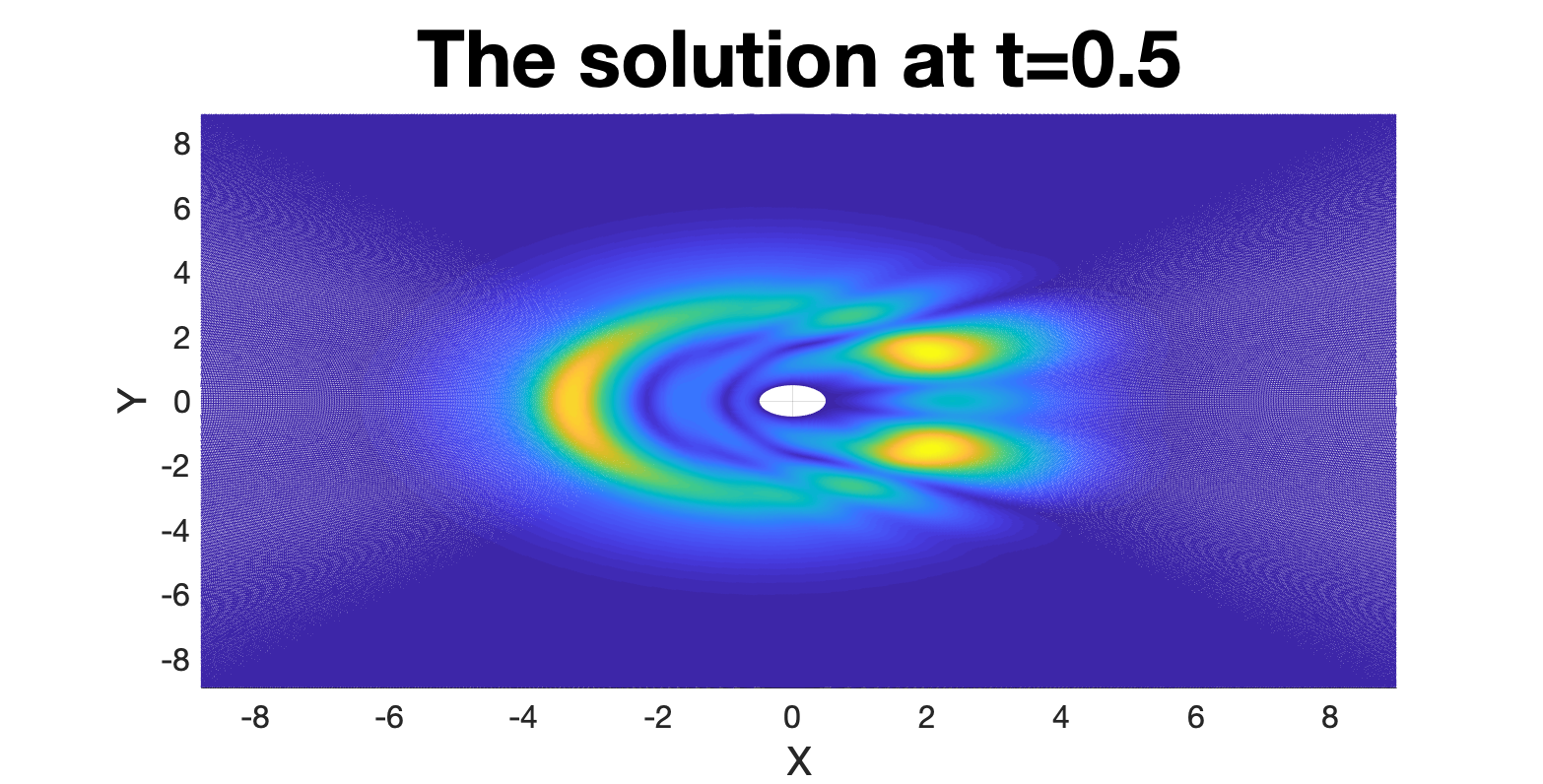}
\includegraphics[width=5.5cm,height=5cm]{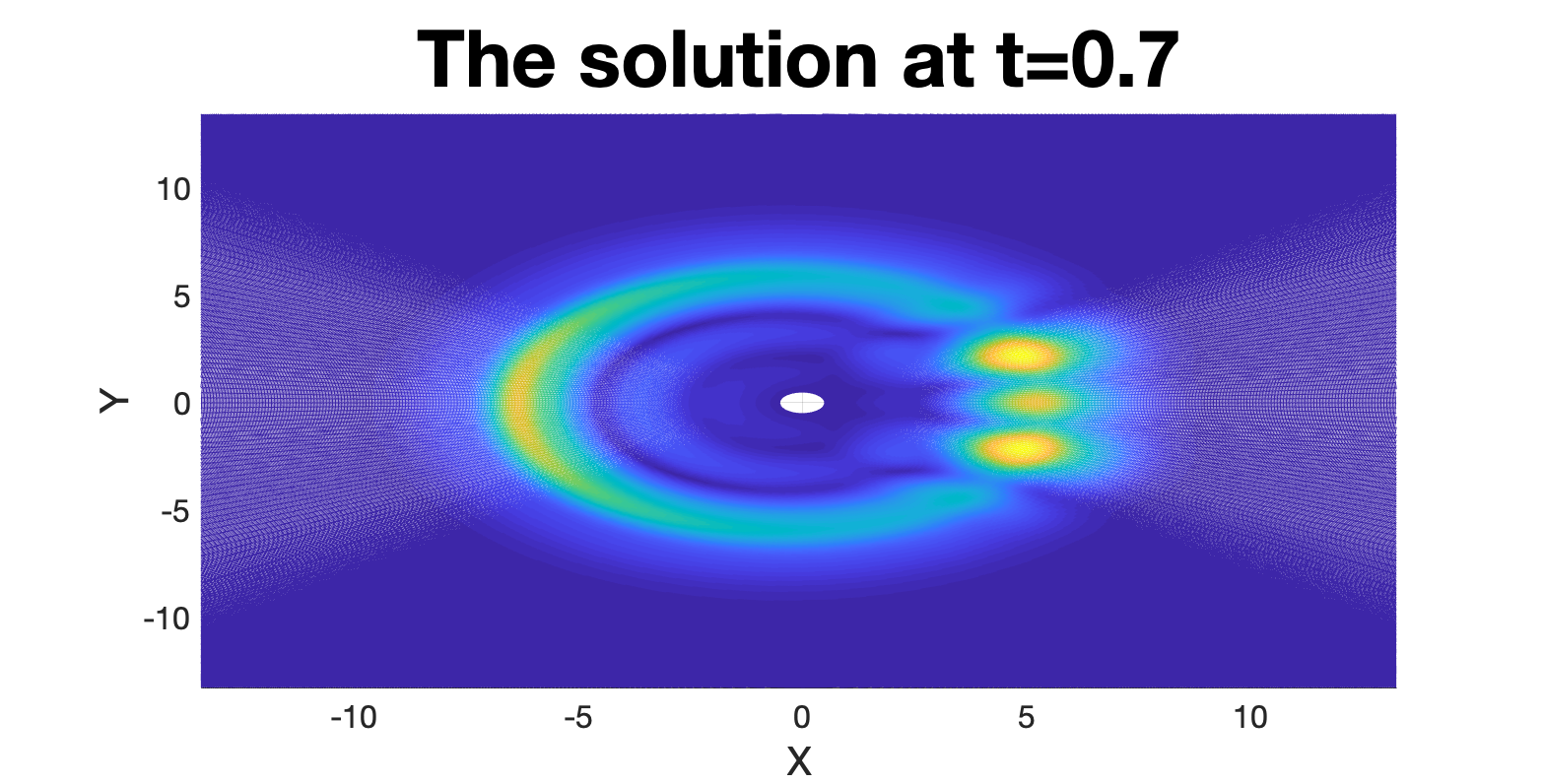}
\includegraphics[width=5.5cm,height=5cm]{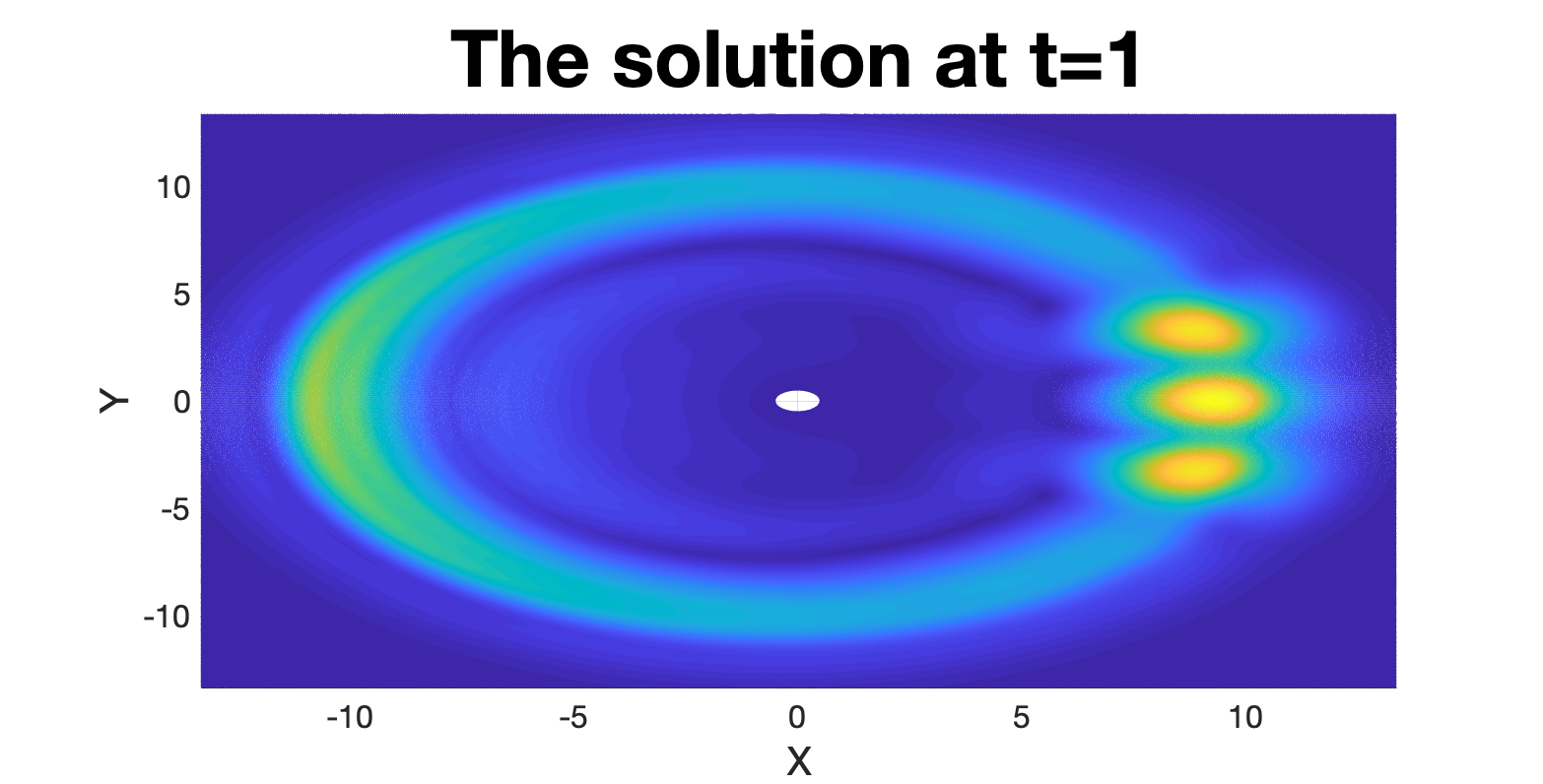}
\caption{Snapshots of the behavior of the solution $u(t)$ to the $2d$ quintic \NNls equation for different times in the strong interaction between the solution and the obstacle, with $(x,y)-$ view.}
\label{F:29}
\end{figure}

\newpage

\section{Dependence on the obstacle size}
\label{sec-Obstacle size dependance}
In this section, we describe the behavior of the solution to the cubic and quintic \NNls equations in the strong interaction case, as in Figures \ref{directionVelocityStrong} and \ref{directStrongInterac}, depending on the size of the obstacle (a disk of radius $r_{\star}$). We study the strong interaction of the solution with the obstacle in terms of the transmitted and the reflected parts of the mass: we call the transmitted mass $M_T[u^n],$ the discrete $L^2$-norm of the solution $u^n$ at time $t=t^n$ on the right-half plane
$$ 
\Omega_{+}:=\{ (r,\theta) \in \left[ r_{\star}, R \right] \times \left[ 0, 2 \pi \right] : 
 0 \leq \theta \leq \frac \pi 2 \; \text{ and } \frac{3 \pi}{2} \leq \theta \leq 2\pi \},    
$$
thus, we obtain the mass on the right-half plane $\{ (x,y)\in \Omega \; :  \; \, x\geq 0\},$
\begin{equation*} 
M_{T}[u^{n}]=    \sum_{k =0}^{N_{r}} \sum_{\substack{ 0 \leq j  \leq \frac{\pi}{ 2} \\  \\ \frac{3\pi}{2} \leq j  \leq 2\pi} }^{N_{\th}} |u^n_{k,j}|^2 \; r_k  \;  \Delta r \; \Delta  \theta, \quad \text{for} \;  n \geq 0.
\end{equation*} 
We call the reflected mass $M_{R}[u^n]$, the discrete $L^2$-norm of the solution $u^n$ at time $t=t^n$ on the left-half plane
$$ 
\Omega_{-}:=\{ (r,\theta) \in \left[ r_{\star}, R \right] \times \left[ 0, 2 \pi \right] : \;     \frac \pi 2  < \theta < \frac{3 \pi}{2} \},    
$$
so that,  we obtain the mass on the left-half plane $\{ (x,y)\in \Omega \; : \;  \,  x \leq 0\}, $ 
\begin{equation*} 
M_{R}[u^{n}]=    \sum_{k =0}^{N_{r}} \sum_{\substack{  \frac{\pi}{ 2} < j  \leq \frac{3\pi}{2} }}^{N_{\th}} |u^n_{k,j}|^2 \; r_k  \;  \Delta r \; \Delta  \theta, \quad \text{for} \;  n \geq 0.
\end{equation*}

We consider different values for the radius $r_{\star}$ of the obstacle and investigate solutions with  initial data \eqref{u0NLS} having the amplitude
$A_0$ and the velocity $v=(v_x,v_y).$ The spatial translation $(x_c,y_c)$ in the initial data depends on the radius $r_{\star}$ of the obstacle so that the initial conditions would satisfy the Dirichlet boundary condition. In \S \ref{Obs-size-Crit}, we first  study 
the cubic, $L^2$-critical, \NNls equation  and  in \S \ref{Obs-size-super-Crit}, we investigate the super-critical \NNls equation.

\subsection{The $L^2$-critical case}
\label{Obs-size-Crit}
We consider the cubic \NNls equation with the initial condition as in \eqref{u0NLS} with the following parameters:  
\begin{equation}
\label{ref-A-2.5}
 A_0=2.5, \quad v_x=15, \quad v_y=15,  \quad 
 x_c=-4-r_{\star}  \quad 
 y_c=-4-r_{\star}.
 \end{equation}
The solution $u$ moves on the line $y=x,$ i.e., in the same direction as the outward normal vector of $\Omega.$ The dependence of the solution behavior on the obstacle size in the strong interaction case is as follows (the summary is given in Table \ref{T:3}):

\begin{itemize}
\item For $ r_{\star}=0.1,$ the solution splits into two bumps with a \textit{small} backward reflection, which has a dispersive behavior. As there is a sufficient mass transmitted, the 
two bumps get back together and behave as a single solitary wave, which blows up in finite time, see Figure  \ref{Fig-critic-r0-0.1}.  

\item For $0.2 \leq r_{\star} \leq 0.4$, we observe that the solution splits into several bumps with a \textit{small} reflected part that leads to some loss of the transmitted mass and results in an overall scattering behavior of the solution (the transmitted part reconstructs back into one bump). 

\item For $0.5 \leq r_{\star} \leq 0.7$, we observe that the solution splits into several \textit{small}  bumps, which later make up the transmitted part of the solution with an  \textit{important} reflected part. 
 The solution has a similar behavior as in the previous case, however, the reflected mass is higher than the half of the total mass $\frac 12 M[u_0]=9.8175,$ which endorses the dispersive (scattering) behavior of the solution.  
 
\item For $0.8 \leq r_{\star}<3,$ the interaction surface between the solution and the obstacle is sufficiently large, thus, the solution scatters: the main part of the solution is mostly reflected back with a small dispersive transmitted part.  
 We observe a reverse behavior of the transmitted and the reflected part of the solution compared to case for $0.2 \leq r_{\star} \leq 0.4,$ see Figure \ref{Fig-critic-r0-2} (for an example with $r_{\star}=2$).
 
\item  For $r_{\star} \geq 4,$ the solution behaves quite different from the previous cases: the solution blows up in finite time at the boundary of the obstacle, as  the interaction region is larger than the contour of the solution 
 and there is almost no transmission of the solution (the transmitted mass is around $10^{-5}$).  The solution can not cross or get around the obstacle boundary, hence, the solution concentrates in its blow-up core at the obstacle's boundary, see Figure \ref{Fig-critic-r0-5} (for an example with $r_{\star}=5$). 
\end{itemize}

{\footnotesize
\begin{table}[h!]
\centering 
\begin{tabular}{|m{1.2cm}|c|c|c|c|}   
  \hline
   $r_{\star}$            & Discrete total mass     &  Behavior of the solution       & Discrete reflected mass                                               & Discrete transmitted mass        \\                                    
  \hline  
  $r_{\star}=0.1$      &  $19.6350$        & Blow-up at $t \approx 0.888 $    &   $3.7746$ at $t \approx 0.888$                     &  $15.8603$ at $t \approx 0.888$           \\
   \hline 
     $r_{\star}=0.2$      &  $19.6350$      & Scattering                                  &    $5.2641 $ at $t \approx 1.2$                        &  $14.3709$ at $t \approx 1.2$              \\
   \hline 
     $r_{\star}=0.3$      &  $19.6350$      & Scattering                                  &   $7.1144   $ at $t \approx 1.2$                        &  $12.5206 $ at $t \approx 1.2$               \\
   \hline 
     $r_{\star}=0.4$      &  $19.6350$      & Scattering                                  &  $8.8085   $ at $t \approx 1.2$                        &  $10.8264 $ at $t \approx 1.2$             \\
   \hline 
    $r_{\star}=0.5$      &  $19.6350$      & Scattering                                   &  $10.1753$ at $t \approx 1.2 $                        &  $9.4596 $ at $t \approx 1.2 $        \\
        \hline 
            $r_{\star}=0.6$      &  $19.6350$      & Scattering                                   &   $11.4460 $ at $t \approx 1.2 $              &  $8.1890 $ at $t \approx 1.2 $        \\
        \hline 
    $r_{\star}=0.7$      & $ 19.6350$      & Scattering                                   &   $12.5008 $ at $t \approx 1.2 $                       &  $7.1341 $ at $t \approx 1.2 $       \\
        \hline 
           $r_{\star}=0.8$      &  $19.6350 $     & Scattering                                   &  $13.4017 $ at $t \approx 1.2 $                 &  $6.2332 $ at $t \approx 1.2 $        \\
        \hline
              $r_{\star}=0.9$      &  $19.6350$      & Scattering                                   &   $14.1658 $ at $t \approx 1.2 $             &  $5.4691 $ at $t \approx 1.2 $        \\
        \hline
     $r_{\star}=1$      &  $19.6350$        &  Scattering                                  &  $ 14.8208 $ at $t \approx 1.2$                        &   $  4.8141$ at $t \approx 1.2 $      \\
\hline
  $r_{\star}=2$          &  $19.6350 $      &  Scattering                                  &   $18.0537$  at $t \approx 1.2$                        &   $ 1.5813$ at $t \approx 1.2 $        \\
  \hline
    $r_{\star}=3$          &  $19.6350$     & Scattering                                   &   $  19.0024$  at $t \approx 1.2$                      &    $ 0.6325  $ at $t \approx 1.2 $       \\
  \hline
   $r_{\star}=4$          &  $19.6350$     &   Blow-up at $t \approx 0.36 $                              &   $  19.6349$  at $t \approx 0.36$                      &    $  5.6838e-05 $ at $t \approx 0.36 $       \\
  \hline
   $r_{\star}=5$          &  $19.6350$     &   Blow-up at $t \approx 0.36 $                              &   $  19.6349$  at $t \approx 0.36$                      &    $  5.6838e-05 $ at $t \approx 0.36 $       \\
  \hline
\end{tabular}
  \caption{Different obstacle size $r_{\star}$ and the corresponding behavior of the solution $u(t)$ to the $2d$ cubic \NNls equation with $u_0$ from \eqref{u0NLS} with \eqref{ref-A-2.5} with the discrete total mass, reflected and transmitted discrete mass after interaction with the obstacle at time $t.$ }
  \label{T:3}
\end{table}
}
Next, we give examples of the solution to the cubic \NNls equation with a slightly smaller initial amplitude $A_0,$ depending on the size of the obstacle (the strong interaction case),  see Figure \ref{directionVelocityStrong}.
We consider the initial condition  \eqref{u0NLS}, with the following parameters: 
\begin{equation}
\label{ref-A-2.25}
 A_0=2.25, \quad v_x=15, \quad v_y=15,  \quad  x_c=-4-r_{\star}  \; \text{ and  } y_c=-4-r_{\star}.
 \end{equation} 
We record all observations in this case in Table \ref{T:3-1}.

\begin{figure}[!h]      
\centering
\includegraphics[width=5.5cm,height=5cm]{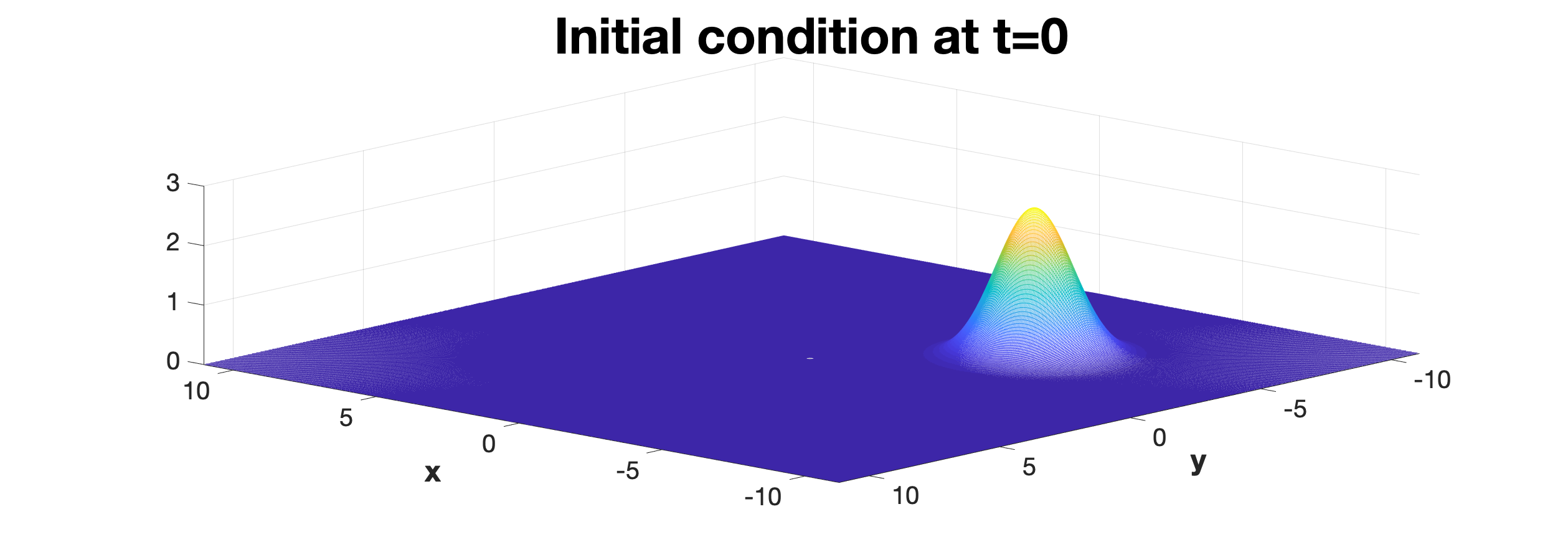}
\includegraphics[width=5.5cm,height=5cm]{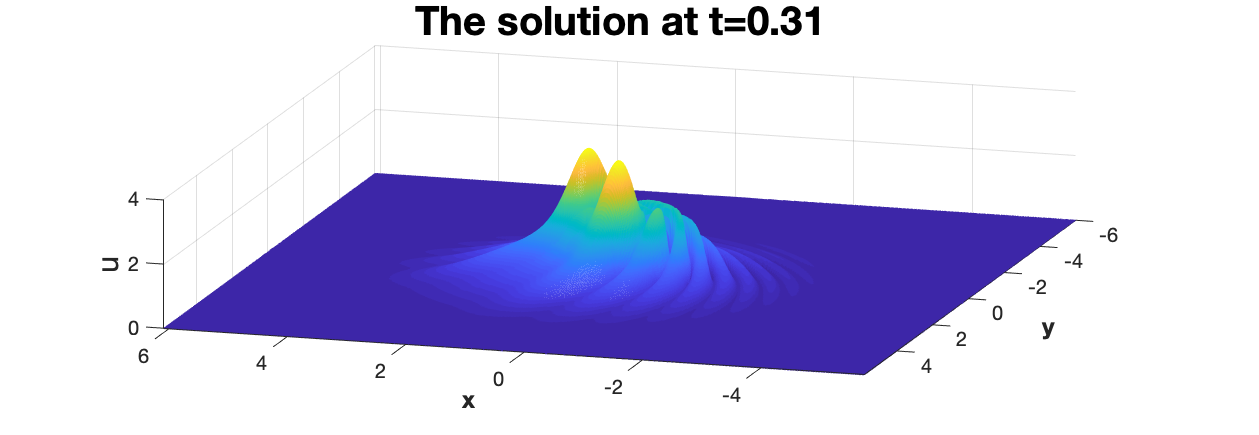}
\includegraphics[width=5.5cm,height=5cm]{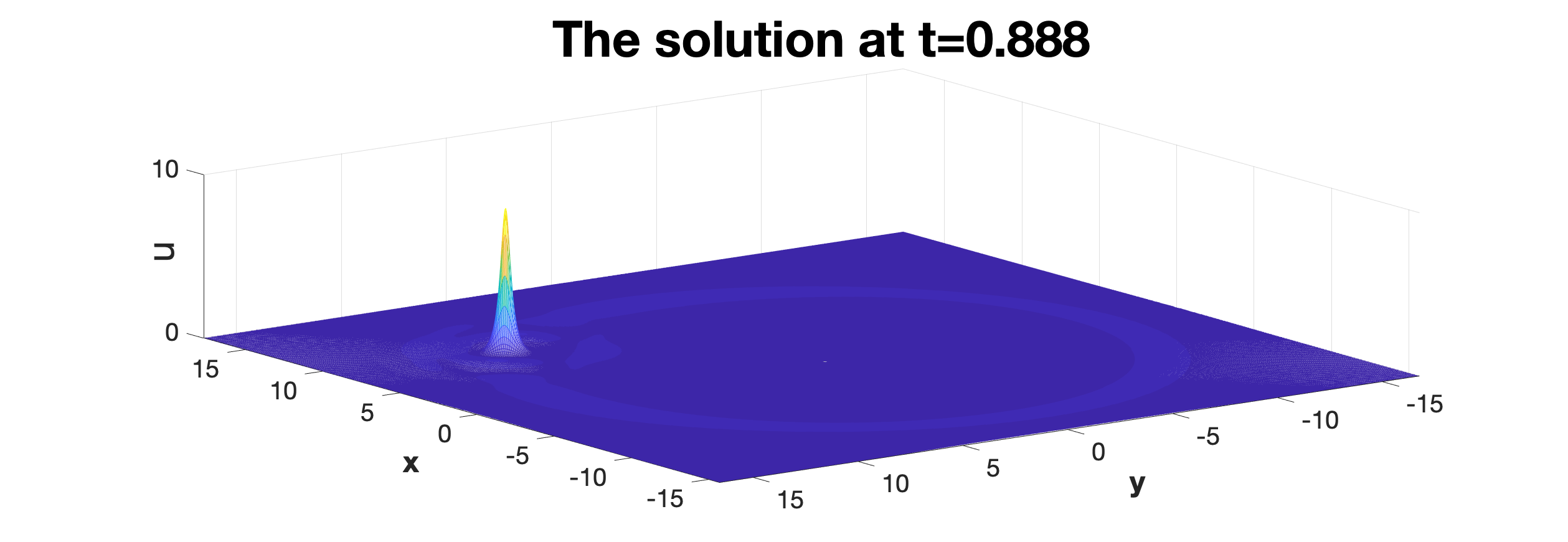}
\includegraphics[width=8.5cm,height=4cm]{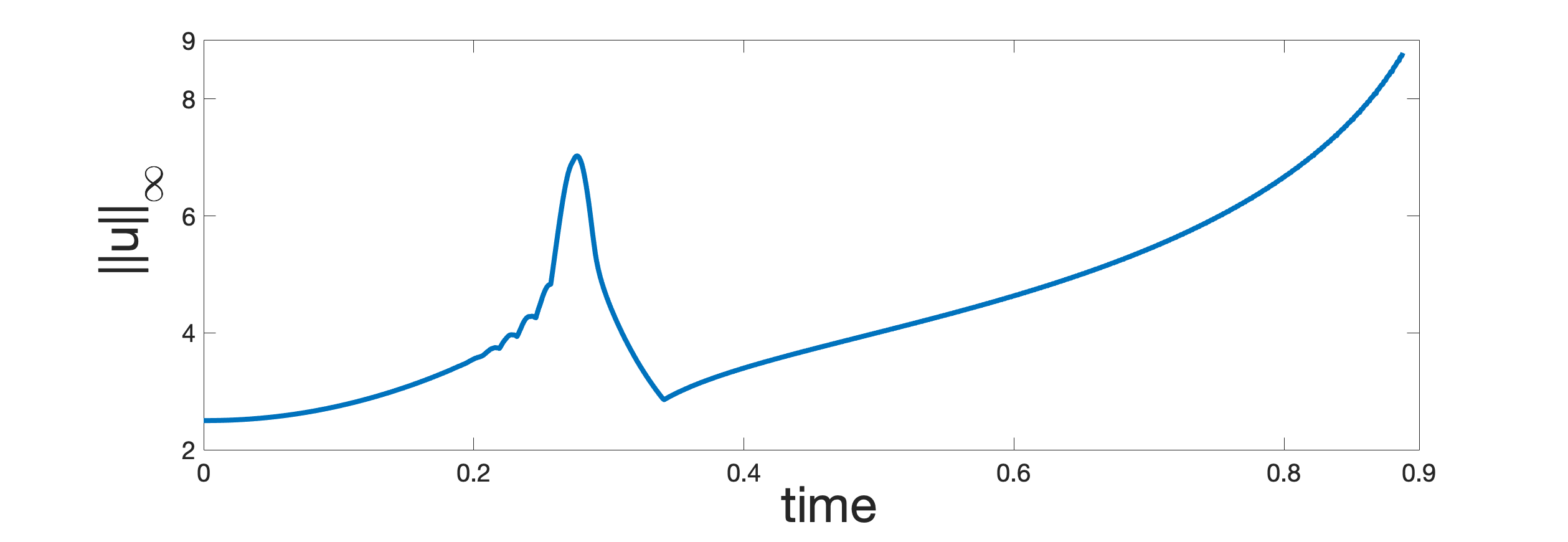}
\includegraphics[width=8.5cm,height=4cm]{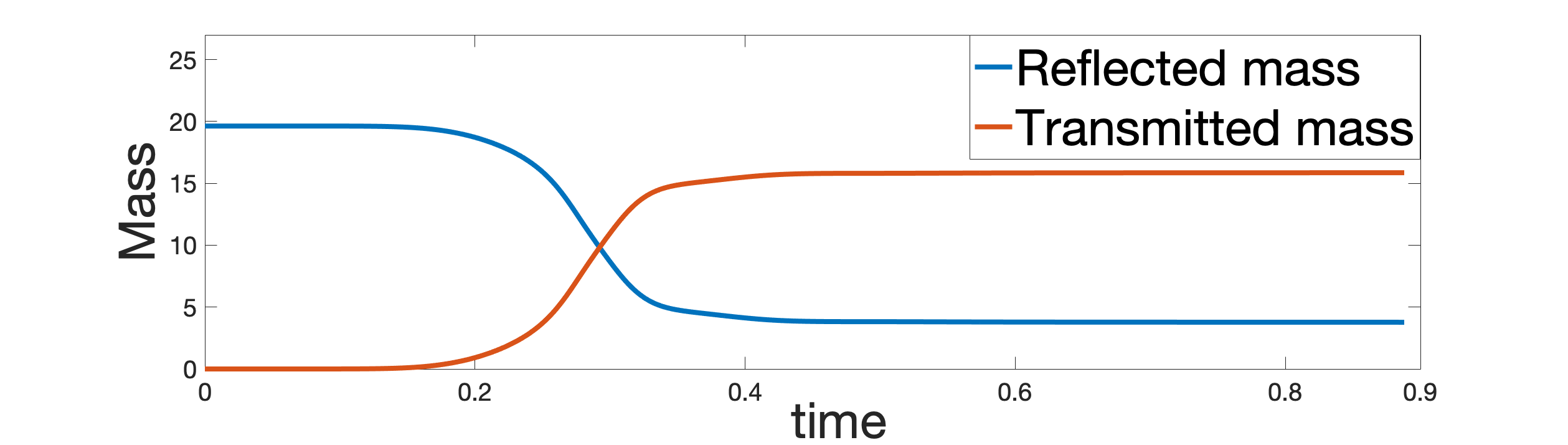}
\caption{Solution to the $2d$ cubic \NNls equation with the radius of the obstacle~$r_{\star}=0.1$ and $u_0$ from \eqref{u0NLS} with \eqref{ref-A-2.5} moving along the line $y=x$. 
Snapshots of the blow-up solution $u(t)$ at $t=0$ (top left), $t=0.31$ (middle top) and $t=0.888$ (top right). 
Time dependence of the $L^{\infty}$-norm (bottom left) and of the transmitted and the reflected mass (bottom right). }
\label{Fig-critic-r0-0.1}
\end{figure}

\begin{figure}[!ht]      
\centering
\includegraphics[width=5.5cm,height=5cm]{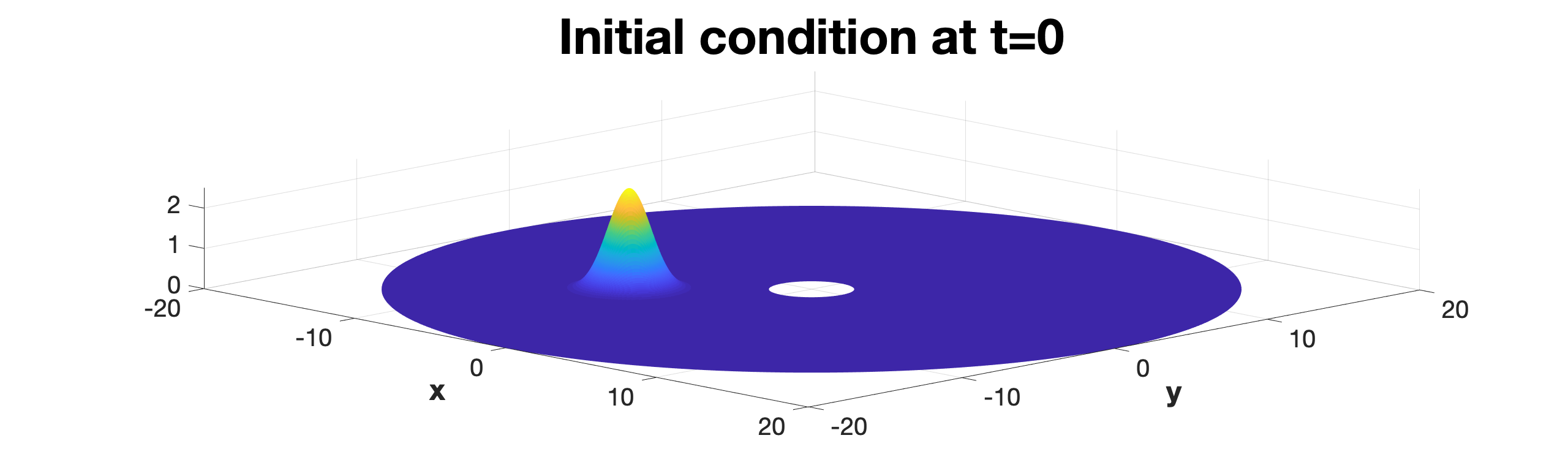}
\includegraphics[width=5.5cm,height=5cm]{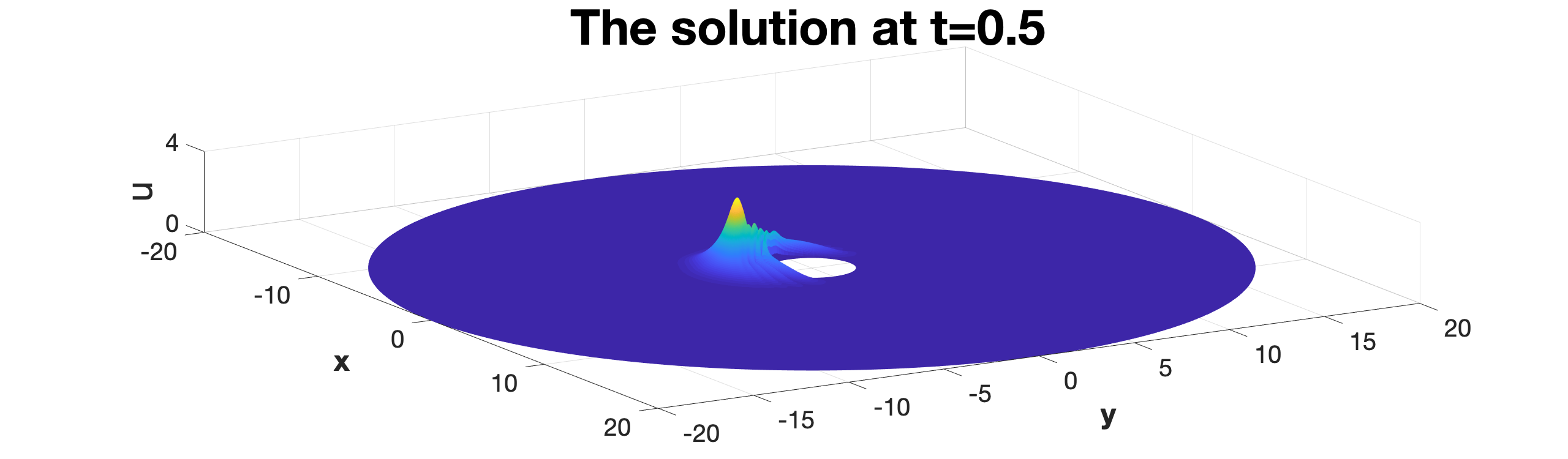}
\includegraphics[width=5.5cm,height=5cm]{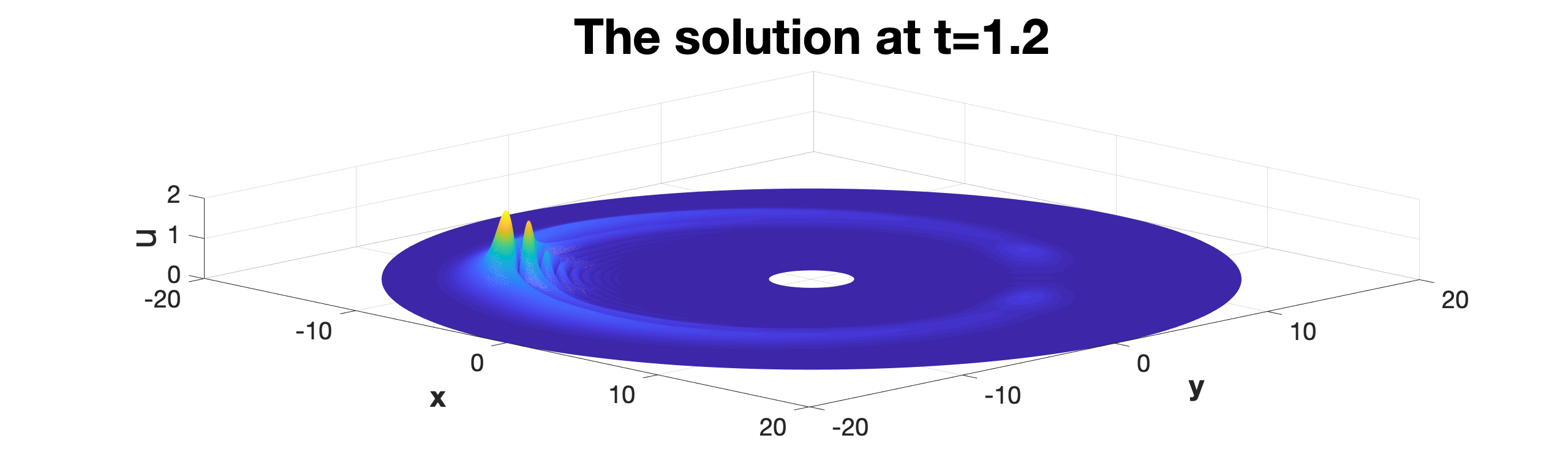}
\includegraphics[width=8.5cm,height=4cm]{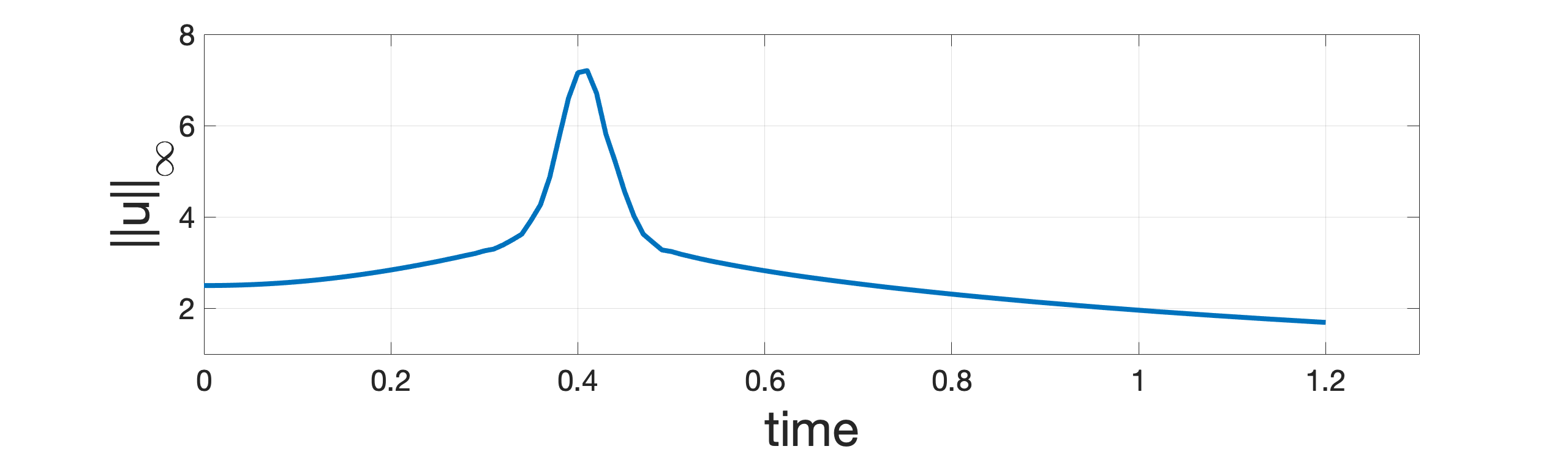}
\includegraphics[width=8.5cm,height=4cm]{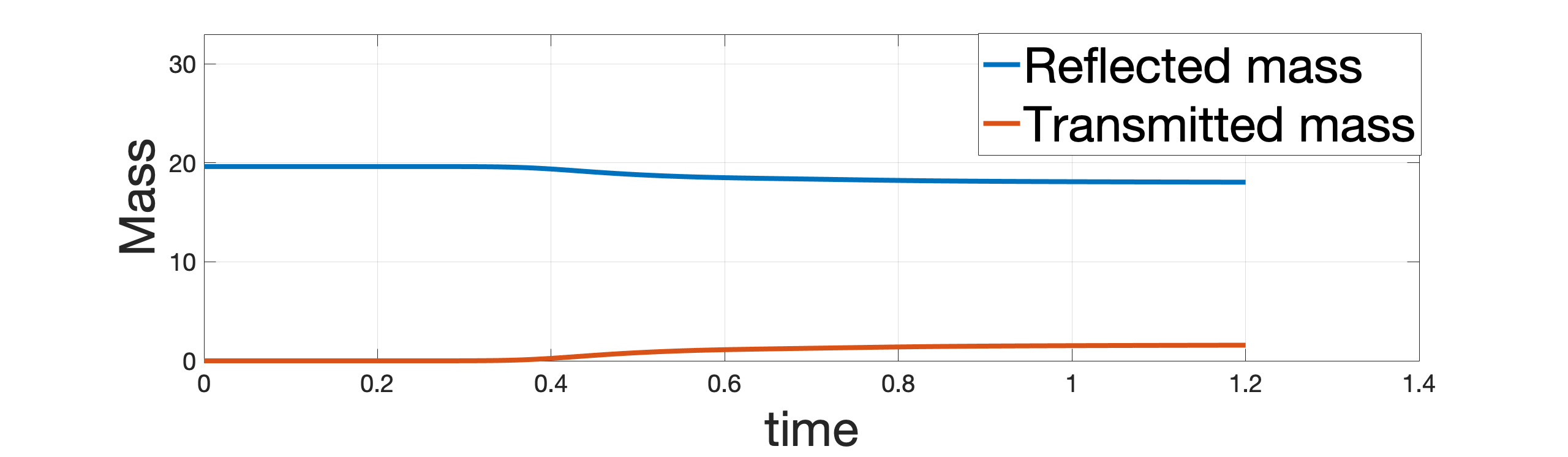}
\caption{Solution to the $2d$ cubic \NNls equation with the obstacle radius  $r_{\star}=2$ and $u_0$ from \eqref{u0NLS} with \eqref{ref-A-2.5} moving along the line $y=x$. 
Snapshots of the scattering solution $u(t)$ at $t=0$ (top left), $t=0.5$ (middle top) and $t=1.2$ (top right). 
Time dependence of the $L^{\infty}$-norm (bottom left) and of the transmitted and the reflected mass (bottom right).
}
\label{Fig-critic-r0-2}
\end{figure}

\begin{figure}[!ht]      
\centering
\includegraphics[width=5.5cm,height=5cm]{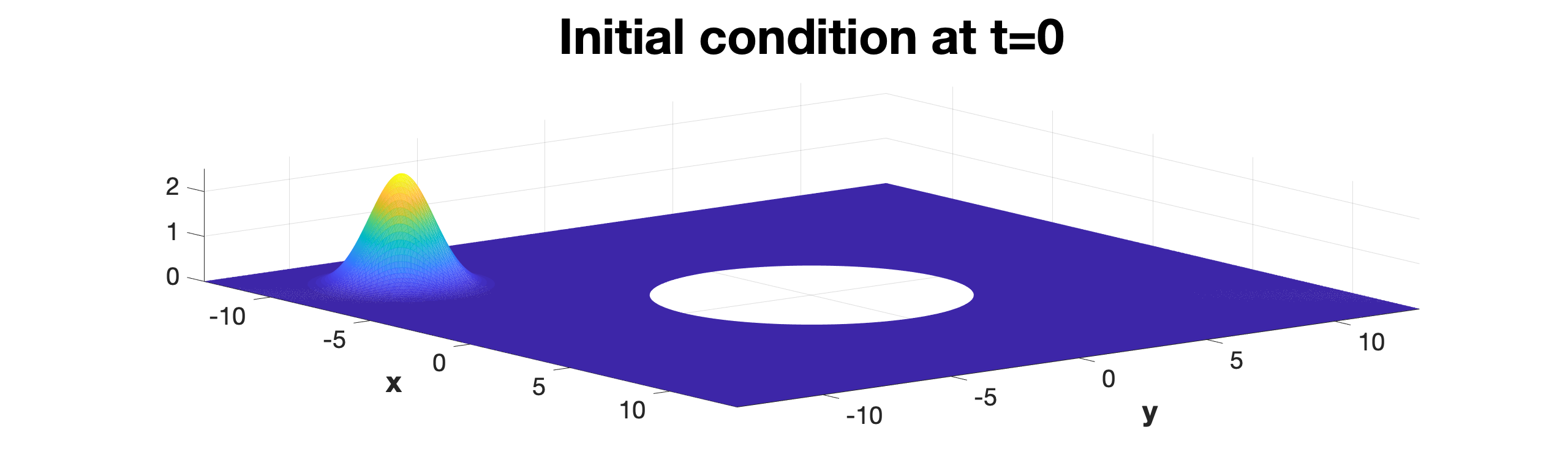}
\includegraphics[width=5.5cm,height=5cm]{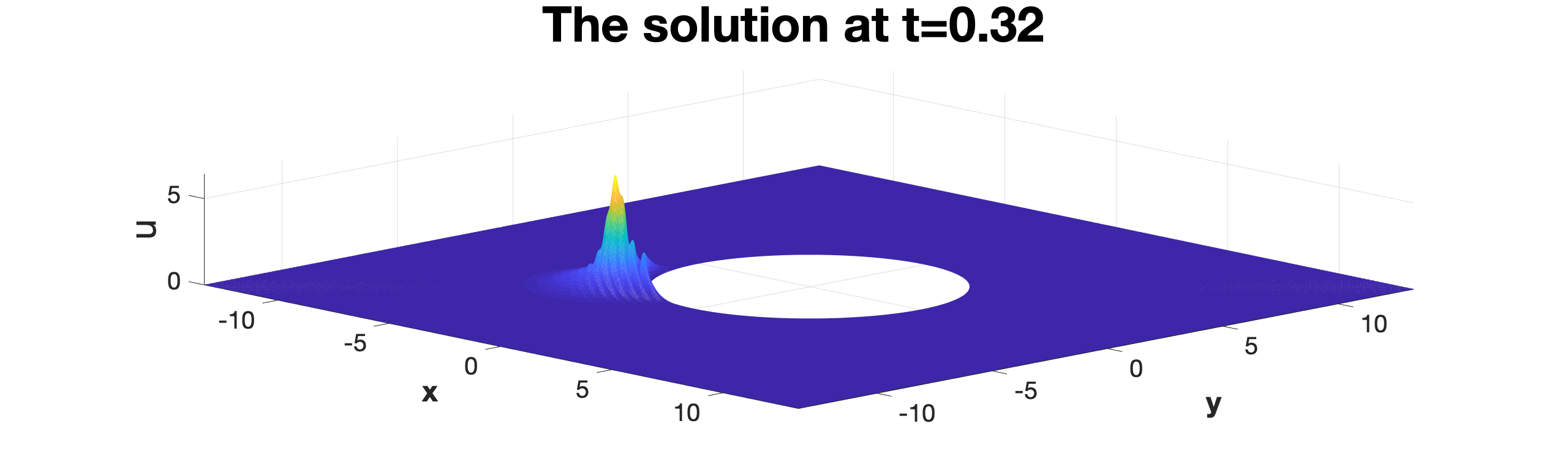}
\includegraphics[width=5.5cm,height=5cm]{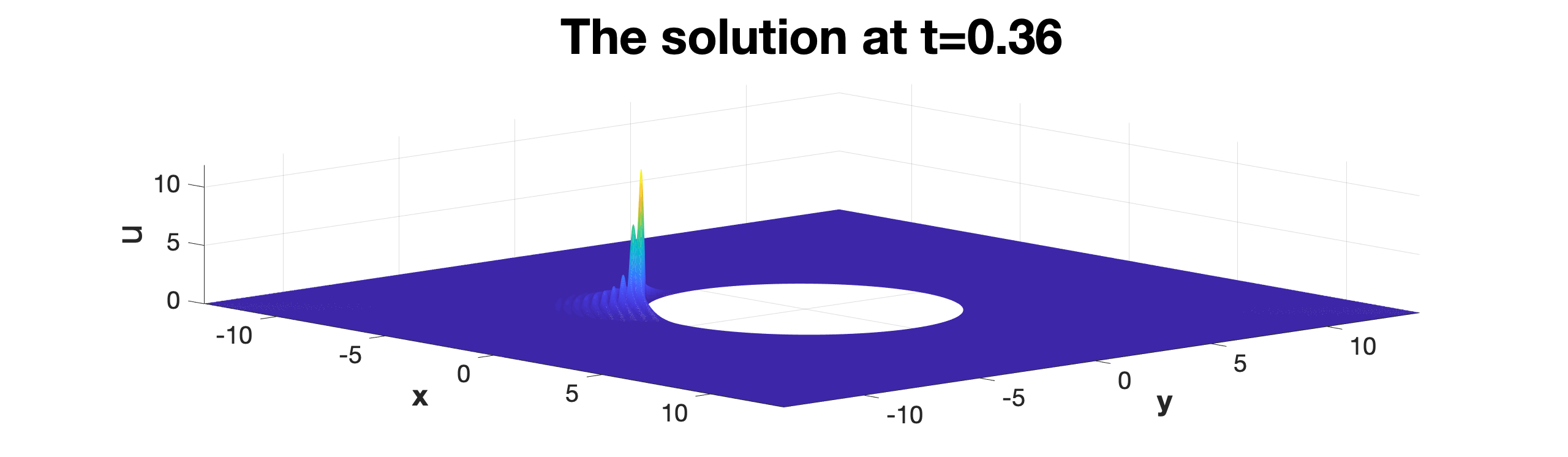}
\includegraphics[width=8.5cm,height=4cm]{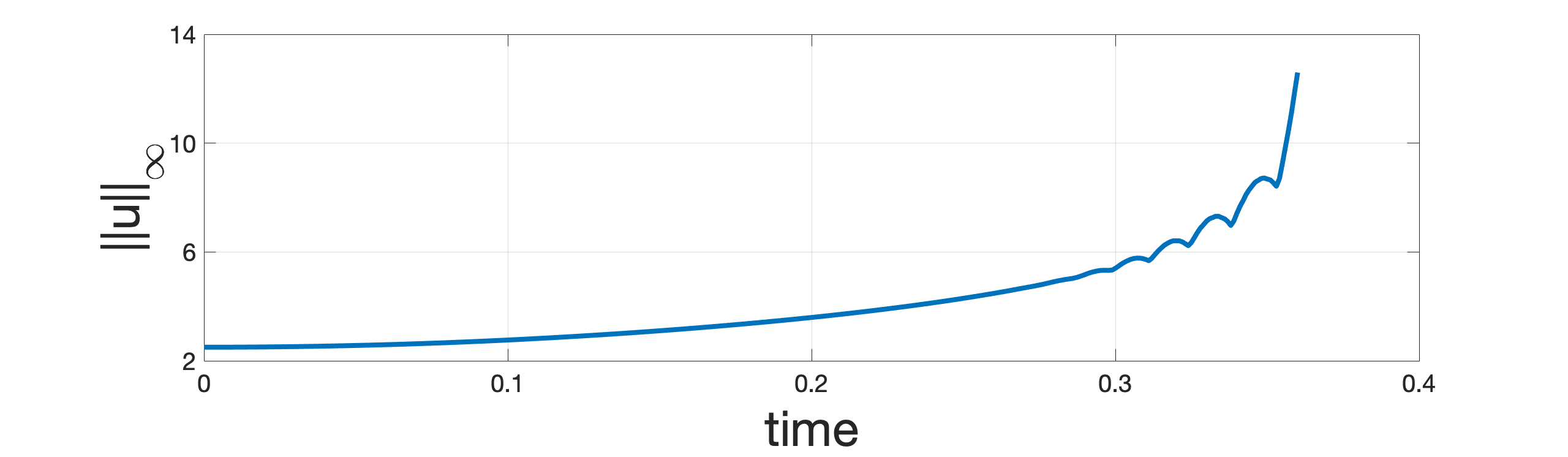}
\includegraphics[width=8.5cm,height=4cm]{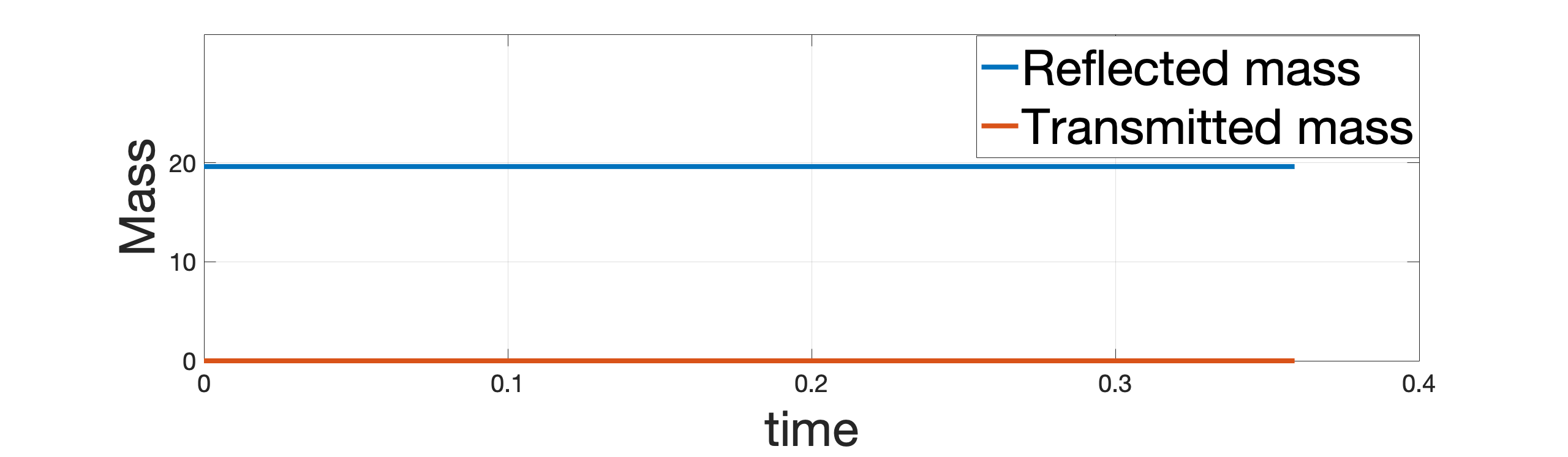}
\caption{Solution to the $2d$ cubic \NNls equation with the obstacle radius $r_{\star}=5$ and $u_0$ from \eqref{u0NLS} with \eqref{ref-A-2.5} moving along the line $y=x$. 
Snapshots of the blow-up solution $u(t)$ at $t=0$ (top left), $t=0.32$ (middle top) and $t=0.36$ (top right). 
Time dependence of the $L^{\infty}$-norm (bottom left) and of the transmitted and the reflected mass (bottom right).
}
\label{Fig-critic-r0-5}
\end{figure}

{\footnotesize
\begin{table}[h!]
\centering 
\begin{tabular}{|m{1.2cm}|c|c|c|c|}   
  \hline
   $r_{\star}$            & Discrete total mass     &  Behavior of the solution       & Discrete reflected mass                                               & Discrete transmitted mass        \\                                    
  \hline  
  $r_{\star}=0.1$      &  $15.9043$        &   Blow-up at $t \approx0.891 $                  &   $ 2.448$ at $t \approx0.891 $                      &  $13.4563$ at $t \approx0.891 $            \\
   \hline 
     $r_{\star}=0.2$      &  $15.9043$      & Scattering                                  &    $ 3.684 $ at $t \approx 1.2$                        &  $12.2203$ at $t \approx 1.2$              \\
   \hline       
     $r_{\star}=0.3$      &  $15.9043$      & Scattering                                  &   $ 5.0335 $ at $t \approx 1.2$                        &  $10.8708$ at $t \approx 1.2$               \\
   \hline 
     $r_{\star}=0.4$      &  $15.9043$      & Scattering                                  &  $6.3137  $ at $t \approx 1.2$                        &  $9.5906$ at $t \approx 1.2$             \\
   \hline 
    $r_{\star}=0.5$      &  $15.9043$      & Scattering                                   &  $7.4179$ at $t \approx 1.2 $                        &  $8.4864$ at $t \approx 1.2 $        \\
        \hline 
            $r_{\star}=0.6$      &  $15.9043$      & Scattering                           &   $8.3902$ at $t \approx 1.2 $              &  $7.5141 $ at $t \approx 1.2 $        \\
        \hline 
    $r_{\star}=0.7$      & $ 15.9043$      & Scattering                                   &   $9.2624$ at $t \approx 1.2 $                       &  $ 6.6419 $ at $t \approx 1.2 $       \\
        \hline 
           $r_{\star}=0.8$      &  $15.9043$     & Scattering                                   &  $10.0278$ at $t \approx 1.2 $                 &  $5.8765$ at $t \approx 1.2 $        \\
        \hline
              $r_{\star}=0.9$      &  $15.90430$      & Scattering                                   &   $10.6936 $ at $t \approx 1.2 $             &  $5.2107$ at $t \approx 1.2 $        \\
        \hline
     $r_{\star}=1$      &  $15.9043$        &  Scattering                                  &  $ 11.2791 $ at $t \approx 1.2$                        &   $4.6253$ at $t \approx 1.2 $      \\
\hline
  $r_{\star}=2$          &  $15.9043$      &  Scattering                                  &   $14.3843$  at $t \approx 1.2$                        &   $ 1.52$ at $t \approx 1.2 $        \\
  \hline
    $r_{\star}=3$          &  $15.9043$     & Scattering                                   &   $  15.3338$  at $t \approx 1.2$                      &    $  0.57049 $ at $t \approx 1.2 $       \\
  \hline
   $r_{\star}=4$          &  $15.9043$     &  Scattering                             &   $  15.65$  at $t \approx 1.2$                      &    $  0.2543$ at $t \approx 1.2$       \\
  \hline
   $r_{\star}=5$          &  $15.9043$     &   Blow-up at $t \approx 0.754$                              &   $  15.8291$  at $t \approx 0.36$                      &    $  0.07523 $ at $t \approx 0.36 $       \\
     \hline
   $r_{\star}=6$          &  $15.9043$     &   Blow-up at $t \approx 0.754$                              &   $  15.8291$  at $t \approx 0.36$                      &    $  0.07523 $ at $t \approx 0.36 $       \\
  \hline
\end{tabular}
  \caption{Influence of the obstacle radius $r_{\star}$ onto the behavior of the solution $u(t)$ to the $2d$ cubic \NNls equation  with $u_0$ from \eqref{u0NLS} and \eqref{ref-A-2.25} with the (initial) discrete total mass, and after the interaction the reflected and transmitted discrete mass parts at time $t$.  }
  \label{T:3-1}
\end{table}
}

We observe that in the case $r_{\star}=0.1,$ in both examples \eqref{ref-A-2.5} and \eqref{ref-A-2.25} with $A_0=2.25$ and $A_0=2.5$, the solutions blow up in finite time,  as the reflection parts of the respective solutions are small and almost all of the solution is transmitted. The time evolution proceeds as follows 
First, the solution hugs around the obstacle, splitting into two bumps and then, since  the radius of the obstacle is small,  the solution gets back together to form a single bump in a few time steps. One can observe that the solution 
has a substantial transmitted mass, which leads to a blow up in finite time. 
It seems as the solitary bump has a similar shape of the solution as if there would be no obstacle interaction near the blow-up time, in particular, it would be interesting to investigate the profile and other features of the blow-up solution after the interaction with the obstacle). 
As the radius of the obstacle increases, we see very different dynamics of the solution compared to the {\rm{NLS}$_{\R^2}$} equation in the whole space, since this solution blows up in finite time when the obstacle is absent (or has a very small radius). 
 
\subsection{The $L^2$-supercritical case}
\label{Obs-size-super-Crit}
In this section, we summarize the behavior of the solution to the quintic \NNls equation depending on the radius $r_{\star}$ of the obstacle and we study the strong interaction. 
As before, we take a set of values of the obstacle radius $r_{\star}$ and investigate the behavior of the solution moving on the line $y=0$ (in the strong interaction case) as shown in Figure \ref{directStrongInterac}. 

We take the initial data as in \eqref{u0NLS} and fix the following parameters: 
\begin{equation}
\label{ref-A-1.25}
A_0=1.25, \quad v_x=15, \quad v_y=0,    \quad  x_c=-4-r_{\star}  \; \text{ and  }  y_c=0.
\end{equation} 
We first discuss behavior of solutions for different sized of the obstacle, specifically, we show the snapshots of the time evolution of the above data for the obsctacle radii $r_{\star}=0.1$, $r_{\star}=1$ and $r_{\star}=3$. Then we provide the summary of results in Table \ref{T:4}.  \\

\begin{figure}[!ht]          
\centering
\includegraphics[width=5.5cm,height=5.0cm]{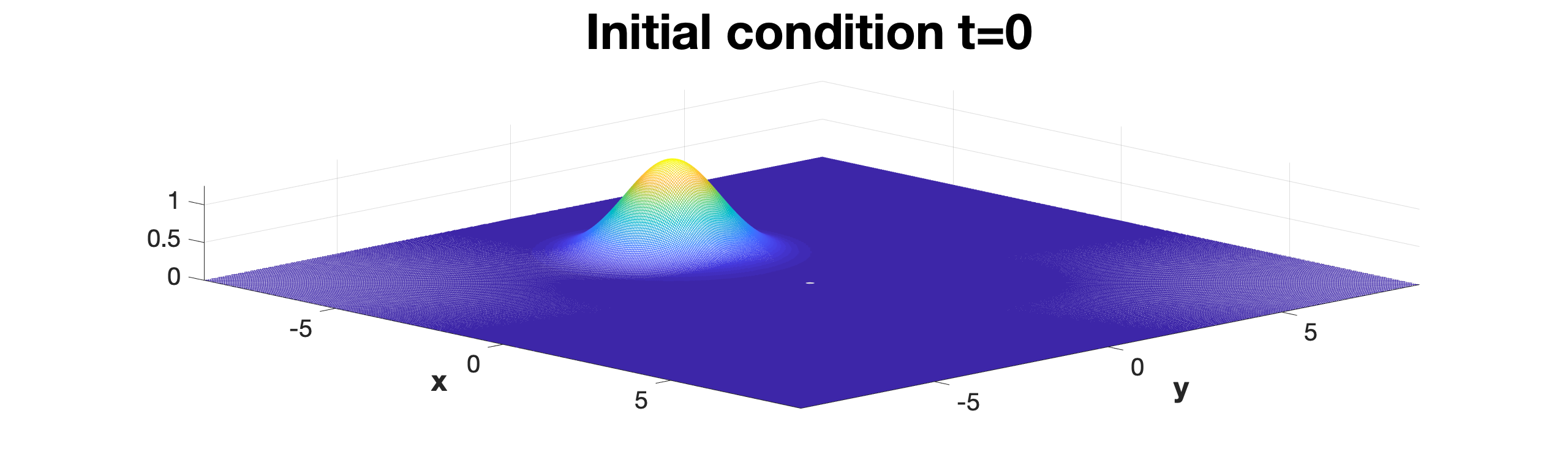}
\includegraphics[width=5.5cm,height=5.0cm]{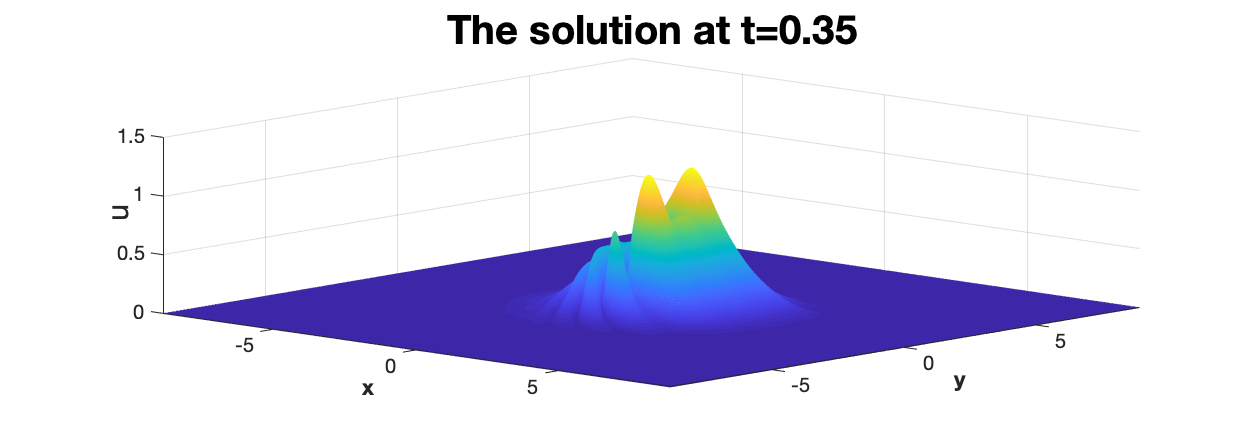}
\includegraphics[width=5.5cm,height=5.0cm]{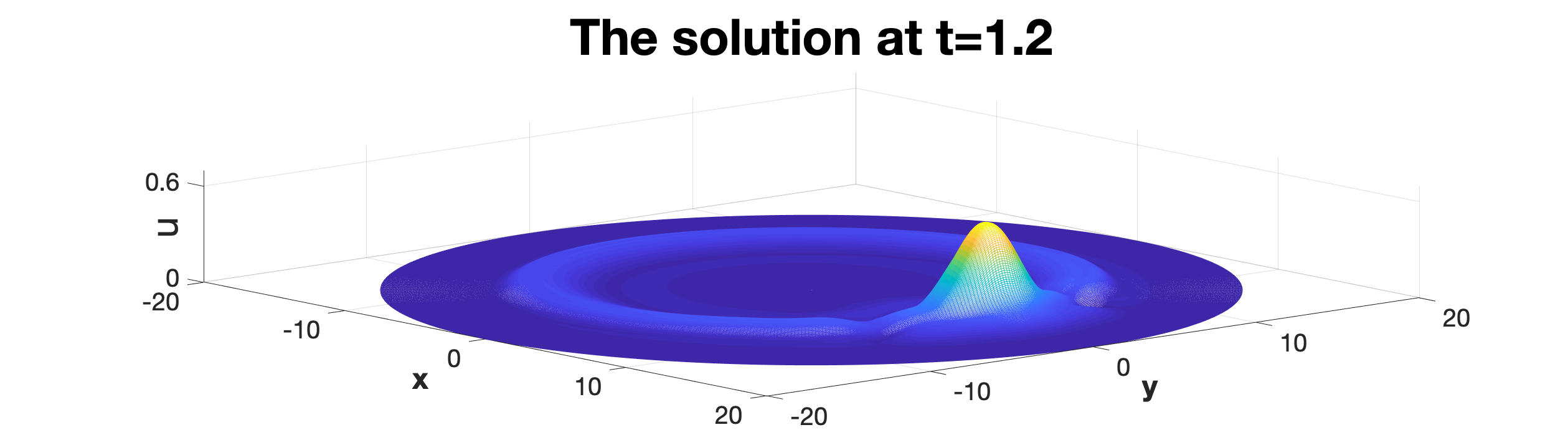}
\includegraphics[width=8.5cm,height=4.0cm]{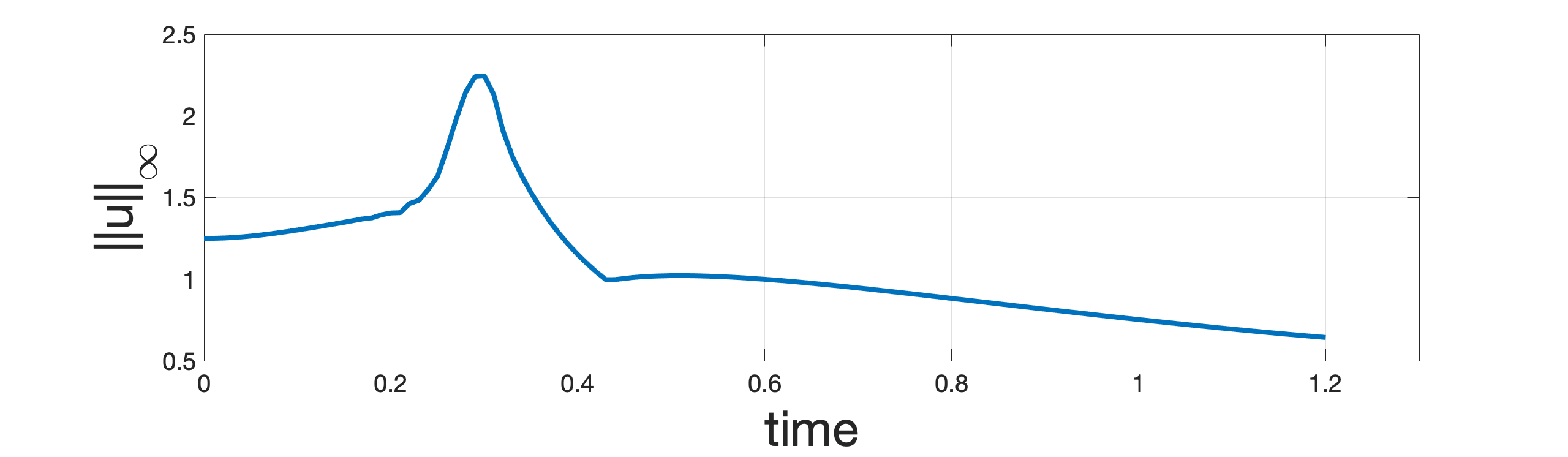}
\includegraphics[width=8.5cm,height=4.0cm]{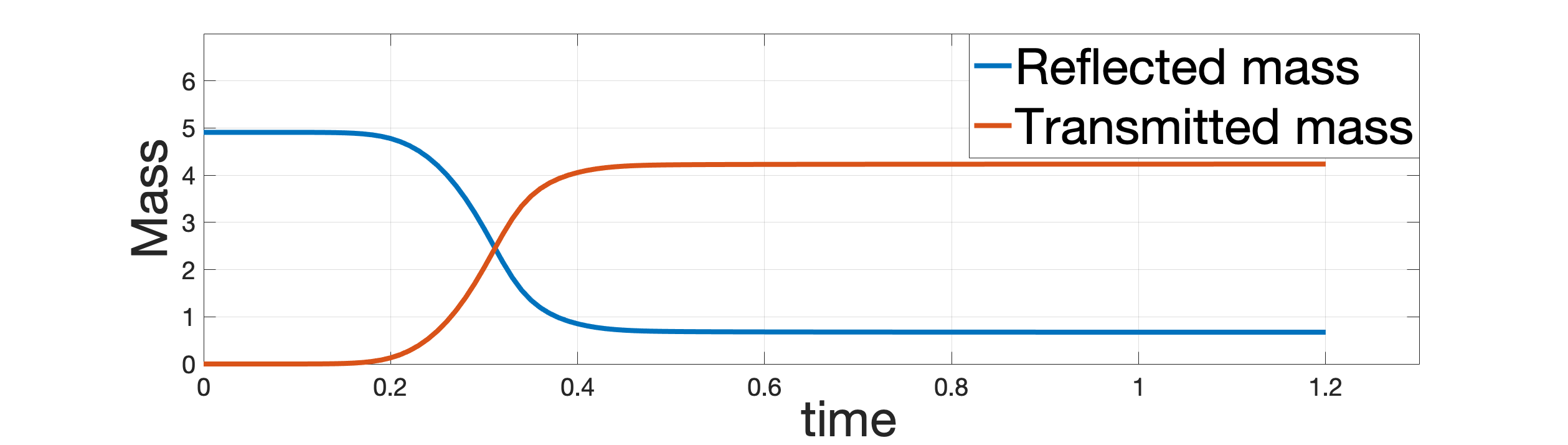}
\caption{Solution to the $2d$ quintic \NNls equation with the obstacle radius $r_{\star}=0.1$ and $u_0$ from \eqref{u0NLS} with \eqref{ref-A-1.25} moving along the line $y=0$. 
Snapshots of the scattering solution $u(t)$ at $t=0$ (top left), $t=0.35$ (middle top) and $t=1.2$ (top right). Time dependence of the $L^{\infty}$-norm (bottom left) and of the transmitted and the reflected mass (bottom right). }
\label{Fig-super-critic-r0-0.1}
\end{figure}

We observe that even for a small obstacle, $r_{\star}=0.1$, the solution scatters with a small backward reflection, see Figure \ref{Fig-super-critic-r0-0.1}. For $r_{\star}=1$ the solution has a similar dispersive behavior as in the previous examples, however, we note that the reflected backward part is more relevant and  important, 
which ensures the dispersive behavior of the solutions, see Figure  \ref{Fig-super-critic-r0-1}. On the other hand, for $r_{\star}=3$ the solution has a different behavior: the solution blows up in finite time at the boundary of the obstacle, as the interaction region becomes larger than the solution contour, and thus, the solution can not be transmitted around the obstacle. It concentrates at its blow-up core at the obstacle's boundary, see Figure  \ref{Fig-super-critic-r0-3}. \\

\begin{figure}[!h]      
\centering
\includegraphics[width=5.5cm,height=5.0cm]{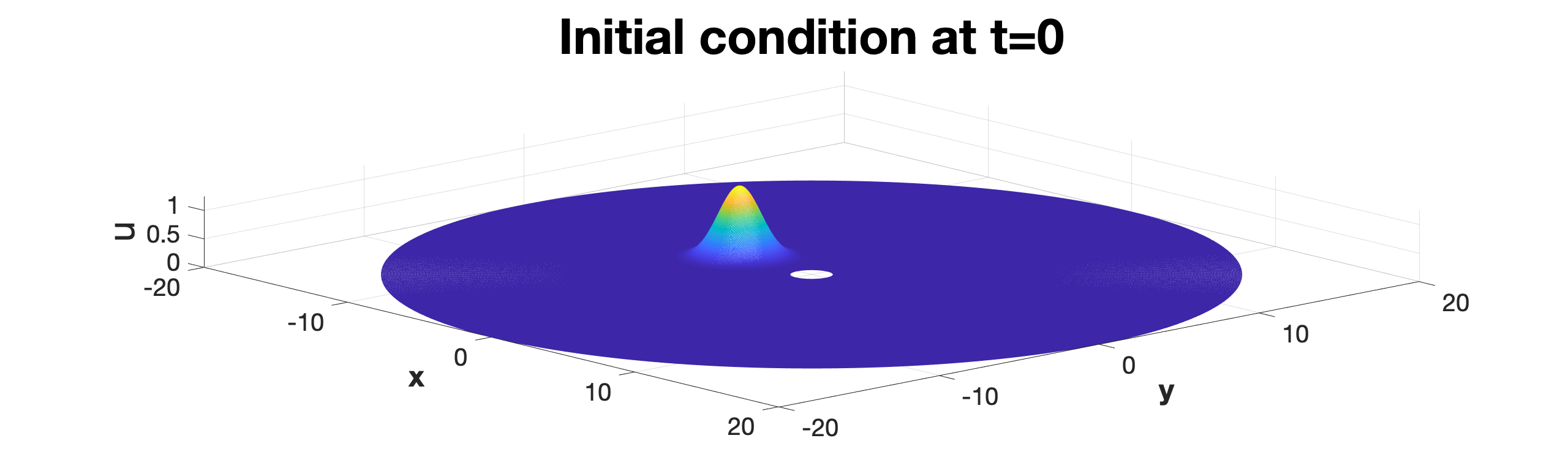}
\includegraphics[width=5.5cm,height=5.0cm]{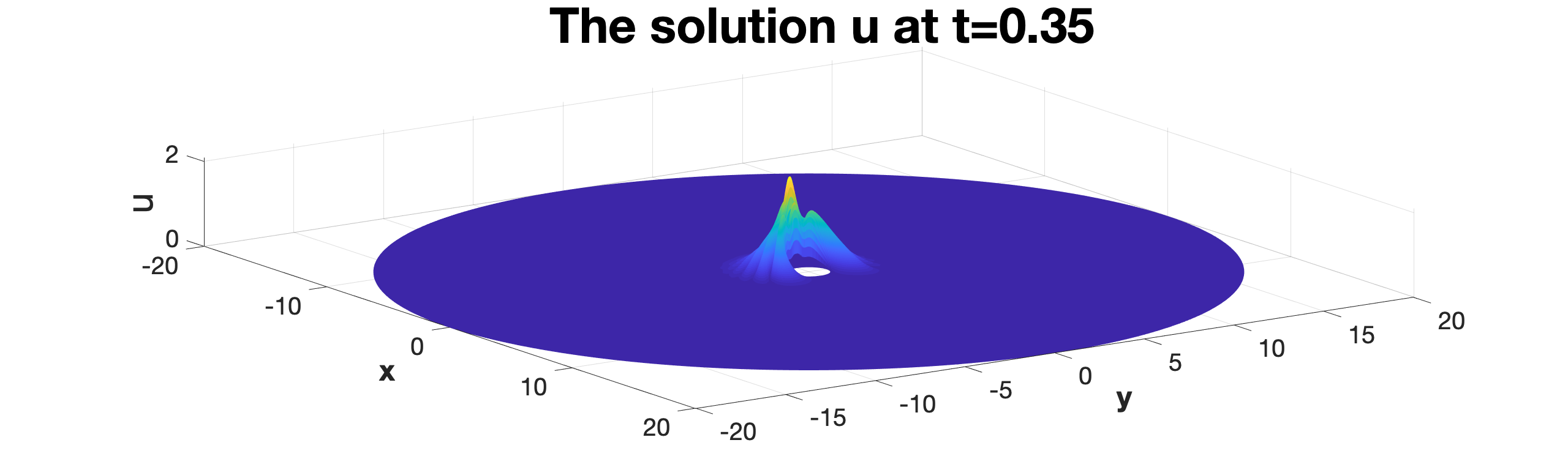}
\includegraphics[width=5.5cm,height=5.0cm]{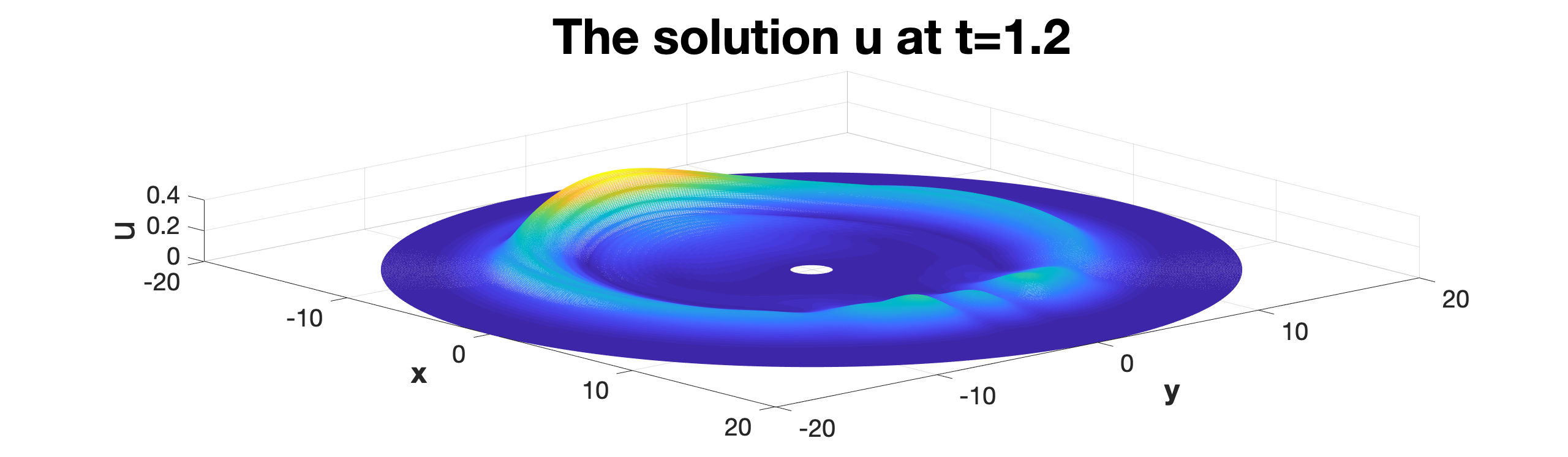}
\includegraphics[width=8.5cm,height=4cm]{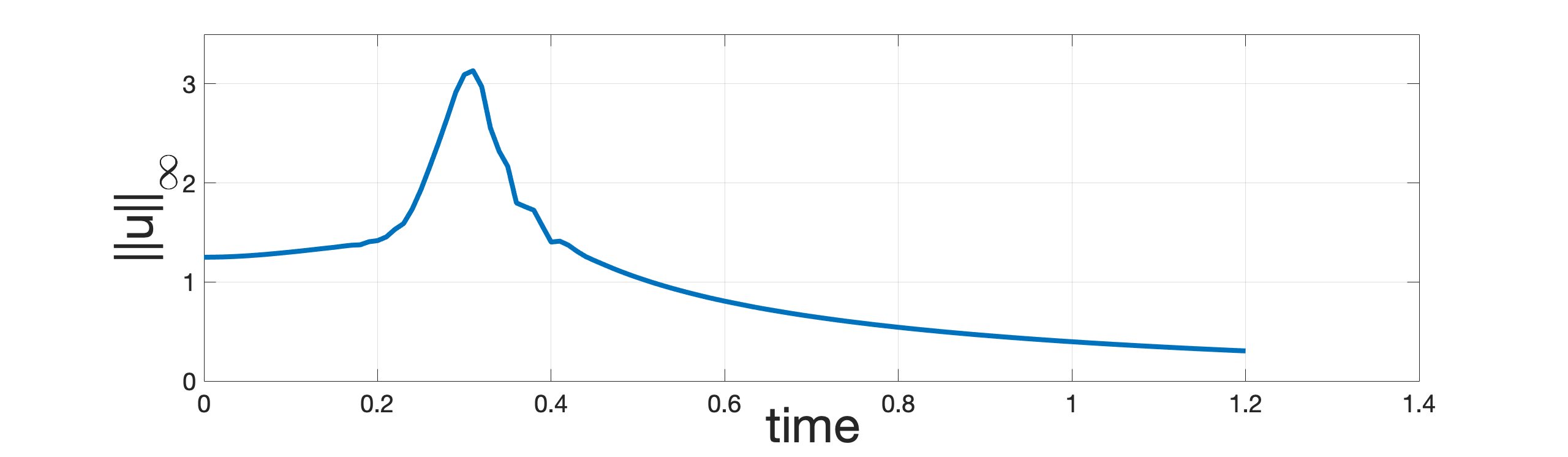}
\includegraphics[width=8.5cm,height=4cm]{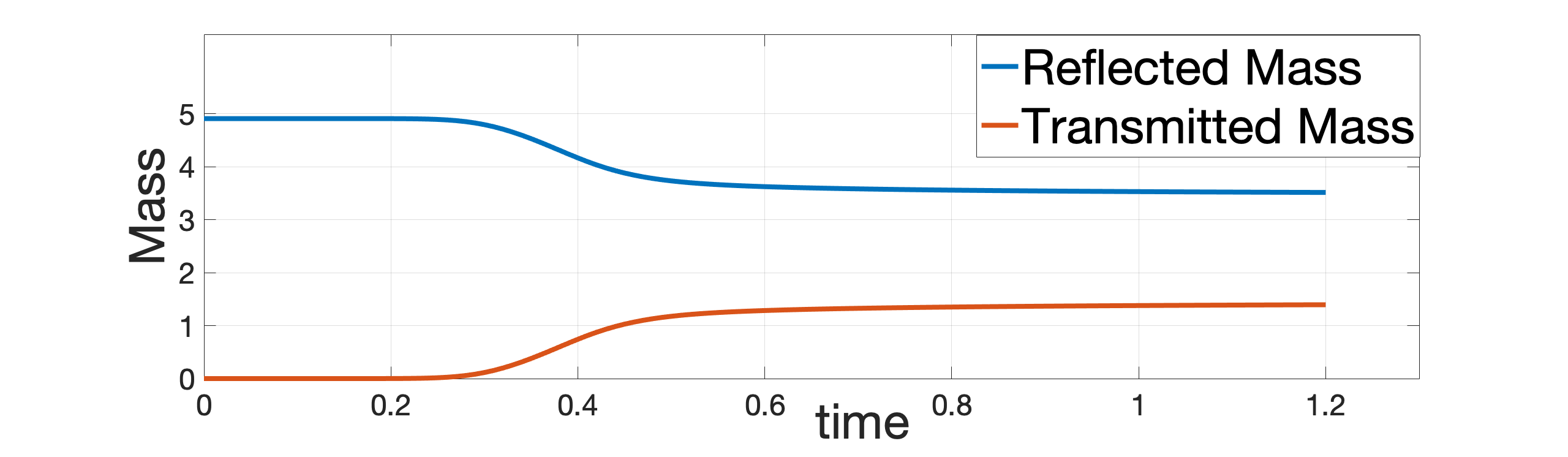}
\caption{Solution to the $2d$ quintic \NNls equation with the obstacle radius $r_{\star}=1$ and $u_0$ from \eqref{u0NLS} and \eqref{ref-A-1.25} moving along the line $y=0$. 
Snapshots of the scattering solution $u(t)$ at $t=0$ (top left), $t=0.35$ (middle top) and $t=1.2$ (top right). Time dependence of the $L^{\infty}$-norm (bottom left) and of the transmitted and the reflected mass (bottom right).
}
\label{Fig-super-critic-r0-1}
\includegraphics[width=5.5cm,height=4.3cm]{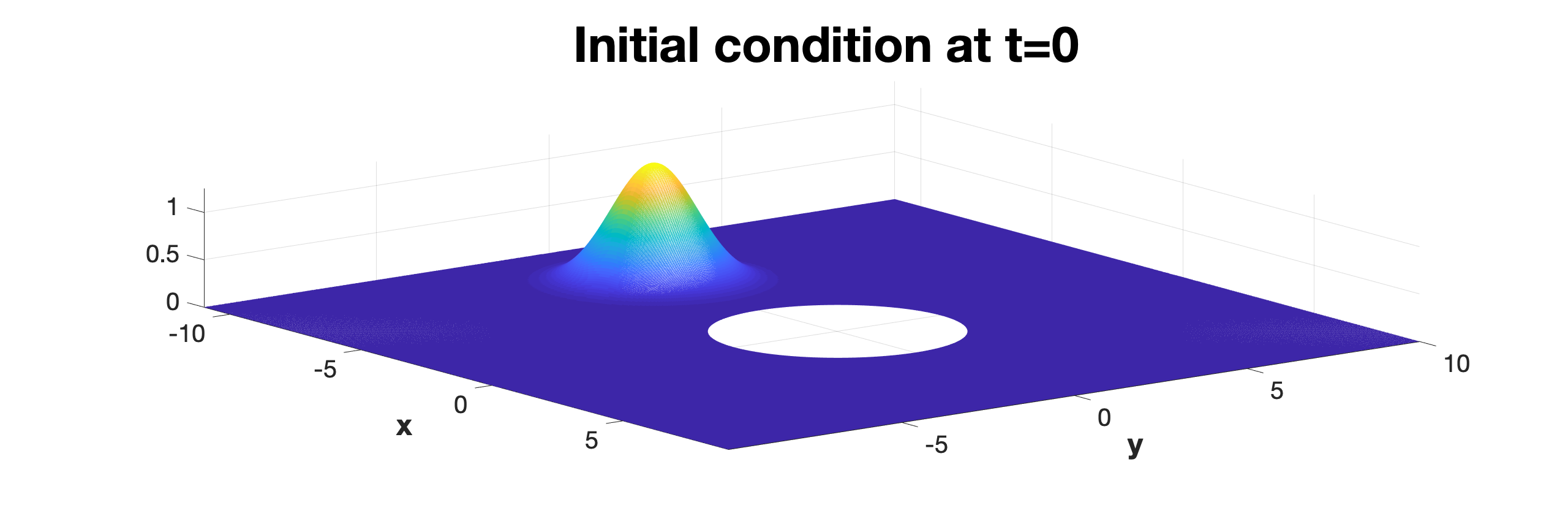}
\includegraphics[width=5.5cm,height=4.3cm]{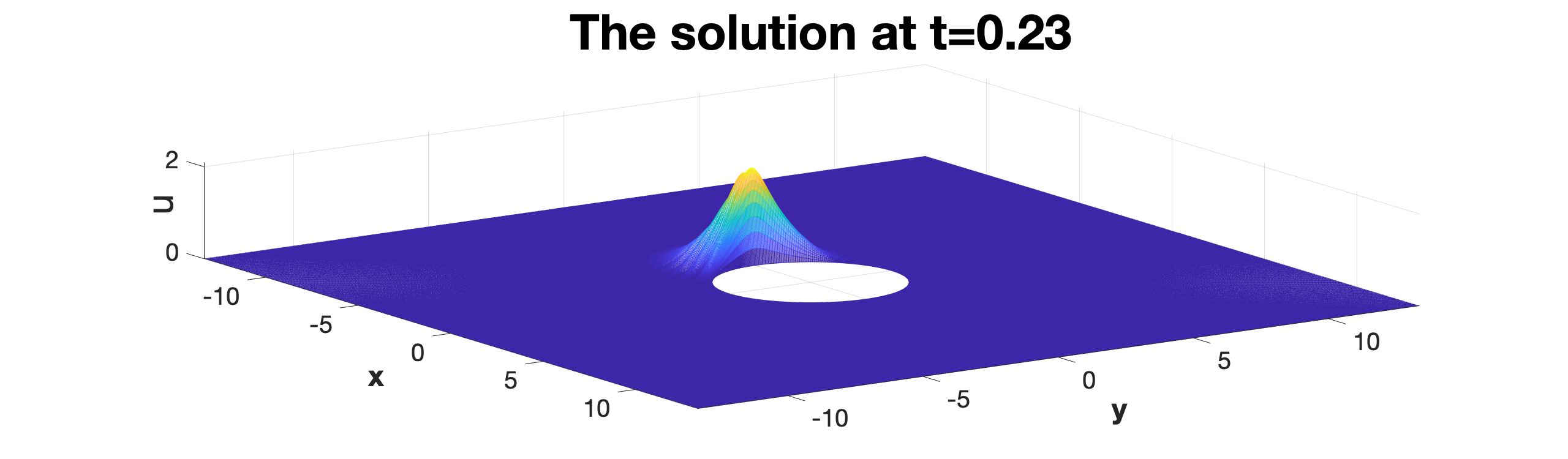}
\includegraphics[width=5.5cm,height=4.3cm]{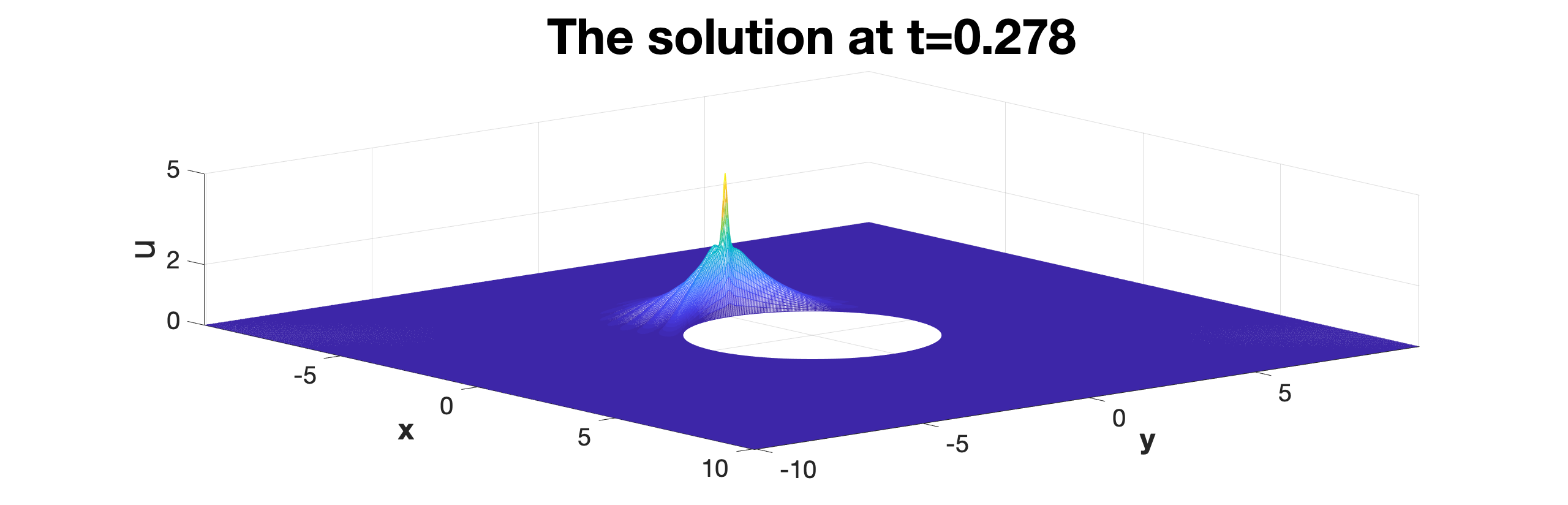}
\includegraphics[width=8.5cm,height=3.7cm]{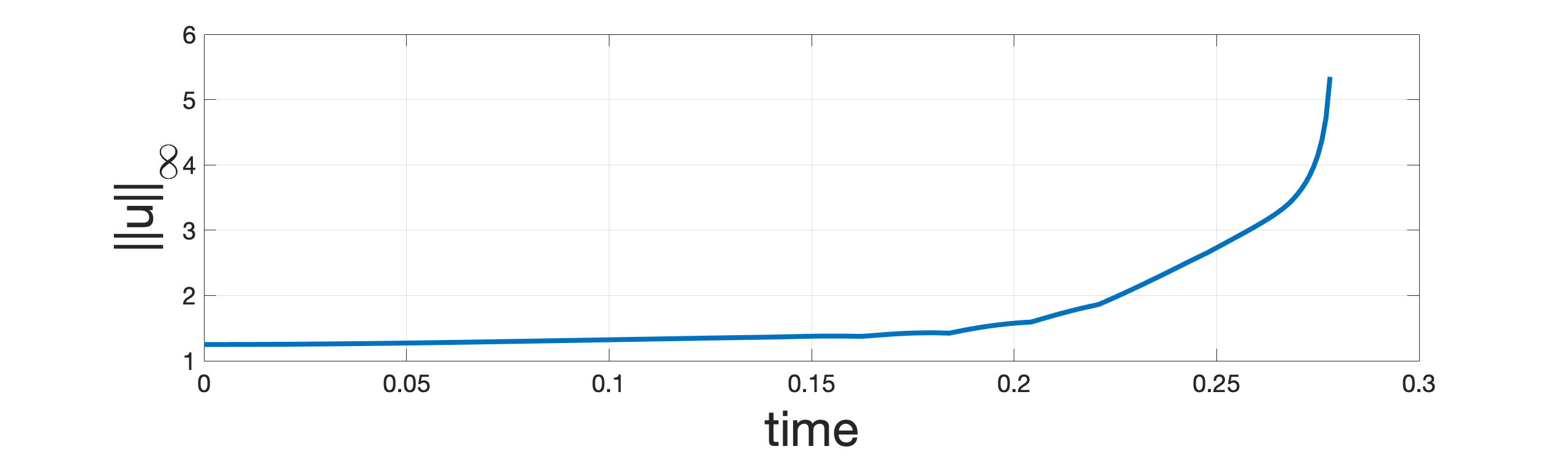}
\includegraphics[width=8.5cm,height=3.7cm]{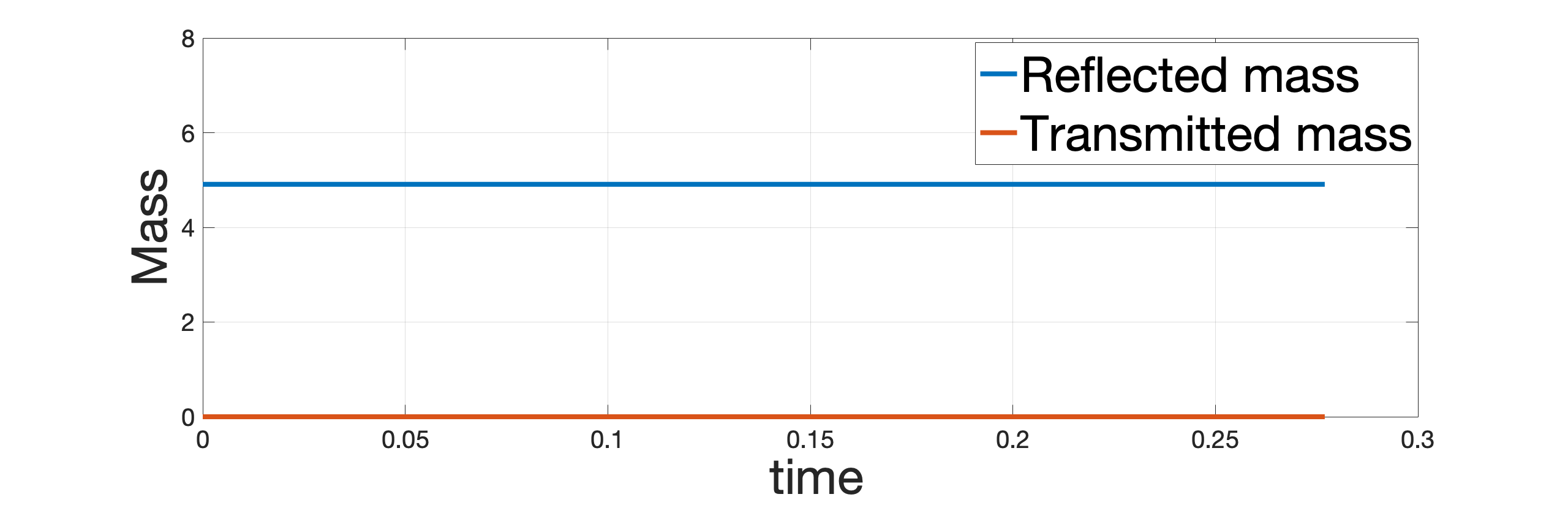}
\caption{Solution to the $2d$ quintic \NNls equation with the obstacle radius $r_{\star}=3$ and $u_0$ from \eqref{u0NLS} with \eqref{ref-A-1.25} moving along the line $y=0$. 
Snapshots of the blow-up solution $u(t)$ at $t=0$ (top left), $t=0.23$ (middle top) and $t=0.278$ (top right). Time dependence of the $L^{\infty}$-norm (bottom left) and of the transmitted and the reflected mass (bottom right).
}
\label{Fig-super-critic-r0-3}
\end{figure}
{\footnotesize
\begin{table}[h!]
\centering 
\begin{tabular}{|m{1.2cm}|c|c|c|c|m{1cm}|}   
  \hline
   $r_{\star}$                    & Discrete total mass        &  Behavior of the solution       & Discrete reflected mass                                & Discrete transmitted mass        \\                                    
  \hline  
 $r_{\star}=0.1$               & $4.9087$                & Scattering                               &  $0.67277 $ at $t \approx 1.2$                     &  $4.236 $ at $t \approx 1.2$           \\
   \hline 
 $r_{\star}=0.2$               &   $4.9087$             & Scattering                                  &  $1.1094   $ at $t \approx 1.2$                     & $ 3.7994 $ at $t \approx 1.2$              \\
   \hline 
 $r_{\star}=0.3$              &   $4.9087$               & Scattering                                  &   $1.5618  $ at $t \approx 1.2$                     &  $3.3469  $ at $t \approx 1.2$               \\
   \hline 
$r_{\star}=0.4$               &  $4.9087$               & Scattering                                  &    $1.9578 $ at $t \approx 1.2$                      & $2.9509$ at $t \approx 1.2$             \\
   \hline   
$r_{\star}=0.5$               & $4.9087$                 & Scattering                                   &   $2.3141$ at $t \approx 1.2 $                      & $2.5946 $ at $t \approx 1.2 $        \\
        \hline 
$r_{\star}=0.6$               &  $4.9087$               & Scattering                                   &   $ 2.6199$ at $t \approx 1.2 $                        &    $2.2888$ at $t \approx 1.2 $        \\
        \hline 
$r_{\star}=0.7$               &  $4.9087$               & Scattering                                   &   $2.8845 $ at $t \approx 1.2 $                       &  $2.0243$ at $t \approx 1.2 $       \\
        \hline 
$r_{\star}=0.8$               &   $4.9087$               & Scattering                                   &   $3.1181$ at $t \approx 1.2 $                      &  $1.7906 $ at $t \approx 1.2 $        \\
        \hline
$r_{\star}=0.9$                & $4.9087$                & Scattering                                   &   $3.3263$ at $t \approx 1.2 $                      &  $1.5824$ at $t \approx 1.2 $        \\
        \hline
$r_{\star}=1$                   &   $4.9087$               &  Scattering                                  &   $3.5143$ at $t \approx 1.2$                     &  $1.3945$ at $t \approx 1.2 $      \\
\hline
$r_{\star}=1.5$                   &   $4.9087$            &  Blow up at $t\approx 0.377  $         &   $4.7085$ at $t \approx 0.377$                     &  $0.20019$ at $t \approx 0.377$      \\
\hline
  $r_{\star}=2$                 &    $4.9087$             & Blow up at $t \approx 0.285$      &   $4.908$      at $t \approx 0.285$                   & $7.5778e^{-4}$ at $t \approx 0.285 $        \\
  \hline
    $r_{\star}=3$              &    $4.9087$                & Blow up at $t \approx 0.278$                                 &   $4.9087$     at $t \approx 0.278$                   &   $ 6.2761e^{-8}  $ at $t \approx 0.278 $       \\
  \hline
\end{tabular}
  \caption{Influence of the obstacle size $r_{\star}$ onto the  behavior of the solution $u(t)$ to the $2d$ quintic \NNls equation  with $u_0$ from \eqref{u0NLS} and \eqref{ref-A-1.25} with the discrete total mass, reflected and transmitted discrete mass parts after interaction with the obstacle at time $t.$}
  \label{T:4}
\end{table}
}

\section{Blow-up: Wall-type initial data}
\label{sec-special-solution}
In this section we study blow-up solutions in the strong interaction case (moving directly towards the obstacle as shown in Figure \ref{directStrongInterac}) for the large obstacle size and, in some cases, reuniting back into one single bump, which then blows up. We investigate the behavior of solutions to the cubic \NNls equation with a special round Wall-type super-Gaussian initial data. We consider the 
initial condition, which is defined by the product of a phase in terms of the angle $\Theta:=(\Theta_j)_{1\leq j  \leq N}$ with a super-Gaussian in terms of $r:=(r_i)_{1\leq i\leq N}:$ 
\begin{equation}
 \label{special-u0}
u_0 (r,\Theta):= A_0  \left( e^{-(r+r_c)^4}    \times e^{-\frac 12(\Theta-\pi)^4}  \right) e^{i\, (\frac 12 (v_x  r \cos(\Theta)+v_y  r \sin(\Theta)))} ,
\end{equation}
where  
\begin{equation}
\label{super-gauss-parameter-1}
 A_0=2.5, \quad v_x=15, \quad v_y=0,  \quad \text{ and }  \; \,  r_c=-4-r_{\star}.   
\end{equation}
Note that, due to the construction of this solution, the $L^2$-norm, or the mass, of $u_0$ depends on the radius of the obstacle $r_{\star}$, i.e., the mass increases as $r_{\star}$ becomes larger. This does not affect the conservation of the mass throughout the simulation for fixed $r_{\star}$.  
\begin{figure}[ht]                          
\centering
\includegraphics[width=5.9cm,height=5.5cm]{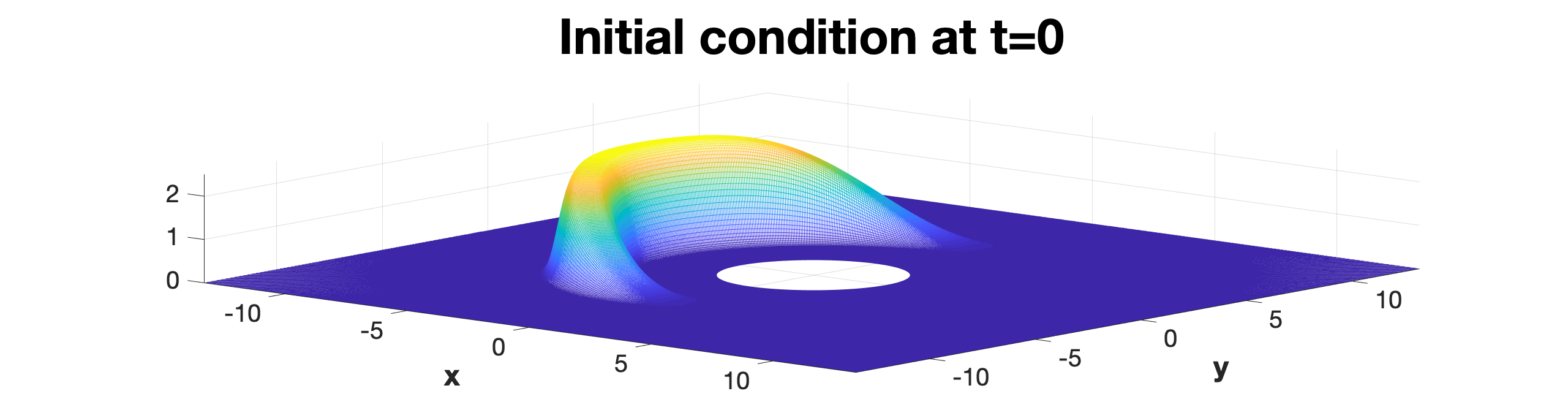}  
\includegraphics[width=5.9cm,height=5.5cm]{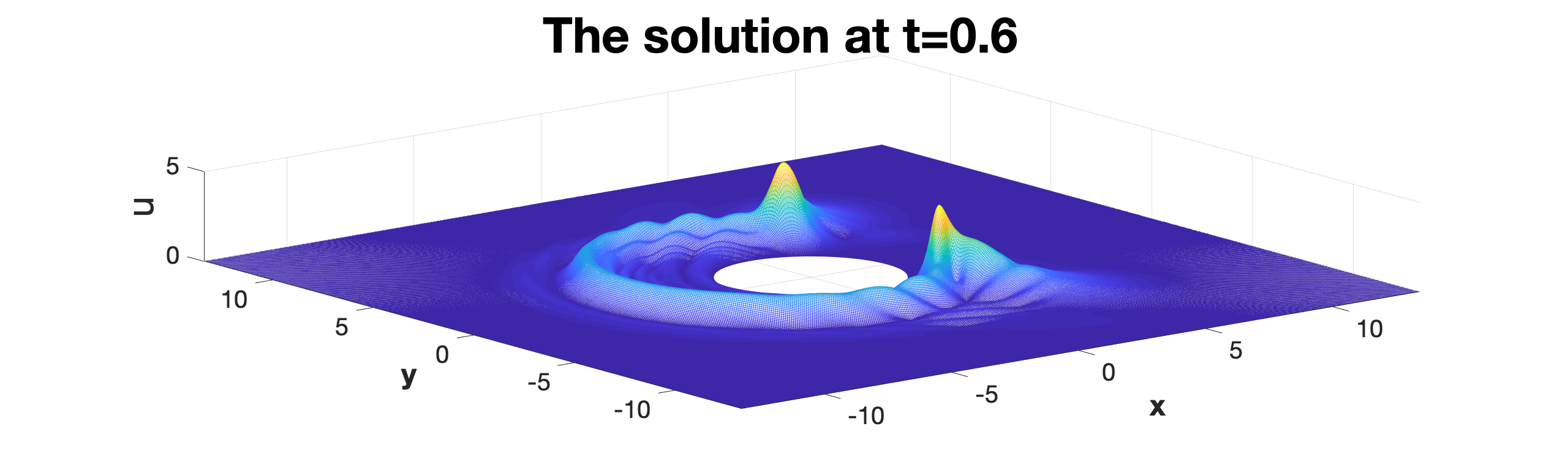} 
\includegraphics[width=6.0cm,height=5.5cm]{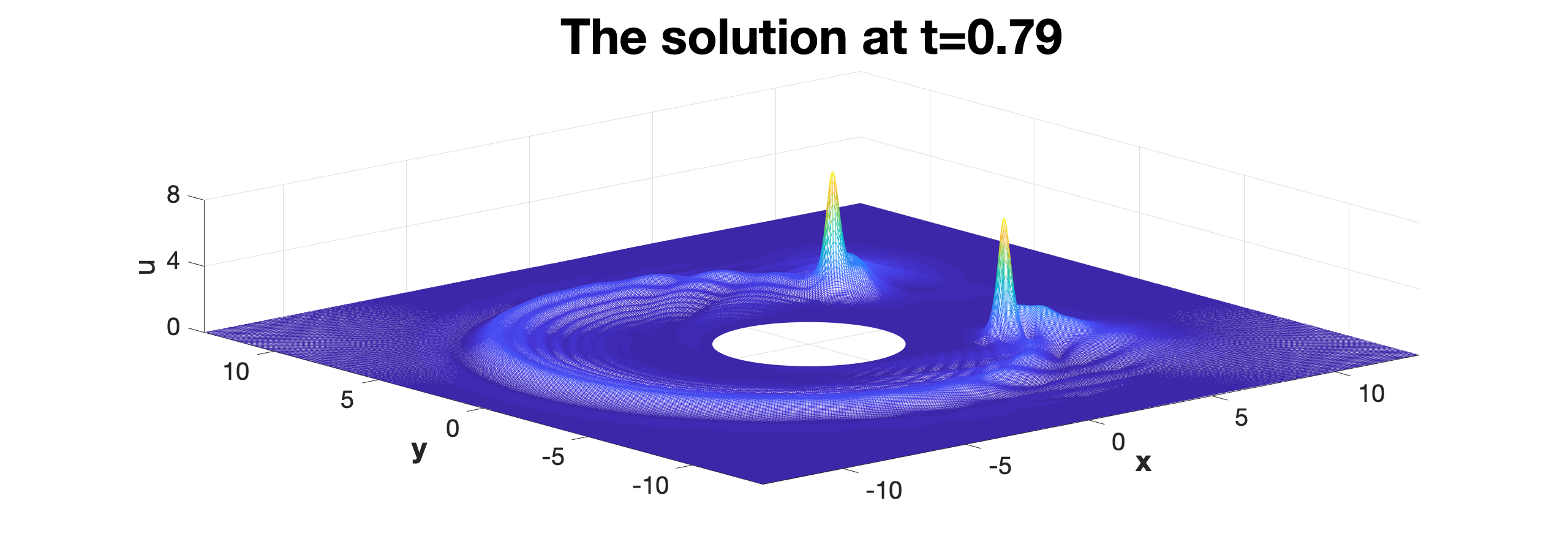}
\caption{Solution to the $2d$ cubic \NNls equation with the obstacle radius $r_{\star}=3$ and $u_0$ from \eqref{special-u0} with \eqref{super-gauss-parameter-1} moving along the line $y=0$. 
Snapshots of the blow-up solution $u(t)$ at $t=0$ (left), $t=0.6$ (middle) and $t=0.79$ (right). 
}
\label{Fig-supecial-critic-r0-3}
\end{figure}
In the following simulations, we consider $r_{\star}=3$ (Figure \ref{Fig-supecial-critic-r0-3}) and $r_{\star}=5$ (Figure \ref{Fig-supecial-critic-r0-5}). We observe that even with the large radius of the obstacle, the solution blows up in finite time. After the interaction, the solution splits into two bumps, with 
an essential backward reflection and a substantial  transmitted mass. Before the two bumps could merge together, they concentrate in their own blow-up core regions, that is, each bump blows up separately at a single point location, see the right plots in Figure \ref{Fig-supecial-critic-r0-3} for  $r_{\star}=3$ and Figure  \ref{Fig-supecial-critic-r0-5} for $r_{\star}=5$; also the growth of the $L^\infty$ norm on the right plots of Figures \ref{Fig-supecial-critic-r0-3b} and \ref{Fig-supecial-critic-r0-5b}.

\begin{figure}[!ht]                          
\centering
\includegraphics[width=8.5cm,height=4.3cm]{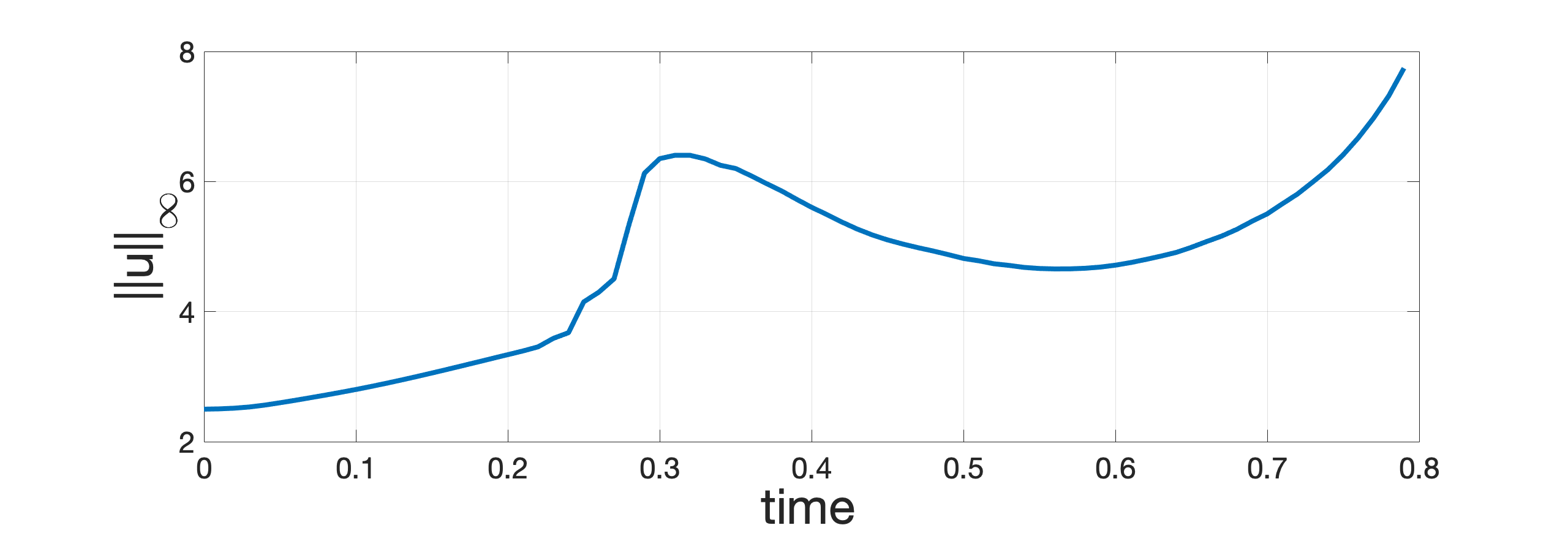}
\includegraphics[width=8.5cm,height=4.3cm]{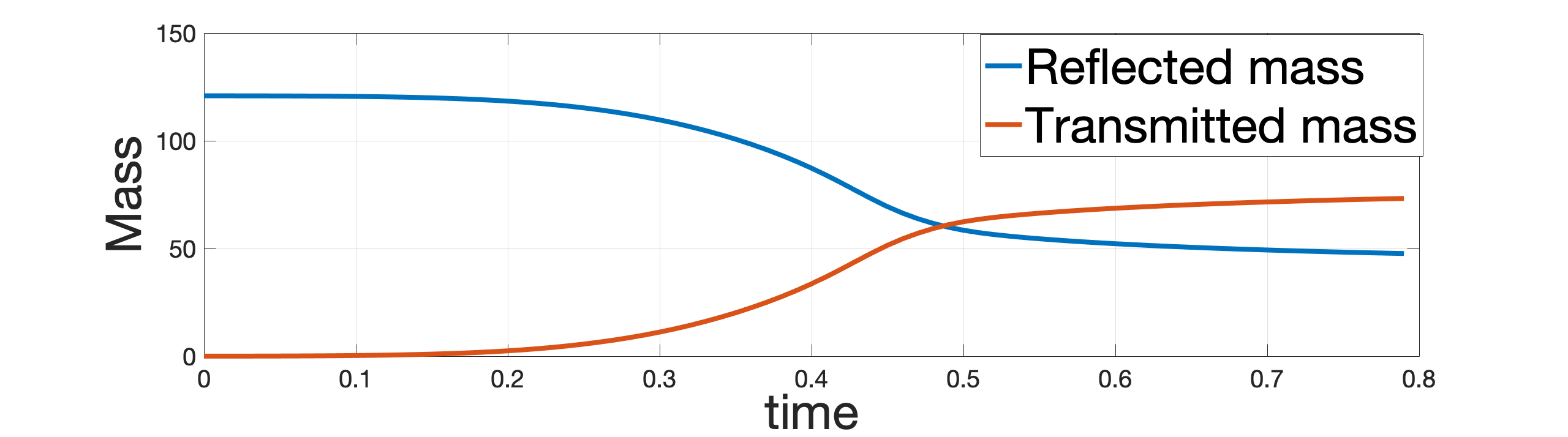}
\caption{
Time dependence of the $L^{\infty}$-norm (left) and of the transmitted and the reflected mass (right) for the solution in Figure \ref{Fig-supecial-critic-r0-3}.}
\label{Fig-supecial-critic-r0-3b}
\end{figure}

\begin{figure}[!ht]      
\centering
\includegraphics[width=5.9cm,height=5.5cm]{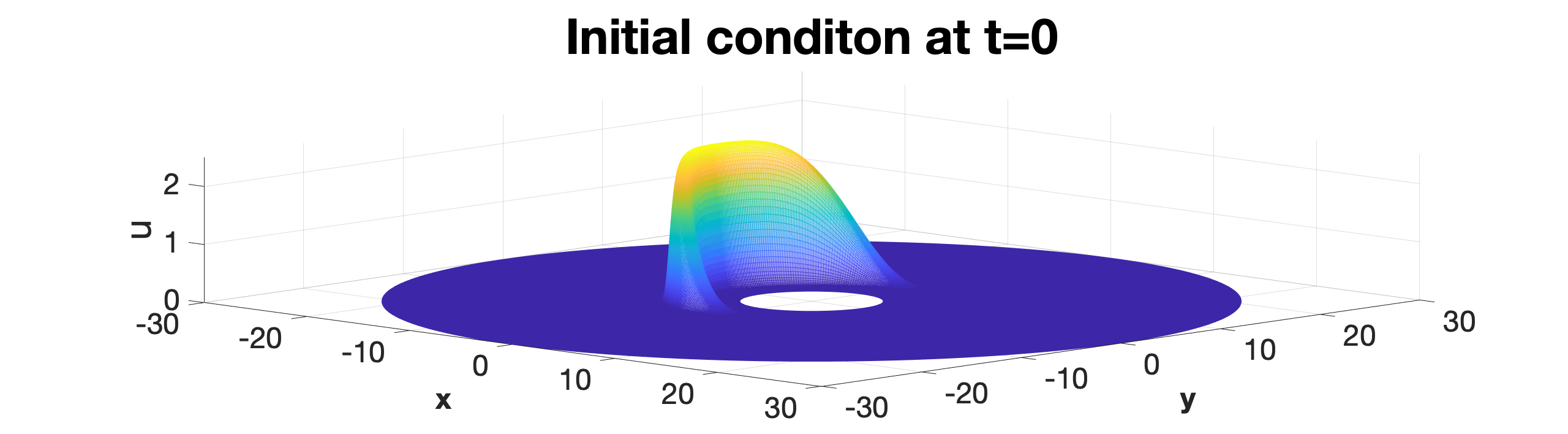}
\includegraphics[width=5.9cm,height=5.5cm]{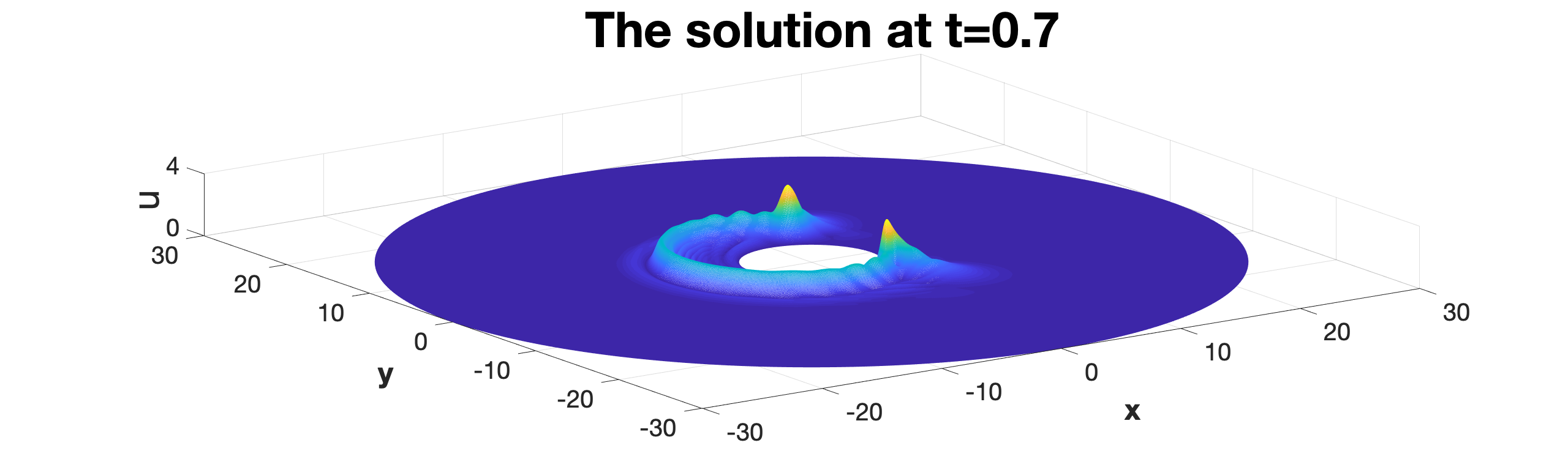}
\includegraphics[width=6.0cm,height=5.5cm]{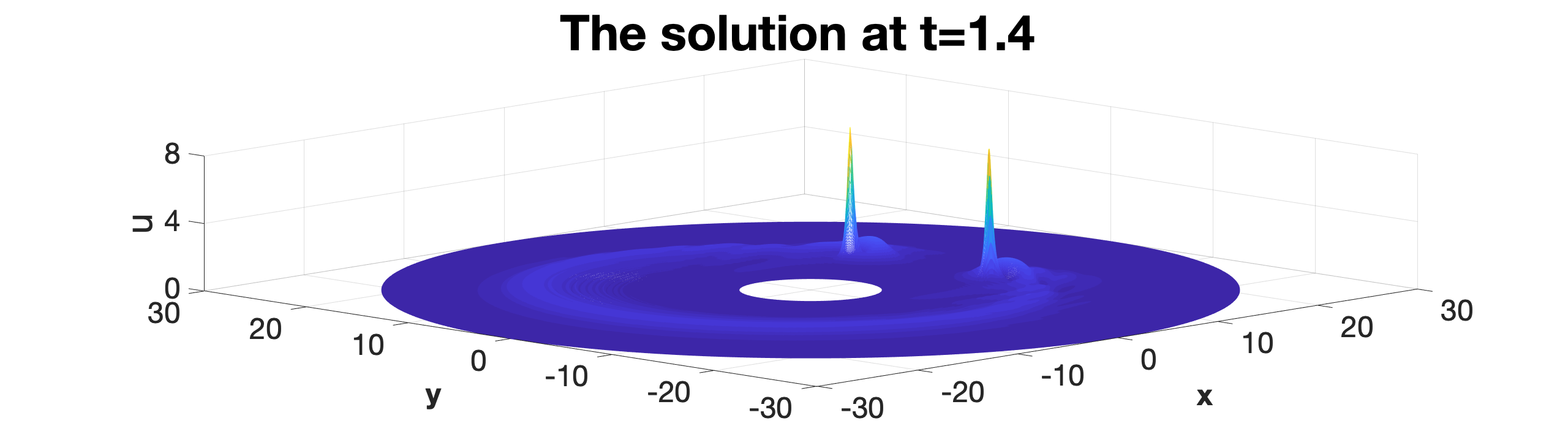}
\caption{Solution to the $2d$ cubic \NNls equation with the radius of the obstacle $r_{\star}=5$ and $u_0$ from \eqref{special-u0} with \eqref{super-gauss-parameter-1} moving along the line $y=0$. 
Snapshots of the blow-up solution $u(t)$ at $t=0$ (left), $t=0.7$ (top) and $t=1.4$ (right). 
}
\label{Fig-supecial-critic-r0-5}
\end{figure}

\begin{figure}[!ht]      
\centering
\includegraphics[width=8.5cm,height=4.3cm]{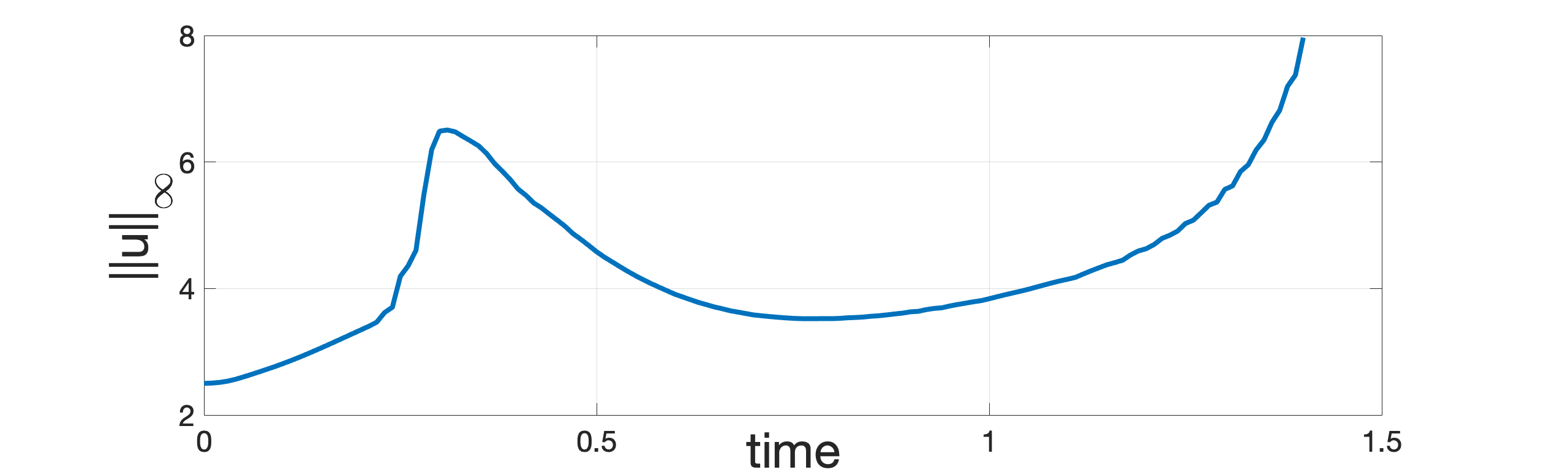}
\includegraphics[width=8.5cm,height=4.3cm]{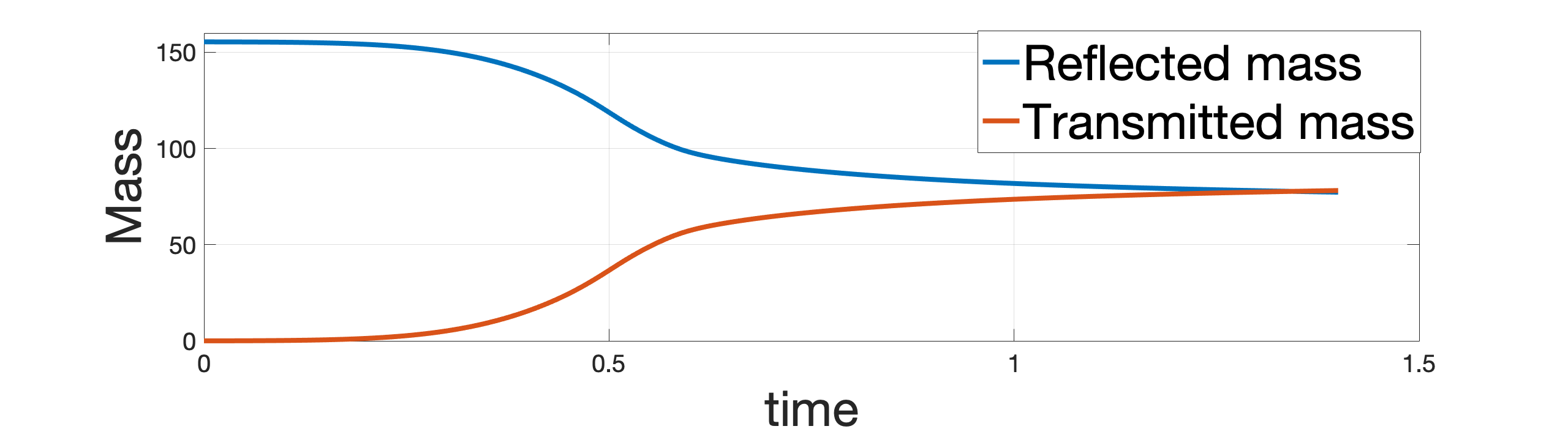}
\caption{
Time dependence of the $L^{\infty}$-norm (left) and of the transmitted and the reflected mass (right) for the solution in Figure \ref{Fig-supecial-critic-r0-5}.
}
\label{Fig-supecial-critic-r0-5b}
\end{figure}

We next consider the data \eqref{special-u0} with the following parameters (changing the amplitude):
 \begin{equation}
 \label{para-spec-A0-1.5}
A_0=1.5 , \quad v_x=15, \quad v_y=0,  \quad \text{ and  } r_c=-4-r_{\star}. 
 \end{equation} 

For the last two examples in this paper, we consider $r_{\star}=2$ and $r_{\star}=10$ with the data in \eqref{para-spec-A0-1.5}. \\

First, we observe that the solution blows up in finite time even if the radius of the obstacle is large (compared to the previous example with the amplitude $A_0=2.5$), see Figure \ref{Fig-supecial-critic-A0-1.5-r0-2}, where $r_{\star}=2$. 

\begin{figure}[ht]            
\centering          
\includegraphics[width=5.9cm,height=6.5cm]{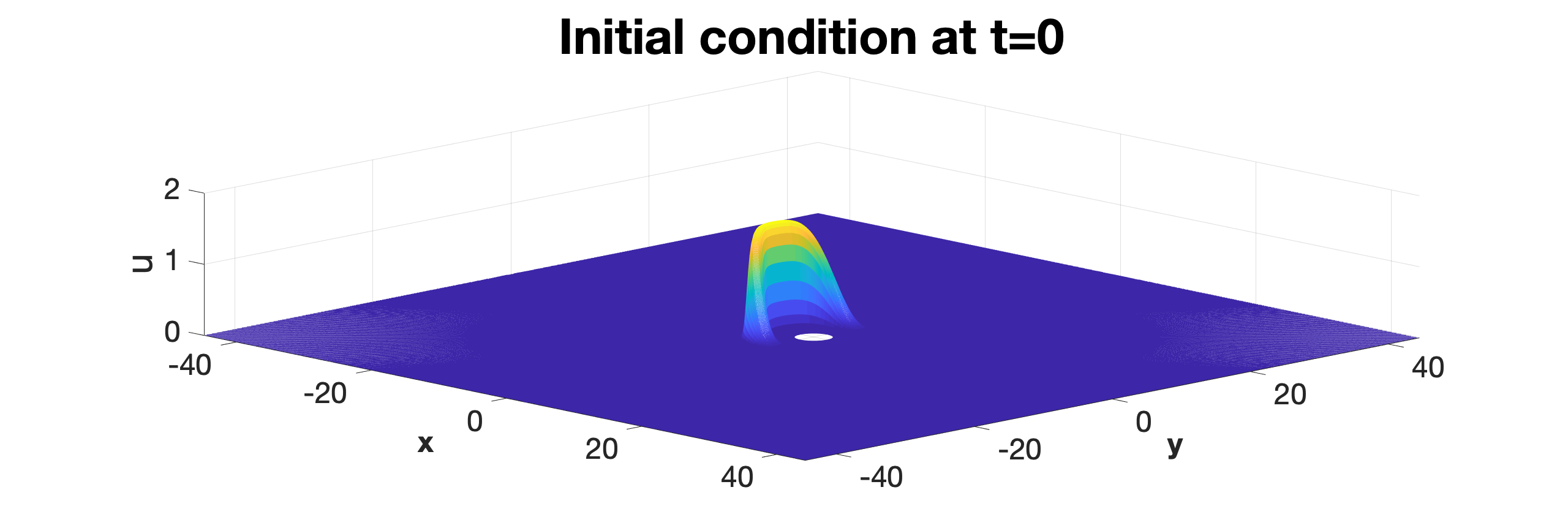}  
\includegraphics[width=5.9cm,height=6.5cm]{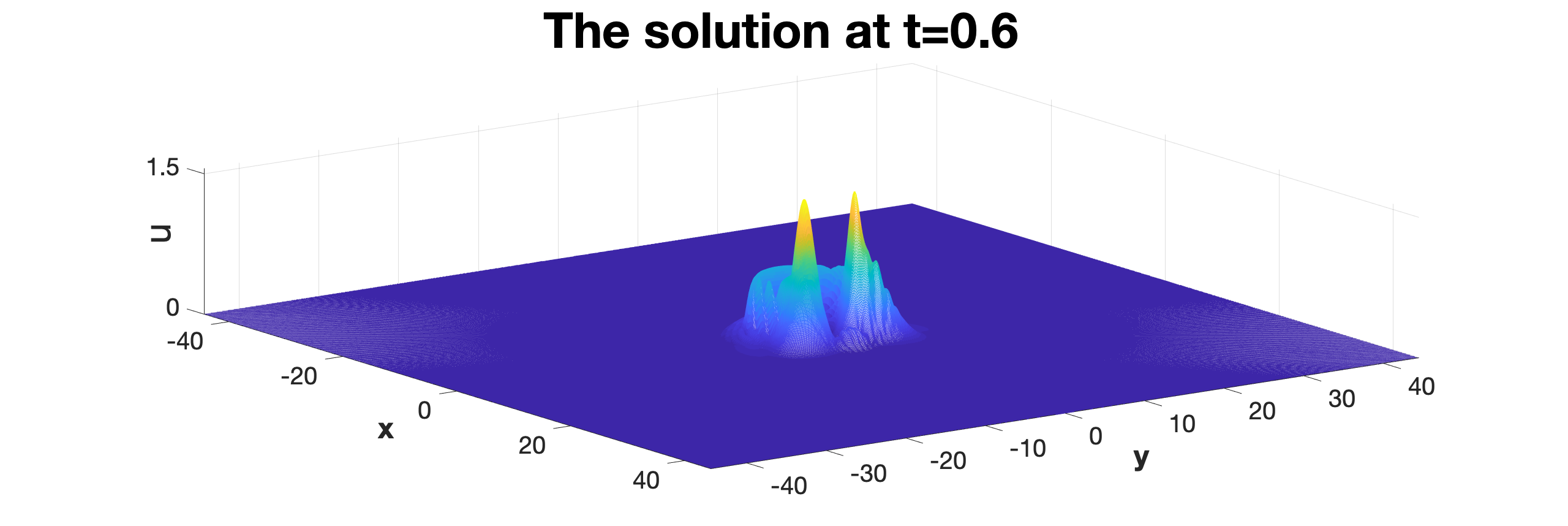}
\includegraphics[width=6.0cm,height=6.5cm]{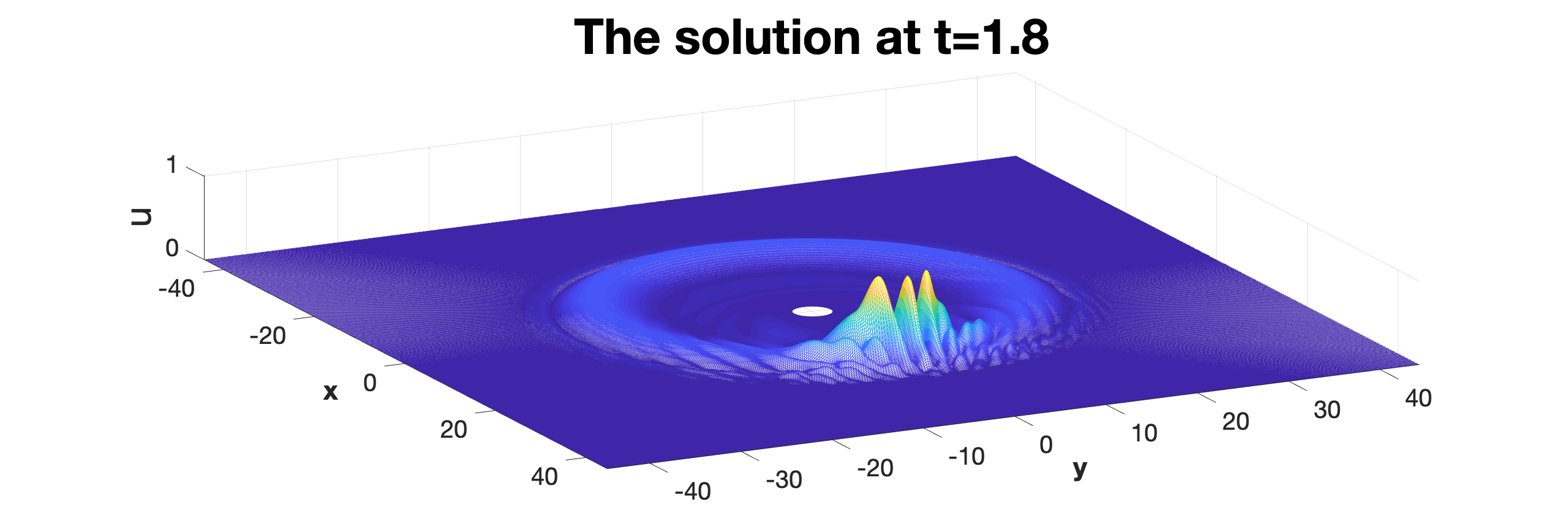}
\includegraphics[width=5.9cm,height=7.3cm]{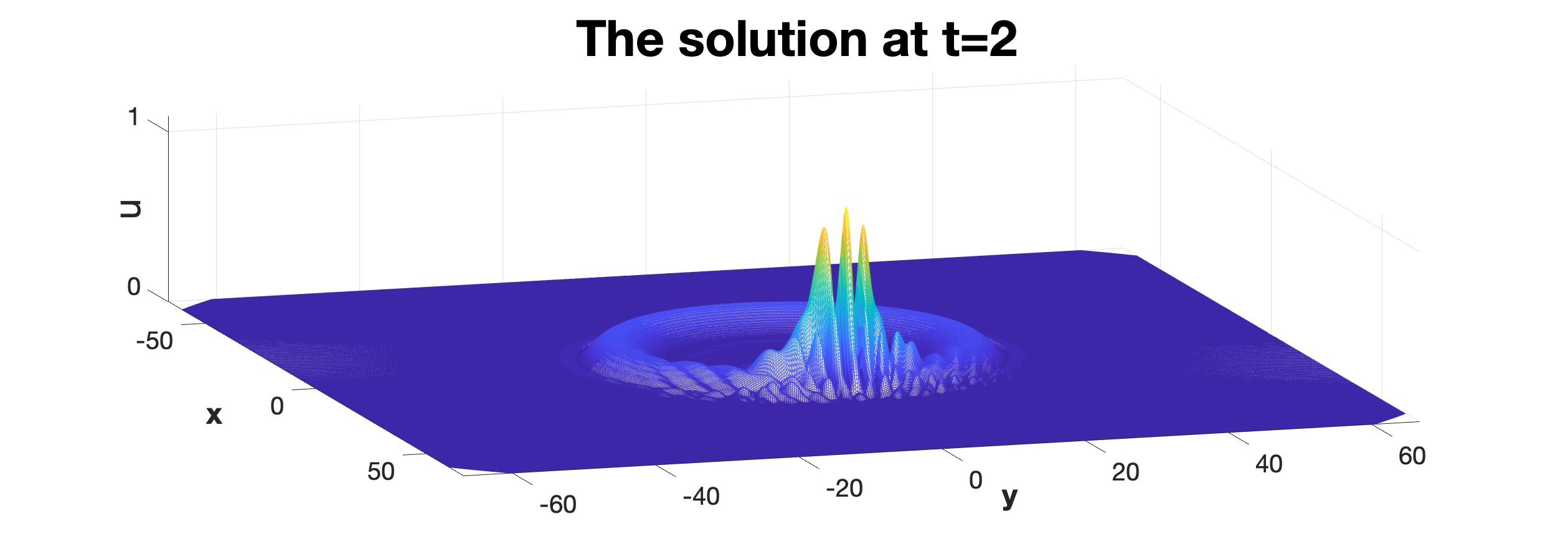}
\includegraphics[width=5.9cm,height=7.3cm]{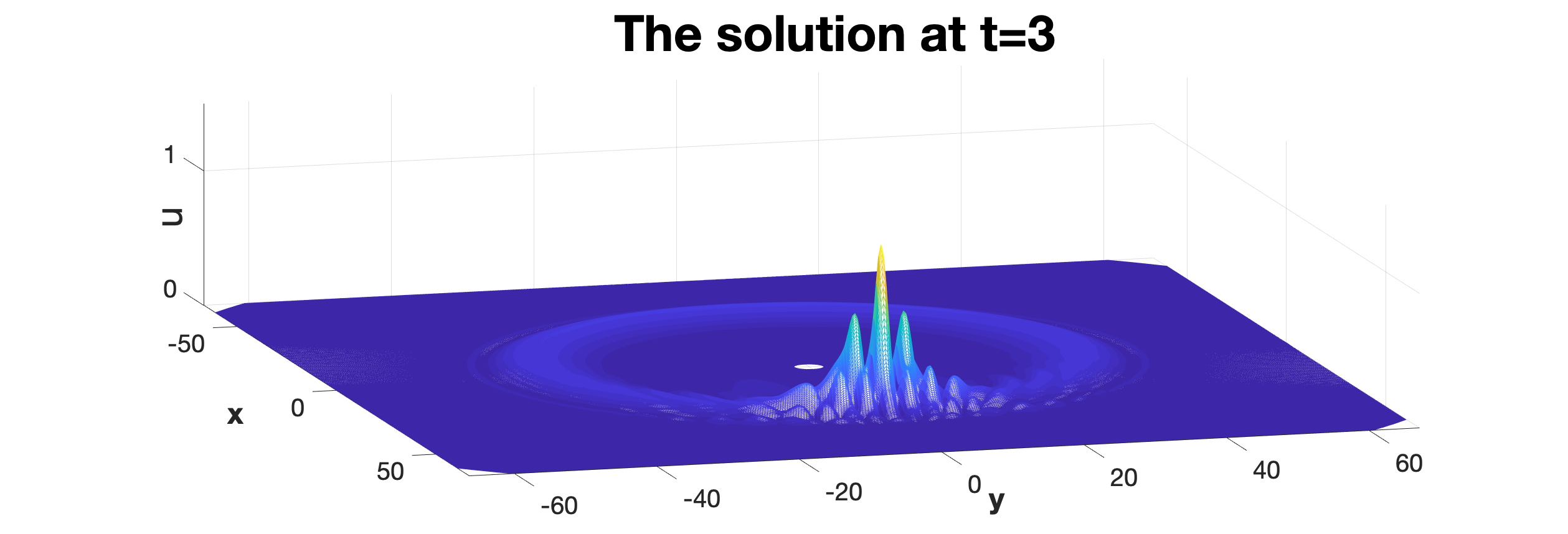}
\includegraphics[width=6.0cm,height=7.3cm]{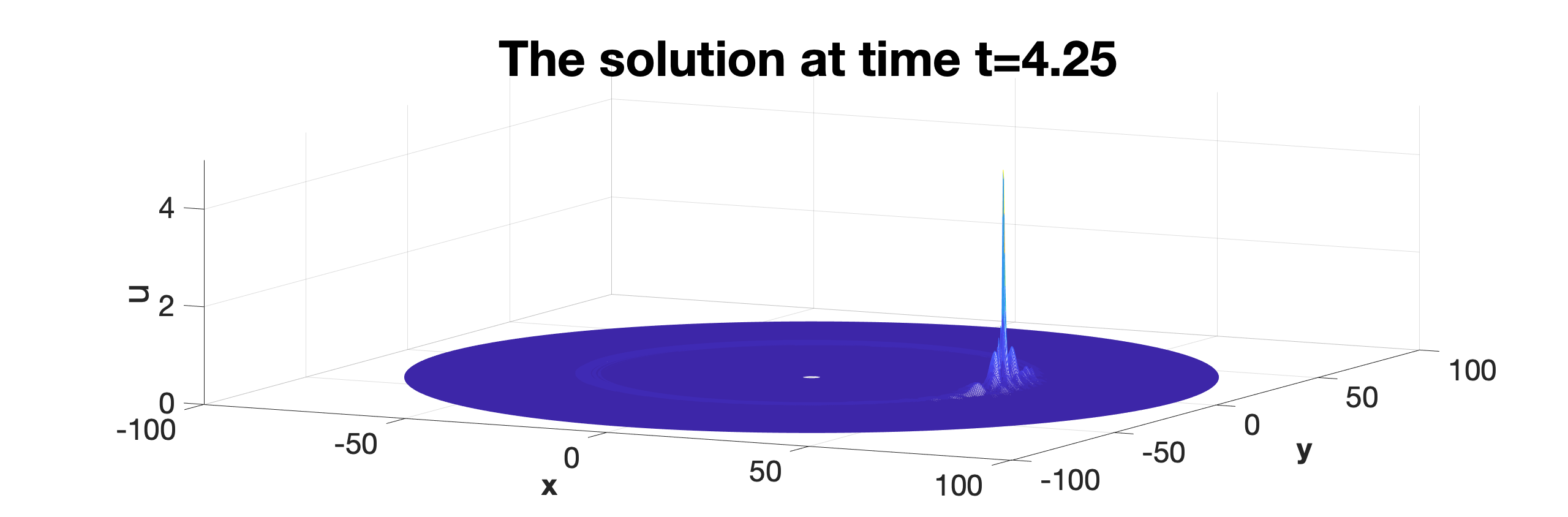}
\caption{Solution to the $2d$ cubic \NNls equation with the obstacle radius 
$r_{\star}=2$ and $u_0$ from \eqref{special-u0} with \eqref{para-spec-A0-1.5}, moving along the line $y=0$. 
Snapshots of the blow-up solution $u(t)$ at $t=0, 0.6, 1.8$ (top row), $t=2, 3, 4.25$ (bottom row). }
\label{Fig-supecial-critic-A0-1.5-r0-2}
\end{figure}

After the interaction, the solution splits into two bumps, with 
the backward reflection having a substantial amount of the transmitted mass. Then the two bumps have sufficient time to merge together and pump the mass from both lumps (as the circle expands) into a single bump, which has enough mass to concentrates in its core to blow up in finite time, 
see Figure \ref{Fig-supecial-critic-A0-1.5-r0-2} and also the $L^\infty$ norm together with the change in time in the transmitted and reflected mass in Figure \ref{Fig-supecial-critic-A0-1.5-r0-2b}. %
\\

\begin{figure}[ht]            
\centering          
\includegraphics[width=8.5cm,height=4.5cm]{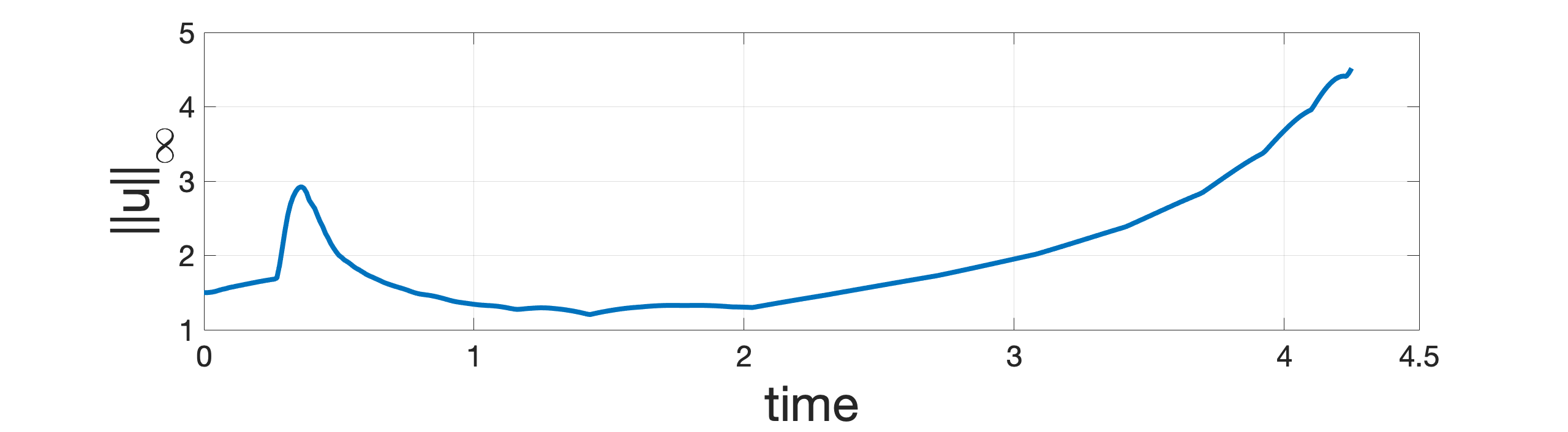}
\includegraphics[width=8.5cm,height=4.5cm]{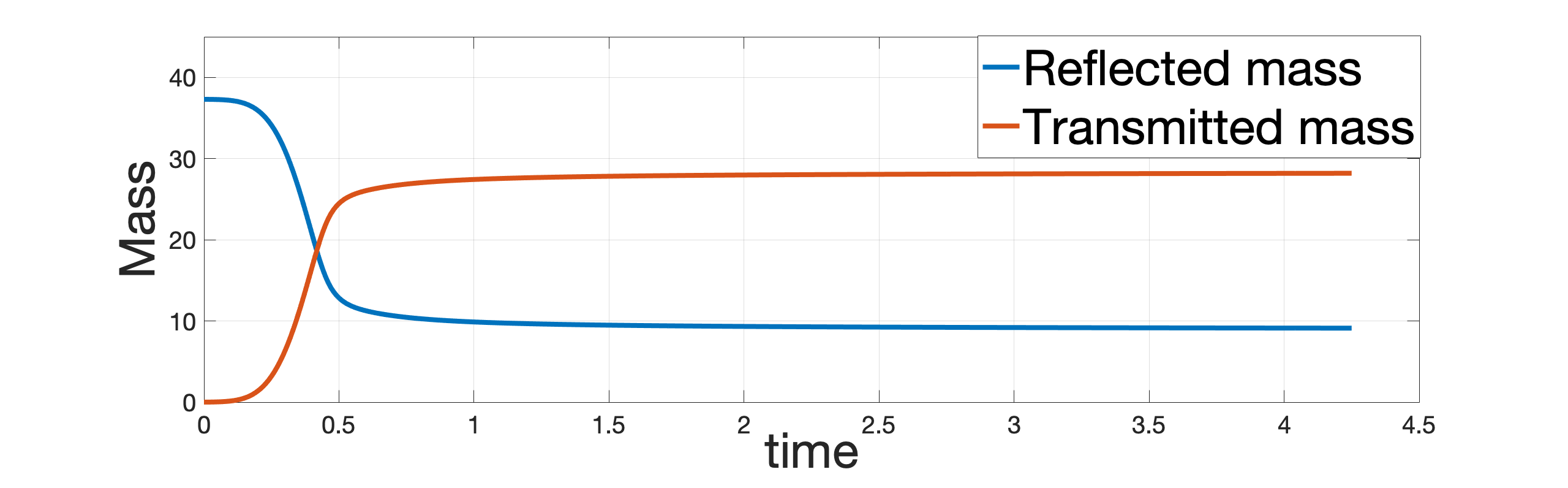}
\caption{Time dependence of the $L^{\infty}$-norm (left) and of the transmitted and the reflected mass (right) for the solution in Figure \ref{Fig-supecial-critic-A0-1.5-r0-2}.
}
\label{Fig-supecial-critic-A0-1.5-r0-2b}
\end{figure}

\newpage

However, for a larger radius, for example $r_{\star} \geq 3$, the solution has to hug around an obstacle with the bigger size, and thus, the mass gets dispersed more around, hence, less of the mass is transmitted, which concludes in the overall scattering behavior. 
See Figure \ref{Fig-supecial-critic-A0-1.5-r0-10} for an example of the obstacle size  $r_{\star}=10$.  \\

\begin{figure}[!ht]              
\centering
\includegraphics[width=5.9cm,height=6.2cm]{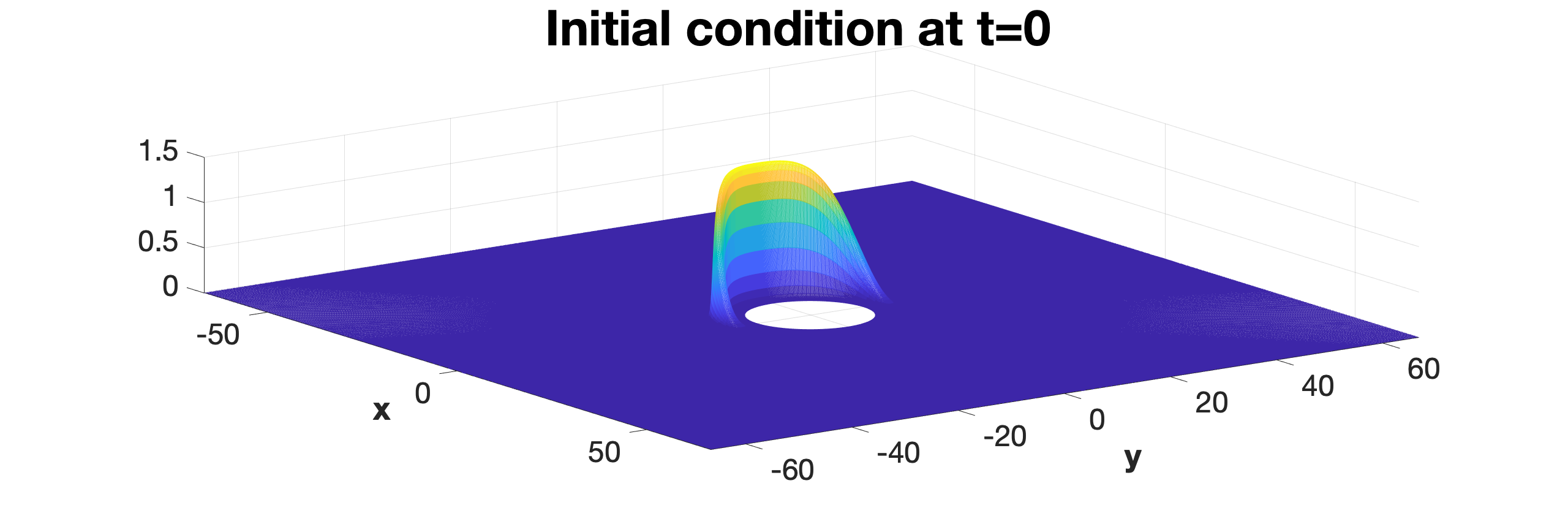}
\includegraphics[width=5.9cm,height=6.2cm]{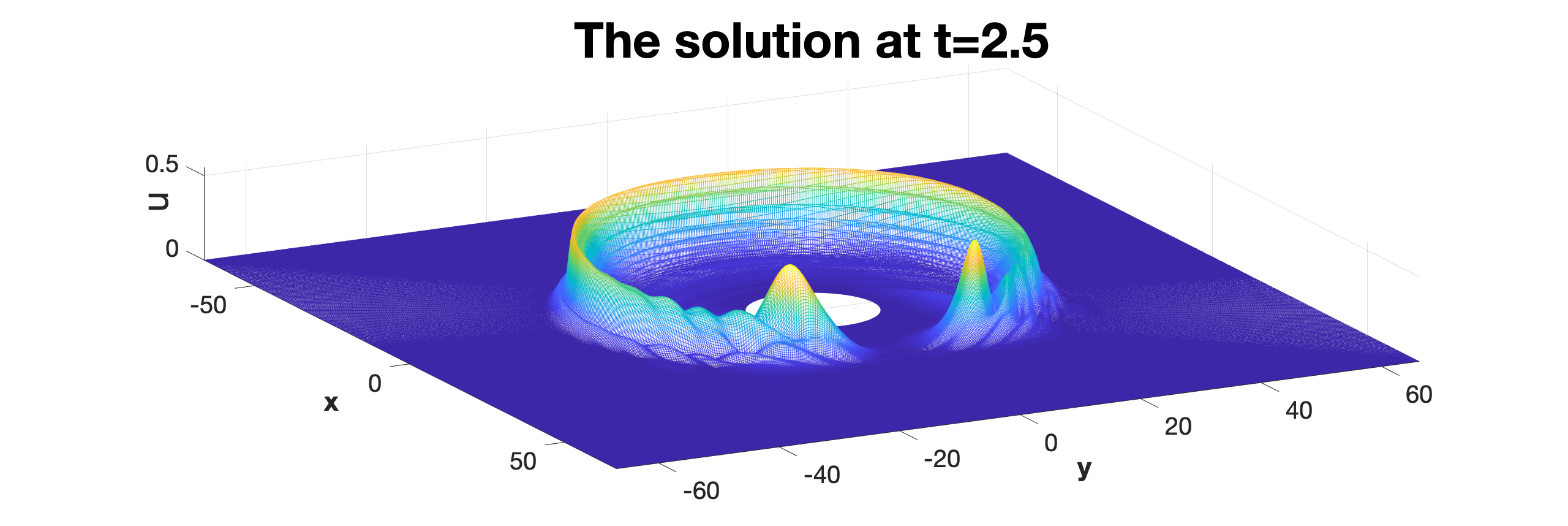}
\includegraphics[width=6.0cm,height=6.2cm]{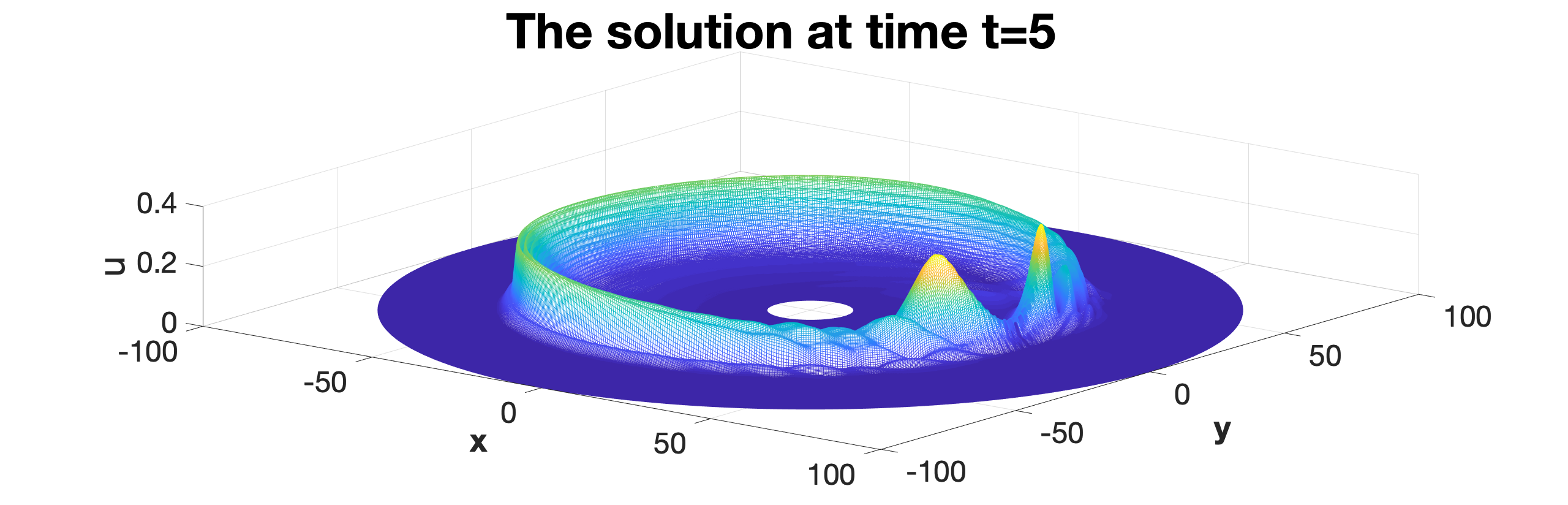}
\includegraphics[width=8.5cm,height=4.5cm]{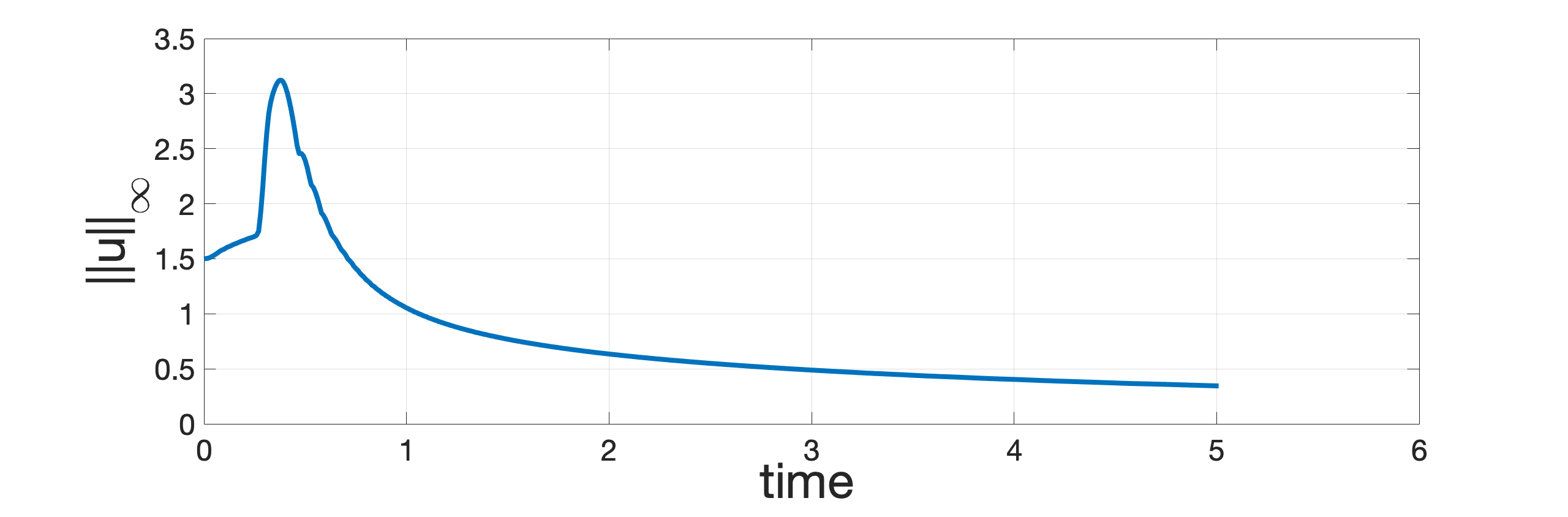}
\includegraphics[width=8.5cm,height=4.5cm]{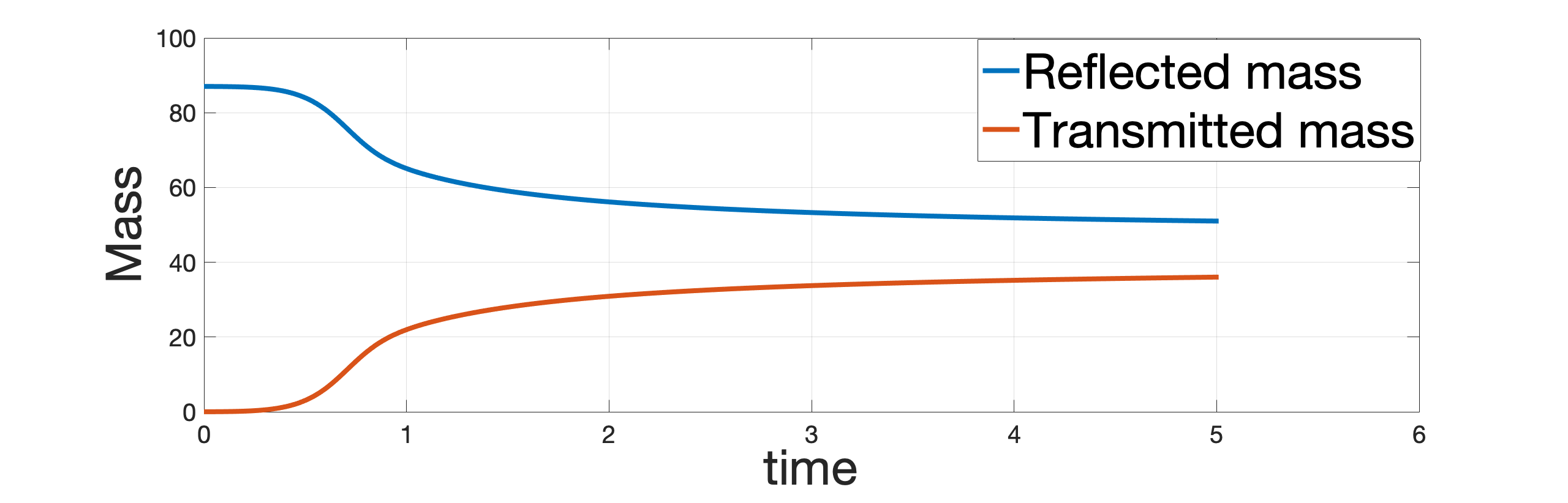}
\caption{Solution to the $2d$ cubic \NNls equation with the radius of the obstacle $r_{\star}=10$ and $u_0$ from \eqref{special-u0} with \eqref{para-spec-A0-1.5}. 
Snapshots of the scattering solution $u(t)$ at $t=0$ (top left), $t=2.5$ (middle top) and $t=5$ (top right). Time dependence of the $L^{\infty}$-norm (bottom left) and of the transmitted and the reflected mass (bottom right).
}
\label{Fig-supecial-critic-A0-1.5-r0-10}
\end{figure}

For the initial data \eqref{para-spec-A0-1.5} to track the dependence as the radius of the obstacle increases we perform numerical simulations for a variety of radii and provide the summary of the results in Table \ref{T:5}.  

{\footnotesize
\begin{table}[h!]
\centering 
\begin{tabular}{|m{1.2cm}|c|c|c|c|m{1cm}|}   
  \hline
   $r_{\star}$                    & Discrete total mass        &  Behavior of the solution                                & Discrete reflected mass                                & Discrete transmitted mass        \\                                    
  \hline  
 $r_{\star}=0.1$               & $25.4924$                & Blow up at $t\approx 1.54 $                              &  $  0.70785 $ at $t \approx 1.54$                     &  $24.7846 $ at $t \approx 1.54$           \\
   \hline 
 $r_{\star}=0.5$               &   $ 27.9795$                &  Blow up at $t\approx 1.75 $                                  &  $ 2.4016   $ at $t \approx 1.75$                     & $ 25.5779 $ at $t \approx 1.75$              \\
   \hline 
 $r_{\star}=1$              &   $31.0883$                  &  Blow up at $t\approx 2.4 $                                 &   $4.6621  $ at $t \approx 2.4$                     &  $26.4262 $ at $t \approx 2.4$               \\
   \hline 
$r_{\star}=1.5$               &  $34.1971$                  & Blow up at $t\approx 3.174 $                                &    $6.8188  $ at $t \approx 3.174 $                      & $27.3784$ at $t \approx 3.174 $             \\
   \hline   
$r_{\star}=2$               & $ 37.3060$                    & Blow up at $t\approx  4.25   $                                &   $ 9.1169$ at $t \approx 4.25  $                      & $28.1891 $ at $t \approx  4.25  $        \\
        \hline 
$r_{\star}=3$                   &  $43.5236$               &  Scattering                                                             &   $14.2976$ at $t \approx 6$                     &  $29.226$ at $t \approx 6$      \\
\hline
  $r_{\star}=4$                 &  $49.7413$               &  Scattering                                                             &   $19.0758$      at $t \approx 6$                   & $30.6655$ at $t \approx 6 $        \\
  \hline
    $r_{\star}=5$              &   $55.9590    $                  & Scattering                                                           &   $ 23.3673$     at $t \approx 5$                   &   $32.5917 $ at $t \approx 5$       \\
  \hline
   $r_{\star}=10$              &   $ 87.0473$                  & Scattering                                                           &   $ 36.0356$     at $t \approx 5$                   &   $ 51.0117$ at $t \approx 5$       \\
  \hline
\end{tabular}
  \caption{Influence of the obstacle size $r_{\star}$ onto the behavior of the solution $u(t)$ to the $2d$ cubic \NNls equation with $u_0$ from \eqref{special-u0} and \eqref{para-spec-A0-1.5}, with the (initial) discrete total mass, after the interaction the reflected and transmitted discrete mass parts at time $t$.}
  \label{T:5}
\end{table}
}

\newpage
\section{Conclusion}\label{S:Conclusions}    
In this work we initiated a numerical study of how the behavior of solutions can change in a presence of a smooth convex obstacle. We observe that the interaction between a solitary wave and the obstacle can significantly influence the overall behavior of the solution to the \NNls equation, which depends on the direction of the velocity vector $\vec{v}=(v_x,v_y)$, the size of the obstacle (radius $r_{\star}$), the initial distance ($d$ vs. $d^{*}$) to the obstacle and the translation parameters $(x_c,y_c)$. The presence of the obstacle yields strong, weak or no interaction. We observed in Sections \ref{L2critWeak} and \ref{L2supercritWeak} that even a small interaction between the obstacle and a single peak solution has some influence on the dynamics (at the least, on the blow-up time). Moreover, we conclude that the strong interaction has a significant effect on the behavior of solutions depending on the size of the obstacle, for example, instead of approaching a solitary wave solution with a single bump, the shape of the solution drastically changes after the collision, splitting it into several bumps with a backward reflected wave. The appearance of the reflection waves due to the presence of the obstacle with Dirichlet boundary conditions prevents the solution from blowing up in finite time in some cases. Furthermore, this backward reflection has always a dispersive character, this might be the reason why the solution scatters in most of the cases after a strong interaction. However, if the obstacle is very small or if the contour of the solution is significantly bigger than the radius of the obstacle, 
we observed the existence of blow-up solutions.  
In this case, the interaction surface is negligible
so that the mass of the solution is almost all transmitted, which is sufficient 
to develop a blow-up.
If the obstacle is large enough or if  there is no transmission of the solution, then, either the solution is completely reflected back or the solution concentrates in its blow-up core at the obstacle's boundary, since the interaction region is relevant and it is larger than the contour of the solution. Furthermore, we construct new (\textit{Wall-type}) initial condition, time evolution of which is characterized by its high mass transmission after a strong interaction, and then blowing up in finite time in a single point or in two separate locations (single points) after the strong interaction. 
For a weak interaction, i.e., when the solution preserves the shape as a traveling solitary wave, the solution behaves either as a solitary wave solution, constructed in \cite{Ou19} (which exists for all positive times), or as the one shown in \cite{Ouss20} (see also \cite{OL20}), a finite time (single point) blow-up solution (as if there would be no obstacle). 

\vspace{0.5cm}

\bibliographystyle{acm}
\bibliography{references}

\end{document}